\documentclass[a4paper,11pt,fleqn]{article}
\usepackage[latin1]{inputenc}  \usepackage[T1]{fontenc}
\usepackage{lscape}
\usepackage{tabularx}
\usepackage{lmodern}       
\usepackage{pstricks}    
\usepackage{rotating}
\usepackage{imakeidx}
\usepackage{amsfonts,amsmath,amsthm,amssymb}
\usepackage{marginnote}
\usepackage{rotate}
\usepackage[arrow,matrix,curve]{xy}
\title{Twisted Teichmüller Curves}
\author{Christian Weiß\thanks{The author is partially supported by the ERC-StG 257137.}}
\newtheorem{thm}{Theorem}[section]
\newtheorem*{thm3}{Theorem~\ref{thm_Hilbert_modular_moduli}}
\newtheorem{defi}[thm]{Definition}
\newtheorem{rem}[thm]{Remark}
\newtheorem{lem}[thm]{Lemma}
\newtheorem{prop}[thm]{Proposition}
\newtheorem{cor}[thm]{Corollary}

\newtheorem{exa}[thm]{Examples}

\newcommand{\SL}{{\rm{SL}}}
\newcommand{\Sp}{{\rm{Sp}}}
\newcommand{\SO}{{\rm{SO}}}
\newcommand{\Id}{{\rm{Id}}}
\newcommand{\PSL}{{\rm{PSL}}}
\newcommand{\GL}{{\rm{GL}}}
\newcommand{\PGL}{{\rm{PGL}}}
\newcommand{\Stab}{{\rm{Stab}}}
\newcommand{\St}{{\rm{St}}}
\newcommand{\Comm}{{\rm{Comm}}}
\newcommand{\Mat}{{\rm{Mat}}}
\newcommand{\Mod}{{\rm{Mod}}}
\newcommand{\Prym}{{\rm{Prym}}}
\newcommand{\Diff}{{\rm{Diff}}}
\newcommand{\Jac}{{\rm{Jac}}}
\newcommand{\Aut}{{\rm{Aut}}}
\newcommand{\inter}{{\rm{int}}}
\newcommand{\Imag}{{\rm{Im}}}
\newcommand{\Real}{{\rm{Re}}}
\newcommand{\End}{{\rm{End}}}
\newcommand{\tr}{{\rm{tr}}}

\newcommand{\diag}{{\rm{diag}}}

\newcommand{\N}{{\mathcal{N}}}
\newcommand{\Aff}{{\rm{Aff}}}
\newcommand{\HH}{{\mathbb{H}}}
\newcommand{\CC}{{\mathbb{C}}}
\newcommand{\RR}{{\mathbb{R}}}
\newcommand{\ZZ}{{\mathbb{Z}}}
\newcommand{\NN}{{\mathbb{N}}}
\newcommand{\QQ}{{\mathbb{Q}}}

\newcommand{\Mg}{{\mathcal{M}_g}}
\newcommand{\OD}{\mathcal{O}_D}

\makeindex

\begin{document}




\newpage

\maketitle

\begin{abstract}
  Let $X_D$ denote the Hilbert modular surface $\HH \times \HH^- / \SL_2(\OD)$. In \cite{HZ76}, F. Hirzebruch     and D. Zagier introduced Hirzebruch-Zagier cycles, that could also be called twisted            diagonals. These are maps $\HH \to \HH  \times \HH^-$ given by $z \mapsto       (Mz,-M^\sigma z)$ where $M \in \GL_2^+(K)$ and $\sigma$ denotes the Galois conjugate. The       projection of a twisted diagonal to $X_D$ yields a Kobayashi curve, i.e. an algebraic curve     which is a geodesic for the Kobayashi metric on $X_D$. Properties of Hirzebruch-Zagier cycles   have been abundantly studied in the literature.\\Teichmüller curves are algebraic curves in the moduli space of Riemann surfaces $\mathcal{M}_g$, which are geodesic for the Kobayashi metric. Some Teichmüller curves in $\mathcal{M}_2$, namely the primitive ones, can also be regarded as Kobayashi curves on $X_D$. This implies that in the universal cover the curve is of the form $z \mapsto (z,\varphi(z))$ for some holomorphic map          $\varphi$. A possibility to construct even more Kobayashi curves on $X_D$ is to consider the    projection of $(Mz,M^\sigma\varphi(z))$ to $X_D$ where again $M \in \GL_2^+(K)$. These new      objects are called twisted Teichmüller curves  because their construction reminds very much     of twisted diagonals. In these notes we analyze  twisted Teichmüller curves in detail and       describe some of their main properties. In particular, we calculate their volume and            partially classify components.
\end{abstract}

\tableofcontents

\newpage




\section{Introduction} \label{chapter_introduction}

Let $X_D$ denote the Hilbert modular surface $\HH \times \HH^- /\SL_2(\mathcal{O}_D)$, where $\HH$ is the complex upper half plane, $\HH^-$ is the complex lower half plane, $\mathcal{O}_D$ is a real quadratic order of discriminant $D$ and $\SL_2(\OD)$ acts on $\HH \times \HH^-$ by the M\"obius transformations given by the two embeddings $\SL_2(\OD) \to \SL_2(\RR)$. The  simplest algebraic curve in $X_D$ is the diagonal, i.e. the image of the composition of the map $z \mapsto (z,-z)$ with the quotient map $\pi: \HH \times \HH^- \to X_D$. In their paper $\cite{HZ76}$, F. Hirzebruch and D. Zagier implicitly introduced \textbf{twisted diagonals}, i.e. the algebraic curves $z \mapsto (Mz,-M^\sigma z)$ for any matrix $M \in \GL_2^+(K)$ where $M^\sigma$ is the Galois conjugate of $M$. Their $\pi$-images are known as modular curves or Hirzebruch-Zagier cycles. These curves have been extensively treated in the literature due to their importance for the geometry and arithmetic of Hilbert modular surfaces (see e.g \cite{vdG88}). In other words, a twisted diagonal is in $\HH \times \HH^-$ given by an equation 
$$\begin{pmatrix} z_2 & 1 \end{pmatrix} \begin{pmatrix} a \sqrt{D} & \lambda  \\ -\lambda^\sigma & b \sqrt{D} \end{pmatrix} \begin{pmatrix} z_1 \\ 1 \end{pmatrix} = 0 $$ 
with $a,b \in \QQ$ and $\lambda \in K$. If $a,b \in \ZZ$ and $\lambda \in \mathcal{O}_D$ such a matrix is called \textbf{integral skew-hermitian}. The (finite) number of integral skew-hermitian matrices which have the same determinant but yield different curves in $X_D$ was calculated by H.-G. Franke and W. Hausmann in \cite{Fra78} and \cite{Hau80}.\footnote{This theorem is one of the main reasons why it makes sense to define the well-known curves $F_N$ and $T_N$, which are in many cases unions of certain twisted diagonals (see e.g. \cite{HV74} or \cite{McM07}). Moreover H.-G. Franke and W. Hausmann gave explicit formulas for the volume of the curves $F_N$ and $T_N$ (\cite{Hau80}, Satz~3.10 and Korollar~3.11). See Section~\ref{sec_special_algebraic_curves}.}  \index{twisted diagonal} It is well-known how to calculate the volume of twisted diagonals. People have also been interested in calculating the intersection numbers of these curves (see \cite{HZ76}). \\[11pt] In the preceding cases, both components of the universal covering are given by Möbius transformations. A curve $C \to X_D$ is still \textit{rather special} if only (at least) one of the components of the universal covering map $\HH \to \HH \times \HH^-$ is a Möbius transformation. It is folklore that it is equivalent to ask that $C \to X_D$ is totally geodesic for the Kobayashi metric (see e.g. \cite{MV10}) and that is the reason why these curves are called \textbf{Kobayashi geodesics}.\footnote{The Kobayashi metric $d$ is the largest pseudo-metric on $X_D$ such that for all holomorphic maps  $f: \mathbb{D} \to X_D$ we have: $d(f(x),f(y)) \leq \rho(x,y)$ where $\rho$ is the Poincaré metric on the unit disk $\mathbb{D}$.} An algebraic curve that is a Kobayashi geodesic is called a \textbf{Kobayashi curve}. Very few examples of Kobayashi curves other than twisted diagonals are known so far.

\paragraph{Teichmüller curves in genus~$2$.} Some of these rare examples are certain Teichmüller curves. Let $\mathcal{M}_g$ denote the moduli space of Riemann surfaces of genus $g$. A \textbf{Teichmüller curve} $C$ is an algebraic curve in $\mathcal{M}_g$ that is totally geodesic with respect to the Teichmüller (or equivalently Kobayashi) metric.  At first sight Teichmüller curves might not seem to have anything to do with Hilbert modular surfaces. By the work in \cite{McM03} (and more generally \cite{Möl06}) both worlds are however closely connected and Teichmüller curves yield indeed Kobayashi curves on $X_D$.  Let us describe Teichmüller curves in a different and more intuitive way. There is a natural bundle $\Omega\Mg$ over $\Mg$: an element in $\Omega\Mg$ is specified by a pair $(X,\omega)$ where $X \in \Mg$ and where $\omega \in \Omega(X)$ is a nonzero, holomorphic 1-form on $X$. The group $\SL_2(\RR)$ acts on  $\Omega\Mg$ in the following way: a matrix $A \in \SL_2(\RR)$ maps the differential form $\omega$ to the differential form $\eta$ where
$$\eta = \begin{pmatrix} 1 & i \end{pmatrix} \begin{pmatrix} a & b \\ c & d \end{pmatrix} \begin{pmatrix} \textrm{Re}\omega \\ \textrm{Im}\omega \end{pmatrix}.$$ 
There is a unique complex structure on $X$ as topological surface such that $\eta$ is holomorphic. We denote the corresponding Riemann surface by $Y$ and define $A(X,\omega):= (Y,\eta)$. It is well-known (see e.g. \cite{Möl11a}) that all Teichmüller curves arise as the projection of an $\SL_2(\RR)$-orbit of a pair $(X,\omega)$ to $\mathcal{M}_g$.\footnote{This is not completely correct: A Teichmüller curve might also be the projection of a $\SL_2(\RR)$-orbit of a half-translation surface $(X,q)$ where $q$ is a quadratic differential. We may however restrict to the case $(X,\omega)$ by the so-called double covering construction.}  The action of $\SO_2(\RR) \subset \SL_2(\RR)$ only rotates the differential form, but does not change the Riemann surface structure. Hence we get the following commutative diagram:
$$
\begin{xy}
 \xymatrix{ \SL_2(\RR) \ar[r] \ar[d] & \Omega\Mg \ar[d]_{\pi} \\ \SO_2(\RR) \backslash \SL_2(\RR) \approx \HH \ar[r]^{\ \ \ \ \ \ \ \ f} & \Mg\\
 	}
\end{xy}
$$
On the other hand, the projection of the $\SL_2(\RR)$-orbit of a pair $(X,\omega)$ to $\Mg$ yields a Teichmüller curve if and only if the stabilizer of the function $f$ under the $\SL_2(\RR)$-action is a lattice (as Fuchsian group). In this case we say that $(X,\omega)$ \textbf{generates} the Teichmüller curve. The stabilizer of the action is called the \textbf{Veech group} of $(X,\omega)$. \\[11pt]
C. McMullen and K. Calta independently constructed a series of Teichmüller curves in $\mathcal{M}_2$ generated by certain L-shaped polygons $P(a,b)$ with opposite sides glued (see Figure~1.1). The holomorphic $1$-form $dz$ on $\CC$ is translation-invariant and hence descends to a natural $1$-form $\omega$ on the glued surface. They proved in \cite{McM03} and \cite{Cal04} that $P(a,a)$ generates a Teichmüller curve in $\mathcal{M}_2$ if $a = (1 + \sqrt{D})/2$ and $D \equiv 1 \mod 4$ and that $P(a,1+a)$ generates a Teichmüller curve if $a = \sqrt{D}/2$ and $D \equiv 0 \mod 4$ (Theorem~\ref{thm_L_shaped}). The number $D$ is called the \textbf{discriminant} of the Teichmüller curve and the mentioned generating surface is denoted by $L_D^1$. The Veech group of such a $L_D^1$ is contained in $\SL_2(\OD)$ (Proposition~\ref{prop_fix_veech_1}, Proposition~\ref{prop_fix_veech_3}) and will be denoted by $\SL(L_D^1)$. 
\begin{center}
\psset{xunit=0.8cm,yunit=0.8cm,runit=1cm}
\begin{pspicture}(-0.5,-0.5)(4.5,3.5) 


\psline[linewidth=0.5pt,showpoints=false]{-}(0,0)(4,0)
\psline[linewidth=0.5pt,showpoints=false]{-}(4,0)(4,3)
\psline[linewidth=0.5pt,showpoints=false]{-}(4,3)(3,3)
\psline[linewidth=0.5pt,showpoints=false]{-}(3,3)(3,1)
\psline[linewidth=0.5pt,showpoints=false]{-}(3,1)(0,1)
\psline[linewidth=0.5pt,showpoints=false]{-}(0,1)(0,0)

\uput[l](0,0.5){1}
\uput[d](2,0){b}
\uput[r](4.0,1.5){a}
\uput[u](3.5,3){1}
\end{pspicture}
\\ Figure~1.1 An L-shaped polygon of the form $P(a,b)$.
\end{center}

\newtheorem*{thme1}{Theorem~\ref{thm_L_shaped}}

We denote these Teichmüller curves by $C_{L,D}^1$ if we want to stress their origin. These are almost all primitive Teichmüller curves in $\mathcal{M}_2$:\footnote{A Teichmüller curve in $\mathcal{M}_g$ is called primitive if it does not arise by a covering construction from a Teichmüller curve in lower genus.} by a theorem of C. McMullen if $D>9$ there is a second primitive Teichmüller curve of discriminant $D$ if and only if $D \equiv 1 \mod 8$ (Theorem~\ref{thm_classification_Teichmüller_genus_2}).\footnote{If $D=5$ then there is also a second primitive Teichmüller curve given by the regular decagon (see \cite{McM06b}).	} The corresponding Veech group is also contained in $\SL_2(\OD)$ (Proposition~\ref{prop_fix_veech_2}) and will be denoted by $\SL(L_D^0)$. The Teichmüller curves are denote by $C_{L,D}^0$ The two Teichmüller curves of the same discriminant are distinguished by their \textbf{spin invariant} (even or odd). Whenever we do not care which of the two Teichmüller curves we actually consider, we just write $\SL(L_D)$ for the Veech group and $C_{L,D}$ for the Teichmüller curve. It is known that such a Teichmüller curve with discriminant $D$ lies on the Hilbert modular surface $X_D$. The reason is shortly speaking the following: by mapping each point $X$ of the Teichmüller curve to its Jacobian $\Jac(X)$ one gets an embedding of the Teichmüller curve into the space of principally polarized Abelian surfaces $\mathcal{A}_2$. C. McMullen proved that all these Jacobians have real multiplication by $\mathcal{O}_D$.\footnote{As the Hilbert modular surface $X_D$ parametrizes all principally polarized Abelian surfaces with real multiplication by $\mathcal{O}_D$ (Theorem~\ref{thm_Hilbert_modular_moduli}) one gets indeed an embedding of the Teichmüller curve into $X_D$.}

\newtheorem*{thme2}{Theorem~\ref{thm_mcm_embedding}}

\begin{thme2} \textbf{(McMullen, \cite{McM03})} Let $f: C_{L,D}^\epsilon \to \mathcal{M}_2$ be one of the Teichmüller curves of discriminant $D$. Then $C_{L,D}^\epsilon$ gives rise to a Kobayashi curve in $X_D$. More precisely we have the following commutative diagram
 $$
\begin{xy}
 \xymatrix{
 	\mathbb{H} \ar@{^(->}[rr]^{\Phi(z)=(z,\varphi(z))} \ar[d]^{/\SL(L_D)}& &	\mathbb{H} \times \mathbb{H}^- \ar[d]^{/\SL_2(\mathcal{O}_D)}	\ar[d] \\ 
 	C_{L,D}^\epsilon \ar@{^(->}[rr] \ar[d]_{f} & & X_D \ar[d]  \\
 	\mathcal{M}_2 \ar@{^(->}[rr]^{\Jac} & & \mathcal{A}_2
 	}
\end{xy}
$$
where $\SL(L_D)=\Stab_{\SL_2(\RR)}(\Phi) \cap \SL_2(\OD)$, the stabilizer of the \textbf{graph of the Teichmüller curve} $(z,\varphi(z))$ in $\SL_2(\OD)$, is the Veech group. Moreover $\varphi$ is holomorphic but not a Möbius transformation. \end{thme2}

\paragraph{Twisted Teichmüller curves.} We have just seen that Teichmüller curves yield some of the very few known examples of Kobayashi curves on $X_D$ that are not twisted diagonals. Using Teichmüller curves another new class of examples of Kobayashi curves can be constructed: these objects remind one very much of twisted diagonals and are therefore called twisted Teichmüller curves. Similar as in \cite{HZ76} we twist Teichmüller curves by a matrix $M \in \GL_2^+(K)$ where $K = \QQ(\sqrt{D})$. These curves are the main objects of these notes. Consider the following diagram
$$
\begin{xy}
 \xymatrix{ \mathbb{H} \ar[rr]^{(Mz,M^\sigma\varphi(z))} \ar[d]^{/\SL_M(L_D)} & & \mathbb{H} \times \mathbb{H}^-              \ar[d]^{/\SL_2(\OD)} \\
             C_M \ar[rr]& & X_D}
\end{xy}
$$
where $\SL_M(L_D)$ is the stabilizer of the graph of the twisted Teichmüller curve $\Phi_M=(Mz,M^\sigma\varphi(z))$ inside $\SL_2(\OD)$. We call such a curve a \textbf{twisted Teichmüller curve}. Note that one may by multiplying all the entries of $M$ with all of the denominators always assume that every entry of $M$ is in $\OD$. It is a first main observation that twisted Teichmüller curves are indeed Kobayashi curves.\footnote{More generally it is, of course, true that all twists of Kobayashi curves are again Kobayashi curves, but we are only concerned about twisted Teichmüller curves in these notes.}

\newtheorem*{thme3}{Proposition~\ref{thm_twisted_finite_volume}}

\begin{thme3} All twisted Teichmüller curves are Kobayashi curves. \end{thme3}

 One might then be interested in the same questions as those that have been answered for twisted diagonals.
\begin{itemize}
\item[(1)] What is the volume of a twisted Teichmüller curve? 
\item[(2)] When do two matrices $M,N \in \GL_2^+(K)$ yield different twisted Teichmüller curves? 
\item[(3)] Are all Kobayashi curves in $X_D$ given by twisted diagonals and twisted Teichmüller curves?
\end{itemize}
These three questions are answered with some restrictions imposed by a simplifying assumption on the class number $h_D$ and some congruence conditions on $D$ in these notes. Moreover one might ask:
\begin{itemize}
\item[(4)] What are the intersection numbers of twisted Teichmüller curves? 
\end{itemize}
This question remains unanswered and is left open as a problem for future research.\\[11pt]When one wants to do explicit calculations for a twisted Teichmüller curve (e.g. of the volume), there arises a major problem: in general, it is very hard to calculate $\SL_M(L_D)$. There are three main reasons for this phenomenon:
\begin{itemize}
\item A theorem of E. Gutkin and C. Judge in \cite{GJ00} implies that the Veech groups $\SL(L_D)$ are all non-arithmetic Fuchsian groups (Theorem~\ref{thm_Gutkin_Judge2}). In particular, this makes it hard to decide whether a matrix in $\SL_2(\OD)$ lies in the Veech group or not. 
\item Furthermore, it is still unknown how to calculate the Veech group for a given flat surface $(X,\omega)$. Although this problem is solved in some special cases (see e.g. \cite{Sch05} and \cite{McM03}), in our case the Veech group can for large $D$ only  be calculated partially.\footnote{There is a recent preprint \cite{Muk12} of R. Mukamel where he claims to give an algorithm to find the Veech group of any Veech surface.}
\item The Taylor expansion of $\varphi$ is known by the theorem of M. Möller and D. Zagier in \cite{MZ11}. Even with this knowledge, it is not easy to decide whether there exist elements of $\SL_2(\RR)$ which are not in the Veech group but lie in $\Stab_{\SL_2(\RR)^2}(\Phi)$, the stabilizer of the graph of the Teichmüller curve, or even if there exist any such element in $\SL_2(K) \smallsetminus \SL_2(\OD)$.
\end{itemize} 
\paragraph{Statement of main results.} If the class number $h_D=1$ then the twist-matrix $M$ can be normalized such that $$M=\begin{pmatrix} m & x \\ 0 & n \end{pmatrix}$$
with $m,n,x \in \OD$ and $(m,n,x)=1$ (Proposition~\ref{prop_matrix_decomposition}). This rests on the fact that the number of cusps of a Hilbert modular surface is equal to the class number of the quadratic order (see e.g. \cite{vdG88}, Proposition~1.1). For $M \in \GL_2^+(K)$ let $$X_D(M):=\HH \times \HH^-/(\SL_2(\OD) \cap M^{-1} \SL_2(\OD) M),$$
 denote a (finite index) cover of $X_D$. If $D \equiv 5 \mod 8$ is fundamental discriminant, then we are able to calculate the Euler characteristic $\chi(C_M)$ of the twisted Teichmüller curve and thus its volume $\mu(C_M)$ since $\mu(C_M) = 2\pi \chi(C_M)$. Our results are cleanest for the case where $D \equiv 5 \mod 8$ and $h_D^+=1$ and therefore we will be mainly concerned about this case now. Nevertheless, some results can be stated (sometimes partially) in a more general setting and so we will later also explain other cases when we describe the results in more detail.

\newtheorem*{thme8}{Theorem~\ref{thm_summarize_euler_II}}

\begin{thme8} Let $D \equiv 5 \mod 8$ be a fundamental discriminant with $h_D^+=1$ and let $m,n,x \in \OD$ be arbitrary elements with $(m,n,x)=1$. Then 
$$ \chi(C_M) = \deg(X_D(M) \to X_D) \cdot \chi(C_{L,D}).$$
\end{thme8}

This result also implies that the Teichmüller curve has minimal volume among all twisted Teichmüller curves.\\[11pt] Surprisingly enough, the surrounding arithmetic of $\SL_2(\OD)$ determines the arithmetic of the twisted Teichmüller curves to a very high degree. It is well-known how to calculate the degree of the covering $X_D(M) \to X_D$ (compare Proposition~\ref{prop_index_congruence}). For instance, if $M= \left( \begin{smallmatrix} \pi & 0\\ 0 & 1 \end{smallmatrix} \right)$ with some prime element $\pi \in \OD$ then the degree of the covering is equal to $|\N(\pi)|+1$. Another interpretation of Theorem~\ref{thm_summarize_euler_II} is that the Veech groups of Teichmüller curves are the opposite of being arithmetic, since it implies in particular that all Hecke congruence subgroups of $\SL(L_D)$ have maximal possible index in $\SL(L_D)$. \\[11pt]
Let us make a short comment on the proof of Theorem~\ref{thm_summarize_euler_II}. The Euler characteristic is calculated in two steps: at first we show that the Euler characteristic of $C_M$ does not change by lifting $C_M$ to a curve $C_M(M)$ on $X_D(M)$. This curve $C_M(M)$ has the same volume as the curve $C^M(M)$ lying over $C_{L,D}$. As second step of the proof we then show that the covering $C^M(M) \to C_{L,D}$ has the maximal possible degree. The main idea of this part of the proof is to calculate indexes of Hecke congruence subgroups in $\SL(L_D)$. The coset representatives are given by words in known parabolic elements of the Veech group. A similar idea is used by G. Weitze-Schmithüsen in \cite{Wei12} to calculate the congruence level of the Veech groups of L-shaped square-tiled surfaces. She shows that these Veech groups are "as far as possible" from being a congruence subgroups.\\[11pt] The knowledge about the volume of twisted Teichmüller curves enables us to classify twisted Teichmüller curves at least partially under some additional assumptions. 
\newtheorem*{thme9}{Theorem~\ref{thm_classificiation1}}
\begin{thme9} Suppose $D \equiv 5 \mod 8$ is a fundamental discriminant with narrow class number $h_D^+=1$. If $M= \left( \begin{smallmatrix} m & x \\ 0 & n \end{smallmatrix} \right)$ and $N = \left( \begin{smallmatrix} a & b \\ 0 & c \end{smallmatrix} \right)$ with $a,b,c,m,n,x \in \OD$ and $(a,b,c)=1$ and $(m,n,x)=1$ define the same twisted Teichmüller curve then $$\det(M) = \det(N).$$
\end{thme9}
Indeed, equality of the determinant is also a sufficient criterion for the twisted Teichmüller curves to coincide whenever the determinant is prime in $\OD$. 
\newtheorem*{thmeprime}{Theorem~\ref{thm_prime_classification}}
\begin{thmeprime} Suppose $D \equiv 5 \mod 8$ is a fundamental discriminant with narrow class number $h_D^+=1$ and let $\pi \in \OD$ be a prime element. Then there is exactly one twisted Teichmüller curve of determinant $\pi$, i.e. all matrices in $\GL_2^+(K) \cap \Mat^{2x2}(\OD)$ of determinant $\pi$ with relative prime entries define the same twisted Teichmüller curve. \end{thmeprime}
For arbitrary determinants we are unfortunately not able to prove a corresponding result in full generality. Nevertheless, we have strong numerical evidence that the following conjecture holds. It is based on computer experiments for many different determinants $n \in \OD$ including all types of splitting behavior of the prime divisors of $n$.

\newtheorem*{conje}{Conjecture}

\begin{conje} Suppose $D \equiv 5 \mod 8$ is a fundamental discriminant with narrow class number $h_D^+=1$. All matrices $M \in \GL_2^+(K) \cap \Mat^{2x2}(\OD)$ of determinant $n \in \OD$ with relative prime entries define the same twisted Teichmüller curve, i.e. there is exactly one twisted Teichmüller curve of determinant $n$.\end{conje}

\paragraph{Comparison to twisted diagonals.} Since twisted diagonals are our main motivation to introduce twisted Teichmüller curves, we now want to shortly compare these two classes of objects. In some aspects they behave similar while they differ in others.\\[11pt]
To a certain extent an analogue of the classification theorem by H.-G. Franke and W. Hausmann for twisted diagonals is also true for twisted Teichmüller curves: after normalizing the involved matrices appropriately, there are only finitely many different twisted Teichmüller curves of a given determinant.\\[11pt]  
Moreover, Theorem~\ref{thm_prime_classification} gives that for arbitrary prime elements $\pi \in \OD$ the number of different twisted Teichmüller curves of determinant $\pi$ is always $1$ and thus agrees with the number of different twisted diagonals of determinant $\pi$ (see \cite{Fra78}, Theorem~2.3.2 or \cite{vdG88}, Chapter~V.3.).\\[11pt] 
On the other hand, Theorem~\ref{thm_summarize_euler_II} implies that the volume of twisted Teichmüller curves behaves very differently than the volume of twisted diagonals. For example, for $M=\left( \begin{smallmatrix} \pi & 0 \\ 0 & 1 \end{smallmatrix} \right)$, where $\pi$ is an inert prime number or a divisor of a ramified prime number, the stabilizer of the twisted diagonal is always $\SL_2(\ZZ)$ and hence unlike the corresponding twisted Teichmüller curve does not depend on $\pi$.\\[11pt]
Furthermore, Theorem~\ref{thm_classificiation1} tells us that if two twisted Teichmüller curves agree, then the determinants of the involved matrices have to agree. As we have seen in the preceding example this is not true for twisted diagonals. However, a corresponding result does hold after passing to skew-hermitian matrices.

\paragraph{More Kobayashi curves.} One might wonder if twisted diagonals and twisted Teichmüller curves yield all Kobayashi curves on $X_D$. This is not the case. For each $D$, C. McMullen constructed two different Kobayashi curves on $X_D$ stemming from certain Teichmüller curves in $\mathcal{M}_3$ and $\mathcal{M}_4$ using Prym varieties (see \cite{McM06a}). The generating surfaces are certain $S$-shaped respectively $X$-shaped polygons with opposite sides glued (Theorem~\ref{thm_polygons_determine_teich}). We denote these curves on $X_D$ therefore by $C_{S,D}$ and $C_{X,D}$.\footnote{To be more precise: While $C_{X,D}$ lies on $X_D$ for all $D$, this is not always the case for $C_{S,D}$. The curves $C_{S,D}$ always do allow real multiplication by $\OD$ but are not principally polarized. Therefore, for some $D$ the curve lies $C_{S,D}$ on a different Hilbert modular surface (see Chapter~\ref{chapter_prym_modular}).} 
\paragraph{The Kontsevich-Zorich cocycle.} There is a very interesting cocycle over the Teichmüller flow on a Teichmüller curve, namely the \textbf{Kontsevich-Zorich cocycle} (see \cite{Zor06}). Oseledet's Multiplicative Theorem (Theorem~\ref{thm_Oseledet}) ensures the existence of some invariants of the Kontsevich-Zorich cocycle, its so-called \textbf{Lyapunov exponents}. To each of the Teichmüller curves which appear in these notes one can uniquely associate a pair of Lyapunov-exponents. The greater of these two Lyapunov exponents is always equal to one. The second Lyapunov exponent of each Teichmüller curve $C_{L,D}^\epsilon$ in $\mathcal{M}_2$ has been calculated by M. Bainbridge in \cite{Bai07} to be equal to $\frac{1}{3}$ (Theorem \ref{thm_lyapunov_M2}). Recent results of D. Chen and M. Möller in \cite{CM11} and of A. Eskin, M. Kontsevich and A. Zorich in \cite{EKZ11} enable us to also calculate the second Lyapunov exponent of the curves $C_{S,D}$ and $C_{X,D}$. It is equal to $\frac{1}{5}$ respectively $\frac{1}{7}$ (Corollary~\ref{cor_lyap1/5} and Corollary~\ref{cor_lyap1/7}). On the other hand, it can be  shown that:
\newtheorem*{core2}{Corollary~\ref{cor_twist_lyapunov}}
\begin{core2} Twists do not change the Lyapunov exponents. \end{core2}
The following statement are then immediate consequences.

\newtheorem*{core3}{Corollary~\ref{cor_lyap1/5}}
\newtheorem*{thme10}{Theorem~\ref{thm_notwist_general}}

\begin{thme10} For all discriminants $D$ the genus~$4$ Teichmüller curve $C_{X,D}$ is neither a twist curve of a $C_{L,D}^\epsilon$ nor a twisted diagonal. \end{thme10}

\begin{core3} For all discriminants $D$ the genus~$3$ Teichmüller curve $C_{S,D}$ is neither a twist curve of a $C_{L,D}^\epsilon$ nor a twisted diagonal. \end{core3}
This result gives a hint that we should not expect to have found all Kobayashi curves on $X_D$ so far. However, this is, of course, still an open question. We now describe more detailed how the results are achieved.
\paragraph{The stabilizer of the graph.} When we want to calculate the Euler characteristic of the twisted Teichmüller curve we have to gain more knowledge on the stabilizer $\SL_M(L_D)$. It is immediately clear from the definition of the Veech group that $M \SL(L_D) M^{-1} \cap \SL_2(\OD) \subset \SL_M(L_D)$ holds for all $M \in \GL_2^+(K)$. For most $M$ we can show that even equality holds because the stabilizer of the graph of a Teichmüller curve in $\SL_2(K)$ cannot be (too much) greater than $\SL(L_D)$. In order to state the results precisely we introduce the technical term of \textbf{pseudo parabolic maximal} Fuchsian groups. The group $\SL(L_D)$ is called pseudo parabolic maximal if there does not exist a Fuchsian group $\Gamma$ containing $\SL(L_D)$ of finite index and also containing a parabolic element which is in $\SL_2(K) \smallsetminus \SL_2(\OD)$. Pseudo parabolic maximal is a good property to consider because of the following theorem.	

\newtheorem*{thme4}{Theorem~\ref{thm_parabolic_maximal_implies_stabilizer}}

\begin{thme4} For all fundamental discriminants $D \equiv 1 \mod 4$ where $\SL(L_D^1)$ is pseudo parabolic maximal $\Stab_{\SL_2(\RR)}(\Phi) \cap \SL_2(K) = \SL(L_D^1)$ holds. This implies that $\SL_M(L_D^1) = M \SL(L_D^1) M^{-1} \cap \SL_2(\OD)$. \end{thme4}
The proof of the theorem heavily relies on Margulis' commensurator theorem (Theorem~\ref{thm_Margulis}) and the arithmetic of $\OD$. Indeed, many groups $\SL(L_D)$ are pseudo parabolic maximal.



\newtheorem*{thme6}{Theorem~\ref{thm_pseudo_parabolic}}

\begin{thme6} For all $D \equiv 5 \mod 8$ the group $\SL(L_D)$ is pseudo parabolic maximal.\end{thme6}

For $D=5$ this theorem goes back to a result of A. Leutbecher in \cite{Leu67}.

\begin{conje} All the groups $\SL(L_D)$ are (pseudo) parabolic maximal. \end{conje}

Due to a lack of information on the Veech group we are however not able to prove this conjecture in the cases left open in Theorem~\ref{thm_pseudo_parabolic}. In general, the results become therefore a bit more complicated to state. In this introduction we will stick to the language of class number $1$ although we will work in the general setting later. We call a Fuchsian group $\Gamma \subset \SL_2(\OD)$ \textbf{$n$-pseudo parabolic maximal} (or shortly	 $n$-ppm) for $n \in \OD$ if there does not exist a Fuchsian group $\Gamma'$ containing $\Gamma$ with finite index and containing also a parabolic element in $\SL_2(K) \smallsetminus (\SL_2(K) \cap \Mat^{2x2}(\frac{1}{n}\OD))$. In other words, there might only appear divisors of $n$ as denominators of the entries of elements in $\Gamma'$. If we set $w:=\frac{1+\sqrt{D}}{2} \in \OD$ if $D \equiv 1 \mod 4$ and set $w:=\frac{\sqrt{D}}{2}$ if $D \equiv 0 \mod 4$ and let $\pi_2$ be the (unique) common prime divisor of $2$ and $w$ if $D \equiv 1 \mod 8$ and let $\widetilde{\pi_2}$ the unique prime divisor of $2$ for $D \equiv 0 \mod 4$then the following table summarizes the situation if $h_D=1$ 

\begin{center} \begin{tabular}{|c|c|c|}
\hline
$D$ & $\SL(L_D^1)$ & $\SL(L_D^0)$ \\
\hline
$5 \mod 8$ & $1$-ppm  &  --- \\
\hline
$1 \mod 8$ & $\pi_2$-ppm  & $\pi_2^\sigma$-ppm \\
\hline
$0 \mod 4$ & $\widetilde{\pi_2}$-ppm & ---\\
\hline
\end{tabular}
\end{center}
This enables us to prove a weaker version of Theorem~\ref{thm_parabolic_maximal_implies_stabilizer} in all the cases where $\SL(L_D)$ is only $n$-pseudo-parabolic maximal.
\newtheorem*{thme11}{Theorem~\ref{thm_cov_degree_10}}

\begin{thme11} Let $M \in \GL_2^+(K) \cap \Mat^{2x2}(\OD)$ and $D$ be a fundamental discriminant and $C_{L,D}^\epsilon$ be a Teichmüller curve of discriminant $D$. Suppose that
\begin{itemize}
\item[(i)] $D \equiv 1 \mod 8$, $C_{L,D}^\epsilon$ has odd spin and $\det(M)$ is not divisible by $\pi_2$ or
\item[(ii)] $D \equiv 1 \mod 8$, $C_{L,D}^\epsilon$ has even spin and $\det(M)$ is not divisible by $\pi_2^\sigma$ or
\item[(iii)] $D \equiv 9 \mod 16$ and $C_{L,D}^\epsilon$ has odd spin or 
\item[(iv)] $D \equiv 5 \mod 8$ or
\item[(v)] $D \equiv 0 \mod 4$ and $\det(M)$ is not divisible by $\widetilde{\pi_2}$
\end{itemize}
then 
\begin{itemize}
\item[(i)] the degree of the covering $\pi: C_M(M) \to C_M$ is equal to $1$ and
\item[(ii)] we have $\SL_M(L_D)=M\SL(L_D)M^{-1}\cap \SL_2(\OD)$. 
\end{itemize}
\end{thme11} 

\paragraph{Euler characteristic.} Using M. Bainbridge's formula in \cite{Bai07} for the volume of Teichmüller curves in genus $2$ (Theorem~\ref{thm_bain_euler}), we are, in the situation of Theorem~\ref{thm_cov_degree_10}, able to give explicit formulas for the Euler characteristic of twisted Teichmüller curves $\chi(C_M)$. We again look at normalized matrices $$M=\begin{pmatrix} m & x \\ 0 & n \end{pmatrix}$$
with $m,n,x \in \OD$ and $(m,n,x)=1$. If $D \equiv 1 \mod 8$ and the spin of the Teichmüller curve is odd let $\widetilde{\eta}$ be the number $\pi_2$ and if $D \equiv 1 \mod 8$ and the spin of the Teichmüller curve is even let $\widetilde{\eta}$ be $(w+1)(w-1)$ and if $D \equiv 5 \mod 8$ let $\widetilde{\eta}$ be $1$ and if $D \equiv 0 \mod 4$ let $\widetilde{\eta}$ be $w(w+1)$.
\newtheorem*{thme7}{Theorem~\ref{thm_summarize_euler_calculations_triangular}}

\begin{thme7} Let $D$ be a fundamental discriminant and let $m,n,x \in \OD$ be arbitrary elements with $(m,n,x)=1$. If $h_D=1$ and $(n,\widetilde{\eta})=1$ and $(m,\widetilde{\eta})=1$, then $$\chi(C_M) = \deg(X_D(M) \to X_D) \cdot \chi(C_{L,D}^\epsilon).$$
\end{thme7}

We stress the fact that for $D \equiv 5 \mod 8$ the result holds with $\widetilde{\eta}=1$ (Theorem~\ref{thm_summarize_euler_II}). We cannot expect such a result to hold without any (congruence) condition on $D$: for instance for $D=17$ and the Teichmüller curve of odd spin there exists a matrix $M_0 \in \GL_2^+(K)$ such that the degree of $\pi: X_D(M_0) \to X_D$ is greater than one, but $\chi(C_{M_0}) = \chi(C_{L,D}^1)$ (see Section~\ref{subsec_non_prime}).
\paragraph{Classification.} We use the knowledge about the volume to prove Theorem~\ref{thm_classificiation1}. Besides the lack of knowledge about the Veech group in the case $D \not \equiv 5 \mod 8$, there arise many more problems in the general situation which prevent our methods from working. In particular, if the class number $h_D$ is greater than one then we may not restrict to upper triangular matrices since the number of cusps of the Hilbert modular surface is equal to $h_D$.
\paragraph{Structure of the notes.} These notes are organized as follows:\\[11pt]In Chapter~\ref{cha_background} we give an overview of the basic concepts which are used everywhere else in these notes. We recall well-known results about real quadratic number fields, Fuchsian groups, moduli spaces and Hilbert modular surfaces. This chapter will be helpful in particular for people who are not already familiar with the mentioned topics.\\[11pt]
In Chapter~\ref{cha_Teichmüller_curves} we introduce Teichmüller curves. In particular, we recall C. McMullen's construction of Teichmüller curves in $\mathcal{M}_2$, which was already mentioned in the introduction.\\[11pt]
In Chapter~\ref{cha_twisted_Teichmüller_curves} the main new objects of these notes, namely twisted Teichmüller curves, are defined. Only some main properties of twisted Teichmüller curves are derived here. Most importantly, it is shown that twisted Teichmüller curves yield indeed Kobayashi curves.  \\[11pt]
In Chapter~\ref{cha_maximality} we describe the relation between the stabilizer of the graph of the Teichmüller curve and the commensurator of the Veech group. Furthermore we introduce the notion of pseudo parabolic maximal Fuchsian groups and show why this property is useful for calculating the stabilizer.\\[11pt]
In Chapter~\ref{chapter_calculations} the volume of twisted Teichmüller curves is calculated for most matrices $M$ if $h_D=1$. From this the classification of twisted Teichmüller curves can be derived. Finally, we present some ideas how quantities like the number of elliptic fixed points, the number of cusps and the genus of twisted Teichmüller curves can be calculated in some special cases.\\[11pt]
In Chapter~\ref{chapter_prym_modular} we recall C. McMullen's construction of Teichmüller curves in $\mathcal{M}_3$ and $\mathcal{M}_4$ using Prym varieties. Also these Teichmüller curves yield Kobayashi curves on $X_D$.\\[11pt]
In Chapter~\ref{chapter_Lyapunov_exponents} we recall Oseledet's Theorem on the existence of Lyapunov exponents and introduce the Kontsevich-Zorich cocycle over Teichmüller curves. Furthermore the connection of the Teichmüller flow to the geodesic flow on $T^1\HH$ is discussed.\\[11pt]
In Chapter~\ref{chapter_comparing} it is proven by using Lyapunov exponents that the Teichmüller curves from Chapter~\ref{chapter_prym_modular} are never twisted Teichmüller curves.
\paragraph{Acknowledgments.} These notes are the expanded version of my PhD-thesis at the Goethe Universität Frankfurt am Main. First and foremost I would like thank my advisor Prof. Dr. Martin Möller for his continuous and motivating support during all stages of this work. Furthermore I am very grateful to my second advisor Prof. Dr. Don Zagier for his idea to consider twisted Teichmüller curves at all and for helpful discussions. Parts of the work was carried out at the Max-Planck-Institute for Mathematics in Bonn, whom I thank for hospitality and an inspiring scientific atmosphere. Many thanks go to Dr. Anke Pohl, who was always willing to answer my mathematical question, to Prof. Dr. Kai-Uwe Bux, who helped me to work out Appendix~\ref{sec_check_commens}, and to my colleagues in Bonn in Frankfurt, in particular to my office mates Soumya Bhattacharya, Dr. Tobias Fritz and Quentin Gendron. I also thank Prof. Dr. Curtis McMullen and JProf. Dr. Gabriela Weitze-Schmithüsen for useful comments on an earlier version of these notes. Moreover, I am grateful to the referees for their appropriate and constructive suggestions and for proposed corrections to improve these notes. Last but not least I would like to thank my friends, my family and in particular my parents for always having a sympathetic ear for me.

\newpage


\section{Background} \label{cha_background}

As we try to keep these notes as self-contained as possible this chapter gathers together many different concepts and results which will be used at various stages of these notes. It will also introduce notation and serve as a reference section. There are only very few proofs included in this chapter. Nevertheless, the reader will find references to the literature where one finds more detailed explanations of the presented material. The reader who is familiar with the topic can without compunction skip the corresponding section.

\subsection{Real Quadratic Number Fields} \label{sec_quadratic_nf}

In this section a short overview over real quadratic number fields is given. We will mainly concentrate on results which are important for the further understanding of these notes. We refer the reader who wants to find out more about the details to \cite{Hec23}, \cite{Neu05}, \cite{Sch07} and \cite{Zag81}, where one finds most of the presented material.
\paragraph{Real quadratic number fields and the ring of integers.} A number field \label{glo_K_real_quadratic} $K$ is a finite field extension of $\mathbb{Q}$. An element $z \in K$ is called \textbf{integer} if there exists a polynomial $f(X) = X^n + a_{n-1} X^{n-1} + ... + a_1X +a_0 \in \mathbb{Z}[X]$ such that $f(z)=0$. The set of all integers in $K$ is denoted by $\mathcal{O}$. In fact $\mathcal{O}$ is known to be a Noetherian ring (see \cite{Neu05}, Theorem 3.1) and therefore called the \textbf{ring of integers}  of $K$. We are interested only in a special type of number fields: let $d>1$ be a square-free integer. Then $K=\mathbb{Q}(\sqrt{d})$ is a \textbf{real quadratic number field}\index{real quadratic number field}. The \textbf{discriminant} of $K$ is
$$D= \left\{ \begin{matrix} d & \textit{if} \ d \equiv 1 \mod 4 \\ 4d & \textit{if} \ d \equiv 2,3 \mod 4 \end{matrix} \right. .$$
Note that $\QQ(\sqrt{d}) = \QQ(\sqrt{D})$. A discriminant (of any - not necessarily real - quadratic number field) is called a \textbf{prime discriminant} if it is divisible by only one prime, i.e. $D=-4,-8,+8$ or $D$ is equal to a prime number $p \equiv 1 \mod 4$ or $D$ is equal to $-p$ where $p$ is a prime number $3 \mod 4$. Every discriminant of a real quadratic number field is the product of prime discriminants, i.e. $D=D_1 \cdots D_t$ where all $D_t$ are prime discriminants.\\[11pt] For all $d$ there exist exactly two embeddings $K \hookrightarrow \RR \subset \CC$. We write $z = x + y\sqrt{d} \mapsto z^{\sigma}=x-y\sqrt{d}$ for the \textbf{(Galois-)conjugation}. The \textbf{norm} \index{real quadratic number field!norm} of an element $z \in K$ is defined as $\textrm{N}(z):=zz^{\sigma}$ \label{glo_N}. In most situation we will only use the \textbf{absolute value of the norm} which we denote by $\N(z):=|\textrm{N}(z)|$. The \textbf{trace} of an element $z \in K$ is defined as $\tr(z):=z+z^{\sigma}$ \label{glo_tr}. The ring of integers \index{real quadratic number field!ring of integers} of the real quadratic number field of discriminant $D$ will be denoted by $\OD$. \label{glo_OD} We know that (see \cite{Sch07}, Theorem 6.1.10):
$$\OD= \left\{ \begin{matrix} \mathbb{Z} + \frac{1+\sqrt{D}}{2} \mathbb{Z} & \textit{if} \ D \equiv 1 \mod 4 \\ \mathbb{Z} + \frac{\sqrt{D}}{2} \mathbb{Z} & \textit{if} \ D \equiv 0 \mod 4 \end{matrix} \right. .$$
We denote this basis of $\OD$ by $(1,w)$. \label{glo_w} As we will always only deal with one real quadratic number field at a time, this notation will not cause any confusion. Kronecker's Approximation Theorem (see e.g. \cite{Mac01}, Theorem 4.1) implies that for all $D$ the ring of integers $\OD$ is dense in $\mathbb{R}$. An element $x \in \OD$ is called \textbf{totally positive (negative)}\index{real quadratic number field!totally positive}, if for all embeddings $\alpha_i: \OD \to \mathbb{R}$, $i=1,2$ the inequality $\alpha_i(x)>0$ ($\alpha_i(x)<0$) holds. If $x$ is totally positive (negative), we write $x \succ 0$ ($x \prec 0$). We denote by $\OD^*$ the group of units in $\OD$. \label{glo_OD*} By \textbf{Dirichlet's unit theorem} (see \cite{Neu05}, Theorem 7.4) there is a unique element $\epsilon >1$ such that $\OD^* = \left\{ \pm 1 \right\} \times \left\{\epsilon^n \mid n \in \ZZ \right\}$. The element $\epsilon$ \label{glo_epsilon} is then called the \textbf{fundamental unit} \index{real quadratic number field!fundamental unit} of $K$. If the discriminant of the real quadratic number field is equal to a prime $p \equiv 1 \mod 4$, then the fundamental unit always has negative norm. If the discriminant has a prime factor $p$ that is congruent to $3 \mod 4$ then the fundamental unit has positive norm (see e.g. \cite{Ste93}). We define the \textbf{inverse different} \index{real quadratic number field!inverse different} as the fractional ideal
$$\OD^\vee = \left\{ x \in K \mid \tr(xy) \in \ZZ \ \forall y \in \OD \right\}.$$
More concretely, \label{glo_ODvee} $\OD^\vee = \frac{1}{\sqrt{D}} \OD$. \\[11pt]Let $J_K$ be the group of (fractional) ideals in $K$, i.e. those subsets $I \neq \left\{ 0 \right\}$ of $K$ for which there exists an element $r \in \OD$ such that $rI \subset \OD$ is an ideal in $\OD$, and let $P_K$ denote the group of fractional principal ideals $(a)=a\OD$ with $a \in K^*$. \label{glo_(a)} The \textbf{ideal class group} is defined as the quotient $Cl_K:=J_K/P_K$. The number $h_D$ \label{glo_hK} of elements of $Cl_K$ is known to be finite (\cite{Neu05}, Theorem 6.3) and is called the \textbf{class number} \index{real quadratic number field!class number} of $K$. The \textbf{norm} \index{real quadratic number field!norm} of an (ordinary) ideal $\mathfrak{a}$ is equal to the number of elements of $\OD  / \mathfrak{a}$. The norm of a fractional ideal $I$ such that $rI$ is an ideal is given by the quotient of the norm of $rI$ and the norm of the principal ideal $(r)$. \\[11pt] If the class number of $K$ is greater than $1$, there do exist ideals which cannot be generated by a single element. However, for Noetherian rings the following fact is well known (see \cite{Bru08}, p. 107): if $\mathfrak{a} \subset K$ is a fractional ideal, then there exist $\alpha, \beta \in K$ such that $\mathfrak{a} = \alpha \OD + \beta \OD$.\\[11pt]
One has to carefully distinguish the class number from the narrow class number. \index{real quadratic number field!narrow class number} Two ideals $\mathfrak{a}$ and $\mathfrak{b}$ are equivalent in the narrow sense if there exists a $\lambda \in K$ with $\textrm{N}(\lambda)>0$ and $(\lambda)\mathfrak{a} = \mathfrak{b}$. The equivalence classes form a group of order $h_D^+$. The number $h_D^+$ is called the \textbf{narrow class number} of $\OD$. By definition, two (fractional) ideals $\mathfrak{a}$ and $\mathfrak{b}$ belong to the same \textbf{genus} if there exists a $\lambda \in K$ with $\textrm{N}(\lambda)>0$ and $\textrm{N}(\lambda)\textrm{N}(\mathfrak{a}) = \textrm{N}(\mathfrak{b})$ (compare \cite{Zag81}, p. 111). If the fundamental unit of $\OD$ has negative norm, then the narrow class number equals the class number. If the fundamental unit of $\OD$ has norm $1$ then $h_D^+ = 2 h_D$.  It goes back to C.-F. Gau\ss \ that the narrow class number $h_D^+$ is always divisible by $2^{t-1}$ if $D=D_1 \cdots D_t$ where all $D_t$ are prime discriminants (see e.g. \cite{Zag81}, p. 112). Hence the only possible discriminants (of a real quadratic number field) of narrow class number $1$ are $8$ and primes $p \equiv 1 \mod 4$. On the other hand, it was conjectured by H. Cohen and H. Lenstra in \cite{CL84} that approximately $76 \%$ of all real quadratic number fields of prime discriminant have (narrow) class number $1$. Nevertheless, it is not even known yet if there are infinitely many real quadratic number fields of class number $1$. 
\paragraph{Euclidean number fields.} An integral ideal domain $R$ is said to be \textbf{Euclidean} if there exists a map $\phi: R \smallsetminus \left\{ 0 \right\} \to \NN$ such that given any $a,b \in R$ with $b \neq 0$ there exist $q,r \in R$ such that $a=bq+r$ with either $r=0$ or $\phi(r) < \phi(b)$. $\phi$ is called \textbf{Euclidean norm function}. We would like to know under which conditions $\OD$ is Euclidean. Clearly a necessary condition for $\OD$ to be Euclidean is that $h_D=1$. More precisely we have the following conditional theorem going back to P. Weinberger: 





\begin{thm} (\cite{Sch07}, Satz 6.6.2) For a real quadratic number field $K=\mathbb{Q}(\sqrt{D})$ the following are equivalent:
\begin{itemize}
\item[(i)] $h_D=1$.
\item[(ii)] $\OD$ is a principal ideal domain.
\item[(iii)] $\OD$ is a unique factorization domain.
\end{itemize}
(\textbf{Weinberger}, \cite{Wei73}) Under the additional assumption of the generalized Riemann hypothesis, also the following is equivalent to $(i)-(iii)$:
\begin{itemize}
\item[(iv)] $\OD$ is Euclidean.
\end{itemize}
\end{thm}

A list of all real quadratic number fields with class number 1 in the range $2 \leq D \leq 100$ can be found in \cite{Neu05}, p. 37. It is conjectured that there are infinitely many real quadratic number fields with $h_D=1$. So far all known examples of Euclidean real quadratic number fields have their norm as Euclidean norm function ($D=2$, $3$, $5$, $6$, $7$, $11$, $13$, $17$, $19$, $21$, $29$, $33$, $37$, $41$, $57$, $73$). It was moreover proven that this list is complete if one only admits the norm as Euclidean norm function (see \cite{JQS85}).
\paragraph{The different types of prime numbers.} Any prime number $p \in \ZZ$ gives rise to an ideal $(p)$ in $\OD$. A prime number is either \textbf{inert} if $(p)$ is a prime ideal, or \textbf{splits} if $(p)$ is a product of two conjugated prime ideals $\mathfrak{p}, \mathfrak{p}^\sigma$, \label{glo_matp} or is \textbf{ramified} if $(p)$ is the square of a prime ideal $\mathfrak{p}$. The splitting theorem (see e.g. \cite{Sch07}, Theorem~6.5.18) describes precisely when each of these cases occur: let $\left( \frac{\cdot}{\cdot} \right)$ denote the Legendre symbol. \label{glo_Legendre} If $p \in \ZZ$ is odd then $(p)$ is inert whenever $\left( \frac{D}{p} \right) =-1$ and $(p)$ is split if $\left( \frac{D}{p} \right) =+1$ and $(p)$ ramified if $\left( \frac{D}{p} \right) =0$, i.e. if $p|D$. The prime number $2$ is inert if $D\equiv 5 \mod 8$, it splits if $D \equiv 1 \mod 8$ and is ramified if $D \equiv 0 \mod 2$. Whenever $h_D=1$ one does not have to distinguish between prime ideals and prime elements in $\OD$. If $h_D>1$ then it only makes sense to speak of a prime decomposition of an element $n \in \OD$ in the sense of ideals. For two ideal $\mathfrak{a}, \mathfrak{b}$ we write $\mathfrak{a}|\mathfrak{b}$ if $\mathfrak{b} \subset \mathfrak{a}$. If $\mathfrak{a}$ is an ideal and $n \in \OD$, then we mean by $\mathfrak{a} | n$ that $\mathfrak{a}| (n)$. 
\paragraph{Quadratic orders.} Let $K$ be a quadratic number field. A \textbf{quadratic order} is a subring $\mathcal{O}$ of $K$ such that $\mathcal{O} \otimes \QQ = K$. Each integer $D \equiv 0 \ \textrm{or} \ 1 \mod 4$ determines a quadratic order
$$\OD = \ZZ[T]/(T^2+bT+c)$$
where $b,c \in \ZZ$ and $b^2-4c = D$. Indeed, these are up to isomorphism all quadratic orders and the isomorphism class of $\OD$ only depends on $D$.\label{glo_ODgen} Note that $\OD$ is the ring of integers of $\QQ(\sqrt{D})$ whenever $D$ is a fundamental discriminant. Thus the notation $\OD$ makes sense. If $D$ is not of the form $f^2E$ for some integers $f>1$ and $E$ with $E \equiv 0 \ \textrm{or} \ 1 \mod 4$ then $D$ is a \textbf{fundamental discriminant} and $f:=1$. In both cases $f$ is called the \textbf{conductor} of $D$.  If $D$ is not a square (but not necessarily fundamental), then $\OD$ is at least a subring of some ring of integers of $\QQ(\sqrt{D})$. Note that if $D \equiv 1 \mod 4$ is not a square then $\OD$ is always spanned by $1$ and $w = \frac{1+\sqrt{D}}{2}$ and if $D \equiv 0 \mod 4$ is not a square then $\OD$ is spanned by $1$ and $w=\frac{\sqrt{D}}{2}$. The notions of class number and narrow class number also make sense in the context of quadratic orders, though the definition of the class group is slightly more sophisticated: instead of considering all fractional ideals of $\OD$ one only considers those fractional ideals $\mathfrak{a}$ which are \textbf{proper}, i.e. fulfill $\mathfrak{a} = \left\{ \beta \in K \mid \beta \mathfrak{a} \subset \mathfrak{a} \right\}$. The proper ideals of $\OD$ are exactly the invertible ones. The class group is then defined as the quotient of the proper ideals by the principal ideals (see e.g. \cite{Cox89}, Chapter~7, for details). Moreover, it is important to note that for $f>1$ and $D=f^2E$ the quadratic order $\OD$ is not a unique factorization domain any more. In particular, if $p \in \ZZ$ is a prime number with $p|f$ then $(p)$ is irreducible but not prime. Thus $p$ behaves to a certain extent similarly as an inert prime number from an arithmetic point of view. On the other hand, any ideal which is relatively prime to $(f)$ can be factored uniquely into a product of prime ideals. For details on the factorization we refer the reader to \cite{Rob09}. If $D$ is a square then everything is slightly different. As we will not deal with this case in the following we refer the reader e.g. to \cite{Bai07}, Chapter~2.2 for details. 
\paragraph{Some arithmetic properties of real quadratic number fields.} \index{real quadratic number field!types of prime numbers}  For the convenience of the reader we recall two (well-known) arithmetic properties of real quadratic number fields, that will be used frequently in the following chapters.

\begin{lem} \label{lem_fund_discriminant_quadratic} \begin{itemize} 
\item[(i)] Let $D$ be a fundamental discriminant and let $K$ be the real quadratic number field of discriminant $D$. If $x \in K$ with $x^2 \in \OD$ then $x \in \OD$.
\item[(ii)] Let $D$ be a fundamental discriminant and let $K$ be the real quadratic number field of discriminant $D$. Let $p \in \ZZ$ be a prime number which is not ramified. If $x \in K$ with $x^2 \in \frac{1}{p} \cdot \OD$ then $x \in \OD$.
\end{itemize}
\end{lem}

\begin{rem} The lemma is not only true for real quadratic number fields, but for all number fields and the corresponding rings of integers.
\end{rem}

\begin{proof} $(i)$ The statement can either be proven by direct calculation or in a more sophisticated way by using valuations, i.e. function $\nu_\mathfrak{p}: K \to \ZZ$ which send $x$ to the multiplicity of the prime ideal $\mathfrak{p}$ in $(x)$ or $\infty$ if the multiplicity is $0$. It is well-known that $x \in \OD$ if and only if $\nu_\mathfrak{p} \geq 1$ for all prime ideals. Therefore, the claim immediately follows. 

$(ii)$ Let $x = \frac{c}{d}$ with $c,d \in \OD$. If the claim was not true, then $d^2$ would necessarily divide $p$ by $(i)$. This is not possible since $p$ is not ramified (and thus has no square divisors).
\end{proof}
The lemma is not true any more if $D$ is not a fundamental discriminant. For instance, if $D=45$ then $(2+5w)/3 \notin \OD$, but $((2+5w)/3)^2 = 31+5w$.


\begin{lem} If $\mathfrak{p}$ is a prime ideal which is a divisor of the principal ideal generated by a split or ramified prime number in $\OD$, then $$\mathfrak{p} \nmid 1,...,\N(\mathfrak{p})-1.$$ \end{lem}
\begin{proof} Suppose $\mathfrak{p} | a$ for $a \in \ZZ$ with $0 < a < \N(\mathfrak{p})$. Then $\N(\mathfrak{p})$ is a prime number strictly greater than $a$ and $\N(a)=a^2$. Therefore, $\N(\mathfrak{p})|\N(a)$ is impossible. 
\end{proof}

As we will make extensive use of the arithmetic of $\OD$ later on, we collect some more of its most important properties in a separate lemma.

\begin{lem} \label{lem_properties_OD} \textbf{(Properties of $w$)} Let $\OD$ be a real quadratic order.
\begin{itemize}
\item[(i)] If $x = a + bw \in \OD$ and $n \in \NN$ then $n|x$ if and only if $n|a$ and $n|b$.
\item[(ii)] If $D \equiv 1 \mod 4$ then $w^\sigma=1-w$ and for the absolute value of the norm we have $\N(w)=\N(w-1)= \frac{D-1}{4}$.
\item[(iii)] If $D \equiv 1 \mod 4$ then all the prime ideal divisors of $w$ are divisors of split prime numbers.
\item[(iv)] For $D \equiv 1 \mod 4$ we have $w^2 = w + \frac{D-1}{4}$.
\item[(v)] For $D \equiv 1 \mod 4$ and all $k \in \NN$ we have $w^k=c_k + d_kw$ with $c_k,d_k \in \ZZ$ and the greatest common divisor of $c_k$ and $d_k$ is $1$.
\item[(vi)] If $D \equiv 0 \mod 4$ then we have $\N(w)=\frac{D}{4}$ and $w^\sigma=-w$. Therefore, all prime ideal divisors of $w$ are divisors of ramified prime numbers.
\end{itemize}
\end{lem}

\begin{proof} $(i)$ follows from the definition.\\$(ii)$ is just a direct calculation.\\$(iii)$: $w$ cannot have a prime divisor which is an inert prime number by $(i)$ and $w$ cannot have a prime divisor which divides a ramified prime divisor by $(ii)$.\\$(iv)$ is again just a calculation.\\$(v)$ follows since $w^k$ has up to multiplicity the same prime divisor as $w$.\\$(vi)$ is again a simple calculation.
\end{proof}



\subsection{Fuchsian Groups} \label{sec_Fuchsian_groups}

Fuchsian groups are certain subgroups of $\PSL_2(\RR)$. They will appear almost everywhere in the following chapters. For a detailed review of this topic the reader may consult for example \cite{Bea89}, \cite{Kat92} or \cite{Mac01}. \\[11pt]
Let $\HH:=\left\{ z \in \CC \mid \Imag(z)>0 \right\}$ denote the \textbf{complex upper half plane} \label{glo_H}. \index{complex upper half plane} The closure of $\HH$ in the \textbf{projective plane} $\mathbb{P}^1(\CC):=\CC \cup \left\{ \infty \right\}$ \label{glo_P1} is $\overline{\HH}=\HH \cup \RR \cup \left\{ \infty \right\}.$ \label{glo_Hbar} We shall use the usual notations for real and imaginary parts, namely $z=x+iy$ for $z \in \CC$. Recall that $\HH$ is equipped with the Riemannian metric derived from the differential $$ds = \frac{\sqrt{dx^2+dy^2}}{y}.$$ 
The metric is called \textbf{Poincaré metric}. The Gaussian curvature of $\HH$ equipped with the Poincaré metric is constant and equal to $-1$ and the geodesics in $\HH$ are then given by vertical lines and semicircles orthogonal to $\RR$ (see e.g. \cite{Bea83}, Chapter~7.3). For a subset $A \subset \HH$ we define $\mu(A)$, the \textbf{hyperbolic area} \index{hyperbolic area} of $A$, by
$$\mu(A):= \int_{A} \frac{dxdy}{y^2} \label{glo_hyp}$$
whenever the integral exists. Another model for the hyperbolic plane is the \textbf{unit disc} $\mathbb{D}:=\left\{ z \in \CC \mid |z|<1 \right\}$ \label{glo_D}. Here the Poincaré metric corresponds to the differential $ds = \frac{2|dz|}{1-|z|^2}$.\\[11pt]
A matrix $M = \left( \begin{smallmatrix} a & b \\ c & d \end{smallmatrix} \right) \in \SL_2(\RR)$, i.e. $a,b,c,d \in \RR$ and $ad-bc=1$, \label{glo_SL2} (or rather the group $\SL_2(\RR)$) acts on $\HH$ by \textbf{Möbius transformation}, namely \index{Möbius transformation}
$$ z \mapsto \frac{az+b}{cz+d} $$ 
and this map is an orientation-preserving isometry (see e.g. \cite{Kat92}, Theorem~1.3.1). Two matrices which differ only by a multiple of $\pm \Id$ define the same Möbius transformation. Indeed, the group of all biholomorphic maps of $\HH$ to itself is given by $\PSL_2(\RR):=\SL_2(\RR)/\left\{\pm \Id \right\}$. \label{glo_PSL2}\\[11pt]
A \textbf{Fuchsian group} \index{Fuchsian group} $\Gamma$ \label{glo_Gamma} is a discrete subgroup of $\PSL_2(\mathbb{R})$. The most prominent example of a Fuchsian group is the \textbf{modular group} $\PSL_2(\ZZ)$. Elements in a Fuchsian group are distinguished by the value of their \textbf{trace} $\tr(M):=|a+d|$ for $M= \left( \begin{smallmatrix} a & b \\ c & d \end{smallmatrix} \right).$ If $\tr(M)<2$ \label{glo_trace} then $M$ is called \textbf{elliptic}, if $\tr(M)=2$ then $M$ is called \textbf{parabolic} and if $\tr(M)>2$ then $M$ is called \textbf{hyperbolic}. Equivalently one could distinguish these elements by their fixed points: an elliptic transformation has exactly one fixed point in $\HH$, a parabolic element has only one fixed point on the boundary and a hyperbolic transformation has two fixed points on the boundary. A closed domain $\mathcal{F} \subset \HH$ is called a \textbf{fundamental domain} \label{glo_F} \index{Fuchsian group!fundamental domain} for the Fuchsian group $\Gamma$ if $(i)$ for every $z \in \HH$ there exists a $M \in \Gamma$ with $Mz \in \mathcal{F}$, $(ii)$ $\inter(\mathcal{F}) \cap M (\inter(\mathcal{F})) = \emptyset$ for all $M \neq \pm \Id$ and $(iii)$ the number of $M \in \Gamma$ with $M\mathcal{F} \cap \mathcal{F} \neq \emptyset$ is finite. A Fuchsian group $\Gamma$ is a \textbf{lattice} or \textbf{cofinite} \index{Fuchsian group!lattice} if the orbit space $\HH/\Gamma$ (or equivalently its fundamental domain) has finite (hyperbolic) area $\mu(\HH/\Gamma)$. A Fuchsian group $\Gamma$ is called \textbf{(co-)compact} \index{Fuchsian group!(co-)compact} if the quotient $\HH/\Gamma$ is compact. It is very often convenient to look at conjugated groups: for a Fuchsian group $\Gamma$ and $M \in \SL_2(\RR)$ we denote by $\Gamma^M$ the group $M^{-1}\Gamma M$. \label{glo_conj}\\[11pt]
Here are two motivations why Fuchsian groups are interesting for us:\footnote{The reader who is not familiar with the objects which are mentioned here should consult the following sections first.} First of all, every Veech group is a Fuchsian group (see e.g. \cite{Vor96}). Secondly recall that the moduli space of Riemann surfaces characterizes the set of isomorphism classes of Riemann surfaces of a given genus. It goes back to ideas of Riemann that the universal cover of every Riemann surface is biholomorphically equivalent either to $\mathbb{P}^1(\CC)$ or to $\CC$ or to $\HH$ (\textbf{Uniformization theorem}, see e.g. \cite{FK92}, Chapter~IV.4). All Riemann surfaces of genus~$g \geq 2$ have $\HH$ as their universal cover. Since the automorphism group of $\HH$ is exactly $\PSL_2(\RR)$, every Riemann surface of genus~$g \geq 2$ can be regarded as the quotient of $\HH$ by a Fuchsian group. On the other hand, for a cofinite Fuchsian group $\Gamma$ the quotient space $\HH/\Gamma$ becomes a Riemann surface, if the charts are chosen properly. The surface $\HH/\Gamma$ has orbifold points at the fixed points of the elliptic elements of $\Gamma$.
\paragraph{Signature of a Fuchsian Group.} The most general presentation of a cofinite Fuchsian group $\Gamma$ is the following (see e.g. \cite{Sin72}): the set of generators is given by some hyperbolic elements $a_1,b_1,...,a_g,b_g$, some elliptic elements $x_1,x_2,...,x_r$ and some parabolic elements $p_1,...,p_s$. Note that the $x_j$ and $p_j$ are representatives of the conjugacy classes of the elliptic and parabolic generators. Furthermore the relations are given by $$x_1^{m_1}=x_2^{m_2}=...=x_r^{m_r}=\prod_{i=1}^g [a_i,b_i] \prod_{j=1}^r x_j \prod_{k=1}^s p_k =1$$
where $[\cdot,\cdot]$ denotes the commutator. Then we say that $\Gamma$ has \textbf{signature} $(g;m_1,...m_r;s)$. \label{glo_signature}\index{Fuchsian group!signature}The signature contains precise information about the Euler characteristic $\chi(\HH/\Gamma)$ \label{glo_eul} and therefore also about its volume since $\mu(\HH/\Gamma)=2\pi\chi(\HH/\Gamma)$.

\begin{thm} \label{thm_Riemann_Hurwitz} \textbf{(Riemann-Hurwitz Formula)} \index{Riemann-Hurwitz formula} The Euler characteristic $\chi(\HH / \Gamma)$ of a cofinite Fuchsian group $\Gamma$ of signature $(g;m_1,...m_r;s)$ is:
$$2 - 2g - \sum_{i=1}^r \left( 1 - \frac{1}{m_i} \right) - s.$$
\end{thm}

\begin{proof} See \cite{Miy89}, Theorem 2.4.3. \end{proof}

\paragraph{Commensurability and Arithmeticity} \index{commensurator} \index{commensurability} Let G be a group and $A,B < G$ two subgroups: \begin{itemize} 
\item[(i)] $A$ and $B$ are called \textbf{directly commensurable} if $A \cap B$ has finite index in both $A$ and $B$, i.e. $[A:A\cap B] < \infty$ and $[B:A\cap B] < \infty$. \label{glo_Index}
\item[(ii)] $A$ and $B$ are called \textbf{commensurable (in G)} if there exits a $g \in G$ such that $A$ and $gBg^{-1}$ are directly commensurable.
\item[(iii)] For a subset $A$ of $G$ we denote by $\Comm_G  (A)$ the \textbf{commensurator of $A$ in $G$} \label{glo_Comm} namely the set of all $g \in G$ such that $A$ and $gAg^{-1}$ are directly commensurable. \index{commensurator}
\end{itemize}
\begin{rem} 
\begin{itemize}
\item[(i)] Two subgroups $A$ and $B$ are directly commensurable if and only if there exists a common subgroup $C$ which has finite index in both $A$ and $B$.
\item[(ii)] If $A,B,C$ are subgroups of $G$ and $A$ and $B$ are directly commensurable then $A \cap C$ and $B \cap C$ are directly commensurable.
\item[(iii)]If $C$ is a subgroup of $B$ of finite index and $A$ is directly commensurable to $C$, then $A$ is also directly commensurable to $B$. 
\item[(iv)] We have $A \subset \Comm_G  (A)$.
\item[(v)]  Always $[\Comm_G  (A):A] \neq 2$ since $A$ cannot be a normal subgroup.
\item[(vi)] The set $\Comm_G (A)$ is a subgroup of $G$.
\item[(vii)] If $A$ and $B$ are commensurable then $\Comm_G  (A) = \Comm_G  (B)$.
\item[(viii)] For all $g \in G$ we have $\Comm_G  (gAg^{-1}) = g \Comm_G  (A) g^{-1}$.
\end{itemize}
\end{rem}


An \textbf{order} $\mathcal{O}$ in a quaternion\footnote{The word quaternion itself is taken from a sentence of the King James Bible: "And when he had apprehended him, he put him in prison, and delivered him to four quaternions of soldiers to keep him"(Acts 12:4); compare \cite{Ebb92}, p.159.} algebra $A$ over a totally real number field $F$ is a subring of $A$ containing $1$ which is a finitely generated $\mathcal{O}_F$-module generating the algebra $A$ over $F$, where $\mathcal{O}_F$ is the ring of integers in $F$. Recall that the group of units in such an order $\mathcal{O}$ of reduced norm $1$ can be embedded into $\SL_2(\RR)$. This group, denoted by $\Gamma(\mathcal{O}^1,A)$, is in fact a Fuchsian group (\cite{Kat92}, Theorem 5.2.7).\\[11pt] A Fuchsian group $\Gamma$ is called \textbf{arithmetic} \index{Fuchsian group!arithmetic} if it is commensurable with the group $\Gamma(\mathcal{O}^1,A)$ for some $\mathcal{O}$ and $A$. Another equivalent definition of arithmeticity was given in \cite{Tak75} using traces and their Galois conjugates.\\[11pt]
A very deep result of G. Margulis (see \cite{Mar89} for a proof) links the commensurator of a Fuchsian group to the notion of arithmeticity. This theorem will play an important role in particular in Chapter~\ref{cha_maximality}. 
\begin{thm} \index{commensurator!Margulis' Theorem} \textbf{(Margulis' Commensurator Theorem)} \label{thm_Margulis}\\ A Fuchsian group $\Gamma \subset \PSL_2(\mathbb{R})$ is arithmetic if and only if its commensurator $\Comm_{\PSL_2(\mathbb{R})}(\Gamma)$ is dense in $\PSL_2(\mathbb{R})$. Furthermore if $\Gamma$ is non-arithmetic, then $\Comm_{\PSL_2(\mathbb{R})}(\Gamma)$ contains $\Gamma$ as finite index subgroup and hence $\Comm_{\PSL_2(\mathbb{R})}(\Gamma)$ is itself a Fuchsian group.
\end{thm}

There is mainly one property beside Margulis' Theorem of arithmetic Fuchsian groups which is important for us (this justifies also the sketchy definition):
\begin{thm} \label{thm_nonarithmetic_commensurable} (\cite{Mac01}, Theorem~5.2) If $\Gamma(\mathcal{O}^1,A)$ is non-compact, then it is commensurable with $\PSL_2(\ZZ)$. \end{thm}
In other words $\PSL_2(\ZZ)$ is a representative for the (single) commensurability class of non-compact arithmetic Fuchsian groups. Vice versa: a non-compact Fuchsian group which is not commensurable to $\PSL_2(\ZZ)$ is non-arithmetic.
\paragraph{Trace fields.} \index{Fuchsian group!trace field} The \textbf{trace field} of a Fuchsian group $\Gamma$ is the subfield of $\RR$ generated by $\tr(A)$ for $A \in \Gamma$. If the volume of the Fuchsian group is finite, then the trace field is always a finite extension of $\QQ$ (\cite{MR03}, Theorem~3.1.2).

\subsection{Congruence Subgroups of $\SL_2(\OD)$} \label{sec_congruence_subgroups}

We denote by $\SL_2(\OD)$ the set of all 2 by 2 matrices with entries in $\OD$ and determinant $1$ and set $\PSL_2(\OD)$ as $\SL_2(\OD)$ modulo diagonal matrices of determinant $1$. Accordingly we define $\SL_2(K)$ and $\PSL_2(K)$.\footnote{$\SL_2(K)$ is generated by the matrices $U(b)=\left( \begin{smallmatrix} 1 & b \\ 0 & 1 \end{smallmatrix} \right), \ b \in K$ and $S=\left( \begin{smallmatrix} 0 & -1 \\ 1 & 0 \end{smallmatrix} \right)$ (see e.g. \cite{Lan85}, p. 209-211).} On the level of matrices Kronecker's Approximation Theorem implies that $\PSL_2(\OD)$ is not a Fuchsian group.\footnote{In the next section we will however see that the group $\PSL_2(\OD)$ acts properly discontinuously on $\HH \times \HH^-$; the group $\PSL_2(K)$ is neither a Fuchsian group nor does it act properly discontinuously on $\HH \times \HH^-$.}  
Beside $\SL_2(\OD)$ we will also be interested in a special type of subgroups of $\SL_2(\OD)$, which will appear in some of the explicit calculations, namely certain congruence subgroups. In this section we will introduce these subgroups of $\SL_2(\OD)$. \cite{Fre90} and \cite{New72} give a broader overview of the topic.
\begin{defi} Two matrices $A$ and $B$ are equivalent modulo $n \in \OD$, if they are entrywise equivalent modulo $n$. Then $$\Gamma^D(n):= \{ M \in \SL_2(\OD) \mid M \equiv E \mod n\} \label{glo_GammaD}$$ is called \textbf{principal congruence subgroup}. A \textbf{congruence subgroup} \index{congruence subgroup|(} is a subgroup of $\SL_2(\OD)$ containing a principal congruence subgroup $\Gamma^D(n)$ for some $n \in \OD$. \end{defi}

Sometimes $n$ is called the \textbf{level} of the congruence subgroup. A principal congruence subgroup is equal to the kernel of the projection map $\SL_2(\OD) \to \SL_2(\OD/n\OD)$ (see e.g. \cite{Fre90}, Chapter~3).

\begin{lem} \label{lem_exact_cong_sub} If $h_D=1$ then for all $n \in \OD$ the following sequence is exact:
$$1 \to \Gamma^D(n) \to \SL_2(\OD) \to \SL_2(\OD  / n \OD) \to 1.$$ \end{lem}

\begin{proof} Evidently $\Gamma^D(n)$ is the kernel of the projection $\pi: \SL_2(\OD) \to \SL_2(\OD  / n \OD)$. The surjectivity follows just like in the $\SL_2(\ZZ)$-case (see e.g. \cite{KK07}, Satz II.3.2) since $h_D=1$. \end{proof}

For general $h_D$ the map is still surjective; to prove this one may use use rather sophisticated techniques (strong approximation). Nevertheless, there is a very nice elementary proof, that we want to sketch here: let $S,R$ be two commutative rings and let $R \to S$ be a surjective homomorphism. If $\SL_2(S)$ is generated by elementary matrices then the induced map $\SL_2(R) \to \SL_2(S)$ is also surjective since an elementary matrix over $S$ obviously lifts to an elementary matrix over $R$. Furthermore, it follows e.g. from the results in \cite{Ros94}, Chapter~2.2, that $\SL_2(S)$ is generated by elementary matrices if $S$ is a local ring. We now apply these facts to $R=\OD$ and $S=\OD/n\OD$. Since $\OD/n\OD$ is finite and since every finite commutative ring is a direct product of local rings, the claim thus follows. \\[11pt]As the absolute value of the norm of an element $n \in \OD$ is equal to the number of elements of $\OD  / n \OD$ and thus finite, the number of elements in $\SL_2(\OD  / n \OD)$ and thereby of $\Gamma^D(n)$ can be very roughly bounded by $\mathcal{N}(n)^4$. If the ring of integers $\OD$ has class number $h_D=1$ then the index of principal congruence subgroups in $\SL_2(\OD)$ can be calculated.  

\begin{thm} \label{thm_cong_index} Let $n \in \OD$ and $\Gamma^D(n)$ be an principal congruence subgroup where 4$$(n) = \prod_{k=1}^{s} (\mathfrak{p}_k)^{e_k}$$ is the unique prime decomposition of $n$ up to multiplication by units. Then
$$\Gamma^D(n) = \Gamma^D(\mathfrak{p}_1)^{e_1}) \times \cdots \times \Gamma^D(\mathfrak{p}_s)^{e_s})$$
and
$$[\SL_2(\OD) : \Gamma^D(n)] =\N(n)^{3} \prod_{\mathfrak{p} | n} \left( 1 - \frac{1}{\N(\mathfrak{p})^2}\right).$$ \end{thm}

\begin{proof} See \cite{New72}, p. 113-115. There it is assumed that the ring is Euclidean. However, the proof also works verbatim if only $h_D=1$. \end{proof}

Note that this formula is a generalization of the $\SL_2(\ZZ)$-case. There are even some more special congruence subgroups, where one can calculate the index. These are the analogs to the Hecke congruence subgroups of $\SL_2(\ZZ)$, which are considered e.g. in the theory of newforms (see e.g. \cite{Wei08}). These groups will also be important for this work.

\begin{defi} \label{glo_GammaD0} The groups $\Gamma^D_0(n)$ and $\Gamma^{D,0}(n)$ are defined by $$\Gamma^D_0(n) := \left\{ \begin{pmatrix} a & b \\ c & d \end{pmatrix} \mid \begin{pmatrix} a & b \\ c & d \end{pmatrix} \equiv \begin{pmatrix} * & * \\ 0 & * \end{pmatrix} \mod n \right\} $$
and
$$\Gamma^{D,0}(n) := \left\{ \begin{pmatrix} a & b \\ c & d \end{pmatrix} \mid \begin{pmatrix} a & b \\ c & d \end{pmatrix} \equiv \begin{pmatrix} * & 0 \\ * & * \end{pmatrix} \mod n \right\}.$$
Finally we set $\Gamma^D(m,n):= \Gamma^D_0(m) \cap \Gamma^{D,0}(n).$ Accordingly, we define all these groups for proper ideals $\mathfrak{m},\mathfrak{n}$.
\end{defi} \index{congruence subgroup|)}

\begin{prop} \label{prop_index_congruence} Let $n \in \OD$ and let $(n) = \prod_{k=1}^{s} (\mathfrak{p}_k)^{e_k}$ be a  decomposition of $n$ into prime ideals then:  $$[\SL_2(\OD) : \Gamma^D_0(n)] = [\SL_2(\OD) : \Gamma^{D,0}(n)]=\N((n)) \prod_{\mathfrak{p}_k} \left( 1 + \frac{1}{\N(\mathfrak{p}_k)}\right).$$ 
If $m \in \OD$ with $(m,n)=1$ and if $(m) = \prod_{k=1}^{r} (\widetilde{\mathfrak{p}_k})^{e_k}$ is a decomposition of $(m)$ into prime ideals then
\begin{eqnarray*} [\SL_2(\OD) : \Gamma^D_0(n) \cap \Gamma^{D,0}(m)] & = & \N((n)) \prod_{\mathfrak{p}_k} \left( 1 + \frac{1}{\N(\mathfrak{p}_k)}\right) \cdot \\ & & \N(m) \prod_{\widetilde{\mathfrak{p}_k}} \left( 1 + \frac{1}{\N(\widetilde{\mathfrak{p}_k})}\right).\end{eqnarray*}
\end{prop}
\begin{proof} In the case $h_D=1$ the proof can be done exactly like e.g. in \cite{Wei08}, Chapter~2.1. If $h_D>1$ it is immediately clear that the claimed equality hold with $\leq$ instead of $=$. As is commonly known in the literature, equality is also true in the latter case (weak approximation). We also implicitly prove this fact in Chapter~\ref{chapter_calculations}. 
\end{proof}

We conclude this section by giving another class of examples of congruence subgroups, which are in general not principal congruence subgroups.

\begin{prop} \label{prop_finite_index_GL2K} If $M \in \GL_2^+(K)$ then  $\SL_2(\OD) \cap M \SL_2(\OD) M^{-1}$ is a congruence subgroup and therefore
$$[\SL_2(\OD) : \SL_2(\OD) \cap M \SL_2(\OD) M^{-1}] < \infty $$
holds.\end{prop}

\begin{proof} Let $M \in \GL_2^+(K)$ and let $d$ be the product of all the denominators appearing in $M$ and $M^{-1}$. Then $M^{-1} \Gamma^D(d) M \subset \SL_2(\OD)$ since $M^{-1}\Id M= \Id$ and since for all $A \in \Mat^{2x2}(\OD)$ we have $dM^{-1}AM \in \Mat^{2x2}(\OD)$. In other words the claim follows from the following exact sequence:
$$1 \to \Gamma^D(d) \to \SL_2(\OD) \to \SL_2(\OD/d\OD) \to 1.$$
\end{proof}
\subsection{Moduli Spaces} \label{sec_moduli_spaces}

One of the objects we will deal with everywhere in these notes is the moduli space of compact Riemann surfaces. It solves the moduli problem of parameterizing all Riemann surfaces of a given genus~$g$. In Section~\ref{sec_moduli_space_rs} this topic will be dealt with. In Section~\ref{sec_flat_surfaces} we introduce flat surfaces. These are Riemann surfaces together with a non-zero holomorphic differential form. Flat surfaces have a very important invariant, the Veech group. Furthermore, they lead to a stratification of the moduli space, that we describe in detail in Section~\ref{sec_strata}. Finally, we consider in Section~\ref{sec_moduli_ablian} another moduli space, namely the one of Abelian varieties.

\subsubsection{Flat Surfaces and Veech Groups} \label{sec_flat_surfaces}

\paragraph{Riemann surfaces.} By a \textbf{Riemann surface} \index{Riemann surface} we shall mean a connected holomorphic manifold of complex dimension one. More precisely, a Riemann surface is a connected Hausdorff space $X$ of real dimension two together with a maximal set $\Sigma$ of \textbf{charts} $\left\{U_\alpha,z_\alpha\right\}$ on $X$ such that the \textbf{transition functions}
$$f_{\alpha\beta} = z_\alpha \circ z_\beta^{-1}: z_\beta(U_\alpha \cap U_\beta) \to z_\alpha(U_\alpha \cap U_\beta)$$
are holomorphic maps. The maps $z_\alpha$ are also called \textbf{local parameters}. $\Sigma$ is called a \textbf{complex structure} on $X$. Note that given any set $M$ of (holomorphic) charts covering the surface, $M$ is contained in a unique $\Sigma$ (see e.g. \cite{Shi71}, p. 15). It is well-known that there is an equivalence of categories between (non-singular, connected, projective) algebraic curves over $\CC$ and compact Riemann surfaces (see e.g. \cite{Rey89}, Chapitre~VII, 3). This is the reason why there is a double terminology, i.e. why the terms \textit{compact Riemann surface} and \textit{algebraic curve} are used as synonyms in the existing literature.
\paragraph{Flat surfaces.} Morally speaking, a flat surface is a Riemann surface with a metric which has curvature zero everywhere with the only exception that the metric may have several singular points. These points are also called conical singularities. An intuitive example of a flat surface in this sense is a cube: it is flat on all of its sides; the vertices of the cube are the conical singularities; they carry all the curvature of the cube (see \cite{Zor06}, Section~1.1 for more details on this example). We now give the precise definition: a \textbf{flat surface} \index{flat surface} is a pair $(X,\omega)$ \label{glo_Xw} where $X$ is a Riemann surface and $\omega$ is a non-zero holomorphic differential form on $X$. One can define charts on $X$ by integrating $\omega$ locally (see e.g. \cite{HS06}, Section~1.1.3). The transition maps are then given by translations.  Let us explain how this gives rise to a metric with the properties mentioned previously. Let $Z(\omega)$ denote the zeroes of $\omega$. On $X \smallsetminus Z(\omega)$ a flat (Riemannian) metric is given by pulling back the Euclidean metric of $\CC$ via the charts. A zero $P \in Z(\omega)$ leads to a singularity of this metric. The total angle around a singularity, called the \textbf{cone angle}, is an integer multiple of $2\pi$. If the form $\omega$ has a zero of degree $d$, then the cone angle at this point is equal to $2\pi(d+1)$ (see e.g. \cite{Zor06}, Section~3.3). A geodesic segment connecting two singular points is called \textbf{saddle connections}. Equivalently, flat surfaces arise from gluing rational angled planar polygons by parallel translations along their faces. This implies in particular that there exists an atlas on $X$ such that all transition maps of $X$ away from the zeroes are given by translations. Therefore, flat surfaces are sometimes also called \textbf{translation surfaces} (see also e.g. \cite{Mas06}). 
a zero.\\[11pt]
The simplest class of examples of flat surfaces are \textbf{square-tiled surfaces} or \textbf{Origamis}, i.e. flat surfaces $(X,\omega)$, where $X$ is obtained as a covering of a square torus ramified over one point only and $\omega$ is the pullback of the holomorphic one-form on the torus. A rather comprehensive and understandable paper on square-tiled surfaces is \cite{Sch05}.\\[11pt]A \textbf{holomorphic quadratic differential} $q$ on a Riemann surface is locally defined by $q=f(z)(dz)^2$ where $f(z)$ is a holomorphic function defined on a chart $(U,z)$. A pair $(X,q)$ is sometimes called a \textbf{half-translation surface}, because there always exists an atlas on $X$ such that the transition functions are given by compositions of $\pm \Id$ and translations. 
\paragraph{Veech groups.} At least from our point of view, the Veech group is the most important invariant of a flat surface. Take the charts on $X \smallsetminus Z(\omega)$ defined by integrating $\omega$ and let $\Aff^+(X,\omega)$ denote the group of orientation-preserving homeomorphisms of $X$ that are affine diffeomorphisms on $X \smallsetminus Z(\omega)$ with respect to these charts. The matrix part of the affine map is independent of the carts and provides a map
$$D: \Aff^+(X,\omega) \to \SL_2(\RR).$$
The image of $D$ is called the \textbf{affine group} or \textbf{Veech group} of $(X,\omega)$ and often denoted by $\SL(X,\omega)$. \label{glo_SLXw} \index{Veech group}Let us point out three well-known basic properties of Veech groups.
\begin{prop} A Veech group $\SL(X,\omega)$ is always discrete, i.e. a Fuchsian group. \end{prop}
\begin{proof} Let $l$ be the length of (one of) the shortest saddle connections $\gamma$ on $X$. There exist only finitely many saddle connection of length at most $2l$. This implies that there are only finitely many possible images of $\gamma$ under $\SL(X,\omega)$. Hence $\SL(X,\omega)$ is discrete. 
\end{proof}
\begin{prop} A Veech group $\SL(X,\omega)$ is never cocompact. \end{prop}
\begin{proof}[Proof (following \cite{HS06})] We only need to find a continuous function on $\SL_2(\RR)/\SL(X,\omega)$ which has no minimum value. Consider the function $\Lambda: \SL_2(\RR) \to \RR^+$ given by $A \mapsto l(A (X,\omega))$ where $l(\cdot)$ denotes the length of the shortest saddle connection. By rotating we normalize $(X,\omega)$ such that its shortest saddle connection is in the vertical direction. Via the geodesic flow $\left( \begin{smallmatrix} e^{t/2} & 0 \\ 0 & e^{-t/2} \end{smallmatrix} \right)$ the length of this geodesic tends to zero. \end{proof} 
Finally, it follows from a theorem by E. Gutkin and C. Judge that $(X,\omega)$ is a square-tiled surface if and only if $\SL(X,\omega)$ is arithmetic (see \cite{GJ96}). \index{Theorem of Gutkin-Judge} We will discuss this theorem in more detail in Section~\ref{sec_definition_tmcurves}. If the Veech group $\SL(X,\omega)$ is a lattice, then $(X,\omega)$ is called a \textbf{Veech surface}. For more general background on Veech groups we refer the reader for instance to \cite{Möl09} and \cite{Sch05}.



\subsubsection{The Moduli Space of Compact Riemann Surfaces} \label{sec_moduli_space_rs}

The uniformization theorem (1851) can be considered as starting point for the moduli problem for Riemann surfaces. Roughly speaking, it states that every connected Riemann surface of genus $g \geq 2$ has $\HH$ as its universal cover (see e.g. \cite{FK92}, p. 194ff). The moduli space of compact Riemann surfaces of genus $g$ parametrizes all compact Riemann surfaces of genus $g$. We will briefly review some of the relevant material. There are some different approaches to construct moduli space. All approaches are similar in the sense that each construction involves a priori curves with additional structure and then taking the quotient by the relation that identifies these additional structures. We present here the topological approach which involves Teichmüller space. \index{Teichmüller space}	There is a wide variety of material on the moduli space and related topics, for example \cite{Ham11}, \cite{HM98}, \cite{IT92}, \cite{Sch89}, \cite{SS92} to name a few. 
\paragraph{Teichmüller space.} Fix an arbitrary closed (i.e. compact without boundary) Riemann surface $S$ of genus $g \geq 2$.\footnote{As the Riemann mapping theorem already indicates, the case $g \leq 1$ is by far easier (see e.g. \cite{Sch05}). We will therefore in the following restrict to the case $g \geq 2$.} A \textbf{marked complex structure} on $S$ is a pair $(X,\varphi)$ consisting of a Riemann surface $X$ and a orientation-preserving diffeomorphism $\varphi: S \to X$. Two such pairs $(X,\varphi)$ and $(X',\varphi')$ are equivalent if $\varphi' \circ \varphi^{-1}$ is homotopic to a biholomorphic mapping $h: X \to X'$. The space of equivalence classes is called the \textbf{Teichmüller space} $\mathcal{T}(S)$. \label{glo_TS} Note that by definition the Teichmüller space has a base point, namely $(S,\Id)$. Teichmüller space has a canonical structure as a complex manifold of dimension $3g-3$. Indeed, it is biholomorphic equivalent to a bounded domain in $\CC^{3g-3}$. The \textbf{mapping class group} is defined as the quotient \label{glo_ModS} \index{mapping class group}
$$\Mod(S) = \Diff^+(S)/\Diff^+_0(S)$$
of orientation-preserving diffeomorphisms of $S$ modulo those orientation-preserving diffeomorphisms of $S$ which are homotopic to the identity. An element $\zeta$ of the mapping class group acts on Teichmüller space by leaving the base-surface invariant and pre-composing the marking with $\zeta^{-1}$.
\paragraph{Moduli space.} The \textbf{moduli space of compact Riemann surfaces of genus g} \index{moduli space of Riemann surfaces} is defined as the quotient \label{glo_Mg}
$$\mathcal{M}_g := \mathcal{T}(S) / \Mod(S)$$
of the Teichmüller space of an arbitrary Riemann surface $S$ of genus $g$ by the action of the mapping class group. The definition does not depend on the chosen base point of Teichmüller space. Since the action of the mapping class group is properly discontinuously, the moduli space is a non-compact orbifold of complex dimension $3g-3$. There is a good compactification $\overline{\Mg}$ \label{glo_Mgbar} of moduli space which goes back to P. Deligne and D. Mumford and which is therefore known as the \textbf{Deligne-Mumford compactification} (for details see e.g. \cite{MB09}).
\paragraph{The Teichmüller metric.} When $S_0$ and $S_1$ are Riemann surfaces with different complex structures, there does (by definition) not exist any conformal map from $S_0$ to $S_1$. The deviation of the complex structures can be measured. An orientation-preserving smooth map $f: S_0 \to S_1$ sends an infinitesimal circle at $x \in S_0$ to an infinitesimal ellipse at $f(x) \in S_1$. The coefficient of quasiconformality of $f$ at $x \in S_0$ is the ratio $K_x(f) = \frac{a}{b}$ of the demi-axis of this ellipse. The \textbf{coefficient of quasiconformality of f} is defined by
$$K(f) := \sup_{x \in S_0} K_x(f).$$
Note that a map is conformal if and only if $K(f)=1$. Hence one can define a metric, the \textbf{Teichmüller metric} \index{moduli space of Riemann surfaces!Teichmüller metric} on $\Mg$, by $d(S_0,S_1)=\inf_{f} \frac{1}{2} \log K(f)$.
\paragraph{A vector bundle over $\Mg$} An element of $\Omega\Mg$ is specified by a pair $(X,\omega)$ where $X \in \Mg$ and where $\omega \in \Omega(X)$ \label{glo_Omx} is a non-zero, holomorphic 1-form on $X$. In other words an element of $\Omega \Mg$ \label{glo_OmMg} is given by a flat surface. So $\Omega\Mg$ is a natural vector bundle minus the zero section over $\Mg$. Moreover, $\Omega\Mg$ naturally has the structure of a complex algebraic orbifold, whose dimension is equal to $4g-3$ (see e.g. \cite{KZ03}). We will treat $\Omega\Mg$ intensively in the following chapters mainly because of the fact that there is a natural action of $\SL_2(\RR)$ on $\Omega\Mg$. We will describe this action in Section~\ref{sec_strata}.
\paragraph{Families of curves.} The concept of families of curves which we will explain now goes back to A. Grothendieck. We do not use the language of categories here but follow closely the textbook \cite{Sch89}. Let $\mathcal{X}, \mathcal{B}$ be complex spaces. A \textbf{family of curves (Riemann surfaces)} \index{moduli space of Riemann surfaces!family of curves|(} over $\mathcal{B}$ is a surjective map $\pi: \mathcal{X} \to \mathcal{B}$ such that $\pi$ is a holomorphic map and the fiber $\mathcal{X}_b:=\pi^{-1}(b)$ is a Riemann surface of genus $g$ for every point $b \in B$. The space $\mathcal{B}$ is also called the \textbf{base space}. Note that a  family of curves $\pi : \mathcal{X} \to \mathcal{B}$ is necessarily locally topologically trivial, i.e. every point $b \in \mathcal{B}$ has a neighborhood $U$ such that there exists a homeomorphism $h: \pi^{-1} (U) \to U \times \mathcal{X}_b$.\\[11pt] If we have such a family $\pi: \mathcal{X} \to \mathcal{B}$ we can define a map $\Psi_{\mathcal{B},\mathcal{X}}: \mathcal{B} \to \Mg$ by assigning to $b \in \mathcal{B}$ the isomorphy class of the fiber $\mathcal{X}_b$ over $b$. The geometric structure of $\Mg$ can be defined such that each $\Psi_{\mathcal{B},\mathcal{X}}$ is a holomorphic map and such that $\Mg$ is universal in the sense described e.g. in \cite{Sch89}, p.69. Therefore, $\Mg$ is called a \textbf{coarse moduli space} although $\Mg$ is not a \textbf{fine moduli space}. This is because there does not exist a universal family of curves over $\Mg$. A family of curves $\pi : \mathcal{U} \to \Mg$ is a \textbf{universal family of curves over $\Mg$} if the fiber over every point $Y \in \Mg$ is a representative for the class $Y$ and every family of curves $\mathcal{X}$ over $\mathcal{B}$ is induced by pulling back $\mathcal{U}$ via the map $\Psi_{\mathcal{B},\mathcal{X}}$. Analogously, the notion of fine moduli space can be defined for every moduli space. The obstruction to the existence of a universal family over $\Mg$ is the existence of curves with nontrivial automorphisms (see \cite{HM98}, Chapter~2.A). A way to bypass this problem is to consider families of curves with level $l$-structures. A \textbf{level $l$-structure} is the choice of an isomorphism from $(\ZZ/(l))^{2g}$ to the $l$-torsion points of the Jacobian (see also Section~\ref{sec_moduli_ablian}). If $l \geq 3$, then there is a fine moduli space representing the moduli problem (see e.g. \cite{Kap11}, Chapter~5). \index{moduli space of Riemann surfaces!family of curves|)}

\subsubsection{Strata} \label{sec_strata}
A \textbf{stratification} of a topological $X$ is a decomposition $X= \bigcup_{i \in I} X_i$ where $I$ is a finite set of indexes and the \textbf{strata} $X_i$ are disjoint orbifolds such that the closure of a stratum is a union of strata. In this section we want to describe a natural stratification of $\Omega\Mg$: from the Riemann Roch Theorem one can deduce a well known formula for the sum of the degrees $k_i$ of the zeroes of a holomorphic 1-form on a Riemann surface of genus~$g$ (see e.g. \cite{For77}, Satz 17.12), namely
\begin{equation} \label{equ_degree_of_zeroes}
\hspace{3.9cm} \sum_{i} k_i = 2g-2.
\end{equation}
Although it is not completely obvious that condition (\ref{equ_degree_of_zeroes}) is also sufficient for the existence of a holomorphic 1-form with corresponding order of zeroes, this is nevertheless true. The key ingredient for the stratification of $\Omega\Mg$ is equation (\ref{equ_degree_of_zeroes}). Let $k_1,...,k_n$ be a finite sequence of positive integers such that the sum $\sum_i k_i$ is equal to $2g-2$. Then denote by $\Omega\Mg(k_1,...,k_n)$ \label{glo_strat} the subspace of $\Omega\Mg$ consisting of equivalence classes of pairs $(X,\omega)$ where $\omega$ has exactly $n$ zeroes with multiplicities $k_1,...,k_n$ (for some ordering of the zeroes). Since the definition does not depend on the ordering of the $k_i$ this yields a decomposition:
$$\Omega\Mg = \bigcup_{\substack{n,(k_1,...,k_n) \\ k_1\leq...\leq k_n \\ k_1+...+k_n = 2g-2}} \Omega\Mg(k_1,...,k_n).$$
It is well-known that each of the $\Omega\Mg(k_1,...,k_n)$ is an orbifold of complex dimension $2g+n-1$ (see e.g. \cite{KZ03}). \index{moduli space of Riemann surfaces!stratum|(} Thus this is indeed a stratification of $\Omega\Mg$.\\[11pt]
Each stratum $\Omega\Mg(k_1,...,k_n)$ can be locally modeled on a cohomology space:  any differential form $\omega$ on $X$ defines an element $[\omega]$ of the relative cohomology $H^1(X,\left\{\textrm{zeroes of }\omega\right\};\CC)$. \label{glo_coh} For a sufficiently small neighborhood of a generic point $(X_0,\omega_0)$ the resulting map from $U$ to the relative cohomology yields  local charts on $\Omega\Mg(k_1,...,k_n)$. More details on these charts can be found e.g. in \cite{EKZ11}.\\[11pt]M. Kontsevich and A. Zorich have been able to calculate the connected components of the strata. It turns out that each stratum decomposes into at most $3$ connected components.\\[11pt] One invariant involved is the so-called \textbf{parity of spin structure}. For a general definition we refer the reader to \cite{BL04} and \cite{McM05}. If the spin structure is determined by an Abelian differential with even degrees of zeroes - which is the only case of interest for us -, there is a geometric way to define the spin structure: let $\alpha$ be a smooth simple closed oriented curve on $(X,\omega)$ which does not contain a zero of $\omega$. The index $ind_\alpha \in \ZZ$ is defined as the total change of the angle between the vector tangent to the curve and the vector tangent to the horizontal foliation divided by $2 \pi$ (note that there is always a well-defined notion of horizontal direction on flat surfaces). Now choose oriented smooth paths $(\alpha_i,\beta_i)_{i=1,...,g}$ representing a symplectic basis of $H_1(X,\ZZ/2)$ \label{glo_hom} with respect to the intersection pairing (see \cite{FK92}, Chapter~III.1). Then the parity of spin structure or \textbf{spin invariant} \index{spin invariant} can be defined as \label{glo_spin}
$$\epsilon(X,\omega):= \sum_{i=1}^g (ind_{\alpha_i}+1)(ind_{\beta_i}+1) \mod 2.$$ 
We say that a connected component of $\Omega\Mg(2k_1,...,2k_n)$ has \textbf{even} or \textbf{odd} spin structure depending on whether the corresponding spin invariant of all pairs $(X,\omega)$ in the connected component is even or odd. Finally $\Omega\Mg^{hyp}(k_1,...,k_n)$ is the \textbf{hyperelliptic locus} which consist of all pairs $(X,\omega) \in \Omega\Mg(k_1,...,k_n)$ such that $X$ has an hyperelliptic involution. Having defined all this, let us now state the classification of the connected components of $\Omega\Mg$.

\begin{thm} (\textbf{Kontsevich, Zorich}, \cite{KZ03}) \label{thm_components_of_strata} The strata of $\Omega\Mg$ have at most three connected components, distinguished by the parity of spin structure and by being hyperelliptic or not. For $g\geq4$, the strata $\Omega\Mg(2g-2)$ and $\Omega\Mg(2k,2k)$ with an integer $k=(g-1)/2$ have three components, namely the component of hyperelliptic flat surfaces and two components with odd or even parity of the spin structure but not consisting exclusively of hyperelliptic curves.\\
The stratum $\Omega\mathcal{M}_3(4)$ has two components $\Omega\mathcal{M}_3(4)^{hyp}$ and $\Omega\mathcal{M}_3(4)^{odd}$.\\ 
Also the stratum $\Omega\mathcal{M}_3(2,2)$ has two components, namely $\Omega\mathcal{M}_3(2,2)^{hyp}$ and $\Omega\mathcal{M}_3(2,2)^{odd}$.\\
Each stratum $\Omega\Mg(2k_1,...,2k_r)$ for $k\geq3$ or $r=2$ and $k_1 \neq (g-1)/2$ has two components determined by even and odd spin structure.\\	
Each stratum $\Omega\Mg(2k-1,2k-1)$ for $k\geq2$ has two components, the component of hyperelliptic flat surface $\Omega\Mg(2k-1,2k-1)^{hyp}$ and the other component $\Omega\Mg(2k-1,2k-1)^{non-hyp}$.\\
In all other cases, the stratum is connected. \index{moduli space of Riemann surfaces!stratum|)}\end{thm}

\paragraph{The action of $\SL_2(\RR)$ on $\Omega\Mg$.} After having repeated some of the main properties of $\Omega\Mg$, we can introduce a natural action of $\SL_2(\RR)$ on $\Omega\Mg$: for $A = \left( \begin{smallmatrix} a & b \\ c & d \end{smallmatrix} \right) \in \SL_2(\RR)$ consider the harmonic 1-form
$$\omega'= \begin{pmatrix} 1 & i \end{pmatrix} \begin{pmatrix} a & b \\ c & d \end{pmatrix} \begin{pmatrix} \textrm{Re}(\omega) \\ \textrm{Im}(\omega) \end{pmatrix}$$
on $X$. There is a unique complex structure with respect to which $\omega'$ is holomorphic. Its charts yield a new Riemann surface $X'$ - topologically the surface has not changed at all. We define $A\cdot(X,\omega):=(X',\omega')$. In other words the action can be described by  identifying $\CC$ with $\RR^2$ and letting $A$ act on the charts by linear transformations. Note that the action of $\SL_2(\RR)$ preserves the stratification of $\Omega\Mg$ and even its connected components. Additional information on this topic can be found e.g. in \cite{MT02}, \cite{McM03} and \cite{Möl11a}.
\paragraph{Projection.} For each $\underline{k}=(k_1,...,k_n)$ with $\sum_{i=1}^n k_i = 2g-2$ there is a natural projection from $\Omega\Mg(\underline{k})$ to $\Mg$ which sends $(X,\omega)$ to $X$. The map only remembers the complex structure on the surface defined by the Abelian differential and forgets the Abelian differential itself.

\subsubsection{The Moduli Space of Abelian Varieties $\mathcal{A}_g^D$} \label{sec_moduli_ablian} \label{section_abelian_varieties}

In this section we want to define the moduli space of Abelian varieties. Before we can do so, we give a general overview over the theory of Abelian varieties. In this section we mainly follow the exposition in \cite{BL04}.\\[11pt]Let $V$ be a complex vector space of dimension $g$ and $\Lambda$ be a \textbf{lattice} in $V$. \label{glo_Lambda} By definition $\Lambda$ is a discrete subgroup of rank $2g$ of V. Then the quotient $X=V/\Lambda$  is called a \textbf{complex torus of dimension $g$}. \index{complex torus} An \textbf{homomorphism} of a complex torus $X$ to a complex torus $X'$ is a holomorphic map $f: X \to X'$, compatible with the group structures of the tori. An \textbf{isogeny} is by definition a surjective homomorphism $f: X \to X'$ with finite kernel.  Recall that we can associate to each Chern class $c_1(L)$ \label{glo_c1} of a \textbf{line bundle} $L$ on $X$ (i.e. a vector bundle of dimension 1) a unique alternating form $E: V \times V \to \mathbb{R}$ with $E(\Lambda,\Lambda) \subset \ZZ$ and $E(iv,iw)=E(v,w)$ for all $v,w$ in $V$. Moreover $E$ is known to be the imaginary part of some unique Hermitian form $H$. The line bundle is \textbf{positive (definite)} if $H$ is a positive definite Hermitian form. The group which consists of all such forms is called the \textbf{Néron-Severi group} (for details, see \cite{BL04}, Chapter~1 and 2). \index{Néron-Severi group}According to the elementary divisor theorem there is a basis $\lambda_1,...,\lambda_g, \mu_1,...,\mu_g$ of $\Lambda$, with respect to which $E$ is given by the matrix
$$\begin{pmatrix} 0 & D \\ -D & 0 \end{pmatrix}$$
where $D=\diag(d_1,...,d_g)$ with integers $d_i \geq 0$ satisfying $d_i|d_{i+1}$ for $i=1,...,g-1$ (see e.g. \cite{Mur93}, Chapter~5). The \textbf{elementary divisors} $d_1,...,d_g$ are uniquely determined by $E$ and $\Lambda$ and thus by $L$. Then $\lambda_1,...,\lambda_g,\mu_1,...,\mu_g$ is called a \textbf{symplectic basis} \index{symplectic basis} of $\Lambda$ for $L$ (or $H$ or $E$).

\begin{defi} \begin{itemize}
\item[(i)] A \textbf{polarization} on $X$ is the Hermitian form $H$ associated to the first Chern class $c_1(L)$ of a positive definite line bundle $L$ on $X$. \index{complex torus!polarization}
\item[(ii)] The vector $(d_1,...,d_g)$ is called the \textbf{type} of the polarization of the line bundle $L$ on $X$. Sometimes also $D$ itself is called the type of the polarization.
\item[(iii)] A polarization is called \textbf{principal} if it is of the type $(1,...,1)$.
\end{itemize} 
\end{defi}

By abuse of notation one sometimes calls the line bundle $L$ itself a polarization. Obviously not every polarization is principal. 

\begin{defi} An \textbf{Abelian variety} \index{Abelian variety} is a complex torus $X$ admitting a polarization $H$. The pair $(X,H)$ is called a \textbf{polarized Abelian variety}. 
A \textbf{homomorphism of polarized Abelian varieties} $f: (Y,K=c_1(M)) \to (X,H=c_1(L))$ is a homomorphism of complex tori $f:Y \to X$ such that $f^*c_1(L)=c_1(M)$. \end{defi}

For example every elliptic curve is an Abelian variety while not every complex torus of dimension $\geq 2$ is an Abelian variety (see \cite{BL04}, Chapter~4). An Abelian variety $X$ of dimension 2 is also called an \textbf{Abelian surface}.  \\
Choose bases $e_1,...,e_g$ of $V$ and $\lambda_1,...,\lambda_{2g}$ of $\Lambda$. Writing all the $\lambda_i$ in terms of the basis $e_1,...e_g$, i.e. $\lambda_i=\sum_{j=1}^g \lambda_{j,i} e_j$ yields a matrix
$$\Pi = \begin{pmatrix} \lambda_{1,1} & \cdots & \lambda_{1,2g} \\ \vdots & & \vdots\\ \lambda_{g,1} & \cdots & \lambda_{g,2g} \end{pmatrix}.$$ This matrix $\Pi$ \label{glo_Pi} is called a \textbf{period matrix} \index{period matrix} of $X$. Note that with respect to these bases we have $X=\mathbb{C}^g/\Pi\ZZ^{2g}$. When one looks at period matrices there is an useful criterion which decides whether a complex torus is an Abelian variety or not: 

\begin{thm} (\cite{BL04}, Chapter~4.2) The space $X$ is an Abelian variety if and only if there is a nondegenerate alternating matrix $A \in M_{2g}(\ZZ)$ such that
\begin{itemize}
\item[(i)] $\Pi A^{-1} \Pi^t=0$
\item[(ii)] $i\Pi A^{-1} \overline{\Pi}^t>0$.
\end{itemize}
\end{thm}

The conditions $(i)$ and $(ii)$ are called \textbf{Riemann relations}. An important class of examples of principally polarized Abelian varieties is the following: recall that $H^0(X,\Omega(X))$, the $\CC$-vector space of holomorphic 1-forms on a compact Riemann surface $X$ of genus $g$, has dimension $g$. The homology group $H_1(X,\ZZ)$ is a free abelian group of rank $2g$. We can associate to any element $\gamma \in H_1(X,\ZZ)$ in a canonical way a linear form on $H^0(X,\Omega(X))$, namely $\gamma: \omega \mapsto \int_{\gamma} \omega$ which does by Stokes' theorem not depend on the choice of the representative. This map is injective. The \textbf{Jacobian} \index{Jacobian} of $X$ is then defined as \label{glo_Jac}
$$\Jac(X):=H^0(X,\Omega(X))^\vee/H_1(X,\ZZ)$$
Jacobians of Riemann surfaces of genus~$g \geq 1$ can be canonically polarized and this polarization is principal (see \cite{BL04}, p. 317, for details).\\[11pt]The notion of real multiplication for Abelian varieties, that we define here only for Abelian surfaces, seems to be very technical at first (the general definition can be found e.g. in \cite{Möl11a}, Chapter~4). Nevertheless, its significance will be seen throughout these notes.

\begin{defi} \label{def_real_multiplication} \begin{itemize}
\item[(i)] Let $\End(X)$ \label{glo_End} denote the endomorphism ring of a polarized Abelian variety $X=V/\Lambda$. Elements of $\End(X)$ can be regarded as complex-linear maps $T: V \to V$ with $T(\Lambda) \subseteq \Lambda$. An endomorphism is called \textbf{self-adjoint} if it satisfies $E(Tx,y)=E(x,Ty)$ with respect to the alternating form $E$ on $x,y \in \Lambda$. 
\item[(ii)] \index{real multiplication} Let $\OD$ be a quadratic order. An Abelian surface has \textbf{real multiplication by $\OD$} if there is a monomorphism $\rho: \OD \to \End(X)$ with the following properties:
\begin{itemize}
\item For each $\lambda \in \OD$ the lift $\tilde{\rho}(\lambda): V \to V$ is self-adjoint.
\item $\rho$ is \textbf{proper} in the sense that it does not extend to a monomorphism $\rho': \mathcal{O}_E \to \End(X)$ for some $\mathcal{O}_E \supset \OD$.
\end{itemize}
\end{itemize}
\end{defi}


It is now crucial to parametrize the set of polarized Abelian varieties of a given type $D$ with symplectic basis: there is a bijection between the set of polarized Abelian varieties of type $D$ with symplectic basis $\lambda_1,...,\lambda_g,\mu_1,...,\mu_g$ and the \textbf{Siegel upper half space} \index{Siegel upper half space} \label{glo_Hg}
$$\mathbb{H}_g := \left\{ Z \in M_g(\mathbb{C}) \mid Z^t=Z, \Imag(Z)>0 \right\}.$$	 
The bijective map is explicitly given in \cite{BL04}, Chapter~8.1. However we just want to parametrize isomorphism classes of Abelian varieties. The matrix group
$$\Gamma_D := \left\{M \in M_{2g}(\ZZ) \mid M \begin{pmatrix} 0 & D \\ -D & 0 \end{pmatrix} M^t = \begin{pmatrix} 0 & D \\ -D & 0 \end{pmatrix} \right\}$$
acts properly discontinuously on $\HH_g$ in the following way: if $M= \left( \begin{smallmatrix} a & b\\c &d \end{smallmatrix} \right) \in \Gamma_D$ then the action on $\HH_g$ is given by $$Z \mapsto M \langle Z \rangle:= (aZ+bD)(D^{-1}cZ+D^{-1}dD)^{-1}$$ (see \cite{BL04}, Proposition 8.2.5). This yields the desired result:
\begin{thm} (\cite{BL04}, Chapter~8.1.-8.2.) The normal complex analytic space $\mathcal{A}_g^D:=\mathbb{H}_g/\Gamma_D$ \label{glo_Agd} of dimension $\frac{1}{2}g(g+1)$ is a (coarse) moduli space for polarized Abelian varieties of type $D$. \end{thm}


Whenever $D$ is omitted in the notation of $\mathcal{A}_g^D$ we mean the moduli space of principally polarized Abelian varieties.  One of the main properties of the moduli space of principally polarized Abelian varieties $\mathcal{A}_g$ \label{glo_Ag} is that for $g \geq 2$ the moduli space of compact Riemann surfaces $\Mg$ can be embedded into $\mathcal{A}_g$ by the \textbf{Torelli map}, that assigns to each element $[X] \in \Mg$ its Jacobian $\Jac([X])$. Indeed the following theorem holds:

\begin{thm} \textbf{(Torelli)} \index{Torelli's Theorem} Let $X$ and $X'$ be two Riemann surfaces of genus~$g \geq 2$. Then $X$ and $X'$ are biholomorphically equivalent if and only if their Jacobians $\Jac(X)$ and $\Jac(X')$ are isomorphic as (principally) polarized Abelian varieties. \\
In other words: the Torelli map $j: \Mg \to \mathcal{A}_g$ is injective. \end{thm}

\begin{proof} See e.g. \cite{Wei57}. 
\end{proof}


The following fact is straightforward but nevertheless rather important:
\begin{prop} Let $D_1=(d_1,d_2,...,d_g)$ and $D_2=(1,\frac{d_2}{d_1},...,\frac{d_g}{d_1})$ then $\mathcal{A}_g^{D_1}$ and $\mathcal{A}_g^{D_2}$ are canonically isomorphic. \end{prop}
Note that $D_2$ is also a type of polarization since $d_1|d_i$ for all $i=1,...,g$.
In fact the type of the polarization of an Abelian variety cannot always be seen immediately. Main tools for calculating polarizations are the following two propositions:
\begin{prop} \label{prop_polar_complementary_abelian} (\cite{BL04}, Corollary~12.1.5) Let $(Y,Z)$ be a pair of complementary Abelian subvarieties of a principally polarized Abelian variety $X$ with $\dim Y \geq \dim Z = r$. If $Z$ has polarization of type $(d_1,...,d_r)$ then $Y$ has induced polarization of type $(1,...,1,d_1,...,d_r)$. 
\end{prop}
\begin{prop} \label{prop_polar_cover} (\cite{BL04}, Lemma~12.3.1) Let $X$, $X'$ be two compact Riemann surfaces. Let $\Jac(X')$ be the canonically polarized Jacobian of $X'$ and let $f: X' \to X$ be a covering map of degree $n$.  Then the induced polarization of $\Jac(X)$ as subvariety of $\Jac(X')$ is $(n,...,n)$. \end{prop}
This means that if one knows the type of the polarization of the smaller subvariety $Z$ one also gets the type of the polarization of the greater subvariety $Y$. The second statement implies that $\Jac(X)$ as subvariety of $\Jac(X')$ is not principally polarized, but has $n$-times a principal polarization.

\subsection{Hilbert Modular Surfaces} \label{sec_hilbert_modular_surfaces}

At the end of the 19th century it was one main aim to create a higher dimensional theory of modular forms (O. Blumenthal, D. Hilbert). This development also led to the definition of Hilbert modular surfaces. During the 1970s Hilbert modular surfaces became more and more popular in modern mathematics, notably by the work of F. Hirzebruch and J.-P. Serre.\\[11pt]
This section will serve as a short summary of some of the most important aspects. More comprehensive accounts to Hilbert modular surfaces can be found e.g. in \cite{Bru08}, \cite{Fre90}, \cite{Hir73} and \cite{vdG88}. We remark that in principal many definitions and results in this section can be generalized to higher dimensions (Hilbert modular varieties).\\[11pt]
Let $\OD$ be a real quadratic order and let $K$ be the real quadratic number field containing $\OD$. The group $\SL_2(K)$ is embedded into $\SL_2(\RR) \times \SL_2(\RR)$ by the two real embeddings of $K$ into $\RR$. Hence it acts on $\HH \times \HH$ via
$$\begin{pmatrix} a & b \\ c & d \end{pmatrix} z := \left( \frac{az_1+b}{cz_1+d}, \frac{a^{\sigma}z_2+b^{\sigma}}{c^{\sigma}z_2+d^{\sigma}} \right),$$
where $z=(z_1,z_2)$ is the standard variable in $\HH \times \HH$. If $\mathfrak{a}$ is a fractional ideal of $K$, we write
$$\SL(\OD \oplus \mathfrak{a}) := \left\{ \begin{pmatrix} a & b \\ c & d \end{pmatrix} \in \SL_2(K) | a,d \in \OD , b \in \mathfrak{a}^{-1}, c \in \mathfrak{a} \right\}$$
for the \textbf{Hilbert modular group}, \label{glo_HMG} respectively  $\PSL_2(\OD \oplus \mathfrak{a})$ for $\SL(\OD \oplus \mathfrak{a})$ modulo the invertible diagonal matrices. Since every Hilbert modular group is discrete in $\SL_2(\RR)^2$, it acts properly discontinuously on $\HH^2$ (see \cite{Fre90}, Proposition 2.1). Thus it makes perfectly sense to consider the quotient $\HH \times \HH / \SL(\OD \oplus \mathfrak{a})$. The \textbf{Hilbert modular surface} \index{Hilbert modular surface} $X_D$ \label{glo_XD} is defined as the quotient $\HH \times \HH / \SL(\OD \oplus \mathcal{O}^\vee_D)$. Moreover any Hilbert modular group $\Gamma$ acts on $\mathbb{P}^1(K)$ by fractional linear transformations. The orbits of $\mathbb{P}^1(K)$ under this action are called the \textbf{cusps} of $\Gamma$. Mostly one choses a set of representatives of these orbits and calls, by abuse of notation, also the representatives  cusps. \index{cusp!Hilbert modular group} For every cusp $x$ of $\Gamma$ there exists an element $\rho \in \PGL_2(\RR)$ with $\rho x = \infty = (1:0)$ such that $\rho\Gamma_x\rho^{-1} \subset \PSL_2(K)$ where $\Gamma_x$ is the isotropy group of the cusp. According to \cite{Hir73} we then have an exact sequence
$$0 \to M \to \rho \Gamma_x \rho^{-1} \to V \to 1$$
where $M$ is a \textbf{complete submodule} in $K$, i.e. an additive subgroup of $K$ which is a free abelian group of rank $2$, and $V$ is a subgroup of the group of totally positive units of $\OD$ of rank 1. $M$ and $V$ do neither depend on the choice of $\rho$ nor on the cusp representative $x$ and are therefore called the \textbf{type of a cusp}. \index{cusp!type} The group $\rho \Gamma_x \rho^{-1}=:G(M,V)$ \label{glo_type} is the semi-direct product $M \rtimes V$ (see also \cite{vdG88}, Chapter~2.1).\\[11pt]An element $(\alpha_1,\alpha_2) \in \Gamma$ is called \textbf{elliptic} if $\tr(\alpha_i)^2 - 4 < 0$ for $i=1,2$. A point $z \in \HH^2$ is called an \textbf{elliptic fixed point} for $\Gamma$ if it is the fixed point of an elliptic element of $\Gamma$.  

\begin{prop} (\cite{vdG88}, p.16) If $\alpha \in \SL(\OD \oplus \mathfrak{a} )$ is elliptic then $\tr(\alpha)$ is confined to the following possibilities:
$$0,\pm 1, \pm \sqrt{2}, \pm \sqrt{3}, \pm \frac{1 \pm \sqrt{5}}{2}.$$
The order of $\alpha$ with such a trace as element in $\PSL(\OD \oplus \mathfrak{a})$ equals $2,3,4,6,5$ respectively.
\end{prop} 
If the discriminant $D$ is different from $5,8,12$ then only order 2 and order 3 occur. For $D=5,8$ and $12$ elliptic elements with order $5,4$ and $6$ do exist. The number of cusps and the number of $\Gamma$-inequivalent elliptic fixed points of a Hilbert modular group $\Gamma$ are always finite (see \cite{Fre90}, Chapter~2). \\[11pt]
There is an isomorphism $\SL(\OD \oplus \OD^\vee) \to \SL_2(\OD)$ given by
$$\begin{pmatrix} a & b \\ c & d \end{pmatrix} \mapsto \begin{pmatrix} a & \sqrt{D}/b \\ \sqrt{D}c & d \end{pmatrix}$$
and the map $T: \HH \times \HH \to \HH \times \HH^-$ induces an isomorphism $X_D \cong \HH \times \HH^- / \SL_2(\OD)$. Both isomorphic descriptions of the Hilbert modular surface will be used frequently in these notes. In general these surfaces are not isomorphic to $\HH \times \HH / \SL_2(\OD)$. Only if $\OD$ contains a unit $\epsilon$ with $\epsilon\epsilon^\sigma = -1$ then multiplication by $(\epsilon,\epsilon^\sigma)$ induces an isomorphism
$$\HH \times \HH^- / \SL_2(\OD) \to \HH \times \HH / \SL_2(\OD).$$
Recall that the occurrence of a unit $\epsilon$ with $\epsilon\epsilon^\sigma=-1$ is impossible if $D$ is divisible by a prime $p \equiv 3 \mod 4$ (see e.g. \cite{Ste93}).\\[11pt]
A classical formula which goes back to C.-L. Siegel implies that the Euler characteristic of a Hilbert modular surface $X_D$ is closely related to the Zeta-function of the number field $K$ (see \cite{Hir73}, Chapters~1.3 and 1.4 for details and references).
\begin{thm} \textbf{(Siegel)} Let $K=\QQ(\sqrt{D})$ be a real quadratic number field. Then the Euler characteristic of the Hilbert modular surface $X_D$ is equal to $2\zeta_K(-1)$ where $\zeta_K$ is the Zeta-function of $K$, i.e. the Euler characteristic is given as \label{glo_zeta}
$$\frac{1}{30} \sum_{|b| < \sqrt{D}, b^2 \equiv D \mod 4} \sigma_1\left(\frac{D-b^2}{4}\right),$$
where $\sigma_1(a)$ \label{glo_sigma} is the sum of the divisors of $a$ (in $\ZZ$).
\end{thm}

There is also a moduli interpretation of Hilbert modular surfaces. Indeed, the Hilbert modular surface $X_D$ is the moduli space of all pairs $(X,\rho)$, where $X$ is a principally polarized Abelian surface and $\rho$ is a choice of real multiplication on $X$ by $\OD$. This claim follows from the much more general work of P. Deligne in \cite{Del70}. Maybe this is the reason why it is hard to find a complete explicit proof of this theorem in the existing literature. We present a very detailed proof in Appendix~\ref{appendix_proof} such that our exposition does not need to be interrupted here for quite a few pages.
\begin{thm} \label{thm_Hilbert_modular_moduli} \textbf{(Deligne)} The Hilbert modular surface $X_D$ is the moduli space of all pairs $(X,\rho)$, where $X$ is a principally polarized Abelian surface and $\rho$ is a choice of real multiplication on $X$ by $\OD$. \end{thm}
In the proof a natural map $j: X_D \to \mathcal{A}_2$ which simply forgets the choice of real multiplication is explicitly constructed. When one wants to prove a higher dimensional analog, i.e. $g>2$, of this theorem, there arises another problem: if one defines \textbf{Hilbert modular varieties} as $\HH^g / \Gamma$, where $\Gamma = \SL(\mathcal{O} \oplus \mathcal{O}^\vee)$ for some order $\mathcal{O}$ of a totally real number field $F$ (see e.g. \cite{Fre90}), then the natural map from $\HH^g / \Gamma$ to the moduli space of Abelian varieties is in general not surjective on the locus of Abelian varieties with real multiplication by $\mathcal{O}$ since the real multiplication locus needs not to be connected (see e.g. \cite{MB09} for details).
\paragraph{Automorphisms of $X_D$.} The group of automorphisms of $\HH \times \HH^-$ is given by the semi-direct product $\SL_2(\RR)^2 \rtimes \left\langle \iota \right\rangle$, where $\iota$ is the involution $(z_1,z_2) \mapsto (-z_2,-z_1)$. For a given $M \in \GL_2^+(K)$ it is now natural to ask whether $\psi_M: \HH \times \HH^- \to \HH \times \HH^-$, with $(z_1,z_2) \mapsto (Mz_1,M^\sigma z_2)$, descends to an automorphism of $X_D$. This is the case if and only if $M$ normalizes $\SL_2(\OD)$, i.e. $M \SL_2(\OD) M^{-1} = \SL_2(\OD)$. If $D$ is a fundamental discriminant then $M$ must therefore be in $\SL_2(\OD)$ and so the only automorphism of $X_D$ coming from a matrix $M \in \GL_2^+(K)$ is the identity.\footnote{If $D$ is not fundamental discriminant, then this is not true any more (compare Lemma~\ref{lem_fund_discriminant_quadratic}).} Nevertheless, the group of automorphisms of $X_D$ might be much bigger than the trivial group. An explicit example of an automorphism of $X_8$ not stemming from an automorphism of $\HH \times \HH^-$ is described in \cite{vdG88}, Chapter~VII.2. For a given $D$, it is in general an unsolved problem how to explicitly write down all automorphisms of $X_D$ . 
\paragraph{Hilbert modular forms.} Hilbert modular forms are a higher dimensional analogue of elliptic modular forms. Many parts of the theory of elliptic modular forms can be translated into the language of Hilbert modular forms. Similarly as the former are related to the modular surface $\HH / \SL_2(\ZZ)$, the latter are related to Hilbert modular surfaces. Let $\Gamma$ be a finite index subgroup of a Hilbert modular group. A holomorphic function $f: \HH^2 \to \CC$ is called a \textbf{Hilbert modular form} of weight $\textbf{k}=(k_1,k_2) \in \ZZ^2$ on $\Gamma$ if for all $\gamma = \left( \begin{smallmatrix} a & b \\ c & d \end{smallmatrix} \right) \in \Gamma$ and for all $z=(z_1,z_2) \in \HH^2$ one has
\begin{equation} 
f(\gamma z) = (cz_1+d)^{k_1} (c'z_2+d')^{k_2} f(z).
\end{equation}
If $\textbf{k}$ is of the form $(k,k)$ then one speaks of a Hilbert modular form of weight $k$. It is a striking fact that, in contrast to the case of elliptic modular forms, a holomorphic Hilbert modular form is automatically holomorphic at the cusps. This is known as the \textbf{Koecher principle} (see \cite{Bru08}, Theorem~1.20).  Non-zero holomorphic Hilbert modular forms only exist if $k=(0,0)$ or if the two $k_i$ are both positive (see \cite{vdG88}, Lemma~6.3). Let $f|_k\gamma:= (cz_1+d)^{-k_1}(c'z_2+d')^{-k_2} f(\gamma z)$. A Hilbert modular form $f$ is a \textbf{cusp form} if the constant term $a_0$ in the Fourier series of $f|_k\gamma$ vanishes for all $\gamma \in \GL_2^+(K)$.  If one extends the definition of Hilbert modular forms to half-integral weight one has (as in the case of elliptic modular forms) to deal with characters. For more background on Hilbert modular forms the reader is invited to consult the books \cite{Bru08}, \cite{Fre90} and \cite{vdG88}.

\subsubsection{Special Algebraic Curves on Hilbert Modular Surfaces} \label{sec_special_algebraic_curves}

A plane \textbf{algebraic curve} over a field $k$ is an equation of the form $F(x,y)=0$ for some $F \in k[X,Y]$. A \textbf{point} on an algebraic curve is a pair $(x,y) \in k^2$ with $F(x,y)=0$. A \textbf{nonsingular} algebraic curve is simply an algebraic curve over $k$ which has no singular points over $k$. Here, we are interested in algebraic curves on Hilbert modular surfaces. For sure, the simplest algebraic curve on $X_D$ is the diagonal, i.e. the image of the composition of the map $z \mapsto (z,-z)$ with the quotient map $\pi: \HH \times \HH^- \to X_D$. Another class of examples are the \textbf{twisted diagonals} \index{twisted diagonal} which F. Hirzebruch and D. Zagier implicitly introduced in their paper \cite{HZ76}. These are the algebraic curves in $\HH^2$ defined by the map $z \mapsto (Mz,-M^\sigma z)$ for an arbitrary matrix $M \in \GL_2^+(K)$, i.e. a matrix of totally positive determinant, \label{glo_gl2+} where $M^\sigma$ \label{glo_Msigma} is the Galois conjugate of $M$. Their $\pi$-images are known as \textbf{modular curves} or \textbf{Hirzebruch-Zagier cycles}. Modular curves have been extensively treated in the literature due to their importance for the geometry and arithmetic of Hilbert modular surfaces. Very good references for an overview on modular curves are \cite{vdG88} and \cite{Bru08}. All modular curves have finite volume. There is a simple trick which makes modular curves rather good accessible: a twisted diagonal in $\HH^2$ is namely given by an equation
\begin{equation} \label{equ_twisted_diagonals} \hspace{2.4cm} \begin{pmatrix} z_2 & 1 \end{pmatrix} \begin{pmatrix} a \sqrt{D} & \lambda \\ -\lambda^\sigma & b \sqrt{D} \end{pmatrix} \begin{pmatrix} z_1 \\ 1 \end{pmatrix} = 0
\end{equation}
with $a,b \in \QQ$ and $\lambda \in K$. Because of its form such a matrix is called \text{skew-hermitian}. The appearing matrix is uniquely determined by $M$ up to scalar multiples in $K$ (see \cite{vdG88}, p.88 for details). If $a,b \in \ZZ$ and $\lambda \in \OD$ the skew-hermitian matrix is called \textbf{integral}. If $a,b,\lambda$ do not have a common divisor $m \in \NN$ the matrix is called \textbf{primitive}. Twisted diagonals can be classified by classifying the corresponding skew-hermitian matrices. For $U$ integral skew-hermitian let $F_U$ denote the image of the curve defined by (\ref{equ_twisted_diagonals}) in $X_D$. H.-G. Franke and W. Hausmann showed that the number of integral skew-hermitian matrices which have the same determinant but yield different curves in $X_D$ is finite (see \cite{Fra78} and \cite{Hau80}). For $N \in \NN$ it hence makes sense to look at the (finite volume) curve $F_N$ which is the union of all modular curves that are defined by an integral primitive skew-hermitian of determinant $N$. If we omit the primitivity condition in the definition we get the curve $T_N$. \label{glo_TN} Also the curves $F_N$ and $T_N$ are of great interest for the arithmetic and geometry of $X_D$ (see e.g. \cite{vdG88} and \cite{Bru08}) and are closely linked to modular forms. This link is given by the intersection numbers of the $T_N$ (see \cite{HZ76}). Also the volume of all these curves is explicitly known (see \cite{Hau80}, Satz~3.10 and Korollar~3.11 or \cite{vdG88}, Theorem~V.5.1). Furthermore, it is known that the stabilizers of the curves $F_N$ inside $\SL_2(\OD)$ are the unit groups of certain quaternion rings (see e.g. \cite{Fra78}, Theorem~2.3.8).\\[11pt] Another example for the importance of these curves is the following: let $\nu(U)$ denote the image of $\lambda$ in the ring $\OD/\mathcal{D}_D$ where $\mathcal{D}_D$, the \textbf{different}, is the principal ideal $(\sqrt{D})$. For $\nu \in \OD/\mathcal{D}_D$, let $T_N(\nu)$ denote the union of all $F_U$ where $U$ is an integral skew-hermitian matrix with $\det(U)=N$ and $\nu(U) = \pm \nu$. Note that $T_N(\nu) \neq \emptyset$ if and only if $\N(\nu)=-N \mod D$. Finally, let $P_D$ denote the \textbf{reducible locus} \label{glo_PD} \index{reducible locus} in $X_D$, i.e. the locus of all abelian varieties which are products of elliptic curves. 
\begin{thm} (\textbf{McMullen}, \cite{McM07}, Corollary~3.5) The reducible locus $P_D$ is given by
$$P_D = \bigcup_N \left\{T_N((e+\sqrt{D})/2): e^2 + 4N =D \right\}.$$
\end{thm}
Note that if $D$ is prime, then $P_D$ can be written as the union of some $F_N$. In general, this is not possible (see \cite{McM07}, chapter 3). \\[11pt]For twisted diagonals both components of the universal covering map are given by Möbius transformations. An algebraic curve $C \to X_D$ is still \textit{rather special} if only (at least) one of the components of the universal covering map $\HH \to \HH^2$ is a Möbius transformation. Equivalently, one may ask that $C \to X_D$ is totally geodesic for the \textbf{Kobayashi metric} (see e.g. \cite{MV10}).\footnote{The reason is the following: a holomorphic map $f: \HH \to \HH$ is a Kobayashi geodesic if and only if it is a Möbius transformation. By the product property of the Kobayashi metric (see e.g. \cite{Kob70}, Proposition~IV.1.5) the claim then follows.}\index{Kobayashi metric} For a complex domain $W$, the Kobayashi metric $k_W(.,.)$ \label{glo_KW} is defined as follows (see \cite{Kob67}): let $\rho$ be the Poincaré metric on the unit disc $\mathbb{D}$. For all $x,y \in W$ we call a \textbf{chain} from $x$ to $y$ points $x_0,x_1,...,x_n \in \mathbb{D}$ together with holomorphic maps $f_i: \mathbb{D} \to W$ such that
$$f_1(x_0)=x , \ \ \ f_j(x_j)=f_{j+1}(x_j), \ \ j=1,...,n-1, \ \ f_n(x_n)=y.$$
Then 
$$k_W(x,y) := \inf \sum_{i=1}^n \rho(x_{i-1},x_i),$$
where the infimum is taken over all chains from $x$ to $y$. In general, the Kobayashi metric is only a pseudo metric although on $X_D$ the Kobayashi metric is known to be a metric. The main property of the Kobayashi metric is that it is distance-decreasing for all holomorphic maps. A very well written introduction to the Kobayashi metric can be found in \cite{Kra90}. A more profound reference is \cite{MV10}. An algebraic curve that is a Kobayashi geodesic is called a \textbf{Kobayashi curve}. \index{Kobayashi curve} Very few examples of Kobayashi curves on $X_D$ other than twisted diagonals are known so far. Nevertheless, it is known by the work of E. Viehweg and M. Möller that each Kobayashi curve is defined over a number field (\cite{MV10}, Corollary~6.2). In the next chapter we will see how Teichmüller curves yield some of those few examples. The main aim of these notes is to construct totally new examples of Kobayashi curves, i.e. also different from Teichmüller curves, and to analyze some of their properties.





\newpage 

\section{Teichmüller Curves} \label{cha_Teichmüller_curves}

Teichmüller curves (or at least the first non-arithmetic examples) were introduced in the groundbreaking paper \cite{Vee89} by W. A. Veech. His main motivation for studying Teichmüller curves came from their relation to billiard dynamics. In Section~\ref{sec_definition_tmcurves} we recall the definition of Teichmüller curves and some of their main properties. A well-known class of examples of Teichmüller curves stems from square-tiled surfaces. In his paper \cite{McM03}, C. McMullen found another class of examples of Teichmüller curves which are \textit{independent} (in a sense to be made more precise in this chapter) from square-tiled surfaces.\footnote{This class of examples was independently also found by K. Calta in \cite{Cal04}.} Moreover he showed that each of these examples lies on a unique Hilbert modular surface (Theorem~\ref{thm_mcm_embedding}). We recapitulate his ideas in Section~\ref{sec_mcmullen_examples1}. Teichmüller curves are generated by a flat surface $(X,\omega)$ which is unique only up to the $\SL_2(\RR)$-action on $\Omega\Mg$. Therefore, the Veech group of the generating surface is unique only up to conjugation. When one wants to consider a Teichmüller curve as a curve on the Hilbert modular surface one thus has to carefully choose an appropriate generating surface. We discuss this problem in Section~\ref{sec_fixing_Veech}.\\ Unfortunately, there is not a standard reference for Teichmüller curves yet. This is why we gather together important facts. Our main references are \cite{Bai07}, \cite{Kap11}, \cite{MB09}, \cite{McM03}, \cite{McM05} and \cite{Möl11a}.

\subsection{Definition and Main Properties} \label{sec_definition_tmcurves}
\begin{defi} A \textbf{Teichmüller curve} \index{Teichmüller curve} $C \to \Mg$ is an algebraic curve in $\Mg$ that is totally geodesic with respect to the Teichmüller metric. \end{defi}
A map $f: X \to Y$ is said to be \textbf{totally geodesic} if for each geodesic $x_t$ in $X$ the image $f(x_t)$ is a geodesic in $Y$. At first glance the definition of Teichmüller curves seems to be rather abstract. However it is well-known (see e.g. \cite{Möl11a}) that all Teichmüller curves arise as the projection of an $\SL_2(\RR)$-orbit of a flat surface $(X,\omega)$ or a half-translation surface $(X,q)$ to moduli space. On the contrary, the projection of the $\SL_2(\RR)$-orbit of a flat surface $(X,\omega)$ (respectively $(X,q)$) to $\Mg$ yields a Teichmüller curve if and only if $(X,\omega)$ (respectively $(X,q)$) is a Veech surface. Hence it stems from a dynamically optimal billiard table (see e.g. \cite{MT02}).\footnote{A billiard table is called dynamically optimal if each billiard trajectory which avoids the corners is either periodic or dense.} In this case we say that $(X,\omega)$ (respectively $(X,q)$) \textbf{generates} the Teichmüller curve. We will in the following restrict ourselves to the case of translation surfaces.\footnote{This is possible to due to the canonical double covering \index{canonical double cover} construction which will be described in Section~\ref{sec_no_twist_Lyapunov}; see e.g. \cite{BM10b} for more details on this fact.}.\\[11pt]Note that the action of the group of rotations $\SO_2(\RR) \subset \SL_2(\RR)$ may change the Abelian differential without touching the Riemann surface structure. We identify $\SO_2(\RR) \backslash \SL_2(\RR)$ with $\HH$ by $\SO_2(\RR) \cdot \rm{B} \mapsto -\overline{B^{-1}(i)}$ (compare \cite{HS07}, Chapter~2.3.2). Hence a Teichmüller curve stems from the unique map $\widetilde{f}: \SO_2(\RR) \backslash \SL_2(\RR) \cong \HH \to \Mg$ which makes the diagram 
$$
\begin{xy}
 \xymatrix{ \SL_2(\RR) \ar[r]^{F} \ar[d] & \Omega\Mg \ar[d]_{\pi} \\ \SO_2(\RR) \backslash \SL_2(\RR) \ar[r]^{\ \ \ \ \ \ \ \ \widetilde{f}} & \Mg\\
 	}
\end{xy}
$$
commute, where $F(A) = A \cdot (X,\omega)$.  More concretely if $z \in \HH$ then $\widetilde{f}(z)$ is defined as $\pi(A_z\cdot(X,\omega))$ where $A_z: \CC \to \CC$ is given by $\left( \begin{smallmatrix} 1 & \Real(z)\\ 0 & \Imag(z) \end{smallmatrix} \right)$ (see \cite{McM03}, Chapter~3). The stabilizer of the curve $\widetilde{f} : \HH \to \Mg$ coincides with the Veech group of $(X,\omega)$ and thus $\widetilde{f}$ factors through $\HH / \SL(X,\omega)$ and yields the Teichmüller curve $f: \HH / \SL(X,\omega) \to \Mg$. Note that the Veech group does not act by Möbius transformations on $\HH$ under this construction. So does only $\textrm{R} \hspace{1.5pt} \SL(X,\omega)\textrm{R}^{-1}$ where $\textrm{R} = \left( \begin{smallmatrix} -1 & 0 \\ 0 & 1 \end{smallmatrix} \right)$ (compare \cite{McM03}, Proposition~3.2.).\\[11pt]
The simplest examples of Teichmüller curves are generated by square-tiled surfaces. 
The Veech group of the torus $E:=\CC/(\ZZ \oplus i \ZZ)$ is $\SL_2(\ZZ)$ and thus the torus generates a Teichmüller curve in $\mathcal{M}_1$. In fact, every square-tiled surface generates a Teichmüller curve: recall that a \textbf{translation covering} $\pi: (X,\omega) \to (Y,\eta)$ is a covering $\pi: X \to Y$ of compact Riemann surfaces such that $\omega = \pi^* \eta$. 
\begin{thm} \label{thm_Gutkin_Judge2} (\textbf{Gutkin, Judge,} \cite{GJ00}) \index{Theorem of Gutkin-Judge} Let $\pi: (X,\omega) \to (Y,\eta)$ be a translation covering. If $Z(\eta)\neq \emptyset$ and $\pi$ is branched over the zeroes of $\eta$ or if the genus $g(X)=1$ and $\pi$ is branched over at most one point then the corresponding Veech groups $\SL(\textit{X},\omega)$ and $\SL(\textit{Y},\eta)$ are commensurable. \end{thm}
Hence square-tiled surfaces give a very big family of Teichmüller curves. Indeed for every $g$ and each connected component of every stratum of $\Omega\Mg$ Teichmüller curves generated by square-tiled surfaces are dense (see e.g. \cite{Möl11a}, Proposition~5.3). 
\paragraph{Primitive Teichmüller curves.} Square-tiled surfaces can be regarded as one big class of examples of Teichmüller curves. A Teichmüller curve in $\Mg$ is called \textbf{(geometrically) primitive} \index{Teichmüller curve!(geometrically) primitive} if it does not arise from a curve in $\mathcal{M}_h$ with $h<g$ via a branched covering construction. More precisely, we call a Teichmüller curve (geometrically) primitive if it is generated by a Veech surface $(X,\omega)$ and there does not exist a translation covering $\pi: (X,\omega) \to (Y,\eta)$ with $g(Y)<g(X)$, where $g(\cdot)$ is the genus. \label{glo_g(X)} All Teichmüller curves stemming from square-tiled surfaces (beside the torus) are hence non-primitive Teichmüller curves. Since every Teichmüller curve has a cusp (see \cite{Vee89} or e.g. \cite{MB09}, Chapter~1), it follows from Theorem~\ref{thm_Gutkin_Judge2} and Theorem~\ref{thm_nonarithmetic_commensurable} that all primitive Teichmüller curves in genus~$g \geq 2$ have non-arithmetic Veech groups. Moreover, every Teichmüller curve has a unique primitive representative in its commensurability class (see \cite{McM05}, Chapter~2). We call a Teichmüller curve \textbf{algebraically primitive} \index{Teichmüller curve!algebraically primitive} if it arises from a flat surface $(X,\omega)$ with $g(X)=[K:\QQ]$ where $K$ is the trace field of $\SL(X,\omega)$. The trace field of a Veech group is well-defined, i.e. does not depend on the explicit choice of the generating surface (\cite{McM03b}, Corollary~9.6). \index{Fuchsian group!trace field} Algebraically primitivity implies primitivity (see \cite{McM03}, Theorem 5.1). The converse does not hold (see Chapter~\ref{chapter_prym_modular} for infinitely many counter examples).
\paragraph{Computational aspects.} In general, it is an unsolved problem how to calculate the Veech group of a flat surface $(X,\omega)$. In her thesis, G. Weitze-Schmithüsen \cite{Sch05} describes an algorithm for finding the Veech group of every square-tiled surface. In \cite{McM03}, C. McMullen gives another algorithm for calculating the Veech group of Teichmüller curves generated by a flat surface $(X,\omega)$ which has some mirror symmetry (as the examples in Chapter~\ref{sec_mcmullen_examples1} have). Nevertheless, this algorithm works only in very few special cases: for instance the genus of the Teichmüller curve has to be equal to $0$. J. Bowman communicated to the author of these notes that he has almost found an algorithm for finding the Veech group of all surfaces that are known to be Veech surfaces (e.g. because they are known to generate a Teichmüller curve) using Iso-Delauney triangulations. Although his idea is very promising it is not yet proven to work in general. If the Veech group is known to be a subgroup of a countable group $\SL_2(S)$ and if its (co-)volume is known, then D. Zagier suggested to calculate the Veech group by brute force, i.e. by checking for each element $A \in \SL_2(S)$ if it lies in the Veech group until one has found a set of generators. Finally, in \cite{Muk12}, R. Mukamel has found an algorithm for computing generators of the Veech group of those Teichmüller curves which will be described next.

\subsection{Examples in Genus~$2$} \label{sec_mcmullen_examples1}

In \cite{Vee89}, W. A. Veech showed that every regular polygon is a Veech surface and hence gives rise to a Teichmüller curve. Many years later in \cite{McM03}, C. McMullen found a big class of  examples of primitive Teichmüller curves different from square-tiled surfaces using L-shaped polygons. These examples were independently also found by K. Calta in \cite{Cal04}. Let us recapitulate the construction: an \textbf{L-shaped billiard table} \index{L-shaped billiard table} $P$ is obtained by removing a small rectangle form the corner of a larger rectangle. It can be normalized by a linear transformation $A \in \GL_2^+(\RR)$ such that it is of the form $P(a,b)$ shown in the following figure:

\begin{center}
\psset{xunit=1cm,yunit=1cm,runit=1cm}
\begin{pspicture}(-0.5,-0.5)(4.5,3.5) 


\psline[linewidth=0.5pt,showpoints=false]{-}(0,0)(4,0)
\psline[linewidth=0.5pt,showpoints=false]{-}(4,0)(4,3)
\psline[linewidth=0.5pt,showpoints=false]{-}(4,3)(3,3)
\psline[linewidth=0.5pt,showpoints=false]{-}(3,3)(3,1)
\psline[linewidth=0.5pt,showpoints=false]{-}(3,1)(0,1)
\psline[linewidth=0.5pt,showpoints=false]{-}(0,1)(0,0)

\uput[l](0,0.5){1}
\uput[d](2,-0,05){b}
\uput[r](4.0,1.5){a}
\uput[u](3.5,3){1}
\end{pspicture}
\\ Figure 3.1. An L-shaped billiard table of the form $P(a,b)$.
\end{center}

The so-called \textbf{Zemljakov-Katok construction} or \textbf{unfolding construction)}, \index{Zemljakov-Katok construction} which was first described in \cite{ZK75}, associates to an L-shaped billiard table a point $(X,\omega) \in \Omega\Mg$: in general, let $P \subset \RR^2 \cong \CC$ be a compact polygon whose interior angles are rational multiples of $\pi$.  Let $G$ be the finite subgroup of the orthogonal group $O_2(\RR)$ \label{glo_O2} generated by the linear parts of the reflections in the sides of $P$. Then we define 
$$X = \left( \coprod_{g \in G} g \cdot P \right) / \sim$$
where $\sim$ is an equivalence relation defined by gluing edges. Glue each edge $E$ of $g \cdot P$ to the edge $r\cdot E$ of $r \cdot g \cdot P$ by a translation, where $r \in G$ is reflection through $E$ (compare \cite{McM03}). The holomorphic $1$-form $dz$ on $\CC$ is translation-invariant and hence descends to a natural $1$-form $\omega$ on $X$. The only possible zeroes of $\omega$ come from the vertices of $P$. 
\begin{center}
\psset{xunit=1.5cm,yunit=1.5cm,runit=1.5cm}
\begin{pspicture}(-0.5,-0.5)(6,2.2)


\psline[linewidth=0.5pt,showpoints=false]{-}(0,0.5)(1.5,0.5)
\psline[linewidth=0.5pt,showpoints=false]{-}(1.5,0.5)(1.5,1.5)
\psline[linewidth=0.5pt,showpoints=false]{-}(1.5,1.5)(1.0,1.5)
\psline[linewidth=0.5pt,showpoints=false]{-}(1.0,1.5)(1.0,1.0)
\psline[linewidth=0.5pt,showpoints=false]{-}(1.0,1.0)(0,1.0)
\psline[linewidth=0.5pt,showpoints=false]{-}(0,1)(0,0.5)


\psline[linewidth=0.5pt]{->}(2.0,1.0)(2.5,1.0)


\psline[linewidth=0.5pt]{-}(3.0,1.0)(6.0,1.0)
\psline[linewidth=0.5pt]{-}(4.5,0)(4.5,2.0)
\psline[linewidth=0.5pt]{-}(3.0,0.5)(3.0,1.5)
\psline[linewidth=0.5pt]{-}(6.0,0.5)(6.0,1.5)
\psline[linewidth=0.5pt]{-}(4.0,2.0)(5.0,2.0)
\psline[linewidth=0.5pt]{-}(4.0,0)(5.0,0)
\psline[linewidth=0.5pt]{-}(3.0,1.5)(4.0,1.5)
\psline[linewidth=0.5pt]{-}(5.0,1.5)(6.0,1.5)
\psline[linewidth=0.5pt]{-}(3.0,0.5)(4.0,0.5)
\psline[linewidth=0.5pt]{-}(5.0,0.5)(6.0,0.5)
\psline[linewidth=0.5pt]{-}(4.0,1.5)(4.0,2.0)
\psline[linewidth=0.5pt]{-}(5.0,1.5)(5.0,2.0)
\psline[linewidth=0.5pt]{-}(4.0,0.0)(4.0,0.5)
\psline[linewidth=0.5pt]{-}(5.0,0.0)(5.0,0.5)

\end{pspicture}
\\ Figure 3.2. Unfolding of an L-shaped billiard table (from \cite{McM03}).
\end{center}

Note that the associated flat surface $(X,\omega)$ of an L-shaped polygon is always in $\Omega\mathcal{M}_2(2)$.

\begin{thm} \label{thm_L_shaped} (\textbf{McMullen}, \cite{McM03}, Theorem~9.2) The Veech group of the L-shaped polygon $P(a,b)$ \label{glo_billiard} is a lattice if and only if $a$ and $b$ are rational or
$$a=x+z\sqrt{D} \ \ \ \textrm{and} \ \ \ b=y+z\sqrt{D}$$
for some $x,y,z \in \QQ$ with $x+y=1$ and $D \geq 0$ in $\ZZ$. In the latter case the trace field of $\SL(X,\omega)$ is $\QQ(\sqrt{D})$. In particular the Veech of group of $P(a,a)$ is a lattice if and only if $a$ is rational or $a=(1 \pm \sqrt{D})/2$ for some $D \in \QQ$. \end{thm}

For $D \equiv 1 \mod 4$, the polygons $P(a,a)$ with $a = (1 + \sqrt{D})/2$ yield an infinite collection of primitive Teichmüller curves in $\Omega\mathcal{M}_2(2)$. For $D \equiv 0 \mod 4$ we consider the polygons $P(\sqrt{D}/2,1+\sqrt{D}/2)$, which also yield an infinite collection of primitive Teichmüller curves. The \textbf{discriminant} of such a Teichmüller curve is defined as the discriminant of the quadratic order generated by $1$ and $a$, i.e. equal to $D$ (compare also \cite{Bai07}, Chapter~1). Let $\SL(L_D^1)$ \label{glo_SLLD1} from now on denote the Veech group of these polygons. To each of them one can associate a unique real quadratic order and a unique real quadratic number field, namely the trace field of the Veech group.\\[11pt]
One might ask if these are all primitive Teichmüller curves in $\mathcal{M}_2$. This is partially true. In \cite{McM05} and \cite{McM06b}, C. McMullen succeeded to completely classify Teichmüller curves in $\mathcal{M}_2$. Before we state this theorem correctly, recall the definition of the spin invariant \index{spin invariant} from Section~\ref{sec_strata}. Again we may restrict our attention to Abelian differentials with zeroes of even degree - for more details on the general case see \cite{BL04}, \cite{KZ03} or \cite{McM05}. The spin invariant of the described L-shaped billiard tables is always odd.
\begin{thm} \label{thm_classification_Teichmüller_genus_2} (\textbf{McMullen}, \cite{McM05}, Corollary~1.2, \cite{McM06b}, Theorem~6.1.) Every Teichmüller curve generated by a form $(X,\omega) \in \Omega \mathcal{M}_2(2)$ is up to isomorphism determined by the discriminant $D$ and, if $D \equiv 1 \mod 8$, by its spin invariant $\epsilon(X,\omega) \in \ZZ/2$. The only primitive Teichmüller curve in $\Omega\mathcal{M}_2(1,1)$ is generated by the regular decagon.\end{thm}
We denote the corresponding Teichmüller curves by $C_{L,D}^\epsilon$ with $\epsilon \in \left\{ 0 ,1 \right\}$ if we want to stress their origin. \label{glo_CLD} By abuse of notation, we will often also say that the Teichmüller curves of discriminant $D \equiv 5 \mod 8$ are of odd spin, because they are also generated by symmetric polygons and therefore behave quite similarly as the odd spin Teichmüller curves. 
\paragraph{Prototypes.} Indeed, an explicit construction of all Teichmüller curves in $\Omega\mathcal{M}_2(2)$ is given in \cite{McM05} using the so-called \textbf{prototypes}. We refer the reader who is interested in details to this paper. Here, we just describe a generating surface for the Teichmüller curve of even spin: it is given by the L-shaped polygon $P(\sqrt{D}/2+3/2,\sqrt{D}/2-1/2)$.\footnote{In the notation of the paper \cite{McM05} this surface corresponds to the prototype $(0,(D-1)/4,1,1)$.} The corresponding Veech group will be denoted by $\SL(L_D^0)$. \label{glo_SLLD0}Whenever we are not interested which of the Teichmüller curves we actually look at, we just write $\SL(L_D)$.\label{glo_SLLD}\\[11pt]
It is known that such a Teichmüller curve with discriminant $D$ lies on the unique Hilbert modular surface $X_D$. We postpone an even more detailed explanation of this fact until Chapter~\ref{chapter_prym_modular}, where we revisit the described examples in a more general setting. At this point, let us just present the following theorem and sketch the proof. We use here the second model of $X_D$, namely $\HH \times \HH^{-} / \SL_2(\OD)$. 

\begin{thm} (\textbf{McMullen}, \cite{McM03}) \label{thm_mcm_embedding} Let $f: C_{L,D}^\epsilon \to \mathcal{M}_2$ be one of the Teichmüller curves of discriminant $D$ generated by an L-shaped polygon. Then $C_{L,D}^\epsilon$ lies on the Hilbert modular surface \index{Hilbert modular surface} $X_D$. More precisely, we have the following commutative diagram: \label{glo_Phi}
 $$
\begin{xy}
 \xymatrix{
 	\mathbb{H} \ar@{^(->}[rr]^{\Phi(z):=(z,\varphi(z)) \hspace{15pt}} \ar[d]^{/\SL(L_D)}& &	\mathbb{H} \times \mathbb{H}^{-} \ar[d]^{/\SL_2(\OD)}	\ar[d] \\ 
 	C_{L,D}^\epsilon \ar@{^(->}[rr] \ar[d]_{f} & & X_D \ar[d]  \\
 	\mathcal{M}_2 \ar@{^(->}[rr]^{\Jac} & & \mathcal{A}_2
 	}
\end{xy}
$$
where $\SL(L_D)$ is the Veech group of the generating surface. Moreover $\varphi$ is holomorphic but not a Möbius transformation. \end{thm}

\begin{proof}[Sketch of the proof] We now explain the main ideas of the proof, following \cite{McM03}, in particular Chapters~7,8 and 10. By mapping each point $X$ of the Teichmüller curve to its Jacobian $\Jac(X)$ one gets an embedding of the Teichmüller curve into the space of principally polarized Abelian surfaces $\mathcal{A}_2$. All these Jacobians have real multiplication by $\OD$ since the trace field of $\SL(L_D)$ is not $\QQ$ (\cite{McM03}, Theorem~7.1). As the Hilbert modular surface parametrizes all principally polarized Abelian surfaces with real multiplication by $\OD$ (see Theorem~\ref{thm_Hilbert_modular_moduli}) one gets indeed an embedding of the Teichmüller curve into $X_D$. Since Teichmüller curves are Kobayashi curves, it is clear that the map $\Phi$ is of the form as described in the statement of the theorem. The map $\Phi$ can also be constructed explicitly as follows: Let $(X_t,\omega_t) = A_t (X,\omega)$ so that $X_t =f(t)$. Then $\Phi$ is of the form
$$\Phi(t) = \left( \frac{\int_{b_1} \eta_1(t)}{\int_{a_1} \eta_1(t)},  \frac{\int_{b_2} \eta_2(t)}{\int_{a_2} \eta_2(t)}, \right),$$
where $(a_i,b_i)$ is a symplectic basis for $H_1(X,\RR)$ adapted to the action of $\OD$ and $(\eta_1(t),\eta_2(t))$ is a dual eigenbasis of $H^0(X_t,\Omega(X_t))$ (see \cite{McM03}, Theorem~10.1 or Section~\ref{section_abelian_varieties}). After a change of the base point, $\Phi$ is then again of the form as mentioned in the theorem. \end{proof}

Note that choosing another surface $(Y,\eta)=\textrm{M}(X,\omega)$ with $\textrm{M} \in \SL_2(\RR)$ as generating surface of the Teichmüller curve yields a slightly different Veech group, namely $\textrm{M}\SL(L_D)M^{-1}$.
\paragraph{Euler characteristic.} In his thesis \cite{Bai07}, M. Bainbridge calculated the Euler characteristic of all of the Teichmüller curves in $\Omega\mathcal{M}_2(2)$. He showed that it is proportional to the Euler characteristic of the Hilbert modular surface on which it lies.

\begin{thm} (\textbf{Bainbridge}, \cite{Bai07}, Theorem~1.1) \label{thm_bain_euler} If $D$ is not a square, then the Euler characteristic of a Teichmüller curve $C_{L,D}^\epsilon \to \mathcal{M}_2$ of discriminant $D$ is given as
$$\chi(C_{L,D}^\epsilon) = -\frac{9}{2} \chi(X_D)$$
if $D \not \equiv 1 \mod 8$ and
$$\chi(C_{L,D}^\epsilon) = -\frac{9}{4} \chi(X_D)$$
if $D \equiv 1 \mod 8$.
\end{thm}

\paragraph{Elliptic Fixed points.} By the work of R. Mukamel \cite{Muk11}, also the number of elliptic fixed points of such Teichmüller curves is known. It is the weighted sum of class numbers of some imaginary quadratic number fields. Moreover for $D>8$ there do not arise any elliptic fixed points of order other than two. We only repeat his results for the case that $D$ is not a square.

\begin{thm} \label{thm_mukamel} (\textbf{Mukamel}, \cite{Muk11}, Theorem~5.5) For $D>8$, all elliptic fixed points of a Teichmüller curve $C_{L,D}^\epsilon$ of discriminant $D$ have order two. The number $e_2(C_{L,D}^\epsilon)$ of such points is given as 
\begin{center} \begin{tabular}{|c|c|}
\hline
$D \mod 16$ &  $e_2(C_{L,D}^\epsilon)$\\
\hline
$1,5,9$ or $13$ & $\frac{1}{4} h_{-4D}$\\
\hline
$0$ & $\frac{1}{2}(h_{-D} + 2 h_{-\frac{D}{4}})$\\
\hline
$4$ & $0$\\
\hline
$8$ & $\frac{1}{2}{h_{-D}}$\\
\hline
$12$ & $\frac{1}{2}(h_{-D} + 3 h_{-\frac{D}{4}})$\\
\hline
\end{tabular}
\end{center}
where $h_{-E}$ is the class number of the imaginary quadratic number field of discriminant $-E$ unless $E=3$ or $4$. In the latter case $h_{-E}$ is half of the class number of the imaginary quadratic number field of discriminant $-E$.
\end{thm}

\subsection{Fixing the Veech Group} \label{sec_fixing_Veech} \index{Veech group}
We now come to the problem of the correct choice of the Veech group. As we have explained on the preceding pages, it follows from the results in \cite{McM03} that there exists $\textrm{M} \in \SL_2(\RR)$ such that $\textrm{M}\SL(L_D)M^{-1}\subset \SL_2(\OD)$. Indeed, we may choose $\textrm{M}=\Id$ for all $D$. This is true for both, $\SL(L_D^1)$ and $\SL(L_D^0)$.
\paragraph{D $\equiv$ 1 mod 4, odd spin.} Let us first analyze the Veech groups of the Teichmüller curves generated by the symmetric L-shaped polygons. Recall our abuse of notation that we include the Teichmüller curves of discriminant $D \equiv 5 \mod 8$ in this case.
\begin{prop} \label{prop_fix_veech_1} For all discriminants $D \equiv 1 \mod 4$ the Veech group $\SL(L_D^1)$ is a subgroup of $\SL_2(\OD)$. \end{prop}
\begin{proof}
Let $D \equiv 1 \mod 4$. Integrating $\omega$ on the L-shaped polygon yields obviously that the periods of $(X,\omega)$ are $1$, $w$, $i$ and $iw$ (compare e.g. \cite{Sil06}). Regarding $\CC$ as $\RR^2$, the Veech group maps by definition the periods to $\ZZ$-linear combinations of the periods.  Hence the Veech group is contained in $\SL_2(\OD)$ since $\OD$ is generated by $1$ and $w$. \end{proof}
One can at least calculate parts of this Veech group.  Immediately one observes that for $D \equiv 1 \mod 4$ \label{glo_S} \label{glo_T}
$$T := \begin{pmatrix} 1 & w \\ 0 & 1 \end{pmatrix} \ \ \ \textrm{and} \ \ \ S:= \begin{pmatrix} 0 & -1 \\ 1 & 0 \end{pmatrix}$$
and hence also $Z:=-STS$ are in $\SL(L_D^1)$. \label{glo_Z} 
Since the parabolic elements of Veech groups correspond to periodic directions on a Veech surface, formulas depending only on $D$ can be given for some more parabolic elements of the Veech group. A calculation for the cylinder decomposition in direction $(1/2w,1)$ implies for $D>5$:
\begin{lem} \label{lem_second_parabolic} If $D>5$ then the following matrices lie in $\SL(L_D^1)$:
\begin{itemize}
\item If $D \equiv 5 \mod 8$: 
$$\begin{pmatrix} 1 - 2 w(w+1) & w^2(w+1) \\ -4 (w+1) & 1 + 2 w(w+1) \end{pmatrix}$$
\item If $D \equiv 1 \mod 16$: 
$$\begin{pmatrix} 1 - w(w+1) & \frac{1}{2}w^2(w+1) \\ -2(w+1) & 1 + w(w+1) \end{pmatrix}$$
\item If $D \equiv 9 \mod 16$:
$$\begin{pmatrix} 1 - \frac{1}{2} w(w+1) & \frac{1}{4}w^2(w+1) \\ -(w+1) & 1 + \frac{1}{2} w(w+1) \end{pmatrix}.$$
\end{itemize} \end{lem}
\begin{center}
\includegraphics[height=4.6cm,width=4.6cm]{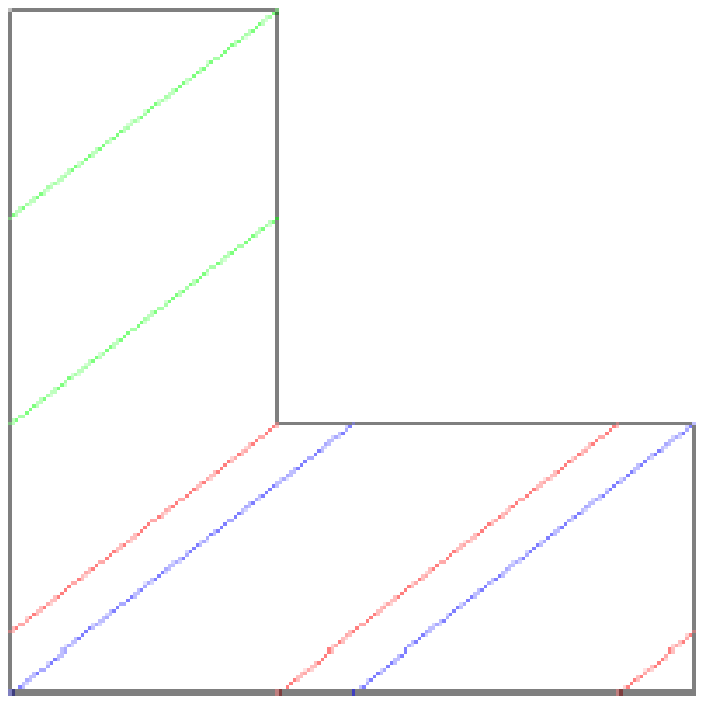} \hspace{2cm}
\includegraphics[height=4.5cm,width=4.5cm]{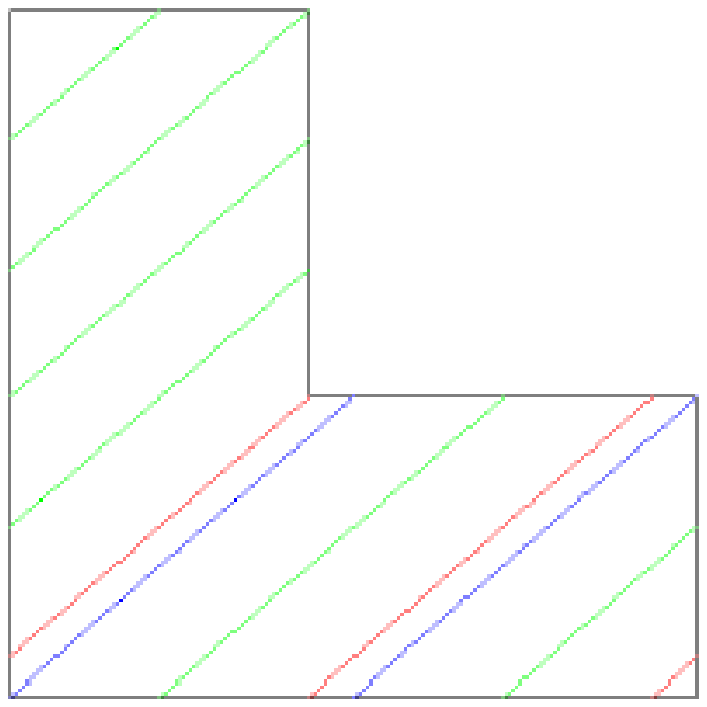}\\
Figure 3.3. Two typical cylinder decompositions in direction $(1/2w,1)$ in the cases $D \equiv 1 \mod 8$ and $D \equiv 5 \mod 8$.
\end{center}
Let us say a few more words about this cylinder decomposition: It can be calculated that in the case $D \equiv 1 \mod 8$ the lengths of the cylinders are $2$ and $w+1$ while the heights of the cylinders are $(w-2)/w$ and $1/w$. Thus the ratio of the moduli of the two cylinders is always $\frac{D-9}{16}$. If $D \equiv 5 \mod 8$ then the lengths of the cylinders are $2$ and $2(w+1)$ while their heights are again $(w-2)/w$ and $1/w$. Thus the ratio of the moduli of the two cylinders is always $\frac{D-9}{4}$.\footnote{In the notation of C. McMullen in \cite{McM05} the cusp $(w/2,1)$ thus corresponds to the prototype $(0,\frac{D-9}{8},2,-3)$ if $D \equiv 1 \mod 16$ and to the prototype $(0,\frac{D-9}{4},2,3)$ if $D \equiv 9 \mod 16$. If $D \equiv 5 \mod 8$ it corresponds to the prototype $(0,\frac{D-9}{4},1,-3)$.} From this data the matrices from Lemma~\ref{lem_second_parabolic} can be calculated (see e.g. \cite{McM03}).\\[11pt]
Additionally we want to describe the parabolic element in $\SL(L_D^1)$ which fixes the cusp $(1,1)$ a little more precisely. It is of the form
$$E = \begin{pmatrix} 1 - e & e \\ -e & 1 +e \end{pmatrix}$$
for some $e \in \OD$. Here the data of the cylinders are slightly more difficult to describe: let $\left\lfloor \cdot \right\rfloor$ be the floor function. Then the lengths of the cylinders are $w - 1+ \left\lfloor w \right\rfloor$ and $w + \left\lfloor w \right\rfloor$ and the heights of the cylinders are $ -w + 1 + \left\lfloor w \right\rfloor$ and $w - \left\lfloor w \right\rfloor$. Hence $e$ is the least common $\ZZ$-multiple of $\frac{w - 1+ \left\lfloor w \right\rfloor }{-w + 1 + \left\lfloor w \right\rfloor}$ and $\frac{w + \left\lfloor w \right\rfloor}{-w + \left\lfloor w \right\rfloor}$.\footnote{We cannot uniformly give a prototype corresponding to the cusp $(1,1)$ because even the ratio of the moduli varies in a rather complicated way, depending on $\left\lfloor w \right\rfloor$.} Note that there thus cannot exist a $n \in \NN$ with $n|e$. \begin{center}
\includegraphics[height=4.58cm,width=4.58cm]{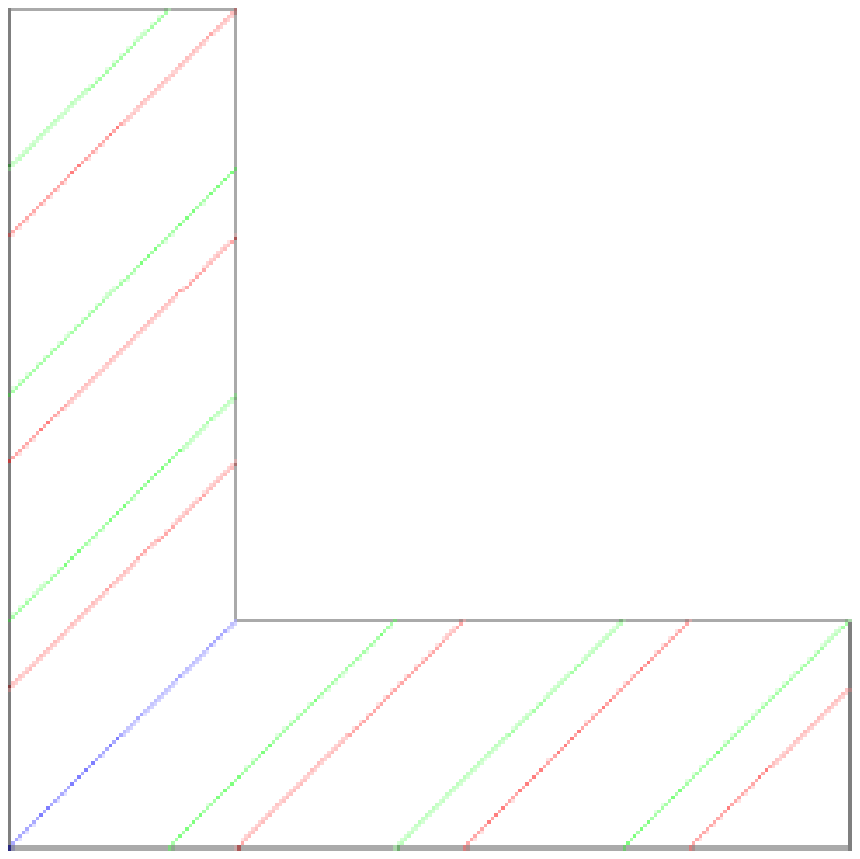}\\
Figure 3.4. A typical cylinder decomposition in direction $(1,1)$.
\end{center}
\paragraph{D $\equiv$ 1 mod 8, even spin.} If $D \equiv 1 \mod 8$, then we also have to analyze the Veech groups of the Teichmüller curves of even spin. We then again have:
\begin{prop} \label{prop_fix_veech_2} For all discriminants $D$ the Veech group $\SL(L_D^0)$ is a subgroup of $\SL_2(\OD)$. \end{prop}
\begin{proof} The periods of the generating surfaces are $1$, $w-1$, $i$ and $i(w+1)$. Since $\left\{1,w\right\}$ is a basis of $\OD$ the claim follows. \end{proof} 
It is again possible to calculate the Veech group $\SL(L_D^0)$ at least partially. More precisely, the matrices \label{glo_T2} \label{glo_Z2}
$$T := \begin{pmatrix} 1 & w-1 \\ 0 & 1 \end{pmatrix} \ \ \ \textrm{and} \ \ \ Z:= \begin{pmatrix} 1 & 0 \\ w+1 & 1 \end{pmatrix}$$
are in $\SL(L_D^0)$. The cylinder decomposition in direction $(1,1/2(w+1))$ yields:
\begin{lem} \label{lem_second_parabolic2} The following matrices lie in $\SL(L_D^0)$:
\begin{itemize}
\item If $D \equiv 1 \mod 16$: 
$$\begin{pmatrix} 1 - \frac{1}{2}w(w+1) & w \\ -\frac{1}{4}w(w+1)^2 & 1 + \frac{1}{2}w(w+1) \end{pmatrix}$$
\item If $D \equiv 9 \mod 16$:
$$\begin{pmatrix} 1 - w(w+1) & 2w \\ -\frac{1}{2}w(w+1)^2 & 1 + w(w+1) \end{pmatrix}$$
\end{itemize} \end{lem}
\begin{center}
\includegraphics[height=5.3cm,width=3.2cm]{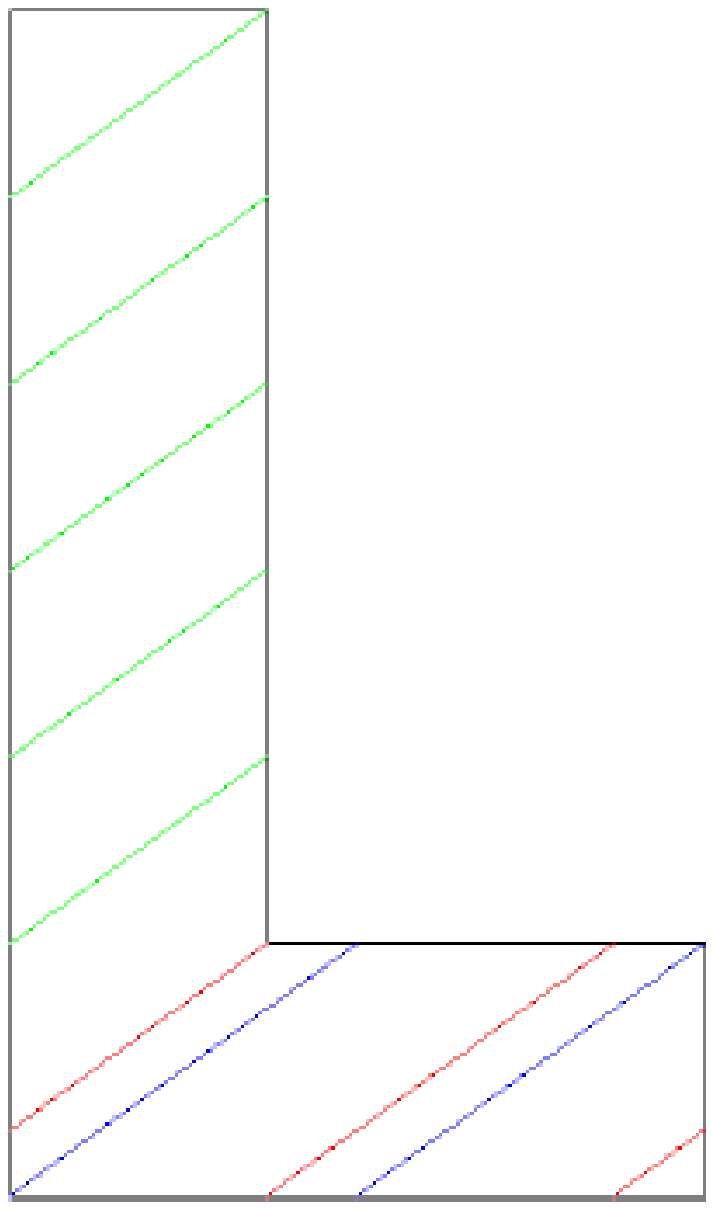}\\ 
Figure 3.5. A typical cylinder decomposition in direction $(1,1/2(w+1))$.
\end{center}
It can be checked that the lengths of the cylinders are $2$ and $w$ while the heights of the cylinders are $1$ and $(w-1)/2$. Thus the ratio of the moduli of the two cylinders is always $\frac{D-1}{16}$.\footnote{In the notation of \cite{McM05} the cusp $(1,1/2(w+1))$ thus corresponds to the prototype $(1,\frac{D-1}{8},2,-1)$ in both cases.}\\[11pt]
Moreover for $D > 17$ the cylinder decomposition in direction $(1/2(w-1),1)$ yields:
\begin{lem} \label{lem_third_parabolic2} The following matrices lie in $\SL(L_D^0)$:
\begin{itemize}
\item If $D \equiv 1 \mod 16$: 
$$\begin{pmatrix} 1 - (2+w)(w-1) & \frac{(2+w)(w-1)^2}{2} \\ -2(2+w) & 1 + (2+w)(w-1) \end{pmatrix}$$
\item If $D \equiv 9 \mod 16$:
$$\begin{pmatrix} 1 - \frac{(2+w)(w-1)}{2} & \frac{(2+w)(w-1)^2}{4} \\ (2+w) & 1 + \frac{(2+w)(w-1)}{2} \end{pmatrix}$$	
\end{itemize} \end{lem}
\begin{center}
\includegraphics[height=5.31cm,width=3.2cm]{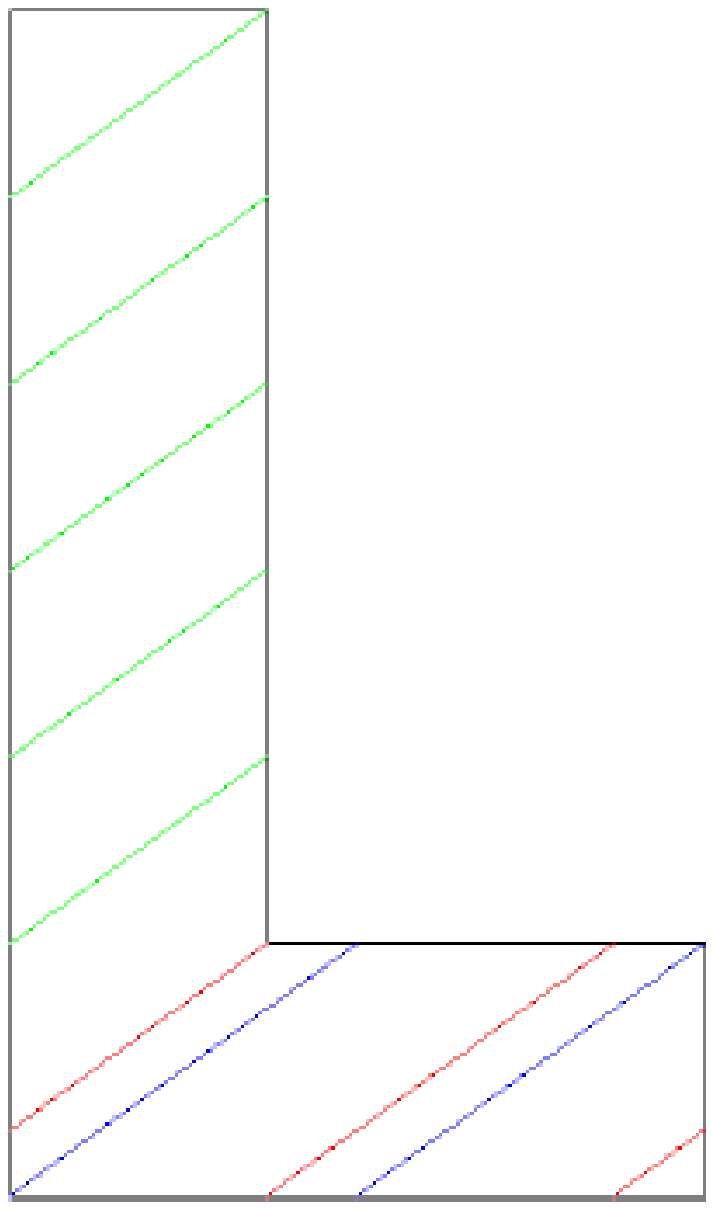}\\ 
Figure 3.6. A typical cylinder decomposition in direction $(1/2(w-1),1)$.
\end{center}
The lengths of the cylinders are $2$ and $2+w$ while the heights of the cylinders are $2/(w-1)$ and $(w-3)/(w-1)$. Thus the ratio of the moduli of the two cylinders is always $\frac{D-25}{16}$.\footnote{In the notation of \cite{McM05} the cusp $(1/2(w-1),1)$ thus corresponds to the prototype $(1,\frac{D-25}{8},2,-5)$ in both cases.}\\[11pt]
Finally we want to describe the parabolic element 
$$E = \begin{pmatrix} 1 - e & e \\ -e & 1 +e \end{pmatrix}$$
with $e \in \OD$, which fixes the cusp $(1,1)$, more precisely.
\begin{center}
\includegraphics[height=5.32cm,width=3.2cm]{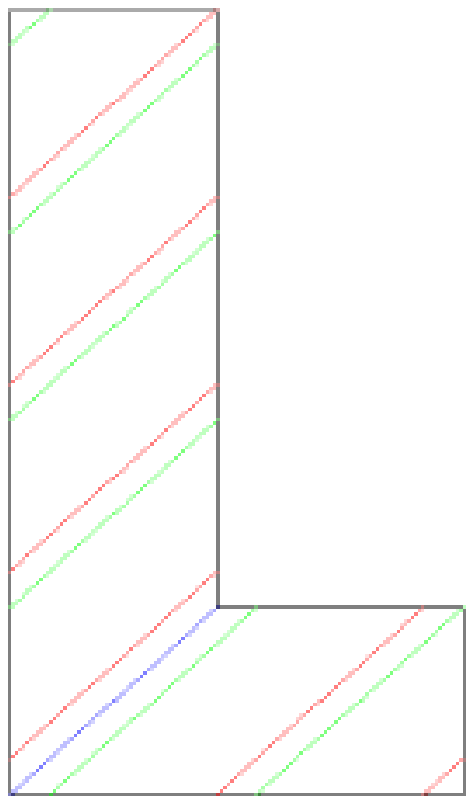}\\
Figure 3.7. A typical cylinder decomposition in direction $(1,1)$.
\end{center}
The lengths and the heights of the two cylinder are then exactly the same as in the case of $\SL(L_D^1)$. Hence $e$ is the least common $\ZZ$-multiple of $\frac{w - 1+ \left\lfloor w \right\rfloor }{-w + 1 + \left\lfloor w \right\rfloor}$ and $\frac{w + \left\lfloor w \right\rfloor}{-w + \left\lfloor w \right\rfloor}$.\footnote{We can again not uniformly give a prototype corresponding to the cusp $(1,1)$.} 

\paragraph{D $\equiv$ 0 mod 4.} Also when $D \equiv 0 \mod 4$ we have chosen the generating surface appropriately such that $\SL(L_D) \subset \SL_2(\OD)$.

\begin{prop} \label{prop_fix_veech_3} For all discriminants $D \equiv 0 \mod 4$ the Veech group $\SL(L_D)$ is a subgroup of $\SL_2(\OD)$. \end{prop}
\begin{proof} If $D \equiv 0 \mod 4$, then the periods of $(X,\omega)$ are $1, 1+w, i$ and $iw$. Since $\left\{1,\omega\right\}$  is a basis of $\OD$ the claim follows. \end{proof}

It can again be immediately seen that $\SL(L_D)$ contains the matrices \label{glo_T3} \label{glo_Z3}
$$T := \begin{pmatrix} 1 & w+1 \\ 0 & 1 \end{pmatrix} \ \ \ \textrm{and} \ \ \ Z:= \begin{pmatrix} 1 & 0 \\ w & 1 \end{pmatrix}$$
and the cylinder decomposition in direction $(1,w/2)$ yields
\begin{lem} \label{lem_second_parabolic_0mod4} For $D>16$ the following matrices lie in $\SL(L_D):$
\begin{itemize}
\item If $D \equiv 4 \mod 8$: 
$$\begin{pmatrix} 1 - 2 (2w+D/4) & 4(w+2) \\ -\frac{(2w+D/4)^2}{w+2} & 1 + 2 (2w+D/4) \end{pmatrix}$$
\item If $D \equiv 8 \mod 16$: $$\begin{pmatrix} 1 - (2w+D/4) & 2(w+2) \\ -\frac{(2w+D/4)^2}{2(w+2)} & 1 + (2w+D/4) \end{pmatrix}$$
\item If $D \equiv 0 \mod 16$: $$\begin{pmatrix} 1 - \frac{1}{2}(2w+D/4) & (w+2) \\ -\frac{(2w+D/4)^2}{4(w+2)} & 1 + \frac{1}{2}(2w+D/4) \end{pmatrix}$$\end{itemize}
\end{lem}
\begin{center}
\includegraphics[height=3.4cm,width=4.85cm]{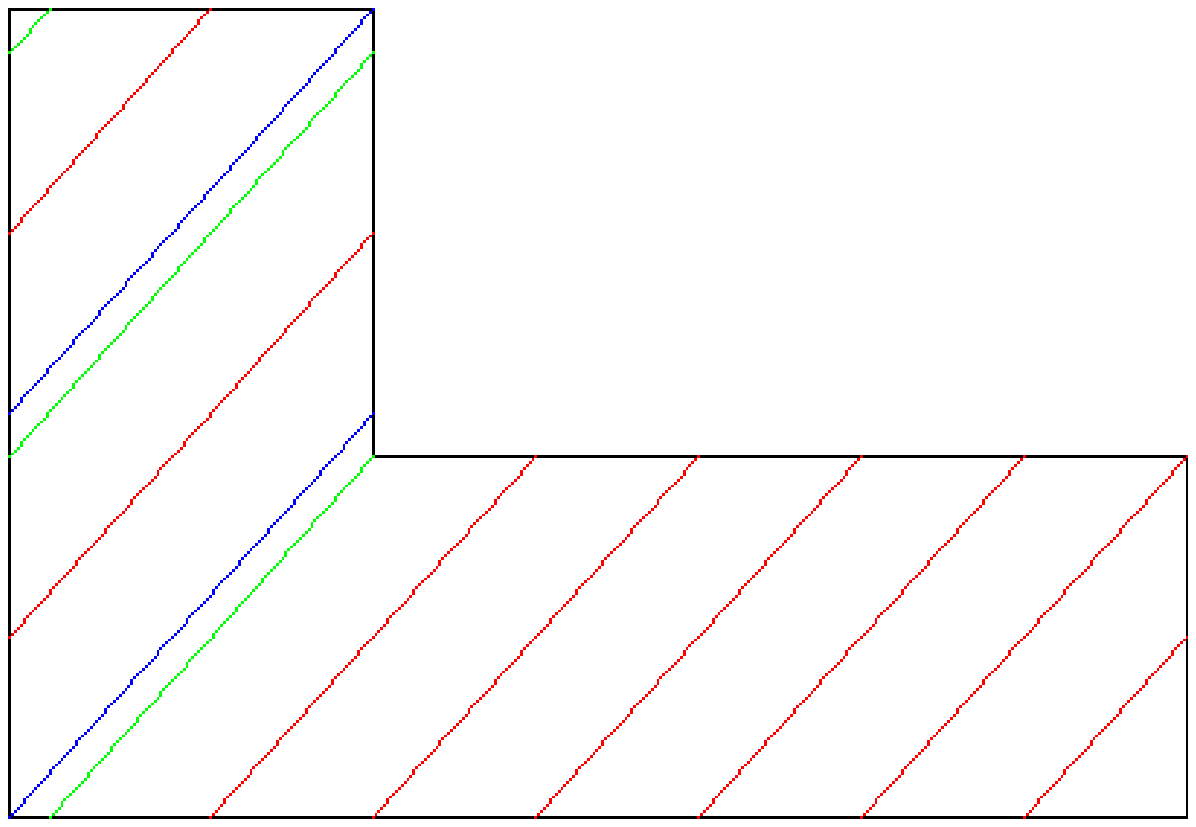} \hspace{2cm}
\includegraphics[height=3.5cm,width=4.95cm]{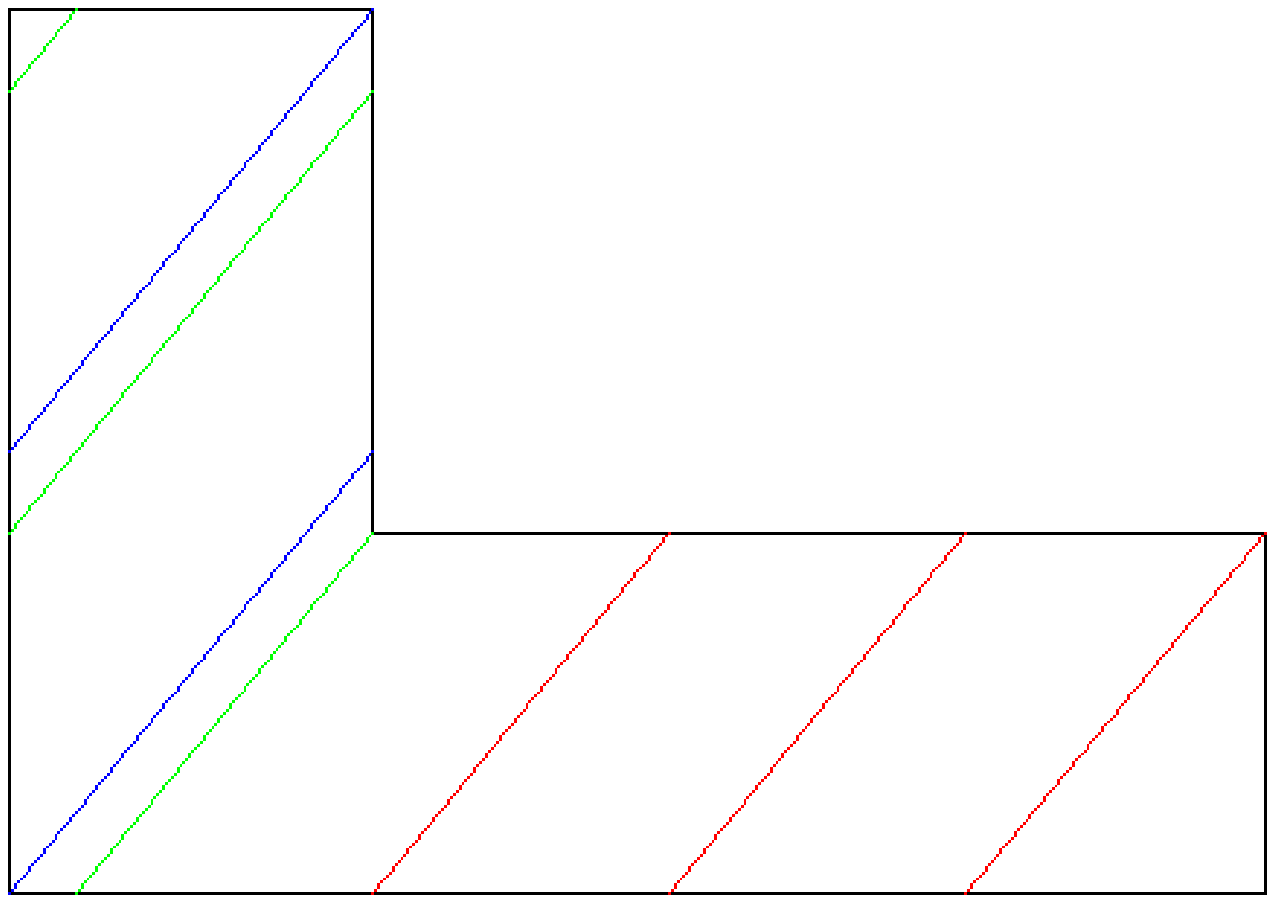}\\
Figure 3.8. Two typical cylinder decompositions in direction $(1,w/2)$ in the cases $D \equiv 4 \mod 8$ and $D \equiv 0 \mod 8$.
\end{center}
If $D \equiv 0 \mod 8$, then the lengths of the cylinders are $w$ and $w+D/8$ while the heights of the cylinders are $(w-2)/2$ and $1$. Thus the ratio of the moduli of the two cylinders is $D/16-1$. If $D \equiv 4 \mod 8$ then the lengths of the cylinders are $w$ and $2w+D/4$ while the heights of the cylinders are $(w-2)/2$ and $1/2$. Thus the ratio of the moduli of the two cylinders is $D/4-4$.\footnote{In the notation of \cite{McM05} the cusp $(1,w/2)$ therefore corresponds to the prototype $(1,\frac{D}{8}-2,2,-4)$ if $D \equiv 0 \mod 8$, and to the prototype $(1,\frac{D}{2}-8,2,-8)$ if $D \equiv 4 \mod 8$.}

\newpage
\section{Twisted Teichmüller Curves} \label{cha_twisted_Teichmüller_curves}

We have just seen that Teichmüller curves yield some of the very few known examples of Kobayashi curves on $X_D$ that are not twisted diagonals.\footnote{Other known Kobayashi curves will be described in Chapter~\ref{chapter_prym_modular}.} Using a Teichmüller curve $C_{L,D}^\epsilon$ a totally new class of examples of Kobayashi curves will be constructed in this chapter: these objects remind very much of twisted diagonals and will therefore be called twisted Teichmüller curves. They will be the main objects of these notes. We derive some of their basic properties here. Most importantly, we will see that twisted Teichmüller curves are indeed 	Kobayashi curves (Proposition~\ref{thm_twisted_finite_volume}). As concrete examples we will do some rather explicit calculations for Teichmüller curves twisted by diagonal matrices. Congruence subgroups will then naturally come into play. This rather big class of examples will be revisited in Chapter~\ref{chapter_calculations}, where we calculate the volume of almost all twisted Teichmüller curves. Aside from the basic properties mentioned in this chapter it requires much work to analyze twisted Teichmüller curves. On account of this, other properties of twisted Teichmüller curves will be derived later in distinct chapters. \\[11pt]
In the last chapter we saw that for the primitive Teichmüller curves in $\mathcal{M}_2$ 
the following diagram commutes
 $$
\begin{xy}
 \xymatrix{
 	\mathbb{H} \ar[rrr]^{\Phi(z)=(z,\varphi(z))} \ar[d]^{/\SL(L_D)} & & &	\mathbb{H} \times \mathbb{H}^{-} \ar[d]^{/\SL_2(\OD)}	\\ 
 	C \ar@{^(->}[rrr] & & & X_D
 	}
\end{xy}
$$
where from now on we write $C$ instead of $C_{L,D}^\epsilon$ to simplify notation. The group $\SL(L_D)=\Stab_{\SL_2(\RR)^2}(\Phi) \cap \SL_2(\OD)$ is the \textbf{stabilizer of the graph of the Teichmüller curve} \index{Teichmüller curve!graph} \index{stabilizer!of the graph of a Teichmüller curve} $\Phi(z)=(z,\varphi(z))$ inside $\SL_2(\OD)$.  As a shortcut, we will from now on always write $\Stab(\Phi)$ \label{glo_stabphi} whenever we mean the stabilizer inside $\SL_2(\RR)^2$ of the graph of $\Phi$. Moreover we saw that $\varphi$ is not a Möbius transformation and that $\SL(L_D)$ is the Veech group of the surface $(X,\omega)$ obtained from the L-shaped polygon, when we consider the Teichmüller curve as projection of the orbit $\SL_2(\RR)(X,\omega)$ to $\Mg$. \\[11pt]
We now mimic the construction of twisted diagonals (compare e.g. \cite{vdG88}): \index{twisted diagonal}Let us twist these Teichmüller curves on $X_D$ by a matrix $M \in \GL_2^+(K)$. In other words, let us consider the \textit{twisted} diagram

$$
\begin{xy}
 \xymatrix{ \mathbb{H} \ar[rrr]^{\Phi_M=(Mz,M^\sigma\varphi(z))} \ar[d]^{/\SL_M(L_D)} & & & \mathbb{H} \times \mathbb{H}^{-}              \ar[d]^{/\SL_2(\OD)} \\
             C_M \ar@{^(->}[rrr]& &  & X_D}
\end{xy}
$$
where \label{glo_C_M}$M^\sigma$ is as always the Galois conjugate of $M$ and where $\SL_M(L_D)=\Stab_{\SL_2(\RR)^2}(\Phi_M) \cap \SL_2(\OD)$ is the stabilizer of $\Phi_M=(Mz,M^\sigma\varphi(z))$ inside $\SL_2(\OD)$, i.e. the \textbf{stabilizer of the graph of the twisted Teichmüller curve}. \index{stabilizer!of a twisted Teichmüller curve} \label{glo_SLMLD} 

\begin{defi} We call $C_M \hookrightarrow X_D$ a \textbf{twisted Teichmüller curve}. \index{twisted Teichmüller curve} \end{defi}

Evidently, the stabilizer $\SL_M(L_D)$ is a group. It is natural to ask what $\SL_M(L_D)$ exactly looks like. A very first starting point for answering this very hard question is the next proposition.

\begin{rem} 
A matrix $N \in \GL_2^+(K)$ is in $\Stab(\Phi)$ if and only if $\varphi(Nz) = N^\sigma \varphi(z)$ for all $z \in \HH$. However, this is hardly of any use to compute the stabilizer.
\end{rem}

\begin{prop} \label{prop_stabilzer_twisted} Let $M \in \GL_2^+(K)$, then \index{stabilizer!of a twisted Teichmüller curve} $$\SL_M(L_D)= \Stab(\Phi)^{M^{-1}} \cap \SL_2(\OD).$$ 
In particular, we have $\textrm{M}  \SL(L_D) M^{-1} \cap \SL_2(\OD) < \SL_M(L_D)$. Equality holds if for all $A \in \SL_2(\RR)$ the condition $\varphi(At) = A^\sigma \varphi(t)$ holds if and only if $A \in \SL(L_D)$.  \end{prop}
\begin{proof} Let $N \in \SL_M(L_D) \subset \SL_2(\OD)$. Then for all $z \in \mathbb{H}$ we have
$$(N \cdot M z, N^\sigma \cdot M^\sigma \varphi(z)) = (Mz^*, M^\sigma \varphi(z^*))$$
for some $z^* \in \mathbb{H}$ depending on $N$. The first component yields:
$$N \cdot M z = M z^* \ \ \textrm{or equivalently} \ \ M^{-1} \cdot N \cdot M z= z^*$$
Inserting this in the second component gives
$$N^\sigma \cdot M^\sigma \varphi(z) = M^\sigma \varphi(M^{-1} \cdot N \cdot M z)$$ 
which is equivalent to
$$M^{\sigma^-1} \cdot N^\sigma \cdot M^\sigma \varphi(z) = \varphi(M^{-1} \cdot N \cdot Mz).$$ 
Thus $\SL_M(L_D) = M \cdot \Stab(\Phi) \cdot M^{-1} \cap \SL_2(\OD)$. Since $\varphi(At) = A^\sigma \varphi(t)$ for all $A \in \SL(L_D)$ we get the claimed inclusion.
\end{proof}

As $M \in \GL_2^+(K)$ and $kM$ act on $\mathbb{H} \times \mathbb{H}^-$ in the same way for totally positive $k \in K$, one can restrict to the case where $M \in \GL_2^+(K) \cap \Mat^{2x2}(\OD)$  by multiplying $M$ with the least common multiple of the denominators of its entries if necessary. Recall that $\GL_2^+(K) \cap \Mat^{2x2}(\OD) \neq \GL_2(\OD)$.  Depending on the question which shall be answered this point of view is sometimes very useful.\\[11pt]Before we proceed with our analysis of twisted Teichmüller curves, let us introduce a bunch of extra notation. The logic behind the notation is such that a matrix $(M)$ in brackets corresponds to a level-covering of the involved curves (compare Lemma~\ref{lem_covering_congruence}) and a matrix $M$ as subscript corresponds to a twist - as for $C_M$. First we denote by $\SL_2(\OD,M)$ the group $\SL_2(\OD) \cap \SL_2(\OD)^{M^{-1}}$ and by $X_D(M):=\HH \times \HH^-/\SL_2(\mathcal{O_D},M)$, \label{glo_XDM} a finite cover of $X_D$. By $\SL(L_D,M)$ \label{glo_SL(L_D,M)} we denote the group $\SL(L_D) \cap \SL_2(\OD)^{M^{-1}}$ which yields a cover of the Teichmüller curve $C$ that is therefore denoted by $C(M)$. \label{glo_C(M)} Finally we set $\SL_M(L_D,M):= \SL_M(L_D) \cap M\SL_2(\OD)M^{-1}$ \label{glo_SL_M(L_D,M)} and let $C_M(M)$ \label{glo_C_M(M)} denote the corresponding cover of the twisted Teichmüller curve.  \\[11pt]
Let us clarify the relation of all these curves in a diagram, where all arrows going down indicate (finite degree) coverings
$$
\begin{xy}
 \xymatrix{ & \HH \ar[d]^{\SL_M(L_D,M)}& \HH \ar[d]^{\SL(L_D,M)}&  & \HH \times \HH^- \ar[dl]_{\SL_2(\OD,M)} \ar[dddl]^{ \SL_2(\OD)}\\
 & C_M(M) \ar@/_1pc/@{^(->}[rr] \ar[dd] & C(M) \ar[dd] 	\ar@{^(->}[r] & X_D(M) \ar[dd] &\\
 & & & &  & \\
 & C_M \ar@/_1pc/@{^(->}[rr]& C \ar@{^(->}	[r]& X_D &
 	}
\end{xy}.
$$ 
In order to facilitate many expressions which include conjugation, we want to introduce even some more notation. For the convenience of the reader we collect the most important notation in the following table.\\[11pt]
\fbox{ \begin{minipage}{12.25cm}
\begin{align*}
\SL_2(\OD,M) = & \ \SL_2(\OD) \cap \SL_2(\OD)^{M^{-1}}\\
\SL(L_D,M) = & \ \SL(L_D) \cap \SL_2(\OD)^{M^{-1}}\\
\SL_M(L_D) = & \ \Stab(\Phi)^{M^{-1}} \cap \SL_2(\OD)\\
\SL^M(L_D) = & \ \SL_M(L_D)^M = \Stab(\Phi) \cap \SL_2(\OD)^M\\
\SL_M(L_D,M) = & \ \SL_M(L_D) \cap \SL_2(\OD)^{M^{-1}}\\
\SL^M(L_D,M) = & \ \SL_M(L_D,M)^M = \SL(L_D) \cap \SL_2(\OD)^{M}\\
\end{align*} 
\end{minipage}}\\[11pt]
Finally we set $C^M(M):= \HH / \SL^M(L_D,M)$.  \label{glo_SLMLD2}Note that this curve has the same geometric properties and in particular the same volume as $C_M(M)$ since the involved groups are conjugated. We are now able to state the main result of this chapter.
\begin{prop} \label{thm_twisted_finite_volume} Every twisted Teichmüller curve is a Kobayashi curve. \index{Kobayashi curve} \end{prop}

\begin{proof}
$C_M$ is an algebraic curve if and only if $\SL_M(L_D)$ is a lattice. Note that this is a purely group theoretic property. It follows from  Proposition~\ref{prop_finite_index_GL2K} that we have $[\SL_M(L_D):\SL_M(L_D,M)] < \infty$ and also that $[\SL(L_D)^M:\SL_M(L_D,M)] < \infty$. Hence $\SL_M(L_D)$ is a lattice. By definition one of the components in the universal covering is given by a Möbius transformation and therefore $C_M$ is a Kobayashi curve. 
\end{proof}
More generally it is, of course, true that all twists of Kobayashi curves are again Kobayashi curves, but as we are only concerned about twisted Teichmüller curves in these notes, we restrict to this case.\\[11pt] It is in general a very hard task to calculate the groups $\SL_M(L_D)$ or the quotients $C_M=\HH/\SL_M(L_D)$. Let us recall the three main reasons for this:
\begin{itemize}
\item[(i)] The theorem of E. Gutkin and C. Judge (Theorem~\ref{thm_Gutkin_Judge2}) \index{Theorem of Gutkin-Judge} implies that the Veech groups $\SL(L_D)$ are all non-arithmetic Fuchsian groups. In particular, this makes it hard to decide whether a matrix in $\SL_2(\OD)$ lies in the Veech group or not. Indeed, we will prove in Chapter~\ref{chapter_calculations} that these Veech groups are somehow the opposite of being arithmetic.
\item[(ii)] It is still unknown how to calculate the Veech group for a given flat surface $(X,\omega)$. Although this problem is solved in some special cases, in our case the Veech group can only (or at least) be calculated partially (see Section~\ref{sec_fixing_Veech}).
\item[(iii)] The Taylor expansion of $\varphi$ is known by a theorem of M. Möller and D. Zagier in \cite{MZ11}. Even with this knowledge, it is not easy to decide whether there exist elements of $\SL_2(\RR)$ which are not in the Veech group but lie in $\Stab(\Phi)$ or even if there exist any such element in $\SL_2(K) \smallsetminus \SL_2(\OD)$.
\end{itemize} 
In Chapter~\ref{sec_stab_comm} we will show that the group $\SL_M(L_D)$ is contained in a unique (finitely maximal) Fuchsian group, namely the commensurator of $\SL(L_D)$. Unfortunately this fact alone does in most cases not help very much since calculating $\Comm_{\SL_2(\RR)}(\SL(L_D))$ remains a very difficult task. In Chapter~\ref{sec_stab_comm}, we will therefore also work out a way to circumvent this problem. Furthermore we will see that the group $\SL^M(L_D,M)$ equals the a priori bigger group $\SL^M(L_D)$ in most cases. Note that this is really surprising when one only considers the definition. This is also a justification why we analyze these groups $\SL^M(L_D,M)$ in more detail in Chapter~\ref{chapter_calculations}.\\[11pt]
The expressions $(Mz,M^\sigma\varphi(z))$ and $(NMz,N^\sigma M^\sigma\varphi(z))$ define the same curve inside the Hilbert modular surface $X_D$ if $M \in \GL_2^+(K)$ and $N \in \SL_2(\OD)$. Thus one  is indeed only interested in analyzing twisted Teichmüller curves for a system of representatives of $\SL_2(\OD) \backslash \GL_2^+(K)$. If the class number $h_D$ is equal to $1$ we are able to give a nice set of matrices that contains a system of representatives. Otherwise the proof of a similar statement would involve some weird arithmetic which we will avoid here.
\begin{prop} \label{prop_matrix_decomposition} If $h_D=1$ then for all $M \in GL_2(K)$ there exists a matrix $N \in \SL_2(\OD)$ and an upper triangular matrix $L \in \GL_2(K)$ with $M=NL$. \end{prop}
\begin{proof} The claim follows immediately from the fact that $X_D$ has only one cusp. Nevertheless, we want to find the matrices $N,L$ explicitly. Let $M$ be of the form 
$$\begin{pmatrix} \frac{a}{x} & \frac{b}{x} \\ \frac{c}{x} & \frac{d}{x} \end{pmatrix}$$
with $a,b,c,d,x \in \OD$. We want to find $e,f,g,h \in \OD$ with $eh-fg=1$ and $k,l,m \in K$ such that 
$$\begin{pmatrix} e & f \\ g & h \end{pmatrix} \begin{pmatrix} k & l\\ 0 & m \end{pmatrix} = \begin{pmatrix} \frac{a}{x} & \frac{b}{x} \\ \frac{c}{x} & \frac{d}{x} \end{pmatrix}$$
By choosing $k$ appropriately, i.e. $k=(a,c)$ we may without loss of generality assume that $(a,c)=1$. Since $h_D=1$ we can can find a $f \in \OD$ such that $h:=\frac{cf+1}{a} \in \OD$. Then we choose $e=a$, $g=c$, $k=\frac{1}{x}$, $l=-\frac{f(ad-bc)-b}{ax}$ and $m=\frac{ad-bc}{x}$ and get the claim. 
\end{proof}

The proposition rests on the fact that the number of cusps of a Hilbert modular surface is equal to the class number of the number field (see e.g. \cite{vdG88}, Proposition~1.1). We may thus from now on always assume that $M$ is an upper triangular matrix if $h_D=1$ .

\begin{rem} \begin{itemize}
\item[(i)] Neither the matrix $N$ nor the matrix $L$ is unique. A different choice of $N$ only changes the parametrization of the twisted Teichmüller curve.
\item[(ii)] The upper triangular matrices are not a system of representatives of $\SL_2(\OD) \backslash \GL_2(K)$ since some of them are equivalent modulo $\SL_2(\OD)$.
\item[(iii)] The matrix $L$ cannot always be chosen as a diagonal matrix.
\end{itemize}
\end{rem}







It is a particularly interesting case to consider twisted Teichmüller curves for a diagonal matrix $M$. One reason why they are so interesting is that the calculation of the volume of Teichmüller curves twisted by upper triangle matrices can be reduced to the case of diagonal matrices (see Chapter~\ref{chapter_calculations}). This is why we give Teichmüller curves twisted by diagonal matrices a special name.

\begin{defi} If $M$ is diagonal, we call $C_M$ a \textbf{diagonal twisted Teichmüller curve}. If $M$ is of the form $\left( \begin{smallmatrix} m & 0 \\ 0 & 1 \end{smallmatrix} \right)$ with $m \in \OD$ we call $C_M$ a \textbf{simple twisted Teichmüller curve}. \end{defi}

Note that every diagonal matrix can be normalized to the form $\left( \begin{smallmatrix} m & 0 \\ 0 & n \end{smallmatrix} \right)$ with $m,n \in \OD$ and $(m,n)=1$ if $h_D=1$. Let us finish this discussion by quickly showing how simple respectively diagonal twisted Teichmüller curves are related to congruence subgroups.

\begin{lem} \label{lem_covering_congruence} For $M= \left( \begin{smallmatrix} m & 0 \\ 0 & 1 \end{smallmatrix} \right)$ with $m \in \OD$ we have $$\SL^M(L_D,M)= \SL(L_D) \cap \Gamma^D_0(m).$$ \end{lem}

\begin{proof} It suffices to show $M^{-1} \SL_2(\OD) M \cap \SL_2(\OD)= \Gamma^D_0(m)$. We have $M^{-1}  \left( \begin{smallmatrix} a & b \\ c & d \end{smallmatrix} \right) M = \left( \begin{smallmatrix} a & b/m \\ c m & d \end{smallmatrix} \right)$. Hence, $M^{-1}  \SL_2(\OD) M \cap \SL_2(\OD) \subset \Gamma^D_0(m)$. We now show the other inclusion: so let $Y=\left( \begin{smallmatrix} \tilde{a} & \tilde{b} \\ \tilde{c} & \tilde{d} \end{smallmatrix} \in \Gamma^D_0(m) \right)$. Then $\tilde{c} = c \cdot m$ with $c \in \OD$. Choosing the matrix $X$ as $\left( \begin{smallmatrix} \tilde{a} & \tilde{b} m  \\ c & \tilde{d} \end{smallmatrix} \right) \in \SL_2(\OD),$  we get $M^{-1} X M = Y$.\end{proof}


Moreover the degree of the covering $C_M(M) \to C_M$ is bounded by the degree of the covering $\pi: X_D(m) \to X_D$, where $X_D(m) := \HH \times \HH^{-} / \Gamma^D_0(m)$.

\begin{lem} 
For $M= \left( \begin{smallmatrix} m & 0 \\ 0 & 1 \end{smallmatrix} \right)$ with $m \in \OD$ the following inequality holds:
$$[\SL_M(L_D):\SL_M(L_D,M)] \leq [\SL_2(\OD): \Gamma^D_0(m)].$$
\end{lem}

\begin{proof} We have
$$\left[\SL_M(L_D):\SL_M(L_D,M)\right] =  \left[\SL^M(L_D):\SL^M(L_D,M)\right] \leq$$
$$\left[ \SL_2(\OD): \SL_2(\OD) \cap M^{-1} \SL_2(\OD) M \right] =$$ $$\left[ \SL_2(\OD): \Gamma^D_0(m) \right].$$

\end{proof}

An analogous calculation as above yields the corresponding result for the case of diagonal twisted Teichmüller curves.

\begin{cor} \label{cor_diagonal_twisted_congruence} For $M= \left( \begin{smallmatrix} m & 0 \\ 0 & n \end{smallmatrix} \right)$ with $m,n \in \OD$ and $(m,n)=1$ we have $$\SL^M(L_D,M)= \SL(L_D) \cap \Gamma^D(m,n)$$ and
$$[\SL_M(L_D):\SL_M(L_D,M)] \leq [\SL_2(\OD): \Gamma^D(m,n)].$$
\end{cor}

\paragraph{A possible generalization.} In these notes we consider twisted Teichmüller curves only for matrices $M \in \GL_2^+(K)$. By dividing each of the entries of the matrix $M$ by the root of the determinant one can also interpret the pair of matrices $(M,M^\sigma)$ as a pair in $\SL_2(\RR)^2$. Hence, there is some scope for generalization of the term twisted Teichmüller curve: One could as well take an arbitrary pair of matrices $(M,M') \in \SL_2(\RR)^2$ and define a twisted Teichmüller curve as the projection of the orbit $(Mz,M'\varphi(z))$ to $X_D$. Of course, this still yields a well-defined curve in $X_D$. Such curves are much harder to treat than the twisted Teichmüller curves which we introduced because of the following main reason: a pair $(M,M')$ lies in the stabilizer of the graph of a Teichmüller curve if and only if $\varphi(Mz)=M'\varphi(z)$. Thus one would need to completely understand the behavior of $\varphi$ under Möbius transformation in order to solve the problem of calculating the stabilizer of the graph of a Teichmüller curve in $\SL_2(\RR)^2$. This seems to be far from being reachable. However, this problem is somehow the starting point when one wants to understand these generalized twisted Teichmüller curves. Even for twisted diagonals people restricted to matrices in $\GL_2^+(K)$ (compare e.g. \cite{Fra78}, \cite{Hau80} or \cite{vdG88}). One reason for this is that for twisted diagonals with an arbitrary pair of matrices $(M,M') \in \SL_2(\RR)^2$ the stabilizer does not need to be a lattice any more. \index{twisted diagonal} 

\newpage

\section{Stabilizer and Maximality} \label{cha_maximality}

In the last chapter we claimed that it is in general very hard to describe the groups $\SL_M(L_D)$ because we do not know the stabilizer of the graph of the Teichmüller curve $\Stab(\Phi)$. However, we could calculate $\SL_M(L_D)$ more easily if we knew that the Veech group $\SL(L_D)$ had certain nice properties. The strongest of these properties is (finite) maximality. This notion is closely linked to the commensurator.\\[11pt]In Section~\ref{sec_max_comm} we show that the Veech groups $\SL(L_D^1)$ for discriminant $D$ equal to $5,13$ and $17$ are indeed (finitely) maximal. We are able to do this because the generators of these groups are explicitly known. Since this is not true for large $D$, we have to introduce the weaker notion of pseudo parabolic maximal Fuchsian groups (see Definition~\ref{def_pseudo_parabolic}) in Section~\ref{sec_pseudo} which suffices for our purposes. And we show (Theorem~\ref{thm_pseudo_parabolic}) that if $D \equiv 5 \mod 8$ then the group $\SL(L_D)$ is pseudo parabolic maximal. This property allows us to deduce a lot of information about the commensurator and the stabilizer of the graph of the Teichmüller curve. In Section~\ref{sec_stab_comm} we discuss the relation between the stabilizer and the commensurator in general: we will show that $\Stab(\Phi)$ is always contained in the commensurator of $\SL(L_D)$ (Corollary~\ref{cor_stabcap}). We will see that we can state stronger and more precise statements for pseudo parabolic maximal $\SL(L_D)$. In particular, we will get that if the discriminant $D \equiv 1 \mod 4$ is fundamental then the degree of the covering $\pi: C_M(M) \to C_M$ is always equal to $1$ for all $M \in \GL_2^+(K)$. The reason for this is the following: being pseudo parabolic maximal suffices to imply $\Stab(\Phi) \cap \SL_2(K) = \SL(L_D)$ (Theorem~\ref{thm_parabolic_maximal_implies_stabilizer}). If $D \not \equiv 5 \mod 8$ then a lack of information on the Veech groups prevents us from being able to prove that the groups $\SL(L_D)$ are pseudo parabolic maximal. Therefore, the techniques are more involved, but the result that the degree of the cover $\pi: C_M(M) \to C_M$ is equal to $1$ still holds for most $M \in \GL_2^+(K)$ if the discriminant is fundamental (Theorem~\ref{thm_cov_degree_10}).

\subsection{Maximal Fuchsian Groups and the Commensurator} \label{sec_max_comm}

A Fuchsian group $\Gamma$ is called \textbf{(finitely) maximal} \index{Fuchsian group!(finitely) maximal} if there does not exist a Fuchsian group $\Gamma'$ properly containing $\Gamma$ with finite index (compare \cite{Sin72}).  If $\Gamma$ is a cofinite Fuchsian group and $\Gamma'$ is a Fuchsian group containing $\Gamma$ then $\left[\Gamma':\Gamma\right] < \infty$ is automatically, because $\mu(\Gamma)>0$.\footnote{Note that the volume of a fundamental domain of a Fuchsian group is bounded by $\pi/21$ (or by $\pi/3$ if the Fuchsian group has a cusp).} For cofinite Fuchsian groups one could therefore have left away the \textit{with finite index} condition in the definition. On the other hand if $\Gamma'$ is any subgroup of $\PSL_2(\mathbb{R})$ which contains the Fuchsian group $\Gamma$ of finite index, then $\Gamma'$ is itself a Fuchsian group. Therefore, we will from now on only speak of maximal Fuchsian groups when we mean (finitely) maximal Fuchsian groups.\\[11pt] 
The notion of maximality is closely linked to the notion of commensurability: if $\Gamma$ is a non-arithmetic Fuchsian group of finite volume, then by Margulis' theorem the commensurator is the unique maximal Fuchsian group containing $\Gamma$ (Theorem~\ref{thm_Margulis}).

\begin{lem} \label{lem_non-arithmetic_commensurator} Let $\Gamma$ be any non-arithmetic Fuchsian group of finite volume. Then $\Comm_{SL_2(\mathbb{R})} (\Gamma)$ is the unique maximal Fuchsian group containing $\Gamma$. \end{lem}

\begin{proof} Let $\Gamma'$ be a maximal Fuchsian group containing $\Gamma$. Thus $\left[\Gamma':\Gamma\right] < \infty$. Hence we have for all $U \in \Gamma'$: 
\begin{equation*}
\infty > \left[\Gamma':\Gamma\right] \geq \left[U \Gamma U^{-1} : \Gamma \cap U \Gamma U^{-1}\right]
\end{equation*}
and 
\begin{equation*}
\infty > \left[\Gamma':\Gamma\right] \geq \left[U^{-1} \Gamma U : \Gamma \cap U^{-1} \Gamma U\right]= \left[\Gamma: \Gamma \cap U \Gamma U^{-1}\right].
\end{equation*}
and so $U \in \Comm_{SL_2(\mathbb{R})} (\Gamma)$. By Margulis' Theorem (Theorem~\ref{thm_Margulis}) the commensurator $\Comm_{\SL_2(\mathbb{R})} (\Gamma)$ is a group containing $\Gamma$ of finite index. Thus the commensurator is also Fuchsian.
\end{proof} 

In this section we mainly want to prove the following theorem:

\begin{thm} \label{thm_maximality} For $D \in \left\{5,13,17 \right\}$ we have:
\begin{itemize}
\item[(i)] $\Comm_{\SL_2(\mathbb{R})} (\SL(L_{D}^1)) = \SL(L_{D}^1)$.
\item[(ii)] $\SL(L_{D}^1)$ is maximal.
\end{itemize} \index{commensurator}
\end{thm}

Applied to our case the last lemma shows that $\Comm_{\SL_2(\mathbb{R})} (\SL(L_{D}^1))$ is the unique maximal Fuchsian group containing $\SL(L_{D}^1)$. This means that $\SL(L_{D}^1)$ is  maximal if and only if $\Comm_{\SL_2(\mathbb{R})} (\SL(L_{D}^1)) = \SL(L_{D}^1)$. Therefore, the two statements of Theorem~\ref{thm_maximality} are equivalent. For $D=5$, namely the Hecke triangle group $\Delta(2,5,\infty)$,\footnote{For the definition of the term \textit{triangle group} see e.g. \cite{Bea89}, Chapter~10.6.} this theorem is due to A. Leutbecher in \cite{Leu67}. We will follow his methods closely. To be more precise, we will use geometric arguments similar to the ones given in \cite{Leu67} combined with the Riemann-Hurwitz formula and some other general statements about Fuchsian groups. All the necessary arguments for proving Theorem~\ref{thm_maximality} will be given in detail for $D=13$. As $D=17$ is very similar, we will then only give a brief sketch of the proof.


\begin{rem} Note that $$\Comm_{\SL_2(\mathbb{R})} (\SL(L_{D}^1)) \subset \SL_2(K)$$ since every element in the commensurator permutes the cusps. \end{rem}
So let us consider the case $D=13$. The first step in the proof will be to show that there does not exist any parabolic element $U \in \SL_2(K)$ such that $1<\left[\left\langle \SL(L_{13}^1),U \right\rangle:\SL(L_{13}^1)\right] < \infty$. We will give three different proofs of this fact. The first one is very geometric and makes use of the explicit knowledge of the fundamental domain of $\SL(L_{13}^1)$, that is known by McMullen's algorithm (compare \cite{McM03}). As a first step in this geometric proof, it is necessary to explicitly give a list of elements in $\SL(L_{13}^1)$ with certain properties.

\begin{lem} \label{lem_double_cosets} \begin{itemize}
\item[(i)] The set of matrices of the form $\left( \begin{smallmatrix} a & b \\ 1 & d \end{smallmatrix} \right)$ inside $\SL(L_{13}^1)$ is given by the double coset
$$\begin{pmatrix} 1 & w \ZZ \\ 0 & 1 \end{pmatrix} \begin{pmatrix} 0 &  -1 \\ 1 & 0 \end{pmatrix} \begin{pmatrix} 1 & w \ZZ \\ 0 & 1 \end{pmatrix}.$$
\item[(ii)] The set of matrices of the form $\left( \begin{smallmatrix} a & b \\ w & d \end{smallmatrix} \right)$ inside $\SL(L_{13}^1)$ is given by the union of the double cosets
$$\begin{pmatrix} 1 & w \ZZ \\ 0 & 1 \end{pmatrix} \begin{pmatrix} 1 & 0 \\ w & 1 \end{pmatrix} \begin{pmatrix} 1 & w \ZZ \\ 0 & 1 \end{pmatrix}$$
$$\begin{pmatrix} 1 & w \ZZ \\ 0 & 1 \end{pmatrix} \begin{pmatrix} -1 & 0 \\ w & -1 \end{pmatrix} \begin{pmatrix} 1 & w \ZZ \\ 0 & 1 \end{pmatrix}$$
\end{itemize}
\end{lem}

\begin{proof} A fundamental domain $\mathcal{F}$ of $\SL(L_{13}^1)$ is:
\begin{center}
\psset{xunit=1cm,yunit=1cm,runit=1cm}
\begin{pspicture}(-0.5,-0.5)(8,4) 
\psline[linewidth=0.5pt]{->}(0,0)(8,0)
\psline[linewidth=0.5pt]{->}(4,-0.1)(4,4)
\psline[linewidth=0.5pt]{-}(6.3,0)(6.3,4)
\psline[linewidth=0.5pt]{-}(1.7,0)(1.7,4)
\psarc[linewidth=0.5pt](4,0){2.0cm}{0}{180}
\psarc[linewidth=0.5pt](6.15,0){0.15cm}{0}{180}
\psarc[linewidth=0.5pt](1.85,0){0.15cm}{0}{180}
\uput[dl](1.7,0){$-\frac{w}{2}$}
\uput[d](2,0){$-1$}
\uput[d](6,0){$1$}
\uput[dr](6.3,0){$\frac{w}{2}$}
\uput[u](3.8,2){$i$}
\uput[u](3.4,3){$\mathcal{F}$}
\pscircle[linewidth=1.0pt](4,2){0.05cm}
\end{pspicture}\\
Figure 5.1. Fundamental domain $\mathcal{F}$ of $\SL(L_{13}^1)$. \index{Fuchsian group!fundamental domain}
\end{center}

Its vertices are $-w/2,-1,1,w/2$ and $\infty$. If we look at $S \mathcal{F}$ this yields another fundamental of $\SL(L_{13}^1)$:

\begin{center}
\psset{xunit=1cm,yunit=1cm,runit=1cm}
\begin{pspicture}(-0.5,-0.5)(8,4) 
\psline[linewidth=0.5pt]{->}(0,0)(8,0)
\psline[linewidth=0.5pt]{->}(4,-0.1)(4,4)
\psline[linewidth=0.3pt,linestyle=dotted](1.5,0.87)(6.5,0.87)
\psarc[linewidth=0.5pt](4,0){2.0cm}{0}{180}
\psarc[linewidth=0.5pt](5.87,0){0.13cm}{0}{180}
\psarc[linewidth=0.5pt](4.87,0){0.87cm}{0}{180}
\psarc[linewidth=0.5pt](2.13,0){0.13cm}{0}{180}
\psarc[linewidth=0.5pt](3.13,0){0.87cm}{0}{180}
\uput[dl](5.95,0){$\frac{2}{w}$}
\uput[dl](2.05,0){$-1$}
\uput[dr](5.95,0){$1$}
\uput[dr](1.95,0){$-\frac{2}{w}$}
\uput[d](3.8,0){$0$}
\uput[u](3.8,2){$i$}
\uput[u](3.4,1.1){$S\mathcal{F}$}
\uput[r](6.5,0.87){$\frac{1}{w}$}
\pscircle[linewidth=1.0pt](4,2){0.05cm}
\pscircle[linewidth=1.0pt](3.13,0.87){0.05cm}
\pscircle[linewidth=1.0pt](4.87,0.87){0.05cm}

\end{pspicture}\\
Figure 5.2. Fundamental domain $S\mathcal{F}$ of $\SL(L_{13}^1)$. 
\end{center}

The radius of the two isometric circles which adjoin $0$ and $2/w$ (respectively $0$ and $-2/w$) is exactly $1/w$. The isometric circles which adjoin $1$ and $2/w$ (respectively $-1$ and $-2/w$) have radius strictly smaller than $1/w$. Those isometric circles and the one which adjoins $-1$ and $1$ are the boundary of $S\mathcal{F}$. All the elliptic fixed point of $\SL(L_{13}^1)$ are exactly $\SL(L_{13}^1)i$. Two points $z_1$ and $z_2$ in the domain $\Imag(z) \geq 1$ are equivalent with respect to $\SL(L_{13}^1)$ if and only if $z_1-z_2 \in w \ZZ$. To see this, one just looks at the tessellation of $\mathbb{H}$ by the copies of the fundamental domain $\mathcal{F}$. The only elliptic fixed points in the domain $\Imag(z) > \frac{1}{w}$ are the points $i + w \ZZ$. Let $V=\left( \begin{smallmatrix} a & b \\ c & d \end{smallmatrix} \right) \in \SL(L_{13}^1)$ with $c \neq 0$. Then 
\begin{eqnarray} \label{eqn_trans}
  V \left(\frac{-d}{c}+\frac{i}{|c|} \right) = \frac{a}{c}+\frac{i}{|c|}
\end{eqnarray}
So we have $|c|\geq 1$. If $|c|=1$ then $\frac{a}{c} \in w \ZZ$ and $\frac{d}{c} \in w \ZZ$ hence $a \in w \ZZ$ and $d \in w \ZZ$. Then $b$ is determined, because the determinant is equal to 1. All those matrices are exactly given by the double coset in $(i)$.
Now suppose that $c=w$, i.e. we are searching for all matrices $V= \left( \begin{smallmatrix} a & b \\ w & d \end{smallmatrix} \right) \in \SL(L_{13}^1)$. At first we now look at the two points $z_1= 1/3w - 1/3 + \frac{i}{w}$ and $z_2= 1/3w-1/3 + \frac{i}{w}$ which lie on $ \partial S \mathcal{F}$ and have imaginary part $\frac{1}{w}$. Then $z_1$ and $z_2$ can  not be mapped to any of the two other points lying on $\partial S\mathcal{F}$ under the action of $\SL(L_{13}^1)$ since $z_1$ and $z_2$ are the midpoints of the isometric circles containing them. Thus all points in $z \in \mathbb{H}$ with $\Imag(z) = \frac{1}{w}$ which might be equivalent to $z_1$ or $z_2$ under the action of $\SL(L_{13}^1)$ are given by $z_1 + w\ZZ$ and $z_2 + w\ZZ$. So suppose that $z_1$ is mapped to $z_2$. By (\ref{eqn_trans}) we then have $a = \pm 1$ and so $V = \left( \begin{smallmatrix} \pm 1 & 0 \\ w & \pm 1 \end{smallmatrix} \right)$. Thus $V$ lies in one of the double cosets given in $(ii)$. Moreover, by looking at (\ref{eqn_trans}) we immediately see that there cannot exist other matrices inside $\SL(L_{13}^1)$ with $c = w$.
\end{proof}

We use this to prove the following key lemma.

\begin{lem} \label{lem_key} If $U \in SL_2(K)$ stabilizes the cusp $\infty$ then 
$$U \in \SL(L_{13}^1) \ \textrm{if and only if} \ U \in \Comm_{SL_2(\mathbb{R})} (\SL(L_{13}^1)).$$ \end{lem}

\begin{proof} Let $z \mapsto z + \alpha \in \Comm_{SL_2(\mathbb{R})} (\SL(L_{13}^1))$, i.e. the map stabilizes $\infty$. Evidently $\alpha \in \mathbb{Q}(\sqrt{13})$. \\[11pt] At first we consider the case $\alpha \in \mathcal{O}_{13}$. As $S \in \SL(L_{13}^1)$ we know that $0$ is a representative of the cusp $\left[\infty\right]$. Since any element in $\Comm_{SL_2(\mathbb{R})} (\SL(L_{13}^1))$ permutes the cusps of $\SL(L_{13}^1)$, we also know that $\alpha$ is a cusp which is equivalent to $\infty$ under $\SL(L_{13}^1)$. Therefore, there exists a matrix $V = \left( \begin{smallmatrix} a & b \\ c & d \end{smallmatrix} \right)$ in $\SL(L_{13}^1)$ with $c>0$ and $\frac{a}{c}=\alpha$. The entries of $V$ lie in $\mathcal{O}_{13}$. For this reason $c$ must be a unit. Thus
$$V \begin{pmatrix} 1 & w \\ 0 & 1 \end{pmatrix} V^{-1} = \begin{pmatrix} * & * \\ -wc^2 & * \end{pmatrix}$$
is the generator of the stabilizer of $\alpha$ in $\SL(L_{13}^1)$. According to the remarks about the commensurator in Section~\ref{sec_Fuchsian_groups} the group generated by
$$\begin{pmatrix} 1 & -\alpha \\ 0 & 1 \end{pmatrix} V \begin{pmatrix} 1 & w \\ 0 & 1 \end{pmatrix} V^{-1} \begin{pmatrix} 1 & \alpha \\ 0 & 1 \end{pmatrix} = \begin{pmatrix} 1 & 0 \\ -w c^2 & 1 \end{pmatrix}$$
is commensurable to the group $\left( \begin{smallmatrix} 1 & 0 \\ w \ZZ & 1 \end{smallmatrix} \right)$. Therefore, $c^2$ is rational (and a unit) and hence $c=1$ and $a=\alpha$. According to Lemma~\ref{lem_double_cosets} (i), then $\alpha \in w\ZZ$.\\[11pt]
For every map $z \mapsto z + \beta \in \Comm_{SL_2(\mathbb{R})} (\SL(L_{13}^1))$ we hence have $\beta = r w$ with $r \in \mathbb{Q}$. To show that $r$ lies in fact in $\ZZ$ we only have to show that for every map $z \mapsto z + \frac{w}{n} \in \Comm_{SL_2(\mathbb{R})} (\SL(L_{13}^1))$ it must be the case that $n=1$.\\
So suppose that $z \mapsto z + \frac{w}{n} \in \Comm_{SL_2(\mathbb{R})} (\SL(L_{13}^1))$. Then there is a matrix $V= \left( \begin{smallmatrix} a & b \\ c & d \end{smallmatrix} \right)\in \SL(L_{13}^1)$ such that $c>0$ and $\frac{a}{c}=\frac{w}{n}$, i.e. $a= \rho w, \ c=\rho n $ with $\rho \in \mathbb{Q}(\sqrt{13})$ and $\rho >0$. Then $ad-bc=1$ implies that $\rho^{-1} \in \mathcal{O}_{13}$. Furthermore the group generated by
$$\begin{pmatrix} 1 & -w/n \\ 0 & 1 \end{pmatrix} V \begin{pmatrix} 1 & w \\ 0 & 1 \end{pmatrix} V^{-1} \begin{pmatrix} 1 & w/n \\ 0 & 1 \end{pmatrix} = \begin{pmatrix} 1 & 0 \\ -w c^2 & 1 \end{pmatrix}$$
is commensurable with $\left( \begin{smallmatrix} 1 & 0 \\ w\ZZ & 1 \end{smallmatrix} \right)$, i.e. $\rho^2 \in \mathbb{Q}$. As $\rho^{-2}$ is a rational, algebraic integer which divides $w^2$ and as $w$ is not divisible by any element in $\ZZ_{>1}$, $\rho=1$ holds. Thus we have $c=n$ and $a=w$. Moreover we know that $S \cdot V = \left( \begin{smallmatrix} -n & -d \\ w & b \end{smallmatrix} \right)\in \SL(L_{13}^1)$ and therefore by Lemma~\ref{lem_double_cosets} $(ii)$ $n=1$.\\[11pt]
Finally we look at an arbitrary map $W(z)=az+b \in \Comm_{SL_2(\mathbb{R})} (\SL(L_{13}^1))$ with $a>0$. Hence $WTW^{-1}$ and $W^{-1}TW$ lie in $\Comm_{SL_2(\mathbb{R})} (\SL(L_{13}^1))$. Then $WTW^{-1}(z)=z+aw$ and $W^{-1}TW(z)=z+a^{-1}w$ yield $a=1$ and $b \in w \ZZ$. This proves the lemma.
\end{proof}

For the other two cusps of $\mathbb{H}/ \SL(L_{13}^1)$, namely $1$ and $w/2$, we are not able to exploit the geometry of the fundamental domain in a similar way because we do not know how to formulate an analogue of Lemma~\ref{lem_double_cosets}. Therefore, we have to use a different approach. 

\begin{lem} \label{lem_roots} \textbf{(Root lemma for parabolic elements)} \index{root lemma} Let $\Gamma$ be a Fuchsian group with $\mu(\mathbb{H} / \Gamma) < \infty$ and let $\gamma_1,...,\gamma_r$ be a system of representatives of the primitive parabolic elements of $\Gamma$ up to conjugacy. If $\Gamma'$ is a Fuchsian group containing $\Gamma$, then for every parabolic element $\gamma' \in \Gamma'$ there is a parabolic element $\gamma \in \Gamma$, such that
\begin{itemize}
\item[(i)] $\gamma$ is conjugated to one of the $\gamma_i$ and 
\item[(ii)] $\gamma'^n=\gamma^k$ for some $k,n \in \ZZ$ with $(k,n)=1$. 
\end{itemize}
\end{lem}

\begin{proof} Let $\gamma'$ be a parabolic element in $\Gamma' \backslash \Gamma$. Since $\mu(\mathbb{H} / \Gamma) < \infty$ and since $\Gamma'$ is a Fuchsian group we have $\left[\Gamma':\Gamma\right] < \infty$. Therefore, there exists a minimal $n \in \mathbb{N}$ with  $\gamma'^n \in \Gamma$. As $\gamma'^n$ is also parabolic it is conjugated to the $k$-th, $k \in \ZZ$, power of one of the $\gamma_i$ inside $\Gamma$, i.e.
$$ \gamma'^n = M^{-1} \gamma_i^k M$$
with $M \in \Gamma$ and $(k,n)=1$ since $n$ is minimal. 
\end{proof}

An immediate consequence of this lemma is:

\begin{rem} \label{rem_root} Let $\Gamma$ be a Fuchsian group with $\mu(\mathbb{H} / \Gamma) < \infty$ and let $\gamma_1,...,\gamma_r$ be a system of representatives of the primitive parabolic elements of $\Gamma$ up to conjugacy. If $\Gamma'$ is a Fuchsian group containing $\Gamma$ and at least one parabolic element not in $\Gamma$ then $\Gamma'$ also contains a $n$-th root of one of the $\gamma_i$, i.e. there is a matrix $\gamma' \in \Gamma'$ with $\gamma' \notin \Gamma$ but $\gamma'^n \in \Gamma$ for some $n \in \NN$. \end{rem} 

For $\SL(L_{13}^1)$ the root lemma allows us to move away from the explicit geometric arguments and to give a second proof of Lemma~\ref{lem_key}. In fact we prove two much more general versions of Lemma~\ref{lem_key}.

\begin{cor} \label{cor_cusp_stab<49} Let $D\equiv 1 \mod 4$ and and let $U$ stabilize the cusp at $\infty$ then
$$U \in \SL(L_D^1) \ \textrm{if and only if} \ U \in \Comm_{SL_2(\mathbb{R})} (\SL(L_D^1)).$$
\end{cor}

\begin{proof} Let $U\in \Comm_{SL_2(\mathbb{R})} (\SL(L_D^1))$ stabilize the cusp at infinity. Without loss of generality we may then assume that $U = \left( \begin{smallmatrix} 1 & a \\ 0 & 1 \end{smallmatrix} \right)$ and that $U \notin \SL(L_D^1)$. By the root lemma we then have that $a=w/n$ with $n \in \mathbb{N}$.\footnote{The claim follows immediately for all $D<49$ and $n>2$: we then have $w/n<2$ which implies (since $S \in \SL(L_D^1)$, since the commensurator is Fuchsian and since $\mathbb{Q}(\sqrt{D}) \subset \mathbb{Q}(\zeta_D)$, where $\zeta_D$ is the primitive $D$th root of unity) that $w/n=\cos(2 \pi /D)$.  However for all $D<49$ and $n>2$ we have $w/n<\cos(2 \pi /D)$. This is a contradiction.} One proves by induction that $(US)^k = \left( \begin{smallmatrix} p_k(a) & q_k(a) \\ r_k(a) & s_k(a) \end{smallmatrix} \right)$ with $p_k,q_k,r_k,s_k$ polynomials in $a$ with $p_k(a) = a^k + b_{k-1} a^{k-1} + ... + b_0$ and $\deg(q_k) , \deg(r_k), \deg(s_k) \leq k-1$. Now suppose that there exists $U \notin \SL(L_D^1)$ such that $\left[\left< \SL(L_D^1), U \right>: \SL(L_D^1)\right] < \infty$. Then there exists a $k \in \NN$ such that $(US)^k \in \SL(L_D^1) \subset \SL_2(\OD)$. In particular $p_k(\frac{w}{n}) \in \OD$ which implies that $n | w^k$. Finally, this contradicts Lemma~\ref{lem_properties_OD}~(v).
\end{proof}

\begin{cor} If $\Gamma$ is a non-arithmetic cofinite Fuchsian-group containing $S$ such that the assertion of Lemma~\ref{lem_double_cosets} (ii) holds, then for all $U \in \SL_2(K)$ which stabilize $\infty$
$$U \in \Gamma \ \textrm{if and only if} \ U \in \Comm_{SL_2(\mathbb{R})} (\Gamma)$$
holds.
\end{cor}

\begin{proof} We may again assume $U= \left( \begin{smallmatrix} 1 & w/n \\ 0 & 1 \end{smallmatrix} \right) \in \Comm_{SL_2(\mathbb{R})} (\Gamma) \smallsetminus \Gamma$ with $n>2$. As in the proof of Lemma~\ref{lem_key} there then exists a matrix $V = \left( \begin{smallmatrix} -n & w \\ -d & b \end{smallmatrix} \right) \in \Gamma$. This contradicts the assertion of Lemma~\ref{lem_double_cosets} $(ii)$.
\end{proof}

We now want to show that a Fuchsian group that contains $\SL(L_{13}^1)$ can never contain any  additional parabolic elements. So suppose that there exists a Fuchsian group $\Gamma$ which contains $\SL(L_{13}^1)$ and at least one parabolic element $U \notin \SL(L_{13}^1)$. As we know by the root Lemma~\ref{lem_roots} the only possible candidates for the parabolic elements $U$ such that $\left\langle \SL(L_{13}^1),U\right\rangle$ might still be Fuchsian are the roots of the (non-conjugated) parabolic elements in $\SL(L_{13}^1)$. However there are just $3$ (non-conjugated) parabolic elements inside $\SL(L_{13}^1)$ and thus we know that only parabolic roots of $$B= \begin{pmatrix} -5-4w & 6+5w \\ -4-4w & 7+4w \end{pmatrix} \quad \text{or} \quad C=\begin{pmatrix} -4-4w & 5+4w \\-5-4w & 6+4w \end{pmatrix}$$ may occur. We can exclude almost all of the candidates with the help of the following two theorems.

\begin{thm} (\cite{Kat92}, Theorem~2.4.8) A non-elementary subgroup $\Gamma$ of $\PSL_2(\mathbb{R})$ is discrete if and only if, for each $V$ and $W$ in $\Gamma$, the group $\left\langle V,W \right\rangle$ is discrete. \end{thm}

A Fuchsian group $\Gamma$ is called \textbf{elementary} if there exists a finite $\Gamma$-orbit in $\overline{\HH}$. Equivalently the commutator of any two elements of infinite order has trace $2$ (see \cite{Ros86}). Note that a Fuchsian group containing parabolic as well as hyperbolic elements is never elementary. From this fact it can be observed that all the two generator Fuchsian which we consider in this chapter are indeed non-elementary. We now only need a criterion to find out whether $\left\langle V,W \right\rangle$ is discrete.

\begin{thm} \label{prop_criterion_discretness} (\cite{Bea83}, Theorem~11.4.2) Let $V,W \in \PSL_2(\mathbb{R})$ such that $V$ is parabolic and $G=\left\langle V,W \right\rangle$ is a non-elementary Fuchsian group. Then $tr[V,W] \geq 3$. If $3 \leq tr[V,W] < 6$ then $tr[V,W]=4+2 \cos (2 \pi / q)$ with $q \in \mathbb{N}$. \end{thm}

By this we are able to exclude all roots of $B$ and $C$ of degree $\geq 3$ by an easy calculation. First define for $n \geq 3$, $V=V_n$ as $$V_n:= \begin{pmatrix} 1 - \frac{6+4w}{n} & \frac{(6+5w)}{n} \\ \frac{-4-4w}{n} & 1 + \frac{6+4w}{n} \end{pmatrix}$$ (the roots of B) and set $W=T \cdot S$. Since $n \geq 3$ it follows that $\tr[V,W]= 2 + \frac{7+5\textit{w}}{n^2} < 6$ and none of the values is equal to $4+2\cos(2 \pi /q)$ for a $q \in \mathbb{N}$. Hence, the subgroups generated by $V_n$ and $W$ cannot be Fuchsian for $n \geq 3$.  Secondly define for $n \geq 3$, $V=\widetilde{V_n}$ as $$\widetilde{V_n}:= \begin{pmatrix} 1 - \frac{5+4w}{n} & \frac{5+4w}{n} \\ - \frac{5+4w}{n} & 1 + \frac{5+4w}{n} \end{pmatrix}$$ (the roots of C) and set $W=T \cdot S$. Then $\tr[V,W]= 2 + \frac{7+5\textit{w}}{n^2}$ and we repeat the same argument as above to exclude the matrices $\widetilde{V_n}$.  It now only remains to check the case $n=2$. This will for both $B$ and $C$ yield a contradiction to the key Lemma~\ref{lem_key}. 

\begin{itemize}
\item[(i)]Suppose that $B_2:=B^{1/2} \in \Comm_{\SL_2(\mathbb{R})} (\SL(L_{13}^1))$, then $$S \cdot B_2^{-1} \cdot (C^{-1} \cdot S^2)^4 = \begin{pmatrix} -2-2w & 2+2w \\ 0 & 1-1/2w \end{pmatrix} \in \Comm_{\SL_2(\mathbb{R})} (\SL(L_{13}^1)).$$ This contradicts Lemma~\ref{lem_key}.
\item[(ii)] Suppose that $C_2:=C^{1/2} \in \Comm_{\SL_2(\mathbb{R})} (\SL(L_{13}^1))$, then $$(S \cdot T)^2 \cdot (S \cdot C \cdot T)^6 \cdot S \cdot C_2 =$$ $$\begin{pmatrix} -7/2-3/2w & -5/2-11/2w \\ 0 & 8+6w \end{pmatrix} \in \Comm_{\SL_2(\mathbb{R})} (\SL(L_{13}^1)).$$ This again contradicts Lemma~\ref{lem_key}.
\end{itemize}

Collecting everything we see that we have established:

\begin{prop} Let $\Gamma$ be any Fuchsian group containing $\SL(L_{13}^1)$. If $U \in \Gamma$ is parabolic then $U \in \SL(L_{13}^1)$.
\end{prop}

\begin{defi} A Fuchsian group $\Gamma$ is called \textbf{parabolic maximal} \index{Fuchsian group!parabolic maximal} if there does not exist a Fuchsian group $\Gamma'$ containing $\Gamma$ and a parabolic element $U \in \Gamma' \smallsetminus \Gamma$. \end{defi}
Using the Riemann-Hurwitz formula (Theorem~\ref{thm_Riemann_Hurwitz}) \index{Riemann-Hurwitz formula} we can now deduce Theorem~\ref{thm_maximality}.


\begin{proof}[Proof of Theorem~\ref{thm_maximality} for $D=13$] Recall that the signature of $\SL(L_{13}^1)$ is $(0;2;3)$. Now we suppose that there exists a maximal Fuchsian group $\Gamma$ of signature $(g;m_1,...,m_r;s)$ containing $\SL(L_{13}^1)$ such that $\left[\Gamma : \SL(L_{13}^1)\right] = k \geq 3$. Since $\SL(L_{13}^1)$ has exactly three cusps and $\Gamma$ cannot have additional parabolic elements, the degree of the covering is exactly equal to $3$. 
Therefore, we must have that $s=1$. For the Euler characteristic of $\SL(L_{13}^1)$ we have $\mathbb{H}/\SL(L_{13}^1) =3/2$ and thus by the Riemann-Hurwitz formula:
\begin{eqnarray*}
\frac{3}{2} & = & 3 \left( 2g-2+1+\sum_{j=1}^r \left( 1 - \frac{1}{m_j} \right)\right)\\
\frac{1}{2} & = & 2g-1+\sum_{j=1}^r \left( 1 - \frac{1}{m_j} \right).
\end{eqnarray*}
In particular
\begin{eqnarray} \label{equ_rh_fin_max}
\frac{3}{2} & \geq & 2g + \sum_{j=1}^r \left( 1 - \frac{1}{m_j} \right).
\end{eqnarray}
Hence $g=0$. Since the degree of the covering is $3$, there can be at most $2$ elliptic elements in $\Gamma$. If inequality~(\ref{equ_rh_fin_max}) is fulfilled then $\Gamma$ must have signature $(0;3,6,\infty;1)$ or $(0;4,4,\infty;1)$. This is a contradiction - these groups do not have an appropriate trace field \index{Fuchsian group!trace field}.
\end{proof}

We now show that the assertion also holds for $\SL(L_D^1)$ if $D=17$ instead of $D=13$. 

\begin{proof}[Proof of Theorem~\ref{thm_maximality} for $D=17$] Since the groups $\SL(L_{13}^1)$ and $\SL(L_{17}^1)$ have the same signature, we only need to check if $\SL(L_{17}^1)$ is parabolic maximal. Recall that the generators of $\SL(L_{17}^1)$ are given by
$$T=\begin{pmatrix} 1 & w \\ 0 & 1 \end{pmatrix}, \quad B:=\begin{pmatrix} -3-2w & 4 + 3w \\ -2-2w & 5 + 2w \end{pmatrix},$$
$$S=\begin{pmatrix} 0 & -1 \\ 1 & 0 \end{pmatrix}, \quad C:=\begin{pmatrix} -2-2w & 3+ 2w\\ -3-2w & 4+2w \end{pmatrix}.$$
By Corollary~\ref{cor_cusp_stab<49} and Lemma~\ref{lem_roots} we thus only have to check if the roots of $B,C$ might lie in the commensurator. If we again define $V_n$ as the $n$-th root of $B$ (respectively $C$) and set $W=T \cdot S$ we get that $\tr[V_n,W] < 6$ for $n\geq 3$. Then we only have to check the square roots of the matrices $B$ and $C$.
\begin{itemize}
\item[(i)]Suppose that $B_2:=B^{1/2} \in \Comm_{\SL_2(\mathbb{R})} (\SL(L_{17}^1))$, then $$B_2^{-1}~\cdot~C~\cdot~B_2^{-1}~= \begin{pmatrix} 1 & 1/4w \\ 0 & 1 \end{pmatrix} \in \Comm_{\SL_2(\mathbb{R})} (\SL(L_{17}^1)).$$ This contradicts Corollary~\ref{cor_cusp_stab<49}.
\item[(ii)] Suppose that $C_2:=C^{1/2} \in \Comm_{\SL_2(\mathbb{R})} (\SL(L_{17}^1))$, then $$F~\cdot~C_2~\cdot~A~\cdot~C_2^{-1}~= \begin{pmatrix} 2-3/4w & 17+53/4w \\ 0 & 5+3w  \end{pmatrix} \in \Comm_{\SL_2(\mathbb{R})} (\SL(L_{17}^1))$$ where $$F=T^2 \cdot S \cdot C \cdot S \cdot A^{-1} \cdot S^2 \cdot C \cdot S^2 \cdot C^2 \cdot A \cdot S^{-1} \cdot C \cdot A \cdot S.$$  This again contradicts Corollary~\ref{cor_cusp_stab<49}.
\end{itemize}
\end{proof}
\begin{cor} \label{cor_max_conjugated} For $D \in \left\{ 5, 13, 17 \right\}$ every group $\Gamma \subset \SL_2(\mathbb{R})$ that is commensurable to $\SL(L_D^1)$ is conjugated to a subgroup of $\SL(L_D^1)$.
\end{cor}
Also note that we get a third proof of Lemma~\ref{lem_key} with the help of the root Lemma~\ref{lem_roots} and Theorem~\ref{prop_criterion_discretness}.

\subsection{Pseudo Parabolic Maximal Groups} \label{sec_pseudo}

Also the property of parabolic maximality seems to be hard to prove if the group $\SL(L_D)$ is not explicitly given. For Fuchsian groups having entries in $\OD$ it is possible to define a weaker notion which is still good enough for our purposes. 
\begin{defi} \label{def_pseudo_parabolic}
\begin{itemize}
\item[(i)]  We call a Fuchsian group $\Gamma \subset \SL_2(\OD)$ \textbf{pseudo parabolic maximal} \index{Fuchsian group!pseudo parabolic maximal} if there does not exist a Fuchsian group $\Gamma'$ containing $\Gamma$ with finite index and also containing a parabolic element in $\SL_2(K) \smallsetminus \SL_2(\OD)$.
\item[(ii)] We call a Fuchsian group $\Gamma \subset \SL_2(\OD)$ \textbf{$\mathfrak{n}$-pseudo parabolic maximal} for a (proper) ideal $\mathfrak{n} \subset \OD$ if there does not exist a Fuchsian group $\Gamma'$ containing $\Gamma$ with finite index and also containing a parabolic element in $\SL_2(K) \smallsetminus (\SL_2(K) \cap \Mat^{2x2}(\mathfrak{n}^{-1}))$. 
\end{itemize}
\end{defi}
Note that the stabilizer cannot distinguish pseudo parabolic maximal groups from parabolic maximal groups since $\Stab(\Phi) \cap \SL_2(\OD) = \SL(L_D)$, but this property is much easier to check. This is why pseudo parabolic maximality is a good property to consider.\\[11pt]
Before going through the rest of this section, we recommend the reader to repeat the arithmetic properties of $\OD$ which we described in Section~\ref{sec_quadratic_nf}.\\[11pt]
In this section we want to prove that all the groups $\SL(L_D^i)$ are almost pseudo parabolic maximal, i.e. there exists at most an ideal $\mathfrak{n}_i \subset \OD$ of norm of absolute value $2$, that depends only on whether $D$ is even or odd and on the spin of the Teichmüller curve, such that $\SL(L_D^i)$ is $\mathfrak{n}_i$-pseudo parabolic maximal.\footnote{Recall our abuse of notation of the term spin, compare Section~\ref{sec_definition_tmcurves}.} Unfortunately, it is not convenient to treat here the even and the odd spin case simultaneously if $D \equiv 1 \mod 4$ for the following reason: the lower left entry of the matrix $Z$ is $w$ in the odd spin case and $w+1$ in the even spin case. While $(w)$ has only prime ideal divisors which are divisors of split prime numbers (Lemma~\ref{lem_properties_OD}), we have $\N(w+1) = \frac{D-9}{4}$ and so if $3|D$ then $(w+1)$ has also a prime ideal divisor which is a divisor of a ramified prime number, namely $3$. This will force us to treat some special cases for even spin Teichmüller curves.
\paragraph{D $\equiv$ 1 mod 4, odd spin.} Let us begin with the Teichmüller curves generated by the symmetric L-shaped polygons. We want to prove the following theorem.

\begin{thm} \label{cor_1_mod_8} For all $D \equiv 1 \mod 4$ the group $\SL(L_D^1)$ is $(2)$-pseudo parabolic maximal.
\end{thm}

If we have more information on the discriminant we can strengthen the result.

\begin{thm} \label{thm_pseudo_parabolic} \begin{itemize}
\item[(i)] For all $D \equiv 5 \mod 8$ the group $\SL(L_D)$ is pseudo parabolic maximal.
\item[(ii)] For all $D \equiv 1 \mod 8$ the group $\SL(L_D^1)$ is $\mathfrak{p}_2$-pseudo parabolic maximal, where $\mathfrak{p}_2$ \label{glo_P2} is the (unique) common prime ideal divisor of $(2)$ and $(w)$.
\end{itemize}
\end{thm}

Both of these facts will be important later on.\\[11pt]Fix a discriminant $D \equiv 1 \mod 4$ and assume that $\SL(L_D^1)$ is not pseudo parabolic maximal and $\Gamma'$ is a corresponding Fuchsian group containing $\SL(L_D^1)$. We have already seen that $\left[\Gamma'~:~\SL(L_D^1)\right]~<~\infty$ implies that only roots of parabolic elements in $\SL(L_D^1)$ might lie in $\Gamma'$. Let 
$$M = \begin{pmatrix} 1 - a & b \\ c & 1 + a \end{pmatrix}$$
be an arbitrary parabolic element in $\SL(L_D^1)$. We may without loss of generality assume that $b \neq 0$ and $c \neq 0$ by Corollary~\ref{cor_cusp_stab<49} and hence $c = \frac{-a^2}{b}$. The n-th root of $M$ is of the form 
$$M_n := \begin{pmatrix} 1 - \frac{a}{n} & \frac{b}{n} \\ \frac{c}{n} & 1 + \frac{a}{n} \end{pmatrix}.$$ As we want to analyze how much $\SL(L_D^1)$ differs from being pseudo parabolic maximal we are just interested in those roots not lying in $\SL_2(\OD)$. Since $n | b$ and $n|c$ yield $n|a$ we may in the following always assume that $n \nmid c$ or $n \nmid b$. Furthermore we may obviously assume that $n$ is a prime number in $\NN$ and that $\Gamma'$ is the group $\left< \SL(L_D^1),M_n \right>$. Given $M_n$ we now construct (for almost all $n$) a matrix $L \in \Gamma'$ such that $L^k \notin \SL(L_D^1)$ for all $k \in \NN$. This contradicts the assumption  $\left[\Gamma'~:\SL(L_D^1)\right]~<~\infty.$

\begin{lem} \label{lem_leading_coefficient_1} Let 
$$N = \begin{pmatrix} 1 - e & f \\ g & 1 + e \end{pmatrix}$$
be any parabolic element in $\SL(L_D^1)$. Then
$$(M_nN)^k = \begin{pmatrix} p_k(\frac{1}{n}) &  * \\ q_k(\frac{1}{n}) & * \end{pmatrix}$$
where $p_k(\frac{1}{n}), q_k(\frac{1}{n})$ are polynomials in $\frac{1}{n}$ of degree $k$ with leading coefficient
$$(-1)^k ((e-1)a + gb)(-fc+2ea+gb)^{k-1} \quad \it{for} \ \it{p_k} $$
and
$$(-1)^k (-(e-1)c + ga)(fc+2ea+gb)^{k-1} \quad \it{for} \ \it{q_k}.$$
\end{lem}

\begin{proof} This follows by induction on $k$ or by \cite{MR03}, Lemma~3.1.3. \end{proof}

Applying Lemma~\ref{lem_leading_coefficient_1} to $T$ already gives a rather restrictive condition on $a,b$ and $c$.

\begin{prop} \label{prop_pseudo_parabolic_hk>1} Let $n$ be an arbitrary prime number in $\NN$. If $n \nmid wc$ or $n \nmid wb$ or $n \nmid wa$, then
$$\left[\left<\SL(L_D^1),M_n\right>:\SL(L_D^1)\right] = \infty.$$ 
\end{prop}

\begin{proof} First assume that $n \nmid wc$. We use the notation from Lemma~\ref{lem_leading_coefficient_1} which we apply to $(M_nT)^k$. The leading coefficient of $q_k$ is in this case $(-1)^k w^{k-1} c^k$. We assume that $\left[\left<\SL(L_D^1),M_n\right>:\SL(L_D^1)\right] < \infty.$ Then there exists $k \in \NN$ with $n|w^{k-1} c^{k}$. This and Lemma~\ref{lem_properties_OD} imply that $n|w^2c^2$ if $n$ ramifies or otherwise that $n|wc$. So $n=\mathfrak{p}^2$ must be ramified and therefore by Lemma~\ref{lem_properties_OD}~(v) we have $n|c^2$. We may thus without of loss of generality assume that $\mathfrak{p}|c$ and $\mathfrak{p}|b$ (otherwise the statement would be true by interchanging the roles of $c$ and $b$) and hence $\mathfrak{p}|a$. Then the (fractional) ideal generated by the lower left entry of $M_n$ is of the form $\mathfrak{c}\mathfrak{p}^{-1}$ for some proper ideal $\mathfrak{c}$ and $\mathfrak{p} \nmid \mathfrak{c}$. Looking at $(M_nT)^k$ once again one then notices that the necessary condition for the existence of some $k_0 \in \NN$ such that $(M_nT)^k$, namely $\mathfrak{p}|(w)^2\mathfrak{c}^2$, contradicts Lemma~\ref{lem_properties_OD}~(iii) and therefore cannot be fulfilled.\\
If $n \nmid wb$ then one again uses the same argument with $T$ replaced by $Z=T^t$. If $n\nmid wa$ the claim follows since $a^2 = -bc$ and hence $n \nmid wb$ or $n \nmid wc$.
\end{proof}

From this it follows immediately that $n$ must be a split prime number, i.e. $(n)=\mathfrak{p}\mathfrak{p}^\sigma$, a product of two conjugated prime ideals in $\OD$, if the index is finite. Moreover $n | \N(w) = \frac{D-1}{4}$. Another infinite series of matrices in $\Gamma'$ can also be explicitly calculated.

\begin{lem} \label{lem_leading_coefficient_2} We have $$(M_nS)^k = \begin{pmatrix} p_k(\frac{1}{n}) & * \\ q_k(\frac{1}{n}) & * \end{pmatrix}$$
where $p_k(\frac{1}{n})$ and $q_k(\frac{1}{n})$ are polynomials in $\frac{1}{n}$ of degree $k$ with leading coefficients $(-1)^k c (c-b)^{k-1}$ for $p_k$ and $(-1)^k a (c-b)^{k-1}$ for $q_k$.
\end{lem}

\begin{proof} Also follows by induction on $k$. \end{proof}

Similarly as above this gives some additional restrictions on the entries of $M_n$.

\begin{prop} \label{prop_pseudo_parabolic2} If $n=\mathfrak{p}\mathfrak{p}^{\sigma}$ is a split prime number with $\mathfrak{p} \nmid(c-b)$ and $\mathfrak{p}^\sigma\nmid(c-b)$ then
$$\left[\left<\SL(L_D^1),M_n\right>:\SL(L_D^1)\right] = \infty.$$ 
\end{prop}

\begin{proof} Since we may assume that $n \nmid c$ the claim follows immediately from Lemma~\ref{lem_leading_coefficient_2} by the same arguments as in Proposition~\ref{prop_pseudo_parabolic_hk>1}. \end{proof}

Without loss of generality we may then assume $\mathfrak{p}^{\sigma} | (c-b)$ and therefore in particular $\mathfrak{p}^{\sigma}| a^2 + b^2$ since $c=-\frac{a^2}{b}$.\\[11pt]
The desired theorem can now be derived. Recall that the parabolic element that fixes the cusp $1$ is of the form
$$E = \begin{pmatrix} 1 - e & e\\ -e & 1+e \end{pmatrix}.$$
where $e \in \OD$. Moreover there cannot exist a prime number with $n \in \ZZ$ and $n|e$ (see Section~\ref{sec_fixing_Veech}). This allows to prove the main result of this section. 

\begin{proof}[Proof of Theorem~\ref{cor_1_mod_8}] We have already seen $(n)=\mathfrak{p}\mathfrak{p}^{\sigma}$ has to be a split prime number with $(n)|wa$, $(n)|wb$, $(n)|wc$ and $\mathfrak{p}^{\sigma}|a^2 + b^2$. We may furthermore assume that $\mathfrak{p} | b$ but $(n) \nmid b$. We now prove that $\SL(L_D)$ is $(2)$-pseudo parabolic maximal. Assume that the claim is not true. We apply Lemma~\ref{lem_leading_coefficient_1} to $M_nE$. If the index is finite, then $\mathfrak{p}^{\sigma}|a$ or $\mathfrak{p}^{\sigma}|(a-b)$. Since $\mathfrak{p}^\sigma | a^2 + b^2$ and $\mathfrak{p}^\sigma \nmid 2$ it thus follows in any case that $\mathfrak{p}^\sigma |b$. This is a contradiction.  \end{proof}

\begin{proof}[Proof of Theorem~\ref{thm_pseudo_parabolic}] If $D \equiv 1 \mod 8$ then $n|wa$, $n|wb$ and $n|wc$ imply that $\SL(L_D)$ is indeed $\mathfrak{p}_2$-pseudo parabolic maximal since $\mathfrak{p}_2^\sigma \nmid w$. If $D \equiv 5 \mod 8$ we also see that the claim is an immediate consequence since $2$ is an inert prime number if $D \equiv 5 \mod 8$.  \end{proof}




For the Veech groups which are explicitly known by C. McMullen's algorithm it now follows that they are indeed parabolic maximal since none of the parabolic elements in $\SL(L_D^1)$ has a root in $\SL_2(\OD)$ in these cases.

\begin{cor} For $D \in \left\{ 5,13,21,29 \right\}$ we have that $\SL(L_D^1)$ is parabolic maximal.  \end{cor}

\begin{cor} In the other cases where $\SL(L_D^1)$ is explicitly known, i.e. $D \in \left\{ 17,33,41 \right\}$, the group is also parabolic maximal. \end{cor}

\begin{proof} None of the parabolic elements in $\SL(L_D^1)$ has a root in $\SL_2(\OD)$ in these cases. It can also be checked that second roots may not appear (e.g. by applying the algorithm from Appendix~\ref{sec_check_elements}). \end{proof}

\paragraph{D $\equiv$ 1 mod 8, even spin.} We now want to show the corresponding result for the Veech groups of even spin Teichmüller curves.
\begin{thm} \label{thm_pseudo_parabolic_even} For all $D \equiv 1 \mod 8$ the group $\SL(L_D^0)$ is $\mathfrak{p}_2^\sigma$-pseudo parabolic maximal. \end{thm}
Recall that the matrices $T= \left( \begin{smallmatrix} 1 & w-1 \\ 0 & 1 \end{smallmatrix} \right)$ and $Z= \left( \begin{smallmatrix} 1 & 0 \\ w+1 & 1 \end{smallmatrix} \right)$ lie in $\SL(L_D^0)$. We assume that $\Gamma'$ is a Fuchsian group containing $\SL(L_D^0)$ with finite index and some additional parabolic element. We use the same notation as in the case of odd spin. The following is then the analogue of Proposition~\ref{prop_pseudo_parabolic_hk>1}.

\begin{prop} \label{prop_pseudo_parabolic2_hk>1} Let $n$ be an arbitrary prime number in $\NN$. If $3 \nmid D$ and if $n \nmid (w-1)c$ or $n \nmid (w+1)b$, then
$$\left[\left<\SL(L_D^0),M_n\right>:\SL(L_D^0)\right] = \infty.$$ 
If $3 | D$ then
$$\left[\left<\SL(L_D^0),M_n\right>:\SL(L_D^0)\right] < \infty$$ 
implies that either $n | (w-1)c$ and $n |(w+1)b$ or that $n=3$ and $3|c$.
\end{prop}

\begin{proof} We proceed similar as in the proof of Proposition~\ref{prop_pseudo_parabolic_hk>1}. We first consider the case that $3 \nmid D$. We again use the notation from Lemma~\ref{lem_leading_coefficient_1} which we apply to $(M_nT)^k$. The leading coefficient of $q_k$ is in this case $(-1)^k (w-1)^{k-1} c^k$. If $\left[\left<\SL(L_D^1),M_n\right>:\SL(L_D^1)\right] < \infty$, then there exists $k \in \NN$ with $n|(w-1)^{k-1} c^{k}$. This and Lemma~\ref{lem_properties_OD} imply that $n|(w-1)^2c^2$ if $n$ ramifies or otherwise that $n|(w-1)c$. So we assume that $(n)=\mathfrak{p}^2$ ramifies and therefore by Lemma~\ref{lem_properties_OD} we have $n|c^2$ and hence $\mathfrak{p}|c$. We now look at the sequence of matrices $(M_nZ)^k$. Note that $\N(w+1) = \frac{D-9}{4}$ implies that $(w+1)$ has no prime ideal divisors which also divide $D$. The assumption that the index is finite therefore yields that we must have $n|b^2$ and hence $\mathfrak{p}|b$ and $\mathfrak{p}|a$. Then the (fractional) ideal generated by the lower left entry of $M_n$ is of the form $\mathfrak{c}\mathfrak{p}^{-1}$ for some proper ideal $\mathfrak{c}$ and $\mathfrak{p} \nmid \mathfrak{c}$. Looking at $(M_nT)^k$ once again one then notices that $\mathfrak{p}|(w-1)^2\mathfrak{c}^2$ contradicts Lemma~\ref{lem_properties_OD}~(iii) since $(w)^\sigma=(w-1)$.\\ 
If $3\nmid D$, then interchanging the roles of $T$ and $Z$ yields $n|(w+1)b$.\\[11pt]
If $3|D$ then $(w+1)$ has a common prime ideal divisor with the ramified prime number $3$. We denote this ideal by $\mathfrak{p}_3$. If $n \neq 3$ then the preceding proof of course still works. But if $n=3$, then the above only yields that either $n | (w-1)c$ and $n |(w+1)b$ or that $n=3$ and $3|c$. \end{proof}

We then consider the parabolic element in $\SL(L_D^0)$ which fixes the cusp 1, i.e.
$$E=\begin{pmatrix} 1 - e & e \\ -e & 1 +e \end{pmatrix}.$$
with $e \in \OD$. It is known that there does not exist $n \in \ZZ$ with $n|e$ (see Section~\ref{sec_fixing_Veech}). This matrix is again the key ingredient for the proof of Theorem~\ref{thm_pseudo_parabolic_even}.
\begin{proof}[Proof of Theorem~\ref{thm_pseudo_parabolic_even}] Since $((w+1),(w-1))=\mathfrak{p}_2^\sigma$ it suffices to prove that $\SL(L_D^0)$ is $(2)$-parabolic maximal. If $D=17$ then we know all the (parabolic) generators of $\SL(L_{17}^0)$ by Lemma~\ref{lem_second_parabolic2} and it can be checked that the group is indeed parabolic maximal. So we may assume that $D>17$. We know that $\N(w-1)=\frac{D-1}{4}$ and $\N(w+1)=\frac{D-9}{4}$.\\[11pt]Let us first assume that $3 \nmid D$. Then it follows from Proposition~\ref{prop_pseudo_parabolic2_hk>1} that $(n)=\mathfrak{p}\mathfrak{p}^\sigma$ is a split number. Lemma~\ref{lem_leading_coefficient_1} implies that the leading coefficient of the polynomial $q_k$ in $(M_nE)^k$ is
$$(-1)^k((e-1)c+ea)e^{k-1}(c+2a-b)^{k-1}$$
and that the leading coefficient of the polynomial $\widetilde{q_k}$ in $(EM_n)^k$ is
$$(-1)^k((e+1)c+ea)e^{k-1}(c+2a-b)^{k-1}.$$
If $\left[\left<\SL(L_D^0),M_n\right>:\SL(L_D^0)\right] < \infty$ then, since $\mathfrak{p}$ and $\mathfrak{p}^\sigma$ do not divide $e$ at the same time, we may assume that
$$\mathfrak{p}^\sigma | ((e-1)c+ea)(c+2a-b)$$
and
$$\mathfrak{p}^\sigma | ((e+1)c+ea)(c+2a-b).$$
Suppose that $\mathfrak{p}^\sigma \nmid (c+2a-b)$. Then summing up both relations yields $\mathfrak{p}^\sigma |(2c)$ and since $n \neq 2$ hence $\mathfrak{p}^\sigma|(c)$ and thus $\mathfrak{p}^\sigma | (a)$. From the leading coefficients of the upper right entry of $(EM_n)^k$ it then follows that $\mathfrak{p}^\sigma | (b)$. Since $n$ does neither divide $b$ nor $c$, it then follows from $(n)|(w-1)(c)$ and $(n)|(w+1)(b)$ that $\mathfrak{p}|((w-1),(w+1)) = \mathfrak{p}_2^\sigma$. So $n=2$ which is a contradiction.\\ 
Therefore, we may assume that $\mathfrak{p}^\sigma | (c+2a-b)$. As $c = -\frac{a^2}{b}$ we hence have that $\mathfrak{p}^\sigma|(a-b)$ and thus $\mathfrak{p}^\sigma|(w+1)(a)$ since $\mathfrak{p}^\sigma|(w+1)(b)$. Similarly $b = -\frac{a^2}{c}$ implies that $\mathfrak{p}^\sigma|(a+c)$ and therefore $\mathfrak{p}^\sigma|(w-1)(a)$. Then either $\mathfrak{p}^\sigma|((w-1),(w+1))$, i.e. $n=2$, or $\mathfrak{p}^\sigma|(a)$ and therefore also $\mathfrak{p}^\sigma|(b)$ and $\mathfrak{p}^\sigma|(c)$. In both cases we derived a contradiction.\\[11pt]
Now let us assume that $3|D$. Then $(w+1)$ has a prime ideal divisor $\mathfrak{p}_3$ which also divides the ramified prime number $3$. This might cause trouble if $n=3$. If $\left[\left<\SL(L_D^0),M_3\right>:\SL(L_D^0)\right] < \infty$, then we have to distinguish two different subcases.\\[11pt]
\textit{Case (1)} If we have $(3)|(w+1)(b)$ and $(3)|(w-1)(c)$ then $3|(c)$, $\mathfrak{p}_3 | (b)$ and, since $a^2 = -bc$, also $(3)|(a)$. Hence 
$$M_3 = \begin{pmatrix} 1 - \widetilde{a} & \frac{b}{3} \\ \widetilde{c} & 1 + \widetilde{a} \end{pmatrix}$$
with $\widetilde{a},\widetilde{c}$ in $\OD$ and $(b)=\mathfrak{p}_3\mathfrak{b}$ for some proper ideal $\mathfrak{b}$ in $\OD$ and $\mathfrak{p}_3 \nmid \mathfrak{b}$. Then we look at the matrix $L$ from Lemma~\ref{lem_third_parabolic2} and consider $(M_3L)^k$. The upper right entry is a polynomial in $\frac{1}{3}$ and the (fractional) ideal generated by the leading term is $((e+1)g^{k-1}\mathfrak{b}^k)\mathfrak{p}_3^{-k}$, where $e=(2+w)(w-1)$ and $g=(2+w)$ or $e=\frac{(2+w)(w-1)}{2}$ and $g=2(2+w)$ (depending on $D$). Note that $\mathfrak{p}_3 \nmid (e+1)$ and $\mathfrak{p}_3 \nmid 2(2+w)$. Thus there exits a $k \in \NN$ with $(M_3L)^k$ only if $\mathfrak{p}_3|\mathfrak{b}$. This is a contradiction.\\[11pt]
\textit{Case (2)} If we have $(3)|(c)$ but $\mathfrak{p}_3 \nmid b$ then, since $a^2 = -bc$, also $\mathfrak{p}_3|(a)$. Hence 
$$M_3 = \begin{pmatrix} 1 - \frac{a}{3} & \frac{b}{3} \\ \widetilde{c} & 1 + \frac{a}{3} \end{pmatrix}$$
with $\widetilde{c} \in \OD$ and $(a)=\mathfrak{p}_3\mathfrak{a}$ for some proper ideal $\mathfrak{a}$ in $\OD$. Then we again look at the matrix $L$ from Lemma~\ref{lem_third_parabolic2} and consider $(M_3L)^k$. The upper right entry is again a polynomial in $\frac{1}{3}$ and the (fractional) ideal generated by the leading term is $((e+1)g^{k-1}b^k)\mathfrak{p}_3^{-2k}$, where again $g=(2+w)$ or $g=2(2+w)$ (depending on $D$).  Thus there exits a $k \in \NN$ with $(M_3L)^k$ only if $\mathfrak{p}_3|(b)$. This is again a contradiction.
\end{proof}

\begin{cor} The Fuchsian group $\SL(L_{17}^0)$ is maximal. \end{cor}

\begin{proof} $\SL(L_{17}^0)$ and $\SL(L_{17}^1)$ have the same signature. \end{proof}

\paragraph{D $\equiv$ 0 mod 4.} Recall that the prime number $2$ ramifies if $D \equiv 0 \mod 4$ and thus $(2) = \widetilde{\mathfrak{p}_2}^2$ for some prime ideal $\widetilde{\mathfrak{p}_2}$ of norm of absolute value $2$. The following analogue of Theorem~\ref{thm_pseudo_parabolic} will be proven next.

\begin{thm} For all $D \equiv 0 \mod 4$ the group $\SL(L_D)$ is $\widetilde{\mathfrak{p}_2}$-pseudo parabo-lic maximal, where $\widetilde{\mathfrak{p}_2}$ is the (unique) prime ideal divisor of $(2)$.
\end{thm}

The proof is very similar to the case $D \equiv 1 \mod 4$. Therefore, we only sketch parts of the arguments.

\begin{proof} The theorem can be checked directly if $D < 16$. So let $D>16$ from now on. As usual, we look at a prime number $n \in \ZZ$ and 
$$M_n = \begin{pmatrix} 1 - \frac{a}{n} & \frac{b}{n} \\ \frac{c}{n} & 1 +\frac{a}{n} \end{pmatrix},$$ 
a root of some parabolic element in $\SL(L_D)$. We assume that $M_n \notin \SL_2(K) \cap \Mat^{2x2}(\widetilde{\mathfrak{p}_2}^{-1})$ and that
$\Gamma:=\left\langle \SL(L_D), M_n \right\rangle$ contains $\SL(L_D)$ with finite index. We now have to consider three different cases:\\[11pt]
\textit{Case (1) $n \nmid D$:} The condition $(M_nT)^{k_0} \in \SL_2(K) \cap \Mat^{2x2}(\widetilde{\mathfrak{p}_2}^{-1})$ for some $k_0\in \ZZ$ implies that $n|(w+1)c$ and by considering $(M_nZ)^k$ we get that $n|wb$. Hence $n$ cannot be inert and $n|b$. More precisely, $(n)=\mathfrak{p}\mathfrak{p}^\sigma$ must be split. The relations $n|(w+1)c$ and $n|b$ imply that we may without loss of generality assume that $\mathfrak{p}^\sigma|b, \mathfrak{p}^\sigma|c, \mathfrak{p}^\sigma|a$ and that $\mathfrak{p}|a$. Hence $M_n$ is of the form $$\begin{pmatrix} 1 - \widetilde{a} & \widetilde{b} \\ c/n & 1 + \widetilde{a}  \end{pmatrix}$$ 
with $\widetilde{a}, \widetilde{c} \in \OD$ and $\mathfrak{p}\nmid c$ and $\mathfrak{p}|(w+1)$. Let $L$ be the matrix from Lemma~\ref{lem_second_parabolic_0mod4}. Then there must exist some $k_1 \in \ZZ$ such that $(M_nL)^{k_1}\in \SL_2(K) \cap \Mat^{2x2}(\widetilde{\mathfrak{p}_2}^{-1})$ and so $\mathfrak{p}|(w+2)c(1-c)$. There also exists $k_2 \in \ZZ$ such that $(LM_n)^{k_2}\in \SL_2(K) \cap \Mat^{2x2}(\widetilde{\mathfrak{p}_2}^{-1})$ and hence $\mathfrak{p}|(w+2)c(1+c)$. Thus $\mathfrak{p}|(w+2)$ which contradicts $\mathfrak{p}|(w+1)$.\\[11pt]
\textit{Case (2) $n|D$ and $n \neq 2$:} Then $(n)=\mathfrak{p}^2$ ramifies. From the sequence $(M_nT)^k$ we see that $n|c$ and hence $\mathfrak{p}|a$. It happens that $(w)=\mathfrak{p} \mathfrak{g}$ for some ideal $\mathfrak{g}$ with $\mathfrak{p} \nmid \mathfrak{g}$ because otherwise we would have $n|w$. By recalling that $c=-\frac{a^2}{b}$ and then analyzing the upper right entry of $(M_nZ)^k$ we get that $\mathfrak{p}|\mathfrak{g}b$ and so $\mathfrak{p}|b$. Thus $M_n$ is of the form 
$$\begin{pmatrix} 1 - \widetilde{a} & b/n \\ \widetilde{c} & 1 + \widetilde{a}  \end{pmatrix}$$ 
with $\widetilde{a},\widetilde{c} \in \OD$ and $\mathfrak{p}| b$ but $\mathfrak{p}^2 \nmid b$. Looking at $(M_nZ)^k$ again, we see that $\mathfrak{p} | \mathfrak{g}b$. That is again a contradiction.\\[11pt]
\textit{Case (3) $n=2$:} The claim immediately follows since $2 \nmid w+1$ and $2 \nmid w$.
\end{proof}

\subsection{Stabilizer and Commensurator} \label{sec_stab_comm}

Recall that we use the abbreviation $C$ for the Teichmüller curve $C_{L,D}^\epsilon$. In this section we want to prove the following theorem:

\begin{thm} \label{thm_cov_degree_10} Let $M \in \GL_2^+(K) \cap \Mat^{2x2}(\OD)$ and $D$ be a fundamental discriminant and $C$ be a Teichmüller curve of discriminant $D$. Suppose that
\begin{itemize}
\item[(i)] $D \equiv 1 \mod 8$, $C$ has odd spin and $\det(M)$ is not divisible by $\mathfrak{p}_2$ or
\item[(ii)] $D \equiv 1 \mod 8$, $C$ has even spin and $\det(M)$ is not divisible by $\mathfrak{p}_2^\sigma$ or
\item[(iii)] $D \equiv 9 \mod 16$ and $C$ has odd spin or 
\item[(iv)] $D \equiv 5 \mod 8$ or
\item[(v)] $D \equiv 0 \mod 4$ and $\det(M)$ is not divisible by $\widetilde{\mathfrak{p}_2}$
\end{itemize}
then 
\begin{itemize}
\item[(i)] the degree of the covering $\pi: C_M(M) \to C_M$ is equal to $1$ and
\item[(ii)] we have $\SL_M(L_D)=M\SL(L_D)M^{-1}\cap \SL_2(\OD)$. \index{stabilizer!of a twisted Teichmüller curve}
\end{itemize}
\end{thm} 

Our task is now to better understand the stabilizer of the graph of the Teichmüller curve $\Stab(\Phi)$. \index{stabilizer!of the graph of a Teichmüller curve}The commensurator will help us a lot do this because it is the unique maximal Fuchsian group which contains the Veech group.  Let us start with the following statement:

\begin{lem} \label{lem_commensurable_implies_subgroup} Let $\Gamma$ be a non-arithmetic, cofinite Fuchsian group. If $H$ is a finite index subgroup of $\Gamma$ then
$$\Comm_{\SL_2(\mathbb{R})} (H) = \Comm_{\SL_2(\mathbb{R})} (\Gamma).$$
If $\Gamma$ is maximal then 
$$\Comm_{\SL_2(\mathbb{R})} (H) = \Gamma.$$ \end{lem}

\begin{proof} The group $H$ is non-arithmetic because it is a finite index subgroup of a non-arithmetic group. The commensurator $\Comm_{\SL_2(\mathbb{R})} (H)$ is the unique maximal group containing $H$ by Lemma~\ref{lem_non-arithmetic_commensurator}. However, $\Comm_{\SL_2(\mathbb{R})} (\Gamma)$ is a  maximal group containing $H$.
\end{proof}

This proposition allows us to show that the conjugated stabilizer groups of twisted Teichmüller curves are always contained in the commensurator.

\begin{cor} \label{cor_stabcap} For all $M \in \GL_2(K)$ 
$$\Stab(\Phi) \cap M^{-1} \SL_2(\OD) M \leq \Comm_{\SL_2(\RR)}( \SL(L_D)) \cap M^{-1} \SL_2(\OD) M$$
holds. 
If $\SL(L_D)$ is maximal then 
$$\Stab(\Phi) \cap M^{-1} \SL_2(\OD) M = \SL(L_D) \cap M^{-1} \SL_2(\OD) M.$$
\end{cor}

\begin{proof} The group $G:=\Stab(\Phi) \cap M \SL_2(\OD) M^{-1}$ is a Fuchsian group with finite index subgroup $H:=\Stab(\Phi) \cap M \SL_2(\OD) M^{-1} \cap \SL_2(\OD)$. Indeed, $H$ is a finite index subgroup of $\Comm_{\SL_2(\RR)}( \SL(L_D))$ too, and therefore the Fuchsian group $G$ is contained in the unique maximal Fuchsian group $\Comm_{\SL_2(\RR)}( \SL(L_D))$. This yields the claim. \end{proof}

\paragraph{Maximal groups.} Before coming to $\mathfrak{n}$-pseudo parabolic maximal groups, we will at first look at the maximal case. An immediate consequence of the preceding results is the following theorem, which we have already mentioned in the introduction of this chapter.

\begin{thm} \label{thm_degree_of_covering_CnM} If $\SL(L_D)$ is maximal then the degree of the covering $\pi: C_M(M) \to C_M$ is equal to $1$ for all $M \in \GL_2(K)$. In other words $$\SL_M(L_D)=\textrm{M} \SL(L_D)M^{-1} \cap \SL_2(\OD).$$
\end{thm}

It is now natural to ask which set of matrices we actually can exclude from lying in the stabilizer if $\SL(L_D)$ is maximal. 

\begin{defi} We define $\SL_2^{tr}(K) := \left\{x \in \SL_2(K) | \tr(x) \in \OD \right\}$. \label{glo_SL2tr} \end{defi}

\begin{rem} Note that $\SL_2^{tr}(K)$ is not a group. \end{rem}

Let us explain why this is the right object to look at.

\begin{prop} \label{prop_conj_set_of_sl} We have $$\bigcup_{M \in \GL_2(K)} \rm{M} \SL_2(\OD) M^{-1} = \SL_2^{tr}(K).$$ \end{prop}

\begin{proof} The union $\bigcup_{M \in \GL_2(K)} \textrm{M} \SL_2(\OD) M^{-1}$ is a subset of $\SL_2^{tr}(K)$ as the trace is preserved under conjugation. On the other hand by the theory of the rational canonical form two matrices $A$ and $B$ in $\SL_2(K)$ such that neither of them has a double eigenvalue and such that they have the same characteristic polynomial are conjugated in $\GL_2(K)$. Note that for all matrices $A \in \SL_2^{tr}(K)$ there exists a matrix $B=\left( \begin{smallmatrix} \tr(A)-1 & 1 \\ \tr(A)-2 & 1 \end{smallmatrix} \right) \in \SL_2(\OD)$ with the same characteristic polynomial as $A$. Finally, any parabolic matrix in $\SL_2(K)$ is in $\GL_2(K)$ conjugated to a matrix of the form $\left( \begin{smallmatrix} 1 & x \\ 0 & 1 \end{smallmatrix} \right)$ with $x \in K$ 
\end{proof}

From this we can deduce:

\begin{cor} We always have that
$$\Stab(\Phi) \cap \SL_2^{tr}(K) \leq \Comm_{\SL_2(\RR)}( \SL(L_D))$$
If $\SL(L_D)$ is maximal then even
$$\Stab(\Phi) \cap \SL_2^{tr}(K) = \SL(L_D)$$
holds. 
 \end{cor}

Since $\SL_2^{tr}(K)$ is not all of $\SL_2(K)$, there still remain some matrices which might additionally lie in the stabilizer. We can also exclude all the other possible candidates from lying in the stabilizer of the graph of the Teichmüller curve. Before we do this we want to switch to the more general setting of $\mathfrak{n}$-pseudo-parabolic maximal groups. We will analyze three different cases separately, namely the odd spin Teichmüller curves with discriminant $D \equiv 1 \mod 4$, the even spin Teichmüller curves with discriminant $D \equiv 1 \mod 8$ and the Teichmüller curves with discriminant $D \equiv 0 \mod 4$.
\paragraph{D $\equiv$ 1 mod 4, odd spin.} So far everything in this section has been true for all Teichmüller curves, in particular for $D \equiv 1 \mod 8$ for both, the odd and the even spin Teichmüller curves. From now on we have to distinguish the two cases. As usual, we start with the family of Teichmüller curves which includes the odd spin case. Recall that the matrices $T,Z=T^t$ and $S$ always lie in $\SL(L_D^1)$.

\begin{thm} \label{thm_parabolic_maximal_implies_stabilizer} For all fundamental discriminants $D \equiv 1 \mod 4$ where $\SL(L_D^1)$ is pseudo parabolic maximal \index{Fuchsian group!pseudo parabolic maximal}
$\Stab(\Phi) \cap \SL_2(K) = \SL(L_D^1)$ holds.
In this case the degree of the covering $\pi: C_M(M) \to C_M$ is equal to $1$ for all $M$.
\end{thm}

\begin{proof} Let us first assume that there does not exist a prime ideal $\mathfrak{p}$ in $\OD$ with $\mathfrak{p}|2$ and $\mathfrak{p}^2|w$.\\[11pt]
Now suppose there exists a matrix $$X:= \begin{pmatrix} a' & b' \\ c' & d' \end{pmatrix} \in \Stab(\Phi) \cap \SL_2(K) \smallsetminus \SL(L_D^1).$$ Since $\Stab(\Phi) \cap \SL_2(\OD)=\SL(L_D^1)$ we know that $X \notin \SL_2(\OD)$. Then we define $e$ to be the common denominator of $a',b',c',d'$ in $K$. We can write $X$ as
$$X= \begin{pmatrix} \frac{a}{e} & \frac{b}{e} \\ \frac{c}{e} & \frac{d}{e} \end{pmatrix}$$
with $a,b,c,d,e \in \OD$. We now compute
$$XTX^{-1} = \begin{pmatrix} 1 - w \frac{ac}{e^2} & w \frac{a^2}{e^2} \\ -w\frac{c^2}{e^2} & 1 + w \frac{ac}{e^2} \end{pmatrix}$$
which is parabolic and therefore in $\SL(L_D^1)$ and in particular $XTX^{-1} \in \SL_2(\OD)$. \\[11pt]
So let us now assume that $(w)$ has (finitely many) prime (ideal) divisors $\mathfrak{q}_1,...,\mathfrak{q}_n$ that divide $(w)$ with order at least $2$. We assume that $e$ is not a unit. Then it follows from Lemma~\ref{lem_fund_discriminant_quadratic} that $(e) = \prod_{i \in J} \mathfrak{q}_i$ with $J \subset \left\{1,...,n\right\}$. Working in an appropriate localization we may without loss of generality assume that $J=\left\{1\right\}$, i.e. $(e)=\mathfrak{q}_1=:\mathfrak{q}$. Evidently $\mathfrak{q}$ does divide at most three of $a,b,c,d$. Let $L$ be the parabolic element in $\SL(L_D^1)$ which stems from Lemma~\ref{lem_second_parabolic}. Then $L$ is of the form $L= \left( \begin{smallmatrix} 1+v & z \\ r & 1-v \end{smallmatrix} \right)$ and
$$(XLX^{-1})_{1,2} = \frac{1}{e^2} \left( z a^2 - 2v ab - r b \right).$$
Therefore, $\frac{1}{e^2} \left( z a^2 - 2v ab - r b \right) \in \OD$.\\[11pt]
a) Let us first assume that $e$ divides at least one of the $a,b,c,d$. Since $S \in \SL(L_D^1)$ we may (by multiplying $M$ with $S$ and taking inverses) assume that $e \nmid b$ and $e|a$ and hence $a=ea'$ with $(a',e)=1$ (if higher order divisors occur, then we do the same argument for the appearing power). Therefore, we must have that $za'^2 - 2v a' \frac{b}{e} - r \frac{b^2}{e^2} \in \OD$ or equivalently $2v a' \frac{b}{e} - r \frac{b^2}{e^2} \in \OD$. From this it follows that $\mathfrak{q}|rb^2$. Inserting the expression for $r$ from Lemma~\ref{lem_second_parabolic} we see that $\mathfrak{q}|4b^2$. We thus have $\mathfrak{q}|b$ since $\mathfrak{q} \nmid 2$. This is a contradiction.\\[11pt]
b) So we may assume that $e$ divides none of $a,b,c,d$. Since $r=-\frac{v^2}{z}$, $$(XLX^{-1})_{1,2} = \frac{1}{e^2}\frac{v^2}{z} (a\frac{z}{v}-b)^2.$$ On the other hand $\frac{z}{v}$ is equal to $\frac{1}{2}w$.  Thus we have that $\frac{1}{4e^2}\frac{v^2}{z}(aw-2b)^2 \in \OD$. Inserting again the expression for $r$ we get that $\mathfrak{q} | aw-2b$ and so - since $\mathfrak{q}|w$ - also that $\mathfrak{q}|2b$. Since $\mathfrak{q} \nmid 2$ and $\mathfrak{q} \nmid b$, this is not possible. \\[11pt]
Finally we remark that $(w,2)=1$ is not at all a restriction since if there exists a prime ideal $\mathfrak{p}$ with $\mathfrak{p}^2|w$ and $\mathfrak{p} | 2$ then $D \equiv 1 \mod 16$ but then $r = 2(w+1)$ and therefore $\mathfrak{p}^2|aw-2b$ which yields again a contradiction.
\end{proof}
An especially interesting consequence of this is that the information about how much the stabilizer in $\SL_2(K)$ might be bigger than $\SL(L_D^1)$, is somehow only hidden in the additional parabolic elements of the commensurator. We can again deduce a statement about the commensurator of pseudo parabolic maximal groups $\SL(L_D^1)$:

\begin{cor} For all fundamental discriminants $D \equiv 1 \mod 4$ such that $\SL(L_D^1)$ is pseudo parabolic maximal \index{commensurator}
$$\Comm_{\SL_2(\RR)} (\SL(L_D^1)) \subset \SL_2(\OD).$$ 
\end{cor}

Now we can prove the main theorem of this section. The theorem tells us how the stabilizer of an arbitrary twisted Teichmüller curve exactly looks like and thus enables us to calculate the volume of most twisted Teichmüller curves in Chapter~\ref{chapter_calculations}. It is valid for all fundamental discriminants $D \equiv 1 \mod 4$ with $D \not \equiv 1 \mod 16$ and not only for those for which we could prove pseudo parabolic maximality. 

\begin{thm} \label{thm_cov_degree_1} For all fundamental discriminants $D \equiv 1 \mod 4$ with $D \not \equiv 1 \mod 16$ the degree of the covering $\pi: C_M(M) \to C_M$ is equal to $1$ for all $M$. In other words  $\SL_M(L_D^1)=\rm{M}\SL(L_D^1)M^{-1} \cap \SL_2(\OD)$. \end{thm}

\begin{proof} 
If $D \equiv 5 \mod 8$ this follows since $\SL(L_D^1)$ is then pseudo parabolic maximal.\\[11pt] 
If $D \equiv 9 \mod 16$ then by Theorem~\ref{cor_1_mod_8} there might only exist second roots of parabolic elements inside the commensurator. Since $2$ is not a ramified prime number in this case, Lemma~\ref{lem_fund_discriminant_quadratic}~(ii) suffices to derive an analogue statement as in Theorem~\ref{thm_parabolic_maximal_implies_stabilizer} since the lower left entry of the matrix from Lemma~\ref{lem_second_parabolic} is $w+1$ and thus not divisible by $2$. 
\end{proof}
If $D \equiv 1 \mod 16$ then the situation might be a little worse. The proof of the corollary is just the same as for Theorem~\ref{thm_cov_degree_1}. 

\begin{cor} \label{cor_cov_degree_1} Let $M \in \GL_2^+(K) \cap \Mat^{2x2}(\OD)$. Then for all fundamental discriminants $D \equiv 1 \mod 16$ the degree of the covering $\pi: C_M(M) \to C_M$ is equal to $1$ if $\det(M)$ is not divisible by $\mathfrak{p}_2$. \end{cor}

This also implies that the stabilizer is only known to be contained in $\SL_2(K) \cap \Mat^{2x2}(\mathfrak{p}_2^{-1})$ if $D \equiv 1 \mod 16$.
\paragraph{D $\equiv$ 1 mod 8, even spin.} In Theorem~\ref{thm_pseudo_parabolic_even} we have shown that the Veech groups $\SL(L_D^0)$ are $\mathfrak{p}_2^\sigma$-pseudo parabolic maximal. Therefore, it is not surprising that the following analogue of Corollary~\ref{cor_cov_degree_1} holds. 
\begin{prop} \label{cor_cov_degree_10} Let $M \in \GL_2^+(K) \cap \Mat^{2x2}(\OD)$. Then for all fundamental discriminants $D \equiv 1 \mod 8$ the degree of the covering $\pi: C_M(M) \to C_M$ is equal to $1$ if $\det(M)$ is not divisible by $\mathfrak{p}_2^\sigma$. \end{prop} 
\begin{proof} We do the proof here only for the case $D \equiv 1 \mod 16$ because $D \equiv 9 \mod 16$ works in the same way. We prove that the commensurator of $\SL(L_D^0)$ must be contained in $\Mat^{2x2}(\mathfrak{p}_2^{-1})$. Suppose there exists a matrix 
$$X:= \begin{pmatrix} \frac{a}{e} & \frac{b}{e} \\ \frac{c}{e} & \frac{d}{e} \end{pmatrix} \in \Stab(\Phi) \cap \SL_2(K) \smallsetminus \Mat^{2x2}(\mathfrak{p}_2^{-1})$$ 
with $a,b,c,d,e \in \OD$. Since $\SL(L_D^0)$ is $\mathfrak{p}_2^\sigma$-pseudo parabolic maximal, it follows from the entries of $XTX^{-1}$, $X^{-1}TX$, $XZX^{-1}$ and $X^{-1}ZX$  that $e|a$ and $e|d$. We may without loss of generality assume that $(e)$ is a prime ideal. Then either $e|c$ and $e|(w+1)$ or $e|b$ and $e|(w-1)$ holds. In the first case we consider the matrix $L$ from Lemma~\ref{lem_third_parabolic2} to get  $$(XLX^{-1})_{1,2}= u + 2(2+w)(w-1) \frac{ab}{e^2} + 2(w+2) \frac{b^2}{e^2}$$ for some $u \in \OD$.   It follows that $e|b$ because $e|a$ and $e|(w+1)$. This is a contradiction. In the second case we consider the matrix $\widetilde{L}$ from Lemma~\ref{lem_second_parabolic2} to get  $$(X\widetilde{L}X^{-1})_{2,1}= u +  (2w + \frac{D-1}{4}) \frac{cd}{e^2} + w \frac{c^2}{e^2}$$ for some $u \in \OD$. Since $e|d$ and $e|(w-1)$ it follows that $e|c$ which is again a contradiction.
\end{proof}
This finishes the proof of Theorem~\ref{thm_cov_degree_10}~(iv).
\paragraph{D $\equiv$ 0 mod 4.} We will again only present a rather short proof, because everything works very similar as for odd discriminants.
\begin{prop} \label{prop_cov_degree_D04} Let $M \in \GL_2^+(K) \cap \Mat^{2x2}(\OD)$. Then for all fundamental discriminants $D \equiv 0 \mod 4$ the degree of the covering $\pi: C_M(M) \to C_M$ is equal to $1$ if $\det(M)$ is not divisible by $\widetilde{\mathfrak{p}_2}$. \end{prop} 
\begin{proof}
We prove that the commensurator of $\SL(L_D)$ must be contained in $\Mat^{2x2}(\widetilde{\mathfrak{p}_2}^{-1})$. Suppose there exists a matrix 
$$X:= \begin{pmatrix} \frac{a}{e} & \frac{b}{e} \\ \frac{c}{e} & \frac{d}{e} \end{pmatrix} \in \Stab(\Phi) \cap \SL_2(K) \smallsetminus \Mat^{2x2}(\widetilde{\mathfrak{p}_2}^{-1})$$ 
with $a,b,c,d,e \in \OD$. Since $\SL(L_D)$ is $\widetilde{\mathfrak{p}_2}$-pseudo parabolic maximal, it follows from the entries of $XTX^{-1}$, $X^{-1}TX$, $XZX^{-1}$ and $X^{-1}ZX$  that $e|a$ and $e|d$. We may without loss of generality assume that $(e)$ is a prime ideal (otherwise we localize). We see from the preceding list of matrices that either $e|b$ or $e^2|w$. Since $D$ is a fundamental discriminant the latter is not possible and so we have $e|b$. Moreover we see that either $e^2|(w+1)$ or $e|c$. Suppose that the former holds. Then $X$ is of the form
$$ \begin{pmatrix} \widetilde{a} & \widetilde{b} \\ \frac{c}{e} & \widetilde{d} \end{pmatrix}$$
with $\widetilde{a},\widetilde{b},\widetilde{d} \in \OD$. Then we take the matrix $L$ from Lemma~\ref{lem_second_parabolic_0mod4} and look at the lower left entry of $XLX^{-1}$ and see that $e|4(w+2)c^2$ and so $e|4c^2$. Since $(e) \neq \widetilde{\mathfrak{p}_2}$ we have $e|c$. This is a contradiction.
\end{proof}
This completes the proof of Theorem~\ref{thm_cov_degree_10}. \\[11pt]
As we have seen being (pseudo) parabolic maximal is a very useful property of $\SL(L_D)$. We shortly describe a criterion how to decide whether $\SL(L_D)$ is parabolic maximal.\\[11pt] 
Let $\Gamma$ be a Fuchsian group containing $\SL(L_D)$. From the parabolic root Lemma~\ref{lem_roots} one immediately knows that there are only finitely many possible additional parabolic elements in $\Gamma$ (up to conjugation). For each candidate $\gamma$ we can construct (parts of) the Dirichlet fundamental domain of $\left\langle \SL(L_D),\gamma\right\rangle$ in the usual way (see e.g. \cite{Kat92}, Chapter~3.2.). If the area of the partially constructed fundamental domain gets smaller than $\frac{\pi}{3}$ then one knows that $\left\langle \SL(L_D),\gamma\right\rangle$ cannot be Fuchsian since the area of a Fuchsian group with a cusp is bounded from below by $\frac{\pi}{3}$. We emphasize that this stopping criterion works if and only if $\SL(L_D)$ is parabolic maximal. If after \textit{very long time} (whatever this really means) the area of the constructed fundamental domains gets \textit{close} (whatever this really means) to a divisor of the area of $\HH / \SL(L_D)$ then one might tend to assume that $\gamma \in \Comm_{\SL_2(\RR)}(\SL(L_D))$ and might try to prove that $\left< \SL(L_D), \gamma \right>$ contains $\SL(L_D)$ with finite index. It would be an interesting task for future research to find a good algorithm which decides if a Fuchsian group $\Gamma \subset \SL_2(\OD)$ is parabolic maximal.

\newpage
\section{Calculations for Twisted Teichmüller Curves} \label{chapter_calculations}

So far we have treated more or less abstract properties of twisted Teichmüller curves. This chapter will be more explicit: some of the main geometric properties of twisted Teichmüller curves will be derived. First and foremost the volume of these objects will be calculated. Note that unlike in the case of twisted diagonals the classification of skew Hermitian forms cannot be used for this purpose because the function $\varphi$ is involved when dealing with twisted Teichmüller curves (compare \cite{Fra78}, \cite{Hau80}). Instead one can make use of Theorem~\ref{thm_cov_degree_10} and calculate certain group indexes.\\[11pt]In this chapter we assume as \textbf{general condition} that $D$ is a fundamental discriminant.\\[11pt] 
In Section~\ref{sec_volume_simple} we calculate the volume of diagonal twisted Teichmüller curves (Theorem~\ref{thm_summarize_euler_calculations}). A table with some numerical data for the volume of diagonal twisted Teichmüller curves in the cases $D=13$ and $D=17$, that of course supports the results, can be found in Appendix~\ref{appendix_tables}. We will generally assume $h_D=1$ starting with Section~\ref{sec_upper_triangular_twists}. In  Section~\ref{sec_upper_triangular_twists} the volume of Teichmüller curves twisted by upper triangular matrices will be calculated. Surprisingly enough, the calculation can then be reduced to the case of diagonal twisted Teichmüller curves (Theorem~\ref{thm_summarize_euler_calculations_triangular} and Theorem~\ref{thm_summarize_euler_II}). Since the class number is assumed to be equal to $1$, twists by upper triangular matrices indeed give all twisted Teichmüller curves (because the number of cusps of $X_D$ is equal to the class number).\\[11pt] 
In many cases, the volume of twisted Teichmüller curves is an invariant which is good enough for deciding when two twisted Teichmüller curves agree. In Section~\ref{sec_classification_twisted} we will assume that the narrow class number $h_D^+=1$. Note that this implies that $D=8$ or $D \equiv 1 \mod 4$ is a prime by what we have said in Section~\ref{sec_quadratic_nf}. In particular, we will then see that an analogue result as the theorem by H.-G. Franke and W. Hausmann does hold for twisted Teichmüller curves: after normalizing the involved matrices appropriately there are only finitely many twisted Teichmüller curves of a given determinant. Moreover we will get that there is exactly one twisted Teichmüller curve if the determinant of the twisting matrix is prime (Theorem~\ref{thm_prime_classification}).  \\[11pt] Finally in Section~\ref{sec_further_calculations}, we give an outlook how other interesting geometric properties of twisted Teichmüller curves like the number of elliptic fixed points, the number of cusps and the genus can be calculated. These calculations are far from being complete, but should rather be regarded as a rough guideline how things can be done in principle.

\subsection{The Volume of Diagonal Twisted Teichmüller Curves} \label{sec_volume_simple}

\index{twisted Teichmüller curve!volume|(} Recall the following fact: if $\Gamma' \subset \Gamma \subset \PSL_2(\RR)$ are two Fuchsian groups, then the Euler characteristics of the surfaces $\HH /\Gamma$ and $\HH / \Gamma'$ and the index of the subgroup $[\Gamma: \Gamma']$ are closely related via the well-known formula $\chi(\HH /\Gamma')=$ $[\Gamma: \Gamma']\chi(\HH /\Gamma)$ since our definition of Euler characteristic takes orbifold points into account (see Theorem~\ref{thm_Riemann_Hurwitz}). We have seen in the last chapter that the degree of the covering $C_M(M) \to C_M$ is $1$ in most cases. When we want to calculate the volume of diagonal twisted Teichmüller curves, it therefore suffices to calculate the indexes of the groups $\SL^M(L_D,M)$ in $\SL(L_D)$. Recall that $\Gamma^D(m,n)$ is defined as $\Gamma^D_0(m) \cap \Gamma^{D,0}(n)$. This gives us the task to calculate the indexes $[\SL(L_D):(\SL(L_D) \cap \Gamma^D(m,n))]$ for $m,n \in \OD$ with $(m,n)=1$ (see Corollary~\ref{cor_diagonal_twisted_congruence}). \index{congruence subgroup} For fundamental discriminants $D \equiv 1 \mod 4$ we know that this calculation yields the volume of the twisted Teichmüller curve  if $(n,2)=1$ and $(m,2)=1$ (compare Theorem~\ref{thm_cov_degree_10}).\\[11pt] 
As we want to treat the cases $D \equiv 1 \mod 4$ with even or odd spin and $D \equiv 0 \mod 4$ in this section simultaneously, we introduce the following notation: let \label{glo_etas} $\eta^+$ denote the upper right entry of the matrix $T$ and $\eta^-$ denote the lower left entry of the matrix $Z$. Finally let $\eta^*$ be the product of $\eta^+$ and $\eta^-$. In this section we will show that for all elements $m,n \in \OD$ with $(m,n)=1$ and $(n,\eta^*)=1$ and $(m,\eta^*)=1$ the index $[\SL_2(\OD):\Gamma^D(m,n)]$ equals the index $[\SL(L_D):(\SL(L_D) \cap \Gamma^D(m,n))]$. The surrounding arithmetic of $\SL_2(\OD)$ determines so to speak the arithmetic of the twisted Teichmüller curves. Another interpretation of this fact is that the Veech groups \label{Veech group} of Teichmüller curves are the opposite of being arithmetic.\\[11pt] To prove this, we have to distinguish between the different types of splitting behavior of prime numbers \index{real quadratic number field!types of prime numbers} over $\OD$. Before going through the proofs of this section we recommend the reader to recall the number theoretic and arithmetic results from Section~\ref{sec_quadratic_nf} and Section~\ref{sec_congruence_subgroups}.\\[11pt]
The aim of this whole section is to prove the following theorem.
\begin{thm} \label{thm_summarize_euler_calculations} Let $m,n \in \OD$ be arbitrary elements with $(m,n)=1$ and
\begin{itemize}
\item[(i)] if $(m,\eta^*)=1$ and $(n,\eta^*)=1$ or, 
\item[(ii)] if $D \equiv 1 \mod 4$ but $D \not \equiv 1 \mod 16$, the spin of the Teichmüller curve is odd, and $m$ and $n$ are arbitrary or
\item[(iii)] if $D \equiv 1 \mod 16$, the spin of the Teichmüller curve is odd, and $m,n \in \OD$ are arbitrary with $(m,\mathfrak{p}_2)=1$ and $(n,\mathfrak{p}_2)=1$
\end{itemize}
then the degree of the covering $\widetilde{\pi}: C^M(M) \to C$ equals the degree of the covering $\pi: X_D(M) \to X_D$. In other words
$$\left[ \SL(L_D^1) : (\SL(L_D^1) \cap \Gamma^D(m,n))  \right] = \left[ \SL_2(\OD) : \Gamma^D(m,n) \right].$$
If the degree of the covering $\pi : C_M(M) \to C_M$ is equal to 1, then the volume of the Teichmüller twisted by $M= \left( \begin{smallmatrix} m & 0 \\ 0 & n \end{smallmatrix} \right)$ in all these cases is
$$-9\pi \left[ \SL_2(\OD) : \Gamma^D(m,n) \right] \chi(X_D).$$
\end{thm}
Note that the formula for the volume of twisted Teichmüller immediately follows from the first part of the result by Bainbridge's formula (Theorem~\ref{thm_bain_euler}) and the relation between group indexes and volumes which we stated at the very beginning of this chapter. We stress the fact that this is the best result we can achieve in general. For discriminant $D=17$ we have $w = (w+2) \cdot (w+2) \cdot (2w-5)$. Note that $2w-5$ is a unit in $\mathcal{O}_{17}$ and that $w+2=\frac{5+\sqrt{17}}{2}$ is $\pi_2$, the (unique) common prime divisor of $2$ and $w$. A calculation yields that
$$ \left[ \SL(L_{17}^1):(\SL(L_{17}^1) \cap \Gamma^D_{0}(\pi_2) ) \right] = \frac{2}{3} \left[ \SL_2(\OD):\Gamma^D_0(\pi_2) \right]$$ 
holds. Accordingly for $\pi_2^\sigma$ we have
$$ \left[ \SL(L_{17}^0):(\SL(L_{17}^0) \cap \Gamma^D_{0}(\pi_2^\sigma) ) \right] = \frac{2}{3} \left[ \SL_2(\OD):\Gamma^D_0(\pi_2^\sigma) \right].$$ 
This means that the theorem is in general false for $(n,\eta^*)\neq 1$ or $(m,\eta^*)\neq 1$. We will revisit this phenomenon in Section~\ref{subsec_nonrelative}. In particular we will get there that the index 
$$[(\SL(L_D) \cap \Gamma^D((m)\mathfrak{p}^k,(n))):(\SL(L_D^1)\cap \Gamma^D((m)\mathfrak{p}^{k+1},(n)))]$$ is always either $\N(p)$ or $\N(p)+1$ for all $k \in \mathbb{N} \cup \left\{ 0 \right\}$ and all $m,n \in \OD$ with $(m,n)=1$ and all prime ideals $\mathfrak{p}$ if $(\mathfrak{p},n)=1$ and $D \equiv 1 \mod 4$.\\[11pt]
Furthermore, we see that the volume of twisted Teichmüller curves behaves very different than the volume of twisted diagonals. \index{twisted diagonal} For example for $h_D=1$ and $M=\left( \begin{smallmatrix} \pi & 0 \\ 0 & 1 \end{smallmatrix} \right)$, where $\pi$ is an inert prime number or a divisor of a ramified prime number, the stabilizer of the twisted diagonal is always $\SL_2(\ZZ)$ and hence does not depend on $\pi$. This follows immediately from the definition of twisted diagonals. On the other hand if $\pi$ is a divisor of a split prime number, then the volume of the twisted diagonal is up to a constant, which does only depend on $D$, equal to $\N(\pi)+1$. This is a consequence of the work by H.-G. Franke and W. Hausmann (\cite{Fra78}, Theorem~2.4.6 and \cite{Hau80}, Satz 3.10).\\[11pt]
The proof of Theorem~\ref{thm_summarize_euler_calculations} will step through four different cases. The first three cases concern the different types of prime ideals. The fourth concerns those prime ideals $\mathfrak{p}$ with $(\mathfrak{p},\eta^*) \neq 1$. This strange condition will naturally arise. Before we start, recall once again that the matrices $$T=\begin{pmatrix} 1 & \eta^+ \\ 0 & 1 \end{pmatrix}, \quad Z=\begin{pmatrix} 1 & 0 \\ \eta^- & 1 \end{pmatrix}$$ 
are always elements of $\SL(L_D)$. Furthermore note that since $\SL(L_D) \subset \SL_2(\OD)$ we have for all $m,n \in \OD$
$$[\SL(L_D):(\SL(L_D) \cap \Gamma^D(m,n))] \leq [\SL_2(\OD ): \Gamma^D(m,n)].$$
In the sequel this inequality will be used constantly.\\[11pt] Before we start with the proofs let us fix some notation. Let
$$(m)= \prod \mathfrak{q}_i^{e_i} \prod \mathfrak{r}_i^{f_i} \prod \mathfrak{s}_i^{g_i}$$ 
be the unique factorization of $(m)$ into prime ideals, where $\mathfrak{q}_i$ are inert prime ideals, $\mathfrak{r}_i$ are prime ideal factors of split prime numbers and $\mathfrak{s}_i$ are prime ideal factors of ramified prime numbers and $e_i,f_i,g_i \in \mathbb{N}$. Since the norm of $m$ plays a special role, we use capital letter and set $M:=\N(m)$.\\[11pt]
As the index of $\SL(L_D) \cap \Gamma^D(m,n)$ in $\SL(L_D)$ does not depend on the ordering of the prime ideal divisors of $(m)$ which we choose, we may at first divide out the prime divisors of split prime numbers, then the prime divisors of inert prime numbers and finally the prime divisors of ramified prime numbers. Moreover we may always assume that we consider the prime divisor $\mathfrak{p}$ of $(m)$ which has the highest order in $(m)$ of all prime ideal divisor $\mathfrak{p}_i$ of the given type.\\[11pt]
Before we start with the approach described above, let us make an observation which will significantly facilitate things.
\begin{lem} \label{lem_facilitate}
Let $m,n \in \OD$ with $(m,n)=1$. If $\Gamma^D_0(mn) \cap \SL(L_D)$ has the maximal possible index in $\SL(L_D)$ then also $\Gamma^D(m,n) \cap \SL(L_D)$ has the maximal possible index.
\end{lem}
\begin{proof}
As both subgroups describe the volume of a certain twisted Teichmüller curve, it suffices to show that the matrix $M= \left( \begin{smallmatrix} m & 0 \\ 0 & n \end{smallmatrix} \right)$ defines the same twisted Teichmüller curve as the matrix $N = \left( \begin{smallmatrix} mn & 0 \\ 0 & 1 \end{smallmatrix} \right)$. For this we show that there exists a matrix $X = \left( \begin{smallmatrix} a & b \\ c & d \end{smallmatrix} \right) \in \SL(L_D)$ with
$$NXM^{-1} = \left( \begin{smallmatrix} na & mb \\ c/m & d/n \end{smallmatrix} \right) \in \SL_2(\OD)$$
(for details we refer the reader to Chapter~\ref{sec_classification_twisted}). Recall the well-known isomorphism (see e.g. \cite{Kil08}, Section~2.4)
$$\SL_2(\OD) / \Gamma^D_0(nm) \cong \mathbb{P}^1(\OD/nm\OD)$$
where $\mathbb{P}^1(\cdot)$ denotes projective space. The isomorphism is given by mapping a coset representative $\left( \begin{smallmatrix} a & b \\ c & d \end{smallmatrix} \right)$ to $(c:d) \in \mathbb{P}^1(\OD/nm\OD)$. Consider the coset representatives $A_1,...,A_k$ of $\SL(L_D) / (\SL(L_D) \cap \Gamma^D_0(nm))$. Since the index of $\SL(L_D) \cap \Gamma^D_0(nm)$ in $\SL(L_D)$ equals the index of $\Gamma^D_0(nm)$ in $\SL_2(\OD)$ (Theorem~\ref{thm_summarize_euler_calculations}) we also have
$$\SL(L_D) / (\SL(L_D) \cap \Gamma^D_0(nm)) \cong \mathbb{P}^1(\OD/nm\OD).$$
Thus there exists an $A_i = \left( \begin{smallmatrix} e & f \\ g & h \end{smallmatrix} \right)$ with $m|g$ and $n|h$ and hence $NA_iM^{-1} \in \SL_2(\OD)$.
\end{proof}
This means that we may restrict to the case of $\Gamma^D_0(m)$ in the following. In the next three subsections we will always proceed very similarly. In each case we show first for a prime ideal $\mathfrak{p} \in \OD$ and arbitrary $m \in \OD$ with $(m,\mathfrak{p})=1$ that $$[(\SL(L_D) \cap \Gamma^D_0((m)\mathfrak{p}^k)):(\SL(L_D^1)\cap \Gamma^D_0((m)\mathfrak{p}^{k+1}))]=\N(\mathfrak{p})$$ for all $k \in \mathbb{N}$
by giving an explicit list of coset representatives. Afterwards we prove that $$[(\SL(L_D) \cap \Gamma^D_0((m))):(\SL(L_D^1) \cap \Gamma^D_0((m)\mathfrak{p}))] = \N(\mathfrak{p})+1$$ holds. This suffices for the first part of Theorem~\ref{thm_summarize_euler_calculations}. Recall that the quotient $\Gamma^D_0((m)\mathfrak{p}^k) / \Gamma^D_0((m)\mathfrak{p}^{k+1})$ is a cyclic group of order $\N(\mathfrak{p})$ while the quotient $\Gamma^D_0((m)) / \Gamma^D_0((m)\mathfrak{p})$ is isomorphic to $\mathbb{P}^1(\OD/\mathfrak{p})$ (compare e.g. \cite{Kil08}, Section~2.4). This fact also explains the indexes above.\\[11pt] The proofs are quite similar for all three types of splitting behavior of prime numbers although the matrices involved in the proofs differ significantly. As an abbreviation we will from now on write $\Gamma^D_0(m\mathfrak{p}^k)$ when we mean $\Gamma^D_0((m)\mathfrak{p}^k)$. For clearness reasons we will more generally leave away brackets indicating principal ideals in this chapter.\\[11pt]
Before we go through the different types of primes, let us clarify what distinguishes the different cases from our point of view: If $p= \mathfrak{p}\mathfrak{p}^\sigma$ is a split prime number and $m \in \OD$ is arbitrary with $(m,p)=1$, then either $\mathfrak{p}$ and $\mathfrak{p}^\sigma$ both divide $k \cdot m$ for $k\in \ZZ$ or none of them does divide $k \cdot m$. Since we can only give general formulas for parabolic matrices in $\SL(L_D)$ and consider their powers (and products), this makes this case particullary difficult. The difficulty about the case where $p$ is an inert prime number, is that we cannot break down the calculation of $[\SL(L_D) \cap \Gamma^D_0(p^{k}): \SL(L_D) \cap \Gamma^D(p^{k+1})]$ into two sub-steps then and therefore have to find a list of $p^2$ (or $p^2 + 1$ if $k=0$) coset representatives and not only of $p$ coset representatives as in the first case. Finally, the case where $p= \mathfrak{p}^2$ is a ramified prime number is somewhere in the middle of the first two cases and therefore combines difficulties and ways the circumvent these difficulties from both. For example, for $k$ odd and $m \in \OD$ with $(m,p) = 1$ either $\mathfrak{p}^k$ and $\mathfrak{p}^{k+1}$ both divide $k \cdot m$ for $k\in \ZZ$ or none of them does. This makes it again hard to break down the calculation of $[\SL(L_D) \cap \Gamma^D_0(p^{2k}) : \SL(L_D) \cap \Gamma^D_0(p^{2(k+1)})]$ into two sub-cases, which would be easier to treat, because we would have to find only fewer coset representatives.

\subsubsection{Divisors of Split Prime Numbers} \label{subsec_splitting}

We begin with the case which is probably the most difficult one namely, prime ideals that are divisors of split prime numbers. Until the rest of this subsection $p \in \ZZ$ always denotes a split prime number, i.e. $\left( \frac{D}{p} \right)=+1 $, with $p=\mathfrak{p}\mathfrak{p}^\sigma$.

\begin{prop} \label{prop_splittingI} Let $p \in \ZZ$ with $\left( \frac{D}{p}\right)=+1$, i.e. $p=\mathfrak{p}\mathfrak{p}^\sigma$ and let $(\mathfrak{p},\eta^*)$ $=1$. Moreover let $m \in \OD$ be an arbitrary element with $(m,\mathfrak{p})=1$. Then for all $k \in \mathbb{N}$ 
$$\left[ (\SL(L_D) \cap \Gamma^D_0(m\mathfrak{p}^k)) :  (\SL(L_D) \cap \Gamma^D_0(m\mathfrak{p}^{k+1})) \right] = p $$	
holds.
\end{prop} 

\begin{proof} The aim is to find $p$ matrices in $\Gamma^D_0(m\mathfrak{p}^k)$ which are inequivalent modulo $\Gamma^D_0(m\mathfrak{p}^{k+1})$. We will now describe the simplest set of matrices which we found. We have to distinguish two cases.\\[11pt]
\textit{1. Case: $\mathfrak{p}$ is not conjugated to any of the $\mathfrak{r}_i$}\\
Then $Z^{Mp^ki}$, $i=1,...,p$ are obviously in $\Gamma^D_0(m\mathfrak{p}^k)$ but inequivalent modulo $\Gamma^D_0(m\mathfrak{p}^{k+1})$. \\[11pt]
\textit{2. Case: $\mathfrak{p}$ is conjugated to a certain $\mathfrak{r}_h$}\\
We set $\mathfrak{m}'=(m)\mathfrak{r}_h^{-f_h}$ and $M' := \N(\mathfrak{m}')$. Let us furthermore suppose that $\mathfrak{r}_h^l | \eta^-$ but $\mathfrak{r}_h^{l+1}\nmid \eta^-$ for some $l \in \ZZ_{\geq 0}$. We now have to distinguish two subcases.\\
\textit{Case (a) $f_h-l>k$}\\
This means in particular that $\mathfrak{r}_h$ has a higher order in $m$ than $k$. We therefore want to divide out powers of $\mathfrak{r}_h$ first. This means that we have to show 
$$ \left[ (\SL(L_D) \cap \Gamma^D_0(\mathfrak{m}'\mathfrak{r}_h^{f_h-1}\mathfrak{p}^k)) :  (\SL(L_D) \cap  \Gamma^D_0(\mathfrak{m}'\mathfrak{r}_h^{f_h}\mathfrak{p}^k)) \right] = p.$$
We set $u:=f_h-l-1$ and $v:=p^u$. The matrices $Z^{M'vi}$, $1 \leq i \leq p$ lie in $\Gamma^D_0(\mathfrak{m}'\mathfrak{r}_h^{f_h-1}\mathfrak{p}^k))$ since $\mathfrak{m}'\mathfrak{r}_h^{f_h-1}\mathfrak{p}^k| M'v\eta^- \cdot i$ but the matrices are incongruent modulo $\Gamma^D_0(\mathfrak{m}'\mathfrak{r}_h^{f_h}\mathfrak{p}^k)$ since $\mathfrak{m}'\mathfrak{r}_h^{f_h}\mathfrak{p}^k| M'v\eta^- \cdot i$ implies $\mathfrak{p}|i$ by the definition of $v$.
This means that we may restrict to the case:\\ \textit{Case (b)} $k \geq f_h - l$.\\
We then set $v = p^k$ and look at the matrices $Z^{M'vi}$, $1 \leq i \leq p$. These matrices are all in $\Gamma^D(m\mathfrak{p}^k,n)$ but are not equivalent modulo $\Gamma^D_0(m\mathfrak{p}^{k+1})$ by definition of $v$.

%
%
%
%
\end{proof}

We now come to the second claim. The proof will also make use of Proposition~\ref{prop_splittingI}.

\begin{prop} \label{prop_splitting_primes} Let $p \in \ZZ$ with $\left( \frac{D}{p}\right)=+1$, i.e. $p=\mathfrak{p}\mathfrak{p}^\sigma$ and let $(\mathfrak{p},\eta^*)$ $=1$. Then for all $m\in \OD$ with $(m,\mathfrak{p})=1$ 
$$\left[ (\SL(L_D) \cap \Gamma^D_0(m)) :  (\SL(L_D) \cap \Gamma^D_0(m\mathfrak{p}) \right] = p+1$$	
holds.
\end{prop} 

The main problem about divisors of split prime numbers is that it is still possible that $\mathfrak{p}^{\sigma}|m$. This will make it seriously harder to prove the result. Therefore, we have to split the proof into two steps. Each of these steps will be dealt with in a separate lemma.

\begin{lem} \label{lem_splitting_primes_first} Let $p \in \ZZ$ with $\left( \frac{D}{p} \right)=+1$, i.e. $p=\mathfrak{p}\mathfrak{p}^\sigma$ and let $(p,\eta^*)=1$. Furthermore let $m \in \OD$ with $(m,p)=1$. Then
$$\left[ (\SL(L_D) \cap \Gamma^D_0(m\mathfrak{p}^{\sigma})) :  (\SL(L_D) \cap \Gamma^D_0(m\mathfrak{p}^{\sigma}\mathfrak{p})) \right] = p+1$$	
holds.
\end{lem}

\begin{proof} We try to find a matrix $W$ which is a word in $T$ and $Z$ such that the matrices
\begin{eqnarray*}
(I) & WT^{i}, & i=1,...,p \\
(II) & \Id &
\end{eqnarray*}
are all in $\Gamma^D_0(m\mathfrak{p}^\sigma)$ but are incongruent modulo $\Gamma^D_0(m\mathfrak{p}^{\sigma}\mathfrak{p})$ or in other words such that $WT^{i-j}W^{-1} \in \Gamma^D_0(\mathfrak{p}^{\sigma}\mathfrak{p})$ implies $i=j$. The simplest $W$ that we found is $W:=ZT^{k}Z^{M}T^{-k}Z^{-1}$ where $k$ is chosen as follows: let $k \in \left\{ 1,...,p \right\}$, such that $\mathfrak{p}^{\sigma}|k\eta^-\eta^++1$. This is always possible since $\mathfrak{p}^{\sigma} \nmid \eta^*$ and $1,..,p$ are incongruent modulo $\mathfrak{p}^{\sigma}$. Furthermore we know that $k \neq p$ because otherwise it would follow that $\mathfrak{p}^{\sigma}|1$. Now suppose that $\mathfrak{p}|k\eta^-\eta^++1$. Since $(\mathfrak{p},\mathfrak{p}^{\sigma})=1$ we would then have $\mathfrak{p}\mathfrak{p}^\sigma|k\eta^-\eta^++1$ which is exactly $p|k\eta^-\eta^++1$ or $p|k(w+f)+1$ for some $f \in \ZZ$ and therefore $p|k$ since $p \in \NN$.\footnote{If $D \equiv 1 \mod 4$ then $f$ is $\frac{D-1}{4}$ for $\SL(L_D^1)$ and it is $\frac{D-5}{4}$ for $\SL(L_D^0)$. If $D \equiv 0 \mod 4$ then $f$ is $D/4$.} This is a contradiction. Hence $\mathfrak{p}\nmid k\eta^-\eta^++1$. For $i \in \left\{ 1,...,p \right\}$ we have
$$(WT^{Ni})_{2,1} = M\eta^-(k\eta^-\eta^++1)^2.$$
Since
$$ m\mathfrak{p}^{\sigma}| M\eta^-(k\eta^-\eta^++1)^2.$$
all the matrices lie in $\Gamma^D_0(m\mathfrak{p}^{\sigma})$ but none of the matrices in $(I)$ is equivalent to the identity modulo $\Gamma^D_0(m\mathfrak{p}^\sigma \mathfrak{p})$. We now show that the matrices in $(I)$ are pairwise incongruent. So suppose $$ WT^{j-i}W^{-1} \in \Gamma^D_0(m\mathfrak{p}\mathfrak{p}^{\sigma}).$$ Set $x:=j-i$. Then
$$ (WT^{Nx}W^{-1})_{2,1}= -M\eta^{-2}\eta^{+}(Mk\eta^-\eta^++1)^4 \cdot x.$$ 
As $\mathfrak{p} \nmid M, \mathfrak{p} \nmid \eta^*$ and $\mathfrak{p} \nmid (k\eta^-\eta^++1)$ we have that $\mathfrak{p}|x$ and hence $i=j$.
\end{proof}


We proceed to the second lemma.

\begin{lem} \label{lem_splitting_primes_last} Let $p \in \ZZ$ with $\left( \frac{D}{p} \right)=+1$, i.e. $p=\mathfrak{p}\mathfrak{p}^\sigma$ and let $(p,\eta^*)=1$. Furthermore let $m \in \OD$ with $(m,p)=1$. Then
$$\left[ (\SL(L_D) \cap \Gamma^D_0(m)) :  (\SL(L_D) \cap \Gamma^D_0(m\mathfrak{p})) \right] = p+1$$	
holds.
\end{lem}

\begin{proof} We want to find a $k \in \NN$ such that the matrices
\begin{eqnarray*}
(I) & Z^{Mi}, & i=1,...,p \\
(II) & Z^{Mk}T&
\end{eqnarray*}
lie in $\Gamma^D_0(m)$ and are pairwise incongruent modulo $\Gamma^D_0(m\mathfrak{p})$. Indeed, we choose $k$ as follows: let $k \in \left\{1,...,p\right\}$ such that $\mathfrak{p}|kM\eta^-\eta^++1$. This is always possible since $\mathfrak{p} \nmid \eta^*$ and $\mathfrak{p} \nmid M$. Furthermore we know that $k \neq p$ because otherwise it would follow that $\mathfrak{p}|1$. By definition, it is clear that all the matrices in $(I)$ and $(II)$ lie in $\Gamma^D_0(m)$ and that the matrices in $(I)$ are pairwise incongruent modulo $\Gamma^D_0(m\mathfrak{p})$. Finally, we calculate
$$(Z^{Mk}TZ^{-Mi})_{2,1} = M\eta^- ( -(kM\eta^-\eta^++1)i+k).$$
Now suppose that $\mathfrak{p}|M\eta^- ( -(kM\eta^-\eta^++1)i+k)$. In other words this means
$\mathfrak{p}|( -(kM\eta^-\eta^++1)i+k)$ which yet implies $\mathfrak{p}|k$. This is a contradiction
\end{proof}

\begin{proof}[Proof of Proposition~\ref{prop_splitting_primes}] By Proposition~\ref{prop_splittingI} (and analogous arguments as in the proof of it) we may assume that prime ideals stemming from split prime numbers divide $m$ only of order one. Then the claim follows from Lemma~\ref{lem_splitting_primes_first} and Lemma~\ref{lem_splitting_primes_last}. \end{proof}

\subsubsection{Inert Prime Numbers} \label{sec_inert_volume}


The second case concerns inert prime numbers $p \in \ZZ$. The condition $(p,\eta^*)=1$ is then automatically fulfilled (compare Lemma~\ref{lem_properties_OD}). The difficulty about this case is that there exist $0 < i < \N(p)=p^2$ with $p|i$. We start with controlling the effect of higher powers of $p$. 

\begin{prop} \label{prop_primepower_inert_eq} Let $p \in \ZZ$ be an inert prime number, i.e. $\left( \frac{D}{p} \right)=-1$. Moreover let $m \in \OD$ be an arbitrary element with $(m,p)=1$. Then for all $k \in \mathbb{N}$ 

$$\left[ (\SL(L_D) \cap \Gamma^D_0(mp^k)) :  (\SL(L_D) \cap \Gamma^D_0(mp^{k+1})) \right] = \N(p) = p^2$$	
holds.

\end{prop} 

\begin{proof} 
We want to find matrix $W$ which is a word in $T$ and $Z$ such that the matrices
$$W^lZ^{Mp^k j}, \quad j=1,...,p \quad l=1,...,p$$
lie $\Gamma^D_0(mp^k)$ but are pairwise incongruent modulo $\Gamma^D_0(mp^{k+1})$. As $p$ is an inert prime number, it does not divide $M$, and hence
$$W:= ZT^{Mp^k}Z^{-1}$$ 
might be a good a choice. 
$$W^hZ^{Mp^k j} \equiv W^lZ^{Mp^k i}$$
$$\text{is equivalent to} \ \ W^hZ^{Mp^k (j-i)} W^{-l} \in \Gamma^D_0(mp^{k+1}).$$
Setting $x:=j-i$, we thus need to check when
\begin{eqnarray*} (W^hZ^{p^k dx} W^{-l})_{2,1} & = & \underbrace{p^{3k}M^3\eta^{-3}\eta^{+2}\cdot hlx}_{v_1}\\ & + & \underbrace{p^{2k}M^2\eta^{-2}\eta^+\cdot (h+l)x}_{v_2}\\ & + & \underbrace{p^kM\eta^-(x+\eta^-\eta^+\cdot(l-h))}_{v_3} 
\end{eqnarray*}
is divisible by $mp^{k+1}$. We already know that $m | v_1+v_2+v_3$. It suffices to check when $p^{k+1}|v_1 + v_2 + v_3$ holds. Obviously $p^{k+1}$ always divides $v_1 + v_2$. So we are just interested in which cases we have $p^{k+1} | p^kM\eta^-(x+\eta^-\eta^+\cdot (l-h))$ or equivalently $p | M\eta^-(x\eta^-\eta^+\cdot(l-h))$. As $p \nmid \eta^-$ and $p \nmid M$ this is equivalent to $p| (x+\eta^-\eta^+\cdot(l-h))$. We have $\eta^-\eta^+ = w + f$, with $f \in \ZZ$. Thus we ask when $p|(w+f)\cdot(l-h)+x$. We now just look at the \textit{imaginary} part of the right hand side (i.e. the part which does not lie in $\ZZ$). This gives us that $l=h$. Thus $p|x$. This yields $x=0$ or in other words $i=j$.
\end{proof}

We now come to the second claim.

\begin{prop} \label{prop_inert_prime_eq1}Let $p \in \ZZ$ be an inert prime number, i.e. $\left( \frac{D}{p} \right)=-1$. Then for all $m \in \OD$ with $(m,p)=1$ 
$$\left[ (\SL(L_D) \cap \Gamma^D_0(m)) :  (\SL(L_D) \cap  \Gamma^D(mp)) \right] = \N(p)+1=p^2+1$$	
holds.
\end{prop} 

\begin{proof} We claim that the matrices
\begin{eqnarray*}
(I) & Z^{-M}T^{i}Z^{M}T^{j}, & i=1,...,p-1, \quad j=1,...,p \\
(II) & Z^{M}T^{i} & i=1,..,p \\
(III) & \Id &
\end{eqnarray*}
lie in $\Gamma^D_0(n)$ and are pairwise incongruent modulo $ \Gamma^D_0(mp)$. The matrices in $(II)$ lie in $\Gamma^D_0(m)$ but are not congruent to the identity modulo $\Gamma^D_0(mp)$. Furthermore 
$$(Z^MT^{i}Z^{-M})_{2,1} = -M^2\eta^{-2}\eta^+ \cdot i$$
implies that the matrices in $(II)$ are pairwise incongruent modulo $\Gamma^D_0(mp)$. Moreover
all the matrices in $(I)$ lie in $\Gamma^D_0(m)$ but none of them is congruent to the identity modulo $\Gamma^D_0(mp)$. Next we show that the matrices in $(I)$ are indeed pairwise incongruent modulo $\Gamma^D_0(mp)$:  
$$Z^{-M}T^{i}Z^MT^{k} \equiv Z^{-M}T^{j}Z^MT^{l}$$
$$\textrm{is equivalent to} \ \ Z^{-M}T^{i}Z^MT^{x}Z^{-M}T^{-j}Z^{M} \in \Gamma^D_0(mp)$$
where $x:=k-l$. However
$$(Z^{-M}T^{i}Z^MT^{x}Z^{-M}T^{-j}Z^{M})_{2,1} = v_{i,j,x}$$  where $v_{i,j,x} = M^2\eta^{-2}\eta^+(-M^2\eta^{-2}\eta^{+2}\cdot ijx +(j-i)).$ So let us suppose that $p|M^2\eta^{-2}\eta^+(-M^2\eta^{-2}\eta^{+2} \cdot ijx +(j-i))$ which holds if and only if we have $p | -M^2\eta^{-2}\eta^{+2}\cdot ijx +(j-i)$. We now have to distinguish three cases:\\[11pt]
\textit{Case (1) $D \equiv 1 \mod 4$, odd spin:} If we consider $\SL(L_D^1)$ for $D \equiv 1 \mod 4$ then we assume that $p \nmid \frac{D+1}{2}$. Then $\eta^{-2}\eta^{+2} = \frac{D+1}{2} w + \frac{D-1}{4} + \frac{(D-1)^2}{16}$ and we look at the \textit{imaginary} part of the right hand side. As $p \nmid \frac{D+1}{2}$ this yields $p|ijx$. Now we look at the \textit{real} part (i.e. the part which lies in $\ZZ$) and get $j=i<p$. Therefore, $x=0$ yields $k=l$. It remains to check whether it is possible that a matrix in $(I)$ is congruent to a matrix in $(II)$. 
$$(Z^{-M}T^{i}Z^MT^{x}Z^{-M})_{2,1} = M\eta^-(M^2\eta^{-2}\eta^{+2}\cdot ix-1).$$
However $p|M\eta^-(M^2\eta^{-2}\eta^{+2}\cdot ix -1)$ implies $p|(M^2\eta^{-2}\eta^{+2}\cdot ix -1)$. As above, this yields $ix=0$ and thus $p|-1$, which is a contradiction.\\
If $p | \frac{D+1}{2}$, then we first prove the following lemma:
\begin{lem} \label{lem_inert_prime_eq2} Let $p \in \ZZ$ be an inert prime number, i.e. $\left( \frac{D}{p} \right)=-1$. Then
$$\left[ \SL(L_D^1)  :  (\SL(L_D^1) \cap \Gamma^D_0(p) \right] = \N(p)+1=p^2+1$$	
holds.
\end{lem} 

\begin{proof} As $m=1$ we are allowed to use the elliptic element $S \in \SL(L_D^1)$. This will help a lot. It is enough to find a list of $\N(p)+1=p^2+1$ matrices in $\SL(L_D^1)$ which are inequivalent modulo $\Gamma^D_0(p)$. We claim that
\begin{eqnarray*}
(I) & Z^iT^j, &i=1,...,p-1, \quad j=1,...,p \\
(II) & ST^j,  &j=1,...,p  \ \\
(III) &\Id 
\end{eqnarray*}
is such a list. Note that
$$(Z^iT^j)_{2,1} = w\cdot i$$
and
$$(ST^j)_{1,2}= 1$$
which means that none of the matrices in $(I)$ and $(II)$ is equivalent to the identity matrix.\\[11pt] So let us assume $Z^l T^{j-i}Z^k \in \Gamma^D_0(p)$. We set $x:=j-i$ and  calculate
$$(Z^lT^xZ^{-k})_{2,1} = -w^3 \cdot klx + w\cdot (l-k) $$
So by assumption $p | -w^3 \cdot klx + w\cdot (l-k)$. Since $(w,p)=1$ this is equivalent to $p | - w^2\cdot klx + (l-k) $. Furthermore $w^2=w+f$, where $f=\frac{D-1}{4}$ depends only on the discriminant, and so $p | (l-k) -(w+f)\cdot klx$. Since both $l$ and $k$ are smaller than $p$, this implies $x=0$ and therefore $l=k$.\\[11pt]
Next we show that all the matrices in $(II)$ are incongruent modulo $\Gamma^D_0(p)$. So suppose $ST^j \equiv ST^i$. Then
$$(ST^{j-i}S)_{2,1} = w \cdot (j-i)$$
implies that $j=i$. \\[11pt]
Finally it has to be shown that the matrices in $(I)$ and $(II)$ are incongruent modulo $\Gamma^D_0(p)$. Again we set $x:=j-i$ and calculate
$$(ST^{x}Z^{-k})_{2,1} = -w^2 \cdot kx+1$$
But, as above, $p|-w^2kx+1$, cannot hold.
\end{proof}
Now suppose that the index $[\SL(L_D^1) : \SL(L_D^1) \cap \Gamma^D(m)]$ is $k$. So we know by the lemma that the index $[\SL(L_D^1) : \SL(L_D^1) \cap \Gamma^D(mp)]$ is $k(\N(p)+1)/l$ with $l \in \NN$. On the other hand, it follows from Proposition~\ref{prop_primepower_inert_eq} that the index is at least $k\N(p)$. This yields $l=1$ and thus the claim.\\[11pt]
\textit{Case (2) $D \equiv 1 \mod 8$, even spin:} If we consider $\SL(L_D^0)$ then $\eta^{-2}\eta^{+2} = \frac{D-3}{2} w + \frac{D-1}{2} + \frac{(D-5)^2}{16}$ and if we assume that $p \nmid \frac{D-3}{2}$, then the claim immediately follows. If we look at $p$ with $p | \frac{D-3}{2}$ then it suffices to prove the following lemma:
\begin{lem} \label{lem_inert_prime_eq3} Let $p \in \ZZ$ be an inert prime number, i.e. $\left( \frac{D}{p} \right)=-1$, with $p|\frac{D-3}{2}$. Then

$$\left[ \SL(L_D^0)  :  (\SL(L_D^0) \cap \Gamma^D_0(p) \right] = \N(p)+1 = p^2+1$$	
holds.
\end{lem} 

\begin{proof} It is enough to find a list of $\N(p)+1=p^2+1$ matrices in $\SL(L_D^0)$ which are inequivalent modulo $\Gamma^D_0(p)$. We claim that
\begin{eqnarray*}
(I) & L^{-1}T^iLT^j, &i=1,...,p-1, \quad j=1,...,p \\
(II) & LT^i,  &j=1,...,p  \ \\
(III) &\Id, 
\end{eqnarray*}
where $L$ is the matrix from Lemma~\ref{lem_second_parabolic2}, is such a list. We assume that $D \equiv 1 \mod 16$. The other case works in the same way. Note that
$$(L^{-1}T^iLT^j)_{2,1} = -4(w+2)^2(w-1)\cdot i$$
and
$$(LT^i)_{2,1} = -2(w+2)$$
which means that none of the matrices in $(I)$ and $(II)$ is equivalent to the identity matrix since $p \neq 2$.\\[11pt] Since   
$$(LT^iL^{-1})_{2,1} = -4(w+2)^2(w-1)\cdot i$$
the matrices in $(II)$ are indeed pairwise incongruent modulo $\Gamma^D_0(p)$. Furthermore,
$$(L^{-1}T^{i}LT^{x}L^{-1}T^{-j}L)_{2,1} = v_{i,j,x}$$ where $v_{i,j,x} = 4(w-1)(w+2)^2 (4(-w^4-2w^3+3w^2+4w-4)\cdot ijx + (j-i)).$ Since $-w^4-2w^3+3w^2+4w-4 = (-D+5)w + \frac{D-1}{4}\cdot\frac{D+7}{4}-4$ and since $p \neq 2$, this implies that $x=0$ if the matrix is in $\Gamma^D_0(p)$ and thus also $i=j$. Therefore, all the matrices in $(I)$ are inequivalent modulo $\Gamma^D_0(p)$.
\\[11pt]
Finally we have
$$(L^{-1}T^iLT^jL^{-1})_{2,1} = 2(w+2) ( (-w^4-2w^3+3w^2+4w-4)\cdot ij +1).$$
As above, this implies that the matrices in $(I)$ and $(II)$ are not equivalent modulo $\Gamma^D_0(p)$.
\end{proof}
\textit{Case (3) $D \equiv 0 \mod 4$:} For $D \equiv 0 \mod 4$ we have $\eta^{-2}\eta^{+2} =  \frac{D}{2} w + \left(\frac{D}{4}\right)^2 + 1.$ On the other hand $p$ does not divide $\frac{D}{2}$ because $p$ is an inert prime number and so the claim follows. 
\end{proof}

\subsubsection{Divisors of Ramified Prime Numbers}

The last type of prime numbers leads us to ramified prime numbers $p \in \ZZ$ or in other words $p|D$. This case is treated here. What also makes ramified prime numbers a little difficult is the fact that if $p=\mathfrak{p}^2$ then $\mathfrak{p}^2|\N(\mathfrak{p})$. 

\begin{prop} \label{prop_ramifiedII} Let $p \in \ZZ$ be a ramified prime number, i.e. $p|D$ and $p=\mathfrak{p}^2$, with $(\mathfrak{p},\eta^*)=1$. Moreover let $m \in \OD$ be an arbitrary element with $(m,\mathfrak{p})=1$. Then for all $k \in \mathbb{N}$ 
$$\left[ (\SL(L_D) \cap \Gamma^D_0(m\mathfrak{p}^k)) :  (\SL(L_D) \cap \cap \Gamma^D_0(m\mathfrak{p}^{k+1})) \right] = \N(p).$$
holds.

\end{prop} 

\begin{proof} Note that $\mathfrak{p} \nmid m$ yields $\mathfrak{p} \nmid M$. We now have to distinguish three cases for $k$:\\[11pt]
\textit{1. Case: k is even}\\
Choose $d:= p^{k/2} \in \mathbb{N}$. Then $Z^{Mdi}$, $i=1,...,p$ are matrices in $\Gamma^D_0(m\mathfrak{p}^k)$ which are inequivalent modulo $\Gamma^D_0(m\mathfrak{p}^{k+1})$.\\[11pt]
\textit{2. Case: k is odd and $k>1$}\\
We first consider the index
$$\left[(\SL(L_D) \cap \Gamma^D_0(m\mathfrak{p}^{k-1}	)) : (\SL(L_D) \cap \Gamma^D_0(m\mathfrak{p}^{k+1})) \right].$$
We claim that 
$$W^l Z^{Mdj}, \quad j=1,...,p, \quad l=1,...,p$$
is a list of $p^2$ matrices in $\Gamma^D_0(m\mathfrak{p}^{k-1})$ which are moreover inequivalent modulo $\Gamma^D_0(m\mathfrak{p}^{k+1})$, if we choose $d:= p^{(k-1)/2}$ and $W:=ZT^{Md}Z^{-1}$. All the matrices lie in $\Gamma^D_0(m\mathfrak{p}^{k-1})$. So assume $W^h Z^{Md(j-i)} W^{-l} \in \Gamma^D_0(\mathfrak{p}^{k+1})$. Set $x:=j-i$. Then
\begin{eqnarray*} (W^h Z^{Mdx} W^{-l})_{2,1} & = & \underbrace{M^3d^3\eta^{-3}\eta^{+2} \cdot hlx+ M^2d^2\eta^{-2}\eta^+ \cdot (l+h)x}_{v_1}\\ & + & \underbrace{Md\eta^-(x+\eta^-\eta^+\cdot(l-h))}_{v_2}.\end{eqnarray*}
Note that always $m\mathfrak{p}^{k+1}|v_1$. So we are just interested in knowing in which cases $m\mathfrak{p}^{k+1} | Md\eta^-(x+\eta^-\eta^+(l-h))$. This expression is equivalent to $p | \eta^-(x+(l-h)N\eta^-\eta^+)$. As $\mathfrak{p} \nmid \eta^-$ this means $p| (x+N\eta^-\eta^+\cdot (l-h))$. We have $\eta^-\eta^+ = w + f$ with $f \in \ZZ$. Thus we ask when $p|N(w+f)\cdot(l-h)+x$. We now just look at the \textit{imaginary} part of the right side. This gives us that $l=h$ and hence $p|x$. This yields $x=0$ or in other words $i=j$. Thus we have
\begin{eqnarray*}
p^2 & = & \left[(\SL(L_D) \cap \Gamma^D_0(m\mathfrak{p}^{k-1})) : (\SL(L_D) \cap \Gamma^D_0(m\mathfrak{p}^{k+1})) \right]\\
& = & \underbrace{\left[(\SL(L_D) \cap \Gamma^D_0(m\mathfrak{p}^{k-1})) :(\SL(L_D) \cap \Gamma^D_0(m\mathfrak{p}^{k})) \right]}_{=p} \cdot\\
&  & \underbrace{\left[(\SL(L_D) \cap \Gamma^D_0(m\mathfrak{p}^{k})) : (\SL(L_D) \cap \Gamma^D_0(m\mathfrak{p}^{k+1})) \right]}_{=:y}
\end{eqnarray*}
which implies $y=p$.\\[11pt] 
\textit{3. Case: $k=1$}\\
Since $\mathfrak{p}$ is the divisor of a ramified prime number we know that $p|D$. Since $(\mathfrak{p},\eta^*)=1$ we may consider the matrices from the proof of Proposition~\ref{prop_inert_prime_eq1}. All these matrices then lie in $\Gamma^D_0(m)$ but are inequivalent modulo $\Gamma^D_0(mp)$ and therefore
$$\left[(\SL(L_D) \cap \Gamma^D_0(m)) : (\SL(L_D) \cap \Gamma^D_0(mp)) \right] \geq p^2+1.$$
On the other hand 
\begin{align*} 
\left[(\SL(L_D) \cap \Gamma^D_0(m)) : (\SL(L_D) \cap \Gamma^D_0(mp)) \right] & = &\\
\underbrace{\left[(\SL(L_D) \cap \Gamma^D_0(m\mathfrak{p}))) : (\SL(L_D) \cap \Gamma^D_0(mp)) \right]}_{=:y} & \cdot (p+1) & \end{align*}
by Proposition~\ref{prop_ramifiedI}. This means that $y(p+1) \geq p^2+1$ and hence $y=p$ since $y \leq p$.
\end{proof}

The other proposition requires only a much shorter proof.

\begin{prop} \label{prop_ramifiedI} Let $p \in \ZZ$ be a ramified prime number, i.e. $p|D$ and $p=\mathfrak{p}^2$, with $(\mathfrak{p},\eta^*)=1$. Then for all $m \in \OD$ with $(m,\mathfrak{p})=1$ 
$$\left[ (\SL(L_D) \cap \Gamma^D_0(m)) :  (\SL(L_D) \cap \Gamma^D_0(m\mathfrak{p}) \right] = p+1$$	holds.
\end{prop} 

\begin{proof} The matrices $$Z^MT^{i}, \quad i=1,...,p$$ all lie in $\Gamma^D_0(m)$. The same calculation as in Proposition~\ref{prop_inert_prime_eq1} yields that these matrices are pairwise incongruent modulo $\Gamma^D_0(m\mathfrak{p})$. So we need to find one more matrix $W$ which is not equivalent to all the $Z^MT^{i}$. Note that $W_k:=Z^MT^{k}Z^M \in \Gamma^D_0(m)$ for all $k \in \ZZ$. We claim that there exists an $k \in \left\{1,...,p \right\}$ with 
$$W_kT^{-Ni}Z^{-M} \notin \Gamma^D_0(m\mathfrak{p})$$ 
for all $i=1,...,p$. One calculates
$$(W_kT^{-i}Z^{-M})_{2,1} = M\eta^-(M^2\eta^{-2}\eta^{+2}\cdot ik+2M\eta^-\eta^+\cdot i+1).$$
Now suppose that for all $k$ there exits an $i_k$ such that $$\mathfrak{p}|M^2\eta^{-2}\eta^{+2}\cdot ki_k+2M\eta^-\eta^+\cdot i_k+1$$ or equivalently $$\mathfrak{p}|i_kM\eta^-\eta^+(M\eta^-\eta^+\cdot k+2)+1.$$
This would imply that $M\eta^-\eta^+\cdot k+2$ is a unit \textrm{mod} $\mathfrak{p}$ for all $k$. This contradicts $\mathfrak{p}\nmid M\eta^-\eta^+$.
\end{proof}



This also finishes the proof of Theorem~\ref{thm_summarize_euler_calculations}~(i).

\subsubsection{Prime Ideals with $(\mathfrak{p},\eta^*) \neq 1$} \label{subsec_nonrelative}

We have only treated the case $(\mathfrak{p},\eta^*)=1$ so far. If  this condition is violated the situation cannot be controlled by only using the matrices $T$ and $Z$. Therefore, we have to go through three different cases, namely odd and even spin Teichmüller curves if $D \equiv 1 \mod 4$ and $D \equiv 0 \mod 4$  
\paragraph{D $\equiv$ 1 mod 4, odd spin.} We look at first at the odd spin Teichmüller curves.\footnote{Recall our abuse of notation of the term spin.} L7et us revisit the example from the beginning of the chapter: for discriminant $D=17$ and $\pi_2=w+2$
$$ \left[ \SL(L_{17}^1):(\SL(L_{17}^1) \cap \Gamma^D_{0}(\pi_2) ) \right] = \frac{2}{3} \left[ \SL_2(\OD):\Gamma^D_0(\pi_2) \right] $$ 
holds. Thus for arbitrary powers of $\pi_2$ we have the inequality
$$\frac{\left[ \SL(L_{17}^1):( \SL(L_{17}^1) \cap \Gamma^D_{0}(\pi_2^n)) \right|}{\left[ 	\SL_2(\mathcal{O}_{17}):\Gamma^D_0(\pi_2^n) \right]} \leq \frac{2}{3}$$
since $\Gamma^D_0(\pi_2^n) = \Gamma^D_0(\pi_2) \cap \Gamma^D_0(\pi_2^n)$.

\begin{prop} For discriminant $D=17$ we have for all $n \in \mathbb{N}$:
\begin{eqnarray} \label{leq23} \left[ \SL(L_{17}^1):(\SL(L_{17}^1) \cap \Gamma^D_{0}(\pi_2^n) ) \right] = \frac{2}{3} \left[ \SL_2(\OD):\Gamma^D_0(\pi_2^n) \right] .\end{eqnarray} 
\end{prop} 
A direct calculation shows that the matrices $Z,T,S$ do not suffice to find all coset representatives of $\SL(L_{17}^1)/(\SL(L_{17}^1) \cap \Gamma^D_{0}(\pi_2^n) )$. In order to prove the proposition we therefore need to make use of the second parabolic generator of $\SL(L_{17}^1)$ from Lemma~\ref{lem_second_parabolic}, that we from now on denote by $L$.

\begin{proof} The index $\left[ \SL_2(\OD):\Gamma^D_0(\pi_2^n) \right]$ is $2^{n-1} \cdot 3$. Moreover we already know that ($\ref{leq23}$) holds with $\leq$ instead of $=$. Thus we only have to give a list with $2^{n}$ elements in $\SL(L_D^1)$ which are incongruent modulo $\Gamma^D_0(\pi_2^n)$. It can be easily checked that  
$$L,...,L^{2^{n-1}}$$
$$LS,...L^{2^{n-1}}S$$
is such a list.
\end{proof}

There is a general pattern which explains the factor $\frac{2}{3}$. Recall that if $\mathfrak{p}$ is a prime ideal in $\OD$ with $(\mathfrak{p},w)=\mathfrak{p}$, then $\mathfrak{p}$ divides a split prime number $p \in \ZZ$. 

\begin{prop} Let $\mathfrak{p}$ be a prime ideal in $\OD$ with $\mathfrak{p}|w$ and $m \in \OD$ with $(m,\mathfrak{p})=1$. Then for all $k \in \NN$
$$\left[ (\SL(L_D^1) \cap \Gamma^D_0(m\mathfrak{p}^k)) :  (\SL(L_D^1) \cap \Gamma^D_0(m\mathfrak{p}^{k+1}) \right] = \N(\mathfrak{p})$$
holds. 
\end{prop}

\begin{proof}  By what we have shown in the last sections we may without loss of generality assume that $(m,w)=m$.  We have that $\mathfrak{p} \nmid M$. Recall that $2$ is an inert prime number if $D \equiv 5 \mod 8$. Let $j=(k-1)$, if $\mathfrak{p} | 2$ and $D \equiv 1 \mod 16$, and $j=k$ otherwise. Then the matrices $L^{M\N(\mathfrak{p})^ji}$ for $i=1,...,\N(\mathfrak{p})$ all lie in $\Gamma^D_0(m\mathfrak{p}^k)$ but are incongruent modulo $\Gamma^D_0(m\mathfrak{p}^{k+1})$. 
\end{proof}

Let us come to the case $k=0$. 
\begin{prop} \label{prop_k=0_1} Let $\mathfrak{p}$ be a prime ideal in $\OD$ with $\mathfrak{p}|w$ and $\mathfrak{p} \neq \mathfrak{p}_2$ if $D \equiv 1 \mod 16$ and $m \in \OD$ with $(m,\mathfrak{p})=1$, then
$$\left[ (\SL(L_D^1) \cap \Gamma^D_0(m)) :  (\SL(L_D^1) \cap \Gamma^D_0(m\mathfrak{p}) \right] = \N(\mathfrak{p})+1.$$
\end{prop}

\begin{proof} By what we have proven so far, we may without loss of generality assume that $(m)=\mathfrak{p}_1 \cdots \mathfrak{p}_k$ where all the $\mathfrak{p}_i$ are distinct prime ideals which divide $(w)$. We prove the claim by induction on the number $k+1$ of prime ideals divisors of $m \mathfrak{p}$. First assume that $k+1=1$, i.e. $m=1$. Then the matrices $L^i$ and $S$ all lie in $\Gamma^D_0(m)$ but are incongruent modulo $\Gamma^D_0(m\mathfrak{p})$. We proceed with the induction step. Since $\mathfrak{p} \nmid M$ the matrices $L^{Mi}$ with $1 \leq i \leq \N(\mathfrak{p})$ are all in $\Gamma^D_0(m)$ but incongruent modulo $\Gamma^D_0(m\mathfrak{p})$. By the induction hypothesis we have that $\N(\mathfrak{p})+1$ as well as $\N(\mathfrak{p}_i)+1$ for all $i$ divide the index. The induction hypothesis also implies that the index is at least $\N(\mathfrak{p})(\N(\mathfrak{p}_1)+1)\cdots(\N(\mathfrak{p}_k)+1)$. Thus the claim follows.
\end{proof}




This also finishes the proof of Theorem~\ref{thm_summarize_euler_calculations}~(ii) and Theorem~\ref{thm_summarize_euler_calculations}~(iii). Finally, let us consider the case $D \equiv 1 \mod 16$ and $\mathfrak{p}=\mathfrak{p}_2$:

\begin{prop} Let $D \equiv 1 \mod 16$ and $m \in \OD$ with $(m,\mathfrak{p}_2)=1$, then
$$\left[ (\SL(L_D^1) \cap \Gamma^D_0(m)) :  (\SL(L_D^1) \cap \Gamma^D_0(m\mathfrak{p}_2) \right] \geq \N(\mathfrak{p}_2)=2.$$
\end{prop}

\begin{proof} By the same arguments as in the proof of Proposition~\ref{prop_k=0_1} one may restrict to the case $m=1$ and $S$ and $T$ are matrices which are incongruent \textrm{mod} $\Gamma^D_0(\mathfrak{p}_2)$. So the index is at least $2$. \end{proof}

For $\SL(L_D^1)$, all cases that may occur have been treated now.

\paragraph{D $\equiv$ 1 mod 8, even spin.} What makes the situation for even spin a lot more complicated is that we cannot in general write down any elliptic element in $\SL(L_D^0)$ explicitly. It is easy to see by considering the matrices $L^i$ and $\widetilde{L}^i$ from Lemma~\ref{lem_second_parabolic2} and Lemma~\ref{lem_third_parabolic2} that for all prime ideals $\mathfrak{p}$ with $(\mathfrak{p},\eta^*)=\mathfrak{p}$, all $k \in \ZZ_{\geq 0}$  and all $m,n \in \OD$ with $(m,n)=1$ we have
$$\left[ (\SL(L_D^0) \cap \Gamma^D_0(m\mathfrak{p}^{k})):( \SL(L_D^0) \cap \Gamma^D_0(m\mathfrak{p}^{k+1})) \right] \geq \N(\mathfrak{p}).$$
The difficulty is to calculate the index $\left[ \SL(L_D^0):( \SL(L_D^0) \cap \Gamma^D_{0}(\mathfrak{p})) \right]$. If $D=17$ and $\pi_2^\sigma|\eta^*$, which is then the unique prime divisor of $\eta^*$, we know that $\left[ \SL(L_D^0):( \SL(L_D^0) \cap \Gamma^D_{0}(\pi_2^\sigma)) \right] = 2$ and hence the index cannot be maximal in general if $\mathfrak{p}|\eta^*$. On the other hand, for all proper prime ideals $\mathfrak{p} \subset \OD$ with $\mathfrak{p}|\eta^*$ and $\mathfrak{p} \neq \mathfrak{p}_2^\sigma$ we conjecture that the index is still maximal. However, the lack of knowledge about the Veech group, prevents us from proving an analogue of Theorem~\ref{thm_summarize_euler_calculations}~(ii): we explicitly calculated that for small discriminants and arbitrary ideals $\mathfrak{p}$ with $\mathfrak{p}|\eta^-$ the index of $\Gamma \cap \Gamma^D_0(\mathfrak{p})$ in the subgroup $\Gamma$ of $\SL(L_D^0)$ generated by $T,Z,L$ and $\widetilde{L}$ is only $\N(\mathfrak{p})$. 

\paragraph{D $\equiv$ 0 mod 4.} Also if $D \equiv 0 \mod 4$ the matrix $L$ from Lemma~\ref{lem_second_parabolic_0mod4} yields that for all prime ideals $\mathfrak{p}$ with $(\mathfrak{p},\eta^*)=\mathfrak{p}$, all $k \in \ZZ_{\geq 0}$  and all $m \in \OD$ with we have
$$\left[ (\SL(L_D^0) \cap \Gamma^D_0(m\mathfrak{p}^{k})):( \SL(L_D^0) \cap \Gamma^D_0(m\mathfrak{p}^{k+1})) \right] \geq \N(\mathfrak{p}).$$
However, the question how to precisely calculate the volume of a diagonal twisted Teichmüller curve in general remains unsolved by the same reasons as in the even spin case.

\subsection{The Volume of Upper Triangular Twists} \label{sec_upper_triangular_twists}

For the rest of the chapter we want to assume as a \textbf{general condition} that the (wide) class number $h_D$ is equal to $1$. In this case we now also calculate the volume of Teichmüller curves twisted by upper triangular matrices. Since the class number is assumed to be equal to $1$ twists by upper triangular matrices yield indeed all twisted Teichmüller curves because the number of cusps of $X_D$ is equal to the class number (compare Proposition~\ref{prop_matrix_decomposition}). This information is therefore very useful since it enables us to calculate the volume of any twisted Teichmüller curve if $h_D=1$. Surprisingly enough, we can reduce the calculation for upper triangular matrices to the case of diagonal twisted Teichmüller curves. So let us consider a Teichmüller curve twisted by the matrix
$$M=\begin{pmatrix} m & x \\ 0 & n \end{pmatrix}$$
with $m,n,x \in K$. By multiplying with the common denominator we may without loss of generality assume that $m,n,x \in \OD$ and $(m,n,x)=1$. Recall that the stabilizer $\SL_M(L_D)$ of the twisted Teichmüller curve \index{stabilizer!of a twisted Teichmüller curve} is then given by $\SL_M(L_D)=\Stab(\Phi)^{M^{-1}} \cap \SL_2(\OD)$ (Proposition \ref{prop_stabilzer_twisted}). We will again consider the conjugated group $\Stab(\Phi) \cap M^{-1}\SL_2(\OD)M$ which is known to be equal to $\SL^M(L_D,M)=\SL(L_D) \cap M^{-1}\SL_2(\OD)M$ in most cases (see Theorem~\ref{thm_cov_degree_10}). Analogously as for diagonal twisted Teichmüller curves, we have to distinguish the case where both $m$ and $n$ are relatively prime to $\eta^*$ and the case where at least one of them has a common prime divisor with $\eta^*$. While we can always give precise formulas in the first case, there again occur different phenomena in the second case.

\subsubsection{The Relatively Prime Case}

In this section the aim is to prove the following theorem:

\begin{thm} \label{thm_summarize_euler_calculations_triangular} Let $m,n,x \in \OD$ be arbitrary elements with $(m,n,x)=1$ and let $M$ be as above. If $h_D=1$ and $(m,\eta^*)=1$ and $(n,\eta^*)=1$ then the degree of the covering $\widetilde{\pi}: C^M(M) \to C$ equals the degree of the covering $\pi: X_D(M) \to X_D$. In other words
$$\left[ \SL(L_D) : (\SL(L_D) \cap M^{-1} \SL_2(\OD) M ) \right] = \left[ \SL_2(\OD) :  \Gamma^D_0(nm)   \right].$$
If the degree of the covering $\pi : C_M(M) \to C_M$ is equal to 1, then the volume of the Teichmüller twisted by $M= \left( \begin{smallmatrix} m & x \\ 0 & n \end{smallmatrix} \right)$ is
$$-9\pi \left[ \SL_2(\OD) : \Gamma^D_0(nm) \right] \chi(X_D).$$
\end{thm}

The theorem tells us that the volume of the Teichmüller curve can be only preserved if the determinant of $M$ is a unit. On the other hand the volume does not \textit{detect} the upper right entry of the matrix.
\paragraph{D $\equiv$ 1 mod 4, odd spin.} We start again with the $\SL(L_D^1)$ case. Let us from now assume that $(m,w)=1$ and $(n,w)=1$. By Theorem~\ref{thm_cov_degree_10} we then have that $\Stab(\Phi) \cap M^{-1}\SL_2(\OD)M = \SL(L_D^1) \cap M^{-1}\SL_2(\OD)M = \SL^M(L_D^1,M)$. In order to find out the volume of the twisted Teichmüller curve, it is therefore necessary to calculate $[\SL(L_D^1): \SL^M(L_D^1,M)]$.\\
Although the following lemma follows immediately from the definition, it plays an important role for the calculation of the involved indexes. 
\begin{lem} For all $m,n,p,x \in \OD$ with $(m,x,n,p)=1$ and
$$M = \begin{pmatrix} m & x \\ 0 & np \end{pmatrix} \quad \text{and} \quad N= \begin{pmatrix} m & x \\ 0 & n \end{pmatrix}$$
we have $\SL^M(L_D,M) \subset \SL^N(L_D,N)$.
\end{lem}
This lemma implies that we can proceed step by step and do similar calculations as in the diagonal twist case. For $i,j \in \NN$ and an arbitrary prime element $\pi \in \OD$ with $(\pi,w)=1$ we set $P:= \left( \begin{smallmatrix} \pi^i & x \\ 0 & \pi^j \end{smallmatrix} \right)$ and $Q:= \left( \begin{smallmatrix} \pi^i & x \\ 0 & \pi^{j-1} \end{smallmatrix} \right)$. Then we show the following fundamental lemma about the indexes:
\begin{lem} \label{prop_triangular_twisted_2} For all $x \in \OD$ and all prime elements $\pi \in \OD$ with $(\pi,w)=1$ and $(\pi,x)=1$ and all $i,j \in \NN$
\begin{eqnarray*} \left[\SL(L_D^1) : \SL^P(L_D^1,P) \right] = \N(\pi) \left[\SL(L_D^1) : \SL^Q(L_D^1,Q)\right]\end{eqnarray*}
holds.
\end{lem}
\begin{proof}
As usual the different types of prime numbers (split, inert, ramified) require different arguments. Recall that $Z=T^t$ for $\SL(L_D^1)$.\\[11pt]
\textit{(i)} If $\pi$ is a divisor of a split prime number and $j>i$ then $QT^{\N(\pi^{j-i-1})}Q^{-1} \in \SL_2(\OD)$, but $PT^{\N(\pi^{j-i-1})}P^{-1} \notin \SL_2(\OD)$. So the matrices $T^{{\N(\pi^{j-i-1})}k}$ for $1 \leq k \leq \N(\pi)$ are elements in the bigger group which are pairwise incongruent modulo the smaller group. If $i \geq j$ then $QZ^{\N(\pi^{j+i-1})}Q^{-1} \in \SL_2(\OD)$ and  $PZ^{\N(\pi^{j+i-1})}P^{-1} \notin \SL_2(\OD)$. \\[11pt]
\textit{(ii)} If $\pi$ is an inert prime number, we get by an argument of the same type as in the case of split prime numbers that we only have to compare the indexes for the case $i=j=1$. Let $L$ be the second parabolic matrix in the Veech group from Lemma~\ref{lem_second_parabolic}. We only do the case $D \equiv 5 \mod 8$ here. The other cases work exactly in the same way. If $\pi \neq 2$ we claim that the matrices $Z^{\pi k}L^{\pi z}$ with $1 \leq k \leq \pi$ and $1 \leq z \leq \pi$ are matrices in the bigger group which are pairwise incongruent modulo the smaller group. Indeed, $QZ^{\pi k}L^{\pi z}Q^{-1} \in \SL_2(\OD)$ and it requires a long and tedious calculation to check that
$$(PZ^{\pi k}L^{\pi z}Z^{-\pi l}P^{-1})_{1,2} = \frac{x^2}{\pi}(4(w+1)z+w(k-l)) +v$$
for some $v \in \OD$. Since $(\pi,x)=1$ we must therefore have $\pi|4z$ and $\pi|(k-l)$. This yields the claim. If $\pi=2$ then a similar calculation yields $$\left[\SL^Q(L_D^1,Q) : \SL^P(L_D^1,P)\right] = 4$$
since the matrices $Z^{2k},E^{2z}$ for $1 \leq k \leq 2$ and $1 \leq z \leq 2$ all lie in $\SL^Q(L_D^1,Q)$ but are incongruent modulo $\SL^P(L_D^1,P)$ where $E$ is the parabolic element fixing the cusp $1$ (compare Section~\ref{sec_fixing_Veech}).\\[11pt]
\textit{(iii)} If $\pi$ is a divisor of a ramified prime number we may again restrict to the case $i=j=1$. Let $U:=\left( \begin{smallmatrix} \pi^2 & x \\ 0 & 1 \end{smallmatrix} \right) = \left( \begin{smallmatrix} 1 & x \\ 0 & 1 \end{smallmatrix} \right) \left( \begin{smallmatrix} \pi^2 & 0 \\ 0 & 1 \end{smallmatrix} \right)$ and $V:=\left( \begin{smallmatrix} \pi^2 & x \\ 0 & \pi \end{smallmatrix} \right)$. It can be checked that $UZ^{\N(\pi)i} U^{-1} \in \SL_2(\OD)$ but $V Z^{\N(\pi)i} V^{-1} \notin \SL_2(\OD)$ for $1 \leq i < \N(\pi)$ and therefore:
\begin{eqnarray*}
(\N(\pi)+1)\N(\pi)^2 & = & \N(\pi) \left[\SL(L_D^1): \SL^U(L_D^1,U) \right]\\
& = & \left[\SL(L_D^1): \SL^V(L_D^1,V)\right] \\
& \leq & \N(\pi) \left[\SL(L_D^1): \SL^P(L_D^1,P)\right]\\
& \leq & \N(\pi) \N(\pi) (\N(\pi)+1)
\end{eqnarray*}
\end{proof}
By passing to appropriate powers of $T,Z$ and $L$ one can - similarly as in Section~\ref{sec_volume_simple} - deduce the more general result: Let $m,n,x \in \OD$ with $(m,n,x)=1$ and $\pi \in \OD$ be a prime element with $(\pi,w)=1$ and set $P:=\left( \begin{smallmatrix} m\pi^i & x \\ 0 & n\pi^j \end{smallmatrix} \right)$ and $Q:= \left( \begin{smallmatrix} m\pi^i & x \\ 0 & n\pi^{j-1} \end{smallmatrix} \right)$.
\begin{prop} \label{prop_triangular_2} Let $P$ and $Q$ be defined as above. Then
\begin{eqnarray*} \left[\SL(L_D^1) : \SL^P(L_D^1,P)\right] = \N(\pi) \left[\SL(L_D^1) : \SL^Q(L_D^1,Q)\right].\end{eqnarray*}
\end{prop}
It is now possible to show that the volume of the Teichmüller curve twisted by $M$ only depends on the determinant of the matrix. 

\begin{prop} \label{prop_triangular_twisted_1} If $M = \left( \begin{smallmatrix} m & x \\ 0 & n \end{smallmatrix} \right)$ with $m,n,x \in \OD$ with $(m,n,x)=1$, then 
$$\left[ \SL(L_D^1) : \SL^M(L_D^1,M) \right] = \left[ \SL_2(\OD) :  \Gamma^D_0(nm)   \right].$$
\end{prop}
\begin{proof} We first assume that $(n,x)=1$. Then let $\pi$ be an arbitrary prime divisor of $n$ and set $U=\left( \begin{smallmatrix} m\pi & x\\ 0 & n \end{smallmatrix} \right)$. Hence
$$\left[ \SL(L_D^1) :  \SL^M(L_D^1,M)  \right] \geq
\frac{1}{\N(\pi)} \left[ \SL(L_D^1) : \SL^U(L_D^1,U) \right].$$
By Proposition~\ref{prop_triangular_2} for $V=\left( \begin{smallmatrix} m\pi & x\\ 0 & n/\pi \end{smallmatrix} \right) $ the latter expression is equal to
$$\left[ \SL(L_D^1) : \SL^V(L_D^1,V) \right]$$
and hence the index $\left[ \SL(L_D^1) : \SL^M(L_D^1,M)  \right]$ is maximal if and only if the index $\left[ \SL(L_D^1) : \SL^V(L_D^1,V) \right]$ is maximal. We may therefore assume that $M$ is of the form $\left( \begin{smallmatrix} mn & x \\ 0 & 1 \end{smallmatrix} \right)$.
Since
$$\begin{pmatrix} 1 & - x \\ 0 & 1 \end{pmatrix} \begin{pmatrix} mn & x \\ 0 & 1 \end{pmatrix} = \begin{pmatrix} mn & 0 \\ 0 & 1 \end{pmatrix}$$
the claim thus follows.\\[11pt]
If there exist common prime divisors of $x$ and $n$, then we can repeat the preceding arguments until the lower right entry divides $x$. Obviously the claim then also follows since $n|x$.
\end{proof}

This completes the proof of Theorem~\ref{thm_summarize_euler_calculations_triangular} for $D \equiv 5 \mod 8$ and for $D \equiv 1 \mod 8$ for the odd spin Teichmüller curves.
\paragraph{D $\equiv$ 1 mod 8, even spin.} It is clear that it suffices to prove an analogue of Lemma~\ref{prop_triangular_twisted_2} in order to show  Theorem~\ref{thm_summarize_euler_calculations_triangular} in the even spin case. Let again be $P:=\left( \begin{smallmatrix} \pi^i & x \\ 0 & \pi^j \end{smallmatrix} \right)$ and $Q:=\left( \begin{smallmatrix} \pi^i & x \\ 0 & \pi^{j-1} \end{smallmatrix} \right)$ two matrices where $x \in \OD$ and $\pi \in \OD$ is a prime element with $(\pi,x)=1$ and $(\pi,\eta^*)=1$. 
\begin{lem} \label{prop_triangular_twisted_3} For all $x \in \OD$ and all prime elements $\pi \in \OD$ with $(\pi,\eta^*)=1$ and $(\pi,x)=1$ and all $i,j \in \NN$
\begin{eqnarray*} \left[\SL(L_D^0) : \SL^P(L_D^0,P)\right] = \N(\pi) \left[\SL(L_D^0) : \SL^Q(L_D^0,Q)\right]\end{eqnarray*}
holds.
\end{lem}
\begin{proof} If $\pi$ is a divisor of a split or a ramified prime number $p \in \ZZ$ then the proof is verbatim the same as in Lemma~\ref{prop_triangular_twisted_2} if we use the appropriate matrices $T$ and $Z$. So we just have do the case, where $\pi$ is an inert prime number and $i=j=1$. We do here only the case $D \equiv 1 \mod 16$. We consider the matrices $Z^{\pi k}L^{\pi z}$ with $1 \leq k \leq \pi$ and $1 \leq z \leq \pi$ where $L$ stems again from Lemma~\ref{lem_second_parabolic2}. Then $QZ^{\pi k}L^{\pi z}Q^{-1} \in \SL_2(\OD)$ and therefore all the matrices lie in the bigger group. It remains to be shown that the matrices are not equivalent modulo the smaller group. We have
$$(PZ^{\pi k}L^{\pi z}Z^{-\pi l}P^{-1})_{1,2} = \nu_{1,2}$$
where $\nu_{1,2} = \frac{x^2}{\pi} \left((l-k)(w+1) + z \left( \frac{D+15}{16}w + \frac{3(D-1)}{16}\right) \right)+ h $ with $h \in \OD$. So assume that this entry lies in $\OD$. Then since $\pi$ is inert and $(\pi,x)=1$ it follows that $\pi$ divides the \textit{imaginary} part of the expression as well as the \textit{real} part and hence $\pi|z\frac{D-9}{8}$. However $\pi$ cannot divide $\frac{D-9}{8}$ since $\pi$ is inert and $\N(w+1)=\frac{D-9}{4}$. Therefore, $\pi|z$ and thus $\pi|(l-k)$.
\end{proof}
\paragraph{D $\equiv$ 0 mod 4.} We again show an analogue of Lemma~\ref{prop_triangular_twisted_2}. So let $P:=\left( \begin{smallmatrix} \pi^i & x \\ 0 & \pi^j \end{smallmatrix} \right)$ and $Q:=\left( \begin{smallmatrix} \pi^i & x \\ 0 & \pi^{j-1} \end{smallmatrix} \right)$ be as usual two matrices where $x \in \OD$ and $\pi \in \OD$ is a prime element with $(\pi,x)=1$ and $(\pi,\eta^*)=1$.
\begin{lem} \label{prop_triangular_twisted_4} For all $x \in \OD$ and all prime elements $\pi \in \OD$ with $(\pi,\eta^*)=1$ and $(\pi,x)=1$ and all $i,j \in \NN$
\begin{eqnarray*} \left[\SL(L_D^0) : \SL^P(L_D^0,P)\right] = \N(\pi) \left[\SL(L_D^0) : \SL^Q(L_D^0,Q)\right]\end{eqnarray*}
holds.
\end{lem}
\begin{proof} If $\pi$ is a divisor of a split or a ramified prime number $p \in \ZZ$ then the proof is again verbatim the same as in Lemma~\ref{prop_triangular_twisted_2}. So we just have do the case, where $\pi$ is an inert prime number and $i=j=1$. We do here only the case $D \equiv 4 \mod 8$. We consider the matrices $Z^{\pi k}L^{\pi z}$ with $1 \leq k \leq \pi$ and $1 \leq z \leq \pi$ where $L$ stems from Lemma~\ref{lem_second_parabolic_0mod4}. Then $QZ^{\pi k}L^{\pi z}Q^{-1} \in \SL_2(\OD)$ and therefore all the matrices lie in the bigger group. It remains to be shown that the matrices are not equivalent modulo the smaller group. We have
$$(PZ^{\pi k}L^{\pi z}Z^{-\pi l}P^{-1})_{1,2} = \frac{x^2}{\pi} \left((l-k)w - z (D/2+w D/4) \right)+ h $$
with $h \in \OD$. So assume that this entry lies in $\OD$. Then since $\pi$ is inert and $(\pi,x)=1$ it follows that $\pi$ divides the \textit{imaginary} part of the expression as well as the \textit{real} part. Since $\pi$ is inert, we hence have $\pi|z$ and therefore $\pi|(l-k)$.
\end{proof}

\subsubsection{The Non-relatively Prime Case} \label{subsec_non_prime}
If $m$ or $n$ is not relatively prime to $\eta^*$ the index of $\SL(L_D) \cap M^{-1}\SL_2(\OD)M$ in $\SL(L_D)$ is not always maximal. In fact, there are examples where the index is maximal, i.e. equal to the index of $\Gamma^D_0(mn)$ in $\SL_2(\OD)$, and examples where it is not maximal. It has been already described in Section~\ref{subsec_nonrelative} that the first happens for instance in the case $D=13$ and $M=\left( \begin{smallmatrix} w & 0 \\ 0 & 1 \end{smallmatrix} \right)$ (compare Theorem~\ref{thm_summarize_euler_calculations}). We will even see that the index is maximal for all twisted Teichmüller curves if $D = 13$. An example for the latter phenomenon to happen is $D=17$, odd spin and $M=\left( \begin{smallmatrix} 1 & 0 \\ w+1 & w+2 \end{smallmatrix} \right)$. Then the volume of the Teichmüller curve twisted by $M$ is equal to the volume of the original Teichmüller curve.\\[11pt]If the discriminant $D$ is $5 \mod 8$, then let $i,j \in \NN$, $x \in \OD$ and let $\pi \in \OD$ be an arbitrary prime element with $(\pi,w)=\pi$ and set $P:= \left( \begin{smallmatrix} \pi^i & x \\ 0 & \pi^j \end{smallmatrix} \right)$ and $Q:=\left( \begin{smallmatrix} \pi^i & x \\ 0 & \pi^j \end{smallmatrix} \right)$. Then we can prove the following lemma: \begin{lem} If $D \equiv 5 \mod 8$ then for all $x \in \OD$ and all prime elements $\pi \in \OD$ with $(\pi,w)=\pi$ and $(\pi,x)=1$ and all $i,j \in \NN$
\begin{eqnarray*} \left[\SL(L_D^1) : \SL^P(L_D^1,P) \right] = \N(p) \left[\SL(L_D^1) : \SL^Q(L_D^1,Q)\right]\end{eqnarray*}
holds.
\end{lem}
\begin{proof} As always we may restrict to the case $i=j=1$. Note that $\pi \nmid2$ since $D \equiv 5 \mod 8$ and since $\pi$ is a divisor of a split prime number. Let $L$ again be the matrix from Lemma~\ref{lem_second_parabolic}. Then $QL^{\N(\pi)}Q^{-1} \in \SL_2(\OD)$, but $PL^{\N(\pi)}P^{-1} \notin \SL_2(\OD)$. Therefore, the matrices $L^{\N(\pi)i}$ for $1 \leq i \leq \N(\pi)$ are elements in the bigger group which are incongruent modulo the smaller group. \end{proof}

Applying this lemma, Theorem~\ref{thm_summarize_euler_calculations_triangular}, Theorem~\ref{thm_cov_degree_10}~(iii) and Proposition~\ref{thm_summarize_euler_calculations}~(ii) yields:

\begin{thm} \label{thm_summarize_euler_II} Let $D \equiv 5 \mod 8$ be a fundamental discriminant and let $m,n,x \in \OD$ be arbitrary elements with $(m,n,x)=1$ and let $$M=\begin{pmatrix} m & x \\ 0 & n \end{pmatrix}.$$
If $h_D=1$ then the degree of the covering $\widetilde{\pi}: C^M(M) \to C$ equals the degree of the covering $\pi: X_D(M) \to X_D$. In other words \index{twisted Teichmüller curve!volume|)}
$$\left[ \SL(L_D) : (\SL(L_D) \cap M^{-1} \SL_2(\OD) M ) \right] = \left[ \SL_2(\OD) :  \Gamma^D_0(nm) \right].$$
The volume of the Teichmüller twisted by $M$ is then
$$-9\pi \left[ \SL_2(\OD) : \Gamma^D_0(nm) \right] \chi(X_D).$$
\end{thm} 

\subsection{Classification of Twisted Teichmüller Curves} \label{sec_classification_twisted}
\index{twisted Teichmüller curve!classification|(}
In Section~\ref{sec_special_algebraic_curves} we have mentioned the theorem of H.-G. Franke and W. Hausmann which states that there are only finitely many different twisted diagonals if the determinant of the associated primitive skew-hermitian matrix is fixed. In this section, we will see that a corresponding result does also hold for twisted Teichmüller curves. To be more precise, we will almost completely classify twisted Teichmüller curves for $h_D^+=1$ and $D \equiv 5 \mod 8$ a fundamental discriminant. Most importantly, we may by Proposition \ref{prop_matrix_decomposition} then restrict to upper triangular matrices in this case and  make use of Theorem~\ref{thm_summarize_euler_II}.
\begin{rem} So far we have only considered the stabilizer of the graph of the Teichmüller curve, \index{stabilizer!of the graph of a Teichmüller curve} i.e. the stabilizer of the universal covering map of the Teichmüller curve. From this the \textbf{stabilizer of the Teichmüller curve} $C$, by which mean we mean its stabilizer in $X_D$, has to be carefully distinguished. If $D$ is a fundamental discriminant then a matrix $M \in \GL_2^+(K)$ and its corresponding map $\psi_M: (z_1,z_2) \mapsto (Mz_1,M^\sigma z_2)$ define an automorphism of $X_D$ if and only if $M \in \SL_2(\OD)$ (see Section~\ref{sec_hilbert_modular_surfaces}). The stabilizer of the Teichmüller curve is then rather boring.\end{rem} 
It is yet another different question, if for a matrix $M \in \GL_2^+(K)$ the twisted Teichmüller curve $C_M$ agrees with the original Teichmüller curve $C$. In some sense this is also a stabilizer and it will therefore be called the \textbf{twisting stabilizer} \index{stabilizer!twisting stabilizer}of the Teichmüller curve.\\[11pt] Still one might have in mind different such twisting stabilizers, i.e. with respect to different set of matrices. Essentially, one might think of three different such sets. The first one is the twisting stabilizer in $\SL_2(\RR)^2$. We denote this twisting stabilizer by $\St_{\SL_2(\RR)^2}(C)$. \label{glo_St} We have discussed in Chapter~\ref{cha_twisted_Teichmüller_curves} that $\St_{\SL_2(\RR)^2}(C)$ is far from being accessible. By dividing matrices in $\GL_2^+(K)$ by the root of their determinant we can also interpret these matrices as matrices in $\SL_2(\RR)^2$ and ask which of them stabilizes the Teichmüller curve by a twist. We denote the corresponding group by $\St_{\GL_2^+(K)}(C)$. Note that neither $\St_{\SL_2(\RR)^2}(C)$ nor $\St_{\GL_2^+(K)}(C)$ has to be a group ex ante since the twisting stabilizer is not defined by an action. Finally all matrices in $\SL_2(\OD)$ by definition stabilize the Teichmüller curve and these matrices can also be interpreted as a subset $\St_{\SL_2(\OD)}(C) \subset \St_{\GL_2^+(K)}(C)$. \\[11pt]
Since the matrices $M\in\GL_2^+(K)$ and $kM$ with $k \in K$ are identified when they are considered in $\SL_2(\RR)$ every matrix $M \in \GL_2^+(K)$ can, by multiplying each entry of $M$ with the common denominator of its entries, be uniquely written as $$M =\begin{pmatrix} a & b \\ c & d \end{pmatrix}$$ with $a,b,c,d \in \OD$ and $(a,b,c,d)=1$ if $h_D=1$ without changing the twisted Teichmüller curve. Let $q=(a,c)$. Recall from Proposition~\ref{prop_matrix_decomposition} that $M$ then yields the same twisted Teichmüller curve as the upper triangular matrix $\left( \begin{smallmatrix} k & l\\ 0 & m \end{smallmatrix} \right)$ where $k = q$, $m=(ad-bc)/q$ and $l=-\frac{f(ad-bc)/q-b}{a/q}$ and where $f$ is chosen such that $\frac{fc/q+1}{a/q} \in \OD$. \\[11pt]
If $h_D=1$ and $D \equiv 5 \mod 8$ is a fundamental discriminant we were able to calculate the volume of all twisted Teichmüller curves in Theorem~\ref{thm_summarize_euler_II}. This invariant of the twisted Teichmüller curve enables us to calculate $\St_{\GL_2^+(K)}(C)$ whenever the fundamental unit $\epsilon$ of $\OD$ has a negative norm. 

\begin{thm} \label{thm_calculation_of_st} If $D \equiv 5 \mod 8$ is a fundamental discriminant with narrow class number $h_D^+=1$ then $$\St_{\GL_2^+(K)}(C)=\SL_2(\OD)$$
holds.\footnote{Recall from Section~\ref{sec_quadratic_nf} that $h_D^+=1$ implies in particular that $\OD$ has a fundamental unit of negative norm.}
\end{thm}
\begin{proof} If $M \in \GL_2^+(K)$ is an upper triangular matrix then we may assume that $M$ is of the form $\left( \begin{smallmatrix} m & x \\ 0 & n \end{smallmatrix} \right)$ with $m,n,x \in \OD$ and $(m,n,x)=1$. By Theorem~\ref{thm_summarize_euler_II}, $M$ preserves the volume if and only if $m \in \OD^*$ and $n \in \OD^*$. Since the fundamental unit $\epsilon$ has negative norm we thus have $m = \pm \epsilon^{2g}$ and $n = \pm \epsilon^{2g'}$ for $g,g' \in \ZZ$. By dividing each entry of $M$ by an appropriate power of $\epsilon$ we see that $M$ is indeed a matrix in $\SL_2(\OD)$.\\[11pt]
If $M$ is not an upper triangular matrix, then $$M =\begin{pmatrix} a & b \\ c & d \end{pmatrix}$$ with $a,b,c,d \in \OD$ and $(a,b,c,d)=1$. Let $q=(a,c)$ or in other words
$$M =\begin{pmatrix} q\widetilde{a} & b \\ q\widetilde{c} & d \end{pmatrix}$$ where $\widetilde{a} = a/q \in \OD$ and $\widetilde{c} = c/q \in \OD$. From Proposition~\ref{prop_matrix_decomposition}, it then follows that $M$ yields the same twisted Teichmüller curve as the upper triangular matrix $\widetilde{M} = \left( \begin{smallmatrix} k & l\\ 0 & m \end{smallmatrix} \right)$ where $k = q$, $m=(\widetilde{a}d-b\widetilde{c})$ and $l=-\frac{f(\widetilde{a}d-b\widetilde{c})-b}{\widetilde{a}}$ where $f$ is chosen such that $\frac{\widetilde{c}f+1}{\widetilde{a}} \in \OD$. Suppose that $C_M$ is equal to $C$. Theorem~\ref{thm_summarize_euler_II} then implies that all the entries of $\widetilde{M}$ are divisible by $q$ and moreover that $q=(\widetilde{a}d-b\widetilde{c})$. Since $q|l$ it follows that $q|b$, i.e. $b = \widetilde{b}q$ with $\widetilde{b} \in \OD$. Then $(a,b,c,d)=1$ and $q=(\widetilde{a}d-b\widetilde{c})$ imply that $q|\widetilde{a}$, i.e. $\widetilde{a} = q \hat{a}$ for some $\hat{a} \in \OD$. Then $$l=-\frac{f(\widetilde{a}d-b\widetilde{c})-b}{\widetilde{a}} = -\frac{fq-\widetilde{b}q}{\hat{a}q} = - \frac{f-\widetilde{b}}{\hat{a}}.$$ Therefore, $\widetilde{M}$ defines the same twisted Teichmüller curve as
$$\begin{pmatrix} 1 & -\frac{f-\widetilde{b}}{q\hat{a}} \\ 0 & 1 \end{pmatrix}$$ and hence in particular $\widetilde{b}=uq\hat{a}+f$ for some $u \in \OD$ since $C_M=C$. Then 
$$d = \frac{\widetilde{b}\widetilde{c}+1}{\hat{a}} =  \frac{(uq\hat{a}+f)\widetilde{c}+1}{\hat{a}} = \widetilde{c}uq + \frac{\widetilde{c}f+1}{\hat{a}}$$
and hence $q|d$ since $\frac{\widetilde{c}f+1}{\widetilde{a}} = \frac{\widetilde{c}f+1}{q\hat{a}} \in \OD$. Therefore, $q=1$ and $ad-bc=1$, i.e. $M \in \SL_2(\OD)$. \end{proof}
From this result the classification of twisted Teichmüller curves in the case $h_D^+=1$ and $D \equiv 5 \mod 8$ fundamental discriminant can be derived. Recall that two twisted Teichmüller curves $C_M$ and $C_N$ are the same curve in $X_D$ if and only if there exists a $J \in \SL_2(\OD)$ such that 
$$\left\{ (Mz,M^\sigma\varphi(z)) \mid z \in \HH \right\} = \left\{ (JNz,J^\sigma N^\sigma \varphi(z)) \mid z \in \HH \right\}.$$
We only consider Teichmüller curves twisted by upper triangular matrices here since $h_D^+=1$; by normalizing the matrix, we only have to check when two matrices $M= \left( \begin{smallmatrix} m & x \\ 0 & n \end{smallmatrix} \right)$ and $N= \left( \begin{smallmatrix} a & b \\ 0 & c \end{smallmatrix} \right)$ with $a,b,c,m,n,x \in \OD$ and $(m,n,x)=1$ and $(a,b,c)=1$ define the same curve in $X_D$. If the associated Möbius transformations of $M$ and $N$ differ by a matrix in $\SL_2(\OD)$ then the twisted Teichmüller curves agree. On the other hand we have:

\begin{thm} \label{thm_classificiation1} \index{twisted Teichmüller curve!classification|)}  Suppose $D \equiv 5 \mod 8$ is a fundamental discriminant with narrow class number $h_D^+=1$. If $M= \left( \begin{smallmatrix} m & x \\ 0 & n \end{smallmatrix} \right)$ and $N = \left( \begin{smallmatrix} a & b \\ 0 & c \end{smallmatrix} \right)$ with $a,b,c,m,n,x \in \OD$ and $(a,b,c)=1$ and $(m,n,x)=1$ define the same twisted Teichmüller curve then $$\det(M) = \det(N).$$
\end{thm}
\begin{proof}
If $M$ and $N$ define the same twisted Teichmüller curve then there exists an $J \in \SL_2(\OD)$ such that
$$(Mz,M^{\sigma}\varphi(z))=(JNz^*,J^{\sigma}N^{\sigma}\varphi(z^*)).$$
for some $z^* \in \mathbb{H}$ depending on $z$. The first component yields:
$$ N^{-1} \cdot J^{-1} \cdot M z= z^*$$
Inserting this in the second component gives
$$M^\sigma \varphi(z) = J^\sigma \cdot N^\sigma \varphi(N^{-1} \cdot J^{-1} \cdot M z)$$ 
which is equivalent to
$$N^{\sigma^{-1}} \cdot J^{\sigma^{-1}} \cdot M^\sigma \cdot \varphi(z) = \varphi(N^{-1} \cdot J^{-1} \cdot M z).$$ 
Therefore, $N^{-1}J^{-1}M$ lies in the stabilizer of the graph of the Teichmüller curve and in particular also in the twisting stabilizer. Note that $$\det(N^{-1}J^{-1}M) = \frac{nm}{ac}.$$ 
Thus it follows from Theorem~\ref{thm_calculation_of_st} that for $k = \frac{\sqrt{ca}}{\sqrt{nm}}$ the matrix $kN^{-1}J^{-1}M$ has to be a properly normalized matrix in $\GL_2^+(K)$, i.e. in $\Mat^{2x2}(\OD)$ and its entries have no common divisor. Since $N^{-1}J^{-1}M \in \GL_2^+(K)$ also $k$ must then lie in $K$, i.e. $k=\frac{p}{q}$ with $p,q \in \OD$ and $(p,q)=1$. Let 
$$J^{-1} = \begin{pmatrix} e & f \\ g & h \end{pmatrix}.$$
Then
$$Z:=kN^{-1}J^{-1}M = \frac{1}{c} \frac{p}{q}  \begin{pmatrix} \frac{1}{a} (m (-gb+ec)) & \frac{1}{a} (-b(gx+hn)+ c(ex+fn))\\ gm & gx+hn \end{pmatrix}.$$
Let $q_1$ be a prime divisor of $q$. From the lower left entry of $Z$ it follows that $q_1|g$ oder $q_1|m$. Assume that $q_1|g$. Then $Z_{2,2}$ implies that $q_1|n$ since $q_1 \nmid h$ because $J \in \SL_2(\OD)$. Considering $Z_{1,1}$ we then have $q_1|m$ since $q_1|e$ would again contradict $J \in \SL_2(\OD)$.\\
We hence always have $q_1|m$. From $Z_{2,2}$ it then follows that $q_1 | gx+hn$ and therefore from $Z_{1,2}$ we get $q_1 | ex+fn$. Then $q_1 | (eh-fg)x$ which implies $q_1|x$ since $\det(J^{-1})=1$ and $q_1|(eh-fg)n$ which means $q_1|n$. Since $(m,n,x)=1$ hence $q_1=1$. This implies that $q=1$. Interchanging the roles of $N$ and $M$ we get $p=1$. This means that $mn=ac$ or in other words  $\det(M) = \det(N)$.
\end{proof}

We can also state a counterpart to the theorem by H.-G. Franke and W. Hausmann. If $M=\left( \begin{smallmatrix} m & x \\ 0 & n \end{smallmatrix} \right)$ with $m,n \in \OD$ fixed and $x \in K$ arbitrary the volume of the twisted Teichmüller curves changes with varying $x$. However, after normalizing the involved matrices appropriately such that $m,n,x \in \OD$ and $(m,n,x)=1$ there are only finitely many different twisted Teichmüller curves of a given determinant. The number of different twisted Teichmüller curves of the same determinant can be easily bounded very roughly.
\begin{prop} If $k \in \OD$ and $k = \prod_i \pi_i^{e_i}$ is a decomposition of $k$ into prime elements with $e_i \in \NN$. Let $f= \sum_i e_i$ then the number of different twisted Teichmüller curves of determinant $k$ is at most $2^f \N(k)$ if $h_D^+=1$.
\end{prop}
\begin{proof} 
Since the class number is equal to one, we may assume that $M= \left( \begin{smallmatrix} m & x \\ 0 & n \end{smallmatrix} \right)$ is an upper triangular matrix with $m,n,x \in \OD$ and $(m,n,x)=1$. Since $mn=k$ each of the $e_i$ factors $\pi_i$ must either divide $m$ or $n$. This gives altogether $2^f$ possibilities for the diagonal. Since $M$ and $JM$ define the same twisted Teichmüller curves if $J= \left( \begin{smallmatrix} 1 & y \\ 0 & 1 \end{smallmatrix} \right) \in \SL_2(\OD)$ there are at most $\N(k)$ different possibilities for the choice of $x$ if $m$ and $n$ are fixed.
\end{proof}

We have strong numerical evidence that the following conjecture holds. It is based on computer experiments for many different determinants $k \in \OD$ including all types of splitting behavior of the prime divisors of $k$.

\begin{conje} Suppose $D \equiv 5 \mod 8$ is a fundamental discriminant with $h_D^+=1$. All matrices $M \in \GL_2^+(K) \cap \Mat^{2x2}(\OD)$ of determinant $n \in \OD$ with relative prime entries define the same twisted Teichmüller curve, i.e. there is exactly one twisted Teichmüller curve of determinant $n$.\end{conje}

We can prove this conjecture in two instances. Indeed, equality of the determinant is also a sufficient criterion for the twisted Teichmüller curves to coincide whenever the determinant is prime in $\OD$. 

\begin{thm} \label{thm_prime_classification}  Suppose $D \equiv 5 \mod 8$ is a fundamental discriminant, $h_D^+=1$. Let $\pi \in \OD$ be a prime element. Then there is exactly one twisted Teichmüller curve of determinant $\pi$. \end{thm}
\begin{proof} Anyway, all twisted Teichmüller curves are given by the matrices
$$M:=\begin{pmatrix} \pi & 0 \\ 0 & 1 \end{pmatrix} \quad \text{and} \quad N_x:=\begin{pmatrix} 1 & x\\ 0 & \pi \end{pmatrix}, \ x \in \OD, \ \pi \nmid x.$$
Let us say a few words why these are indeed all possible matrices: since we may multiply any twist-matrix $W$ from the left by a matrix $V \in \SL_2(\OD)$ without changing the twisted Teichmüller curve, $M$ yields indeed the only twisted Teichmüller curve with upper left entry $\pi$. Also note that we can multiply a twist-matrix $W$ from the right by an element $V$ in the Veech group without changing the twisted Teichmüller curve. Hence the matrices $M$ and $N_0$ define the same twisted Teichmüller curve because $M=-SN_{0}S$ and $S \in \SL(L_D)$. This means that we may assume $\pi \nmid x$.\\[11pt]
We now show that all the above matrices define the same twisted Teichmüller curve. This is by definition equivalent to the existence of some matrix $W_x \in \SL(L_D)$ for all $x \in \OD$ such that 
$N_xW_xM^{-1} \in \SL_2(\OD)$. As usual, we now have to distinguish the different types of prime numbers.\\[11pt]
\textit{(i)} If $\pi$ is a divisor of a split prime number $p \in \ZZ$ and $\pi \nmid w$ then
$$N_xT^lSM^{-1} = \begin{pmatrix} \frac{lw+x}{\pi} & -1 \\ 1 & 0 \end{pmatrix}.$$
Since $\pi \nmid w$ there thus exists $l \in \ZZ$ with $N_xT^lSM^{-1} \in \SL_2(\OD)$. If $\pi |w$ then $\pi \nmid 2$ since $D \equiv 5 \mod 8$. We consider again the matrix $L$ from Lemma~\ref{lem_second_parabolic}. Then
$$(ML^kN_x^{-1})_{1,1} = \frac{1}{\pi} - \frac{4kx}{\pi} - \frac{2wk}{\pi} (1+w+2x)$$
and all the other entries of $ML^kN_x^{-1}$ automatically lie in $\OD$ for all $k$. There exists a $k_0 \in \ZZ$ such that $\pi|(1-4k_ox)$ since $\pi \nmid x$ and $\pi \nmid 2$. Then $ML^kN_x^{-1} \in \SL_2(\OD)$ as $\pi|w$.\\[11pt]
\textit{(ii)} If $\pi$ is an inert prime number, then the matrices 
$$P:= \begin{pmatrix} 1 & u + iw \\ 0 & \pi \end{pmatrix} \quad \text{and} \quad Q:= \begin{pmatrix} 1 & u +jw \\ 0 & \pi \end{pmatrix}$$
with $u,i,j \in \ZZ$ define the same twisted Teichmüller curve since $PT^{j-i}Q^{-1} = \Id$. So we only have to consider those $N_x$ with $x\in \ZZ$ and $\pi \nmid x$. We now show that for all $x \in \ZZ$ the matrices $N_x$ and $N_0$ define the same twisted Teichmüller curve. Then
$$(N_xT^kST^lN_0^{-1})_{1,2} = -\frac{1}{\pi} + \frac{l}{\pi}(xw+kw^2)$$
and all the other entries always are in $\OD$. Note that $xw+k_0w^2 = w(x+k_0) + \frac{D-1}{4}k_0 $. We choose $k_0 \in \ZZ$ such that $\pi|(x+k_0)$ and thus in particular $\pi \nmid k_0$. Then
$$(N_xT^{k_0}ST^lN_0^{-1})_{1,2} = -\frac{1}{\pi} + \frac{lk_0(D-1)}{4\pi} + iw$$
with $i \in \ZZ$. As $\pi$ is an inert prime number $\pi \nmid \frac{D-1}{4}$. Therefore, there exists $l_0 \in \ZZ$ such that $\pi| - 1 + l_0k_0 \frac{D-1}{4}$ and then $(N_xT^{k_0}ST^{l_0}N_0^{-1}) \in \SL_2(\OD).$ \\[11pt]
\textit{(iii)} If $\pi$ is a divisor of a ramified prime number $p \in \ZZ$ then $\pi \nmid w$ and the proof is verbatim the same as in the split prime case.
\end{proof}

Furthermore, we are able to show the corresponding result for diagonal twists.

\begin{prop} Suppose $D \equiv 5 \mod 8$ is a fundamental discriminant with $h_D^+=1$. Then all diagonal matrices $M = \left( \begin{smallmatrix} m & 0 \\ 0 & n \end{smallmatrix} \right)$ with $m,n \in \OD$ and $(m,n)=1$ which have the same determinant define the same twisted Teichmüller curve. \end{prop}

\begin{proof} This is the assertion of Lemma~\ref{lem_facilitate}.
 \end{proof}

\subsection{An Outlook on Further Calculations} \label{sec_further_calculations}

In Section~\ref{sec_upper_triangular_twists} we have been able to precisely calculate the volume of almost any diagonal twisted Teichmüller curve. Naturally one is also interested in knowing the number of elliptic fixed points, the number of cusps and the genus of the surface $\HH / (\SL(L_D) \cap M^{-1} \SL_2(\OD) M)$. In this section we present some ideas how these quantities can be calculated if the narrow class number $h_D^+=1$ and if the spin of the Teichmüller curve is odd. This section is far from solving any of these problems completely, but should rather be regarded as a rough guideline how practical calculations work. In general, it requires a new idea to give formulas for any of these quantities. This is why we confine ourselves to simple twisted Teichmüller curves where we can at least make some statements, i.e. $M= \left( \begin{smallmatrix} m & 0 \\ 0 & 1 \end{smallmatrix} \right)$.
\paragraph{Elliptic fixed points.} \index{twisted Teichmüller curve!elliptic fixed points} Let us describe the number of elliptic fixed points of simple twisted Teichmüller curves. By R. Mukamel's result (Theorem~\ref{thm_mukamel}) for $D \neq 5$ the elliptic elements inside $\SL(L_D^1)$ have order two. We even restrict to the case that $S$ is the only elliptic element in $\SL(L_D^1)$ up to conjugation. It might be possible to use the method described here also in a more general situation. We want to prove the following theorem:

\begin{thm} \label{thm_number_of_elliptic} For all $m \in \OD$ with $(m,w)=1$ there are elliptic elements of order 2 if and only if the following two conditions are satisfied:
\begin{itemize}
\item[(i)] For all prime divisors $\pi \notin \ZZ$ of $m$ with $\textrm{N}(\pi)$ odd we have $\textrm{N}(\pi) \equiv 1 \mod 4$.
\item[(ii)] For all prime divisors $\pi$ of $m$ with $(\pi,2)\neq 1$ we have $\pi^2 \nmid m$.
\end{itemize}
In this case there are exactly $2^k$ elliptic elements of order $2$ where $k$ is the number of prime divisors of $m$ that do not divide $2$. If $(m,w)\neq 1$ the number of elliptic elements of order 2 is bounded by $2^k$.
\end{thm}
Note that the result generalizes the well-known formula for the number of elliptic fixed points of order 2 of the modular curves $\HH / \Gamma_0(m)$ for congruence subgroups $\Gamma_0(m) \subset \SL_2(\ZZ)$ (compare \cite{Miy89}, Theorem~4.2.7).\\[11pt]
Since $S$ has order two it is immediately clear that any elliptic element in $\SL(L_D^1) \cap \Gamma^D_0(m)$ has to have order two if it exists. To calculate the number of elliptic fixed points recall that a fundamental domain for $\SL(L_D^1) \cap \Gamma^D_0(m)$ consist of $[\SL(L_D^1) : \SL(L_D^1) \cap \Gamma^D_0(m)]$ copies of the fundamental domain of $\SL(L_D^1)$. A boundary component of the fundamental domain of $\SL(L_D^1) \cap \Gamma^D_0(m)$ might be glued to another boundary component by an elliptic element. This happens if and only if there exists a coset representative $B \in \SL(L_D^1) / (\SL(L_D^1) \cap \Gamma^D_0(m))$ such that $BSB^{-1} \in \SL(L_D^1) \cap \Gamma^D_0(m)$. As $B,S \in \SL(L_D^1)$ we thus just have to decide whether $BSB^{-1} \in \Gamma^D_0(m)$. In order to do this, let us look at an arbitrary coset representative $B = \left( \begin{smallmatrix} a & b \\ c & d \end{smallmatrix} \right)$ of $\SL(L_D^1) / (\SL(L_D^1) \cap \Gamma^D_0(m))$. We then have
$$BSB^{-1} = \left( \begin{smallmatrix} ca+db & -a^2-b^2 \\ c^2+d^2 & -ca -db \end{smallmatrix} \right).$$
If $B \in \SL(L_D^1)$m then it is enough to decide if $m \in \OD$ divides $c^2 + d^2$. As a number is divisible by $m$ if and only if it is divisible by all its prime divisors we first consider  the case $m=\pi^n$, where $\pi \in \OD$ is a prime element  and $n \in \mathbb{N}$. We now use the well-known isomorphism (see e.g. \cite{Kil08}, Section~2.4)
$$\SL_2(\OD) / \Gamma^D_0(\pi^n) \cong \mathbb{P}^1(\OD/\pi^n\OD)$$
where $\mathbb{P}^1(\cdot)$ \label{glo_P12} denotes projective space. The isomorphism is given by mapping a coset representative $\left( \begin{smallmatrix} a & b \\ c & d \end{smallmatrix} \right)$ to $(c:d) \in \mathbb{P}^1(\OD/\pi^n\OD)$. A more precise description of $\mathbb{P}^1(\OD/\pi^n\OD)$ will be very helpful.





\begin{lem} A system of representatives of $\mathbb{P}^1(\OD/\pi^{n+1}\OD)$ is given by
\begin{eqnarray*}
(1:0) & & \\
(k:1), & k \in \OD/\pi^{n+1}\OD & \\
(a_1+...+a_n\pi^{n-1}:\pi), & a_1 \in \OD/\pi\OD  \backslash \left\{ 0 \right\}, & a_i \in \OD/\pi\OD \ \textrm{for} \ i > 1 \\
... & &  \\
(a_1:\pi^n), &  a_1 \in \OD/\pi\OD  \backslash \left\{ 0 \right\} &
\end{eqnarray*}
\end{lem}

\begin{proof}
The number of elements of the list is equal to $\N(\pi^{n+1}) ( 1 + \frac{1}{\N(\pi)})$. So it just remains to prove that the elements are not equal in $\mathbb{P}^1(\OD/\pi^{n+1}\OD)$. This is easy to check. 
\end{proof}
Most of these elements in $\mathbb{P}^1(\OD/\pi^{n}\OD)$ do never yield elliptic elements of order $2$. 

\begin{lem} For representatives of $\SL(L_D^1) / (\SL(L_D^1)\cap\Gamma^D_0(\pi^n))$ with lower row equivalent to $(a_1+..+a_m\pi^{m-1},\pi^{n-m})$ or $(1,0)$ the prime element $\pi$ never divides $c^2+d^2$. \end{lem}

\begin{proof}
If the lower row of the matrix is equivalent to $(1,0)$ this is clear. 
If the lower row is equivalent to $(a_1+..+a_m\pi^{m-1},\pi^{n-m})$ then $c^2+d^2 = (a_1+a_2\pi+...+a_m\pi^{m-1})^2+(\pi^{n-m})^2$ but $\pi \nmid a_1$. 
\end{proof}

Thus the only coset representatives of interest inside $\SL(L_D^1) / (\SL(L_D^1) \cap \Gamma^D_0(\pi^n))$ are those with lower row $(k,1)$. The relation $\pi^n|k^2+1$ is equivalent to $-1$ being a quadratic residue mod $\pi^n$. It follows from Hensel's lemma (see e.g. \cite{MSP06}, Satz~13.15) that the congruence $-1 \equiv k^2 \mod \pi^n$ is solvable if and only if the congruence $-1 \equiv k^2 \mod \pi$ is solvable and that the number of solutions coincides. So we may restrict to the case $n=1$. Analogously to the first supplement of the quadratic reciprocity law for all prime elements $\pi$ with $(\pi,2)=1$, it is true that $-1$ is a quadratic residue if and only if 
$$(-1)^{\frac{\textrm{N}(\pi)-1}{2}} = 1.$$
As $\OD / \pi \OD$ is an integral domain we know that the polynomial $z^2+1 = 0$ has at most two roots. On the other hand since $(\pi,2)=1$ we always have that if $z_0$ is a root of this polynomial then also $-z_0$ is a root and that $z_0$ and $-z_0$ do not coincide. Hence we have established the following proposition:

\begin{prop} For all prime elements $\pi \in \OD$ with $(\pi,2)=1$ and all $n \in \NN$ there are at most 2 elliptic elements of order $2$ inside $\SL(L_D^1) \cap \Gamma^D_0(\pi^n)$. If in addition $(\pi,w)=1$ holds, there are exactly $2$ elliptic elements of order $2$ if and only if $\textrm{N}(\pi) \equiv 1 \mod 4$ and none if and only if $\textrm{N}(\pi) \equiv 3 \mod 4$. \end{prop}

\begin{cor} For all $m \in \OD$ with $(m,2)=1$ and $(m,w)=1$ there are elliptic elements of order 2 inside $\SL(L_D^1) \cap \Gamma^D_0(m)$ if and only if for all prime divisors $\pi \notin \ZZ$ of $m$ we have $\textrm{N}(\pi) \equiv 1 \mod 4$. In this case the number of elliptic elements of order 2 is equal to $2^{k}$ where k is the number of prime divisors of $m$. If $(m,w)\neq 1$ then the number of elliptic elements of order 2 is bounded by $2^k$. \end{cor}

Now we want to analyze the case $(\pi,2)\neq 1$. As usual we first assume that $(\pi,w)=1$. There are two possibilities: the first one is that $2$ is an inert prime number (which means $D \equiv 5 \mod 8$). Then $\OD/2\OD \cong \ZZ/2\ZZ \times \ZZ/2\ZZ$ and therefore $2|c^2+d^2$ is true only for $c=1,d=1$. This means that there is exactly one elliptic element of order 2 in $\SL(L_D^1) \cap \Gamma^D_0(2)$. The second possibility is that $2$ splits (which means $D \equiv 1 \mod 8$). Let $\pi$ be a prime divisor of $2$. Then $\OD/\pi\OD \cong \ZZ/2\ZZ$. Thus there is exactly 1 elliptic element of order $2$. For higher powers of $\pi$ there is a copy of $\ZZ/4\ZZ$ involved in both cases but $-1 \equiv 3 \equiv k^2 \mod 4$ is not possible. Hence there are no elliptic elements for higher powers of $\pi$. Combining everything yields \ref{thm_number_of_elliptic}.
\paragraph{The number of cusps.} \index{twisted Teichmüller curve!cusps} The number of cusps of some simple twisted Teichmüller curves will be calculated next. Let $m=\pi$ be a prime element with $(\pi,w)=1$. Since we know that a cusp $s$ of $\SL(L_D^1) \cap \Gamma^D_0(\pi)$ lies over a cusp $s'$ of $\SL(L_D^1)$ it makes sense to speak of the relative width of the cusp $s$: the width of a cusp $s$ is its local ramification index over $s'$. Then the sum of the width of all cusps lying over $s'$ is exactly equal to the index $[\SL(L_D^1) : \SL(L_D^1) \cap \Gamma^D_0(\pi)] = \N(\pi) + 1$. Let us first consider the cusps $s$ lying over $\infty$.
\begin{prop} If $\pi \in \OD$ with $(\pi,w)=1$ is a prime element and $\pi \notin \ZZ$, then $\SL(L_D^1) \cap \Gamma^D_0(\pi)$ has exactly two cusps lying over $\infty$. One of them has relative width $1$, the other has relative width $\N(\pi)$. \end{prop}
\begin{proof} Consider the fundamental domain of $\SL(L_D^1) \cap \Gamma^D_0(\pi)$. It consists of $\N(\pi)+1$ copies of the fundamental $\mathcal{F}$ of $\SL(L_D^1)$. Since $T \in \Gamma^D_0(\pi)$ glues the vertical left side and the vertical right side of $\mathcal{F}$, it follows that $\SL(L_D^1) \cap \Gamma^D_0(\pi)$ has a cusp of relative width $1$. Now suppose that the copy of the vertical left side of $\mathcal{F}$ given by the coset representative $A_1$ and the copy of the vertical right side of $\mathcal{F}$ given by the coset representative $A_2$ are glued. Then $A_1TA_2^{-1} \in \Gamma^D_0(\pi)$. If the corresponding cusp has width $k$ then $A_1TA_2^{-1} \in \Gamma^D_0(\pi)$, $A_2TA_3^{-1} \in \Gamma^D_0(\pi)$,..., $A_kTA_1^{-1} \in \Gamma^D_0(\pi)$. Since $\pi$ is prime we have that $A_1$ is equivalent to $\left( \begin{smallmatrix} 1 & 0 \\ a & 1 \end{smallmatrix} \right)$ with $a \in \NN$ and $1 \leq a < \N(\pi)$. Then we must have that $A_1T^kA_1^{-1} = \left( \begin{smallmatrix} * & * \\ -kwa^2 & * \end{smallmatrix} \right) \in \Gamma^D_0(\pi)$. It follows that $\pi|k$ and hence $k=\N(\pi)$, i.e. we have a second cusp of width $\N(\pi)$.
\end{proof}
If $\pi \in \ZZ$ is an inert prime number then we can immediately deduce that there are between $2$ and $\pi+1$ cusps lying over $\infty$ and one of these cusps has relative width $1$.
\begin{cor} If $\pi \in \ZZ$ is an inert number then there are exactly $\pi+1$ cusps lying over $\infty$. One of them has relative width $1$ all the others have relative width $\pi$. \end{cor}
\begin{proof} Using the same notation as above, it is clear that the relative width of a cusp is divisible by $\pi$ if it is greater than $1$. Now assume that there exists a minimal $l \in \NN$ and  coset representatives $A_1,...A_{l\pi}$ such that $A_1 T^{l\pi} A_{l\pi}^{-1} \in \Gamma^D_0(\pi)$. Then it follows that $A_{l\pi}=A_1$ and $l=1$. This implies the claim. \end{proof}
In particular the number of cusps of simple Teichmüller curves is not universally bounded. We can also gain some general information about the cusp $s = 1/2w$ 
\begin{prop} If $\pi\in \OD$ is a prime element, $\pi \notin \ZZ$ with $(\pi,2)=1$ and $(\pi,w)=1$, then there are exactly $2$ cusps of $\SL(L_D^1) \cap \Gamma^D_0(\pi)$ lying over $1/2w$. If $\pi \in \ZZ$ is an inert prime number, then there are exactly $p+1$ cusps lying over $1/2w$. \end{prop}
\begin{proof} We first look at the non-inert case. Let $L$ be the parabolic element stabilizing $1/2w$ which we found in Lemma~\ref{lem_second_parabolic}. Similarly as in the last proof one then has a cusp of width $k$ if there are coset representatives $A_1,...,A_k$ with $A_1 T^{-1} L T A_2^{-1} \in \Gamma^D_0(\pi)$,..., $A_k T^{-1} L T A_1^{-1} \in \Gamma^D_0(\pi)$. So we need to find out for which coset representatives $A$ one has $A T^{-1} L^k T A \in \Gamma^D_0(\pi)$. If $A$ is a coset representative, then so is $AT^{-1}$. So we may assume that $B:=AT^{-1} = \left( \begin{smallmatrix} 1 & 0\\ a & 1 \end{smallmatrix} \right).$ Then
$$(BL^kB^{-1})_{2,1} = lk\left(\frac{a^2}{4}\left(\frac{D+7}{4}+\frac{D-1}{2}\right) + a\left(2w+\frac{D-1}{4}\right)+w+1\right).$$ 
where $l$ is a power of $2$ depending only on the discriminant.  There exist at most two $a$ such that the lower left entry is equivalent to $0 \mod \pi$. If $\pi \in \OD$ is a prime element and $\pi \notin \ZZ$, there thus exists a cusp of relative width $\N(\pi)$. Therefore, there are exactly two cusps in this case.\\
If $\pi$ is an inert prime number, then one sees by the same arguments as in the last proof that there exist $\pi$ cusps of width $\pi$ and $1$ cusp of relative width $1$. \end{proof}
If we want to gain information on the cusps $s \neq \infty, 1/2w$, we have to restrict to the cases $D=13$ and $D=17$ since we do not have a general formula for parabolic elements inside $\SL(L_D^1)$. There is one more cusp in each case, namely $1$. The geometry of the fundamental domains of these Veech groups is so similar that one more or less automatically treats $D=13$ and $D=17$ at the same time. Almost the same calculation as before gives the following result:
\begin{prop} If $\pi \in \OD$ is a prime element, $\pi \notin \ZZ$, with $(\pi,2)=1$ and $(\pi,w)=1$, then there are exactly $2$ cusps in $\SL(L_D^1) \cap \Gamma^D_0(\pi)$ lying over $1$. If $\pi \in \ZZ$ is an inert prime number, then there are exactly $\pi+1$ cusps lying over $1$. \end{prop}
Let us collect everything in one theorem:
\begin{thm} If $D=5$ and if $\pi \in \OD$ is a prime element with $(\pi,w)=1$ and $(\pi,2)=1$ then
\begin{itemize} 
\item[(i)] if $\pi \in \ZZ$, then $\HH / (\SL(L_D^1) \cap \Gamma^D_0(\pi))$ has exactly $2\varphi_1(\pi)$ cusps.
\item[(ii)] If $\pi \notin \ZZ$, then $\HH / (\SL(L_D^1) \cap \Gamma^D_0(\pi))$ has exactly $4$ cusps.
\end{itemize} 
If $D \in \left\{ 13,17 \right\}$ and if $\pi \in \OD$ is a prime element with $(\pi,w)=1$ and $(\pi,2)=1$ then 
\begin{itemize} 
\item[(i)] if $\pi \in \ZZ$, then $\HH / (\SL(L_D^1) \cap \Gamma^D_0(\pi))$ has exactly $\varphi_1(2\pi)$ cusps.
\item[(ii)] If $\pi \notin \ZZ$, then $\HH / (\SL(L_D^1) \cap \Gamma^D_0(\pi))$ has exactly $6$ cusps.
\end{itemize} 
\end{thm}

Also this formula generalizes the well-known formula for the number of cusps of the modular curves $\HH / \Gamma_0(m)$ for congruence subgroups $\Gamma_0(m) \subset \SL_2(\ZZ)$ (compare again \cite{Miy89}, Theorem~4.2.7). Summing up, this indicates that simple twisted Teichmüller curves behave very much like modular curves from an arithmetic point of view.
\paragraph{Genus.} \index{twisted Teichmüller curve!genus} The Riemann-Hurwitz formula enables us now to calculate the genus of simple twisted Teichmüller curves for prime elements $\pi \in \OD$. Let $g_\pi$ denote the genus of $\HH/(\SL(L_D^1) \cap \Gamma^D_0(\pi))$.
\begin{cor} 
If $D = 5$ and if $\pi \in \OD$ is a prime element with $(\pi,2)=1$ then 
\begin{itemize} 
\item[(i)] if $\pi \in \ZZ$, then $g_\pi = \frac{3(\N(\pi)+1)+10-20\sigma_1(\pi)}{20}$.
\item[(ii)] If $\pi \notin \ZZ$ and $\textrm{N}(\pi) \equiv 1 \mod 4$ then $g_\pi = \frac{3(\N(\pi)+1)-20}{10}$.
\item[(iii)] If $\pi \notin \ZZ$ and $\textrm{N}(\pi) \equiv 3 \mod 4$ then $g_\pi = \frac{3(\N(\pi)+1)-30}{10}$.
\end{itemize} 
If $D \in \left\{ 13,17 \right\}$ and if $\pi \in \OD$ is a prime element with $(\pi,w)=1$ and $(\pi,2)=1$ then
\begin{itemize} 
\item[(i)] if $\pi \in \ZZ$, then $g_\pi = \frac{3(\N(\pi)+1)+2-2\sigma_1(2\pi)}{4}$.
\item[(ii)] If $\pi \notin \ZZ$ and $\textrm{N}(\pi) \equiv 1 \mod 4$ then $g_\pi = \frac{3(\N(\pi)+1)-8}{4}$.
\item[(iii)] If $\pi \notin \ZZ$ and $\textrm{N}(\pi) \equiv 3 \mod 4$ then $g_\pi = \frac{3(\N(\pi)+1)-10}{4}$.
\end{itemize} 
\end{cor}

\newpage

\section{Prym Varieties and Teichmüller Curves} \label{chapter_prym_modular}

We have just seen that there are many different twisted Teichmüller curves (Theorem~\ref{thm_classificiation1}). So it is now natural to ask, whether these curves yields all Kobayashi curves. This question will be answered negatively later. The natural candidates for Kobayashi curves which are not a twisted Teichmüller curve stem from certain Teichmüller curves in $\mathcal{M}_3$ and $\mathcal{M}_4$. The main purpose of this chapter is to repeat how these primitive Teichmüller curves of low genus~$2 \leq g \leq 4$ are constructed in \cite{McM06a}. These curves lie on Hilbert modular surfaces. The embedding is given by assigning to each point lying on these Teichmüller curves its Prym variety.\\[11pt] We have already introduced Abelian varieties in Section~\ref{section_abelian_varieties}. 
A special class of polarized Abelian varieties are the Prym varieties in Section~\ref{section_prym}. In Section~\ref{section_prym_teich} we recall C. McMullen's construction of an infinite collection of primitive Teichmüller curves. We can associate to all the points lying on these Teichmüller curves their Prym variety. All these Prym varieties have real multiplication (see Definition~\ref{def_real_multiplication}). In view of the fact that Hilbert modular surfaces parametrize polarized Abelian surfaces with real multiplication (compare Theorem~\ref{thm_Hilbert_modular_moduli}), it follows that all the Teichmüller curves described in this chapter lie on a Hilbert modular surface (Theorem~\ref{thm_modular_embedding}).

\subsection{Prym Varieties} \label{section_prym}

Prym varieties form a special class of examples of polarized Abelian varieties. Although they are from a certain viewpoint more general than Jacobians, they are still accessible geometrically (see \cite{Bea89} for details). Classically they were introduced as a certain subvariety of the Jacobian of a Riemann surface $X'$ which double covers another Riemann surface $X$ with at most $2$ branch points. 
This classical definition is not sufficient for our purposes. That is why we use the notion of Prym varieties from \cite{McM06a}.  The definition just involves double covers without any restriction on the number of branched points.


\begin{defi} Let $X'$ be a compact Riemann surface with an automorphism $\rho: X' \to X'$ of order two, i.e. an involution. Recall that $\Jac(X')=H^0(X',\Omega(X'))^{\vee}/H_1(X',\ZZ)$. The action of $\rho$ \label{glo_rho} determines a splitting of the Abelian differentials into even and odd forms $\Omega(X') = \Omega(X')^+ \oplus \Omega(X')^- $, \label{glo_Om+-} the $+1$ and the $-1$ eigenspace of the involution $\rho$ on $\Omega(X')$,  as well as sublattices $H_1(X',\ZZ)^{\pm}\subset H_1(X',\ZZ)$. Then we define the \textbf{Prym variety} to be \index{Prym!variety}
$$P=\Prym(X',\rho)= H^0(X',\Omega(X')^-)^{\vee}/H_1(X',\ZZ)^{-}. \label{glo_Prym}$$ 
The automorphism $\rho$ is called the \textbf{Prym involution}.  Finally we refer to elements in $\Omega(X')^-$ as \textbf{Prym forms} on $(X',\rho)$. \end{defi}

\begin{rem} There is a very general notion of Prym varieties (see e.g. \cite{BL04}, Chapter~12), which we do not need in its full generality here. \end{rem}

Note that the covering projection $\pi: X' \to X'/\rho$ may have $2n$ branched points. The Prym variety $P$ is canonically polarized by the intersection pairing on $H_1(X',\ZZ)^-$. By the Riemann-Hurwitz formula we have that if $\dim(\Prym(X',\rho))=h$ then $h \leq g(X) \leq 2h+1$ (see \cite{McM06a}, Theorem~3.1).

\subsection{Prym Varieties and Teichmüller Curves} \label{section_prym_teich}
Let us continue by collecting a few important results relating Teichmüller curves and Prym varieties. In this section we will leave out those technicalities that are not of great significance for us. The reader who wants to know more about the details, in particular \textbf{Prym systems} and the \textbf{Thurston-Veech construction}, is referred to \cite{McM06a}, Chapter~4 and 5, as well as to \cite{Möl11a}, Chapter~2.2.\\[11pt]
For $2 \leq g \leq 4$ Prym varieties tell us where to search for an infinite collection of Teichmüller curves: let $P=\Prym(X,\rho)$ be a Prym variety with real multiplication by $\OD$ (especially $\dim_\mathbb{C}P=2$). Then $\OD$ also acts on $\Omega(P) \cong \Omega(X)^-$. 
\begin{defi} We call $\omega \in \Omega(X)^-$ a \textbf{Prym eigenform} \index{Prym!eigenform} if $0 \neq \OD \omega \subset \mathbb{C}\omega$. By $\Omega E_D^g \subset \Omega \Mg$ we denote the space of all pairs $(X,\omega)$ such that there exists a Prym involution $\rho$ such that $P=\Prym(X,\rho)$ admits real multiplication by $\OD$ and $\omega$ is a Prym eigenform of $P$. The space is called the \textbf{space of all Prym eigenforms for real multiplication by $\OD$}.
\end{defi}
\begin{rem} Neither $\rho$ nor the action of $\OD$ is uniquely determined by $\omega$. \end{rem}
We define the \textbf{Weierstrass locus} \index{Weierstrass locus} \label{glo_Wei} by $$\Omega W_D^g:=\Omega E_D^g \cap \Omega \Mg(2g-2).$$ 
The projection of $\Omega W_D^g$ to moduli space is called the \textbf{Weierstrass curve} $W_D^g$. 
\begin{thm} (\cite{McM06a}, Theorem~3.2) The Weierstrass curve $\Omega W_D^g$ is a $\SL_2(\mathbb{R})$-invariant subset of $\Omega \Mg$. \end{thm}
Another lemma reveals why this construction only works in the case $2 \leq g \leq 4$.
\begin{lem} The Weierstrass curve $W_D^g$ is nonempty only if $2 \leq g \leq 4$. \end{lem}
\begin{proof} By the definition of real multiplication $\dim_{\mathbb{C}}P=2$. Therefore, $\Omega E_D \neq \emptyset$ only if $2 \leq g \leq 5$. If $g=5$ then $X \to X/\rho$ is an unramified double cover. Hence $\omega$ has an even number of zeroes and so $\omega \notin \Omega \Mg(2g-2)$. \end{proof}
$W_D^g$ is itself the main link between Teichmüller curves and Prym varieties.
\begin{thm} (\textbf{McMullen}, \cite{McM06a}, Theorem~3.4) \label{thm_weierstrass_is_union} The Weierstrass curve $W_D^g \subset \Mg$ is a finite union of Teichmüller curves. Each such curve is primitive, provided $D$ is not a square. \end{thm}
In fact, there are infinitely many \textit{primitive} Teichmüller curves \index{Teichmüller curve!(geometrically) primitive} contained in the union of all the $W_D^g$ for fixed $2 \leq g \leq 4$. Without going into technical details (Thurston-Veech construction, Prym systems), we want to sketch how such an infinite collection of primitive Teichmüller curves can be found for genus~$2 \leq g \leq 4$. Consider the L, S, and X-shaped polygons in Figure~7.1 (corresponding to the case of genus~$2,3$ respectively $4$):

\begin{center}
\psset{xunit=1cm,yunit=1cm,runit=1cm}
\begin{pspicture}(-0.5,-0.5)(11,3)


\psline[linewidth=0.5pt,showpoints=true]{-}(0,3)(0,0.75)
\psline[linewidth=0.5pt,showpoints=true]{-}(0,0.75)(0,0)
\psline[linewidth=0.5pt,showpoints=true]{-}(0,0)(0.75,0)
\psline[linewidth=0.5pt,showpoints=true]{-}(0.75,0)(3,0)
\psline[linewidth=0.5pt,showpoints=true]{-}(3,0)(3,0.75)
\psline[linewidth=0.5pt,showpoints=true]{-}(3,0.75)(0.75,0.75)
\psline[linewidth=0.5pt,showpoints=true]{-}(0.75,0.75)(0.75,3)
\psline[linewidth=0.5pt,showpoints=true]{-}(0.75,3)(0,3)


\psline[linewidth=0.5pt,showpoints=true]{-}(4,0)(5.25,0)
\psline[linewidth=0.5pt,showpoints=true]{-}(5.25,0)(5.75,0)
\psline[linewidth=0.5pt,showpoints=true]{-}(5.75,0)(5.75,0.75)
\psline[linewidth=0.5pt,showpoints=true]{-}(5.75,0.75)(5.75,2.25)
\psline[linewidth=0.5pt,showpoints=true]{-}(5.75,2.25)(7,2.25)
\psline[linewidth=0.5pt,showpoints=true]{-}(7,2.25)(7,3)
\psline[linewidth=0.5pt,showpoints=true]{-}(7,3)(5.75,3)
\psline[linewidth=0.5pt,showpoints=true]{-}(5.75,3)(5.25,3)
\psline[linewidth=0.5pt,showpoints=true]{-}(5.25,3)(5.25,2.25)
\psline[linewidth=0.5pt,showpoints=true]{-}(5.25,2.25)(5.25,0.75)
\psline[linewidth=0.5pt,showpoints=true]{-}(5.25,0.75)(4,0.75)
\psline[linewidth=0.5pt,showpoints=true]{-}(4,0.75)(4,0)


\psline[linewidth=0.5pt,showpoints=true]{-}(8,1.5)(8,2)
\psline[linewidth=0.5pt,showpoints=true]{-}(8,2)(9,2)
\psline[linewidth=0.5pt,showpoints=true]{-}(8,2)(9.5,2)
\psline[linewidth=0.5pt,showpoints=true]{-}(9.5,2)(9.5,3)
\psline[linewidth=0.5pt,showpoints=true]{-}(9.5,3)(10,3)
\psline[linewidth=0.5pt,showpoints=true]{-}(10,3)(10,2)
\psline[linewidth=0.5pt,showpoints=true]{-}(10,2)(10,1.5)
\psline[linewidth=0.5pt,showpoints=true]{-}(10,1.5)(11,1.5)
\psline[linewidth=0.5pt,showpoints=true]{-}(11,1.5)(11,1)
\psline[linewidth=0.5pt,showpoints=true]{-}(11,1)(10,1)
\psline[linewidth=0.5pt,showpoints=true]{-}(10,1)(9.5,1)
\psline[linewidth=0.5pt,showpoints=true]{-}(9.5,1)(9.5,0)
\psline[linewidth=0.5pt,showpoints=true]{-}(9.5,0)(9,0)
\psline[linewidth=0.5pt,showpoints=true]{-}(9,0)(9,1)
\psline[linewidth=0.5pt,showpoints=true]{-}(9,1)(9,1.5)
\psline[linewidth=0.5pt,showpoints=true]{-}(9,1.5)(8,1.5)

\uput[l](0.05,2.0){\scriptsize{1}}
\uput[l](0.05,0.35){\scriptsize{2}}
\uput[u](0.35,-0.05){\scriptsize{3}}
\uput[u](1.85,-0.05){\scriptsize{4}}

\uput[l](5.3,2.65){\scriptsize{1}}
\uput[l](5.3,1.45){\scriptsize{2}}
\uput[l](4.05,0.35){\scriptsize{3}}
\uput[u](4.6,-0.05){\scriptsize{4}}
\uput[u](5.5,-0.05){\scriptsize{5}}
\uput[u](6.5,2.2){\scriptsize{6}}

\uput[l](9.55,2.5){\scriptsize{1}}
\uput[l](8.05,1.75){\scriptsize{2}}
\uput[u](8.5,1.45){\scriptsize{5}}
\uput[l](9.05,1.25){\scriptsize{3}}
\uput[l](9.05,0.5){\scriptsize{4}}
\uput[u](9.25,-0.05){\scriptsize{6}}
\uput[u](9.75,0.95){\scriptsize{7}}
\uput[u](10.5,0.95){\scriptsize{8}}

\end{pspicture}
\\ Figure 7.1. Generators for Teichmüller curves in genus 2,3 and 4 (from \cite{McM06a}).
\end{center}

For $n \in \mathbb{N}$ we write $L_n$ for the L shaped polygon with side lengths $(\frac{1+\sqrt{1+4n}}{2},1,1,\frac{1+\sqrt{1+4n}}{2})$ and $S_n$ for the S shaped polygon with side lengths $(1+\sqrt{1+2n},2n,1+\sqrt{1+2n},n,1+\sqrt{1+2n},n)$ and $X_n$ for the X shaped polygon with side lengths $(n,1+\sqrt{1+n},1+\sqrt{1+n},n,n,1+\sqrt{1+n},1+\sqrt{1+n},n)$. \\[11pt] For each such polygon $Q$ we get a flat surface $(X,\omega) = (Q,dz)/\sim$ when we identify opposite sides of $Q$. The vertices of $Q$ are identified to a single point $p$, namely the unique zero of $\omega$. For each genus there is a unique involution $\rho \in \Aut(X)$ fixing $p$. For $S_n$ and $X_n$ the involution $\rho$ is realized by a $180^\circ$ rotation about the center of the polygon. For $L_n$ the involution $\rho$ is the hyperelliptic involution.  
\begin{thm} \label{thm_polygons_determine_teich} (\textbf{McMullen}, \cite{McM06a}, Theorem~5.4) Each of these polygons $L_n, S_n$ and $X_n$ together with the corresponding involutions determines an element $(X,\omega) \in \Omega W_D^g$ with $D=n$ for $L_n$, $D=1+2n$ for $S_n$ and $D=1+n$ for $X_n$. The projection of the $\SL_2(\RR)$-orbit of $(X,\omega)$ to moduli space $\Mg$ determines a Teichmüller curve. \index{Teichmüller curve!Prym} \end{thm}
The \index{Fuchsian group!trace field} trace fields of the corresponding Veech groups are $\mathbb{Q}(\sqrt{1+4n})$ for $L_n$ and $\mathbb{Q}(\sqrt{1+2n})$ for $S_n$ and $\mathbb{Q}(\sqrt{1+n})$ for $X_n$ (\cite{McM06a}, Theorem 5.4). Recall that Theorem~\ref{thm_classification_Teichmüller_genus_2} completely classifies Teichmüller curves in genus $2$. A similar classification is also known to hold in $\Omega E_4^4$ by the recent work of E. Lanneau and D.-M. Nguyen (see \cite{LN11}, Theorem~1.1). A classification in the genus~$3$ case is still open. Furthermore, it is not known whether there exist other primitive Teichmüller curves in genus~$3$ and $4$ than those described in this section.

\subsection{Hilbert Modular Embeddings} \label{section_modular_embeddings}

Having established all the necessary background, we are now able to explicitly describe an embedding of Teichmüller curves into Hilbert modular surfaces. In particular, it will turn out that for genus~$2$ and $4$ Teichmüller curves of discriminant $D$ can always be embedded into the same Hilbert modular surface $X_D$. We again mainly repeat results from \cite{McM03} and \cite{McM06a}. \\[11pt]
First let us explain how primitive Teichmüller curves of genus~$2 \leq g \leq 4$ can be embedded into the moduli space of Abelian varieties. The last two propositions of Section~\ref{section_abelian_varieties} enable us to calculate the types of the polarizations of the Prym varieties \index{Prym!variety} associated to the surfaces described in Section~\ref{section_prym_teich}: consider the polygons $L_n$, $S_n$ and $X_n$ with associated involutions $\rho_L$, $\rho_S$ and $\rho_X$. Each of these yields a flat surface $(X,\omega)$. For each of these flat surfaces we now consider the double cover $\pi: X \to Y:=X/\rho$. The corresponding Prym varieties $\Prym(X,\rho)$ are always $2$-dimensional, because $(X,\omega) \in \Omega W_D^g$ (Theorem~\ref{thm_polygons_determine_teich}). Since $\Jac(X) \sim \Jac(Y) \oplus \Prym(X,\rho)$ we get $g(Y)=g(X)-2$. 
\begin{lem} \label{lem_polygon_polarization} For the $L$-shaped polygon $\Prym(X,\rho)$ has polarization of type $(1,1)$, for the $S$-shaped polygon $\Prym(X,\rho)$ has polarization of type $(1,2)$, for the $X$-shaped polygon $\Prym(X,\rho)$ has polarization of type $(2,2)$. \end{lem}
\begin{proof} We just prove the case of $S$-shaped polygons since the other two cases work in the same way. The surface $X$ has genus~$3$ in this case and thus $Y$ has genus~$1$. Using Proposition~\ref{prop_polar_cover} we can observe that $\Jac(Y)$ has polarization of type $(2)$. Then Proposition~\ref{prop_polar_complementary_abelian} implies that the polarization of $\Prym(X,\rho)$ is of type $(1,2)$. \end{proof}
\begin{cor} By mapping each point of the Teichmüller curve $C$ generated by $L_n$ (respectively $S_n$, $X_n$) to its Prym variety one gets an embedding $V \hookrightarrow \mathcal{A}_2^D$ with $D$ as in the lemma. \end{cor}
However by this construction we lost a little information, namely the choice of real multiplication. This is why we get an embedding into a space with even more structure: by Theorem~\ref{thm_Hilbert_modular_moduli} the Teichmüller curves in $\mathcal{M}_2$ lie on the Hilbert modular surfaces $X_D$. Recalling that $\mathcal{A}_2^{(2,2)} \cong \mathcal{A}_2$, also the described Teichmüller curves in $\mathcal{M}_4$ lie on $X_D$ and we get eventually the desired theorem. 

\begin{thm} (\textbf{McMullen}, \cite{McM03}, \cite{McM06a}) \label{thm_modular_embedding}Let $C \subset \mathcal{W}_D^g$ be one of the described Teichmüller curves of discriminant $D$ in $\Mg$ with $g \in \left\{2,4 \right\}$. Then there exists an embedding of $C$ into the Hilbert modular surface $X_D$ \index{Hilbert modular surface} such that the following diagram commutes:

 $$
\begin{xy}
 \xymatrix{
 	\mathbb{H} \ar@{^(->}[rr]^{\Phi(z)=(z,\varphi(z))} \ar[d]^{/\SL(X,\omega)}& &	\HH \times \HH^- \ar[d]^{/\SL_2(\OD)}	\ar[d] \\ 
 	C \ar@{^(->}[rr] \ar[d]_f & & X_D \ar[d]^j  \\
 	\mathcal{M}_g \ar@{^(->}[rr] & & \mathcal{A}_2
 	}
\end{xy}
$$
Moreover $\varphi$ is holomorphic but not a Möbius transformation.

\end{thm}

Such an embedding is called a \textbf{Hilbert modular embedding}. The case of Teichmüller curves of genus~$3$ is a little more difficult to describe than the genus~$2$ and the genus~$4$ case since the polarization is of type $(1,2)$: although it can be proven exactly in the same way that a similar embedding of Teichmüller curves $C \subset \mathcal{W}_D^3$ into a certain Hilbert modular surface $\mathbb{H} \times \mathbb{H}^- / \SL(\OD \oplus \mathfrak{a})$ exists, it is a quite hard problem to calculate $\mathfrak{a}$ and thus the surface parameterizing the space of all pairs $(X,\rho)$, with $X$ an Abelian surface of type $(1,2)$ and $\rho$ a choice of real multiplication by $\OD$, explicitly (see e.g. \cite{IO97}). Fortunately, this problem does not play an important role for these notes. However, it happens for some $D$ that $\mathfrak{a} = \mathfrak{b}^2$ for some fractional ideal $\mathfrak{b}$. In this case the involved Hilbert modular surface is isomorphic to $X_D$ and thus also the genus $3$ Teichmüller curve lies on $X_D$ (compare \cite{Zag81}, \cite{vdG88}).\\[11pt] Generally, Abelian varieties of a given type with real multiplication are in a natural connection with Hilbert modular varieties (see \cite{McM03}, Chapter~6; \cite{BL04}, Chapter~9). The mentioned result of C. McMullen has been generalized by M. Möller in \cite{Möl06}. There it is shown that every primitive Teichmüller curve of genus~$g\geq 2$ has real multiplication and therefore can be embedded into a Hilbert modular variety.\\[11pt]In \cite{Möl11c}, M. Möller calculates the Euler characteristic of the genus $3$ and $4$ Teichmüller curves. The Euler characteristic is as in M. Bainbridge's theorem a rational multiple of the Euler characteristic of the Hilbert modular surface.

\newpage

\section{Lyapunov Exponents} \label{chapter_Lyapunov_exponents}
\index{Lyapunov exponents|(}
Before we are able to prove that the genus $3$ and the genus $4$ Teichmüller curves are never twists of the genus $2$ Teichmüller curve, we have to introduce another concept, namely Lyapunov exponents. These give to a certain extent information about the long time average behavior of a dynamical system (see e.g. \cite{AB08}, \cite{Möl11a} or \cite{Zor06}). The existence of Lyapunov exponents follows from Oseledets's Multiplicative Ergodic Theorem (Theorem~\ref{thm_Oseledet}). However, Lyapunov exponents are usually very hard to evaluate. Often it is just possible to do numerical approximations. Only on rare occasions Lyapunov exponents are known explicitly. In Section~\ref{sec_Lyap_basic} we introduce the basic theory. We will be especially interested in the Lyapunov exponents of a certain cocycle, namely the Kontsevich Zorich cocycle (Section~\ref{sec_Kontsevich_Zorich}), over the Teichmüller flow. We define the latter in Section~\ref{sec_Teichmüller_flow}. An important tool for our calculations is to relate the Teichmüller flow on a Teichmüller curve (Section~\ref{Teich_flow_Teich_curves}) to the geodesic flow on the tangent space $T^1\mathbb{H}/\SL(X,\omega)$ where $\SL(X,\omega)$ is the Veech group of the flat surface generating the Teichmüller curve (and hence in particular a Fuchsian group). This is why we present the relevant material about the geodesic flow in Section~\ref{sec_geodesic_flow} before we come to the Teichmüller flow.

\subsection{Basic Theory} \label{sec_Lyap_basic}

For the convenience of the reader we recall the definition of Lyapunov exponents for 
continuous time. A good introduction to the discrete case can be found in \cite{AB08}. For the rest of this section let $(X,\mu)$ be a \textbf{probability space}, \index{probability space} i.e. $X$ is a measurable space and $\mu$ a measure on $X$ with $\mu(X)=1$ \label{glo_Xmu}and denote by $L^1(X,\mu)$ the vector-space of integrable, measurable functions $f: X \to \RR$ with respect to $\mu$.

 \begin{defi} Let $\varphi_t: X \to X$ be a measurable, measure-preserving flow on $(X,\mu)$. Let $p: E \to X$ be a finite-dimensional vector bundle over $X$ endowed with a norm. A \textbf{(linear) cocycle over $\varphi_t$} is a measurable flow extension $F_t: E \to E$ such that $p \circ F_t = \varphi_t \circ p$ and such that the action $A^t(x) : E_x \to E_{\varphi_t(x)}$ of $F_t$ on every fiber is a linear isomorphism.\footnote{See \cite{Arn98}, p.6, for a nice figure which describes the situation.} \end{defi}


 Sometimes one also  calls $A$ a \textbf{cocycle over $\varphi_t$}. Note that $A^t(x)$       satisfies the following conditions: 
   \begin{itemize}
    \item[(i)] $A: \mathbb{R} \times E \to E$ is measurable.
    \item[(ii)] $A^0(x)=\Id$ for all $x \in M$
    \item[(iii)] $A^{s+t}(x) = A^t(\varphi_s(x))A^s(x)$ for all $s,t \in \mathbb{R}$.
   \end{itemize}
 By replacing $\mathbb{R}$ by $\ZZ$ in the upper definition, we get  with a grain of salt the definition of a discrete linear cocycle. Albeit it is possible to define Lyapunov exponents in a very general setting (see e.g.   \cite{KH97}, Chapter~S), we restrict to the case when the cocycle $A$ over $T$ fulfills a   certain integrability condition. Then one is able to define Lyapunov exponents with the help of Oseledets's Multiplicative Ergodic Theorem. 

  \begin{thm} \textbf{(Oseledets's Multiplicative Ergodic Theorem)} \label{thm_Oseledet} Let $F_t$ be   a cocycle over a measurable, measure-preserving flow $\varphi_t$ on $(X,\mu)$ and let $A$ be as described above. Assume that the functions
  $$g(x):= \sup_{0 \leq t \leq 1} \log^+ ||A^t(x)|| \quad \textit{and}$$
  $$h(x):= \sup_{0 \leq t \leq 1} \log^+ ||A^{1-t}(\varphi_t(x))||$$
  are in $L^1(X,\mu)$. Then there exists a $\varphi$-invariant set $X' \subset X$ with             $\mu(X')=1$ and the following holds for all $x \in X'$:
   \begin{itemize}
    \item[(i)] The limit $\Lambda_x := \lim_{t \to \infty} ((A^t(x))(A^t(x))^*)^{1/2t}$ exists (* is the      adjoint operator).
    \item[(ii)] Let $\exp \lambda_1(x) < ... < \exp \lambda_{k}(x)$, where $k=k(x)$, be the         eigenvalues of $\Lambda_x$. Then all $\lambda_i(x)$ are real and $\lambda_1(x)$ can be          $-\infty$. Let $U_1(x),...,U_k(x)$ be the corresponding eigenspaces and $l_i(x)= \dim U_i(x)$.     The functions $x \mapsto \lambda_i(x)$ and $x \mapsto l_i(x)$ are $\varphi_t$-invariant.        Let $E_0(x)=\left\{ 0 \right\}$ and $E_r(x) = U_1(x) \oplus ... \oplus U_r(x)$ for              $r=1,...,k$. Then for $v \in E_r(x) \smallsetminus E_{r-1}(x)$, $1 \leq r \leq s$, we have
    $$\lim_{t \to \infty} \frac{1}{t} \log ||A^t(x)v|| = \lambda_r(x).$$
   \end{itemize}  
  \end{thm}

  \begin{proof} See \cite{Arn98}, Chapter~3.4. \end{proof}
  
  A similar result for products of random matrices was proven earlier by H. Furstenberg       in \cite{Fur63}.
  
  \begin{cor} The subspaces $(E_r(x))_{r=0}^k$ form a filtration of $E$ and
  $$E_r(x) = \left\{v \in \mathbb{R}^n \mid \lim_{t \to \infty} \frac{1}{t} \log ||A^t_xv|| \leq     \lambda_r(x) \right\}.$$ 
  \end{cor}


 
 \begin{defi} The numbers $\lambda_i(x)$ appearing in the theorem are called the                 \textbf{Lyapunov exponents at $x$}  with respect to the cocycle $A$ (or $F$). The collection     of all $\lambda_i(x)$ is called the \textbf{Lyapunov spectrum} at $x$. Finally $l_i(x) = \dim E_{\lambda_i}(x) - \dim E_{\lambda_{i-1}}(x)$ is the \textbf{multiplicity of the Lyapunov exponent $\lambda_i(x)$}. \end{defi}
 
 As convention we will, if $l_i(x)=m>1$, consider the Lyapunov exponent $\lambda_i$ as $m$        distinct Lyapunov exponents, such that the total number of Lyapunov exponent is always equal to the dimension of the vector bundle $E$ over $X$.\\     Recall that a measure preserving map $\varphi_t: X \to X$ is called \textbf{ergodic} if all measurable  subsets $Y \subset X$ with $\varphi_t^{-1}(Y)=Y$ have measure $0$ or $1$. Note that if $\varphi_t$ is ergodic,  then the $\lambda_i(x)$ are constant almost everywhere. In this case we are able to speak (globally) of  the Lyapunov  exponents $\lambda_i$ of the cocycle $F$ (or $A$). All maps which we will look at  are ergodic. \label{glo_lambdanox}\\[11pt]
  We will be mostly interested in \textbf{symplectic cocycles}. This means that there exists some symplectic form $w_x$ on each fiber $E_x$ which is preserved by the linear cocycle $F$, i.e.
 $$w_{\varphi_t(x)}(A^t(x)u,A^t(x)v)=w_x(u,v) \quad \textnormal{for all } x \in X \textnormal{ and } t \in \RR \textnormal{ and } u,v \in E_x.$$

 
 \index{Lyapunov exponents|)}

\subsection{The Geodesic Flow} \label{sec_geodesic_flow}

 In this section we briefly recall the definition of the geodesic flow \index{geodesic flow} on                  $T^1\mathbb{H}/\Gamma$ where $\Gamma$ is a Fuchsian group. The reader who wants to find out more about the details is referred to \cite{Ein06}, \cite{Kat96}, \cite{Kat08} and \cite{KU07}.\\[11pt]
 An element $\zeta \in T_z\mathbb{H}$ of \label{glo_tangent_loc} the tangent plane of the hyperbolic plane at $z$ is an   element of $\mathbb{R}^2$. For $z=x+iy \in \mathbb{H}$ an inner product on $T_z\mathbb{H}$ can  be defined by $(\zeta,w)_z=\frac{1}{4y^2}(\zeta \cdot w)$. This yields a norm $||\cdot||_z$   on $T_z\mathbb{H}$. Consider the unit tangent bundle $T^1\mathbb{H}$ \label{glo_tangent} of the hyperbolic plane.   It is defined as the collection of all vectors $\zeta \in T_z\mathbb{H}$ of length one, i.e.    $||\zeta||_z=1$, for all $z \in \mathbb{H}$. An element $g = \left( \begin{smallmatrix} a & b \\ c & d     \end{smallmatrix} \right) \in \SL_2(\mathbb{R})$ acts on $T^1\mathbb{H}$ by $(z,\zeta) \in T^1\mathbb{H}    \mapsto (\frac{az+b}{cz+d},Dg(z)(\zeta))$, where $Dg(z)$ is the derivative of $g$ at $z$.       
 This map is well defined. The action is transitive        (see \cite{Ein06}, p. 3-4). $T^1\mathbb{H}$ can be identified with $\PSL_2(\mathbb{R})$ by        sending $v=(z,\zeta)$ to the unique $g \in \PSL_2(\mathbb{R})$ such that $z=g(i)$ and           $\zeta=Dg(z)(\iota)$, where $\iota$ is the unit vector at the point $i$ pointing upwards (see   \cite{Kat92}, Theorem 2.1.1).\\[11pt]
 There is of course a bunch of geodesics going through a point $z \in \mathbb{H}$. For a fixed  direction $\zeta$ there is however a unique geodesic going through $z$ in the direction  $\zeta$. For a unit vector $\zeta$ based at $z \in \mathbb{H}$ the \textbf{geodesic flow}    \label{geodesic}   can therefore be defined as follows: equip (for technical reasons which       become visible   later) $\mathbb{H}$ with the Poincaré metric with constant scalar curvature    $-4$. The geodesic flow is the flow with unit speed along the geodesic which goes         through $(z,\zeta)$ at time $t=0$. For $(i,\iota) \in T^1\mathbb{H}$ the geodesic is just the   imaginary axis and applying \label{glo_geo_flow} $a_t = \left( \begin{smallmatrix} e^{t} & 0 \\ 0 & e^{-t} \end{smallmatrix} \right)$   moves the vector along the geodesic. 
 Applying elements $g \in \PSL_2(\mathbb{R})$ gives a unit speed parametrization of any other geodesic line in $\mathbb{H}$. Since we apply the isometry   corresponding to $g$ after applying the parametrization $a_t$, the geodesic flow corresponds    to right multiplication by $a_t$ on $\PSL_2(\mathbb{R})$. The orbit $ga_t$ projects to a        geodesic through $g(i)$.\\[11pt]
 Neither is this flow on $T^1\mathbb{H}$ especially interesting from a dynamical point of view   nor will it be the right object for us to look at. Instead we analyze geodesic flows on         $T^1\mathbb{H}/\Gamma$, where $\Gamma$ is a Fuchsian group. The geodesic flow on $T^1\mathbb{H}$ descends to the       \textbf{geodesic flow on $T^1\HH/\Gamma$} via the quotient map $\pi: T^1\mathbb{H} \to              T^1\HH/\Gamma$ of the unit tangent bundles.





 
\subsection{The Teichmüller Flow} \label{sec_Teichmüller_flow}
\subsubsection{The Teichmüller Flow on $\Omega\Mg$} \label{sec_Teichmüller_flow_on_Mg}
The Teichmüller (geodesic) flow is a Hamiltonian flow on $\Omega \Mg$ defined as the geodesic flow with respect to the natural metric, namely the Teichmüller metric. For details we refer the reader to \cite{Zor06} and \cite{For06}. Some aspects of this topic are very nicely presented in \cite{HS06}.\\[11pt]
As has been described in Section~\ref{sec_strata} the action of $\SL_2(\RR)$ on $\Omega\Mg$ preserves all topological characteristics of the flat surface (like genus, number and type of conical singularities) and therefore the action of $\SL_2(\RR)$ preserves each stratum. 
\begin{defi} The \textbf{Teichmüller flow} \index{Teichmüller flow} $g_t$ \label{glo_teich_flow} is given by the action of the diagonal subgroup $\left( \begin{smallmatrix} e^t & 0 \\ 0 & e^{-t} \end{smallmatrix} \right) \subset \SL_2(\RR)$ on $\Omega \Mg$. \end{defi}
Geometrically the Teichmüller flow can be realized as follows: take a polygon pattern of the flat surface $X$ by unwrapping it along some geodesic cuts and expand the polygon in one direction and contract it in the other direction with the same factor (see \cite{Zor06},  Chapter~3.2. and 3.3.).\\[11pt]
By a continuity argument, it is evident that the Teichmüller geodesic flow cannot leave the connected component of any stratum. Let $\Omega\Mg^{(1)} \subset \Omega\Mg$ \label{glo_OmMg1} denote the subspace of surfaces with normalized area $\int |\omega|^2 =1$. The Teichmüller flow preserves any hypersurface of constant area, especially $\Omega \Mg^{(1)}$, since the action of $\SL_2(\RR)$ is also area preserving,  \\[11pt]
It was independently shown by H. Masur and W. Veech that there exists a distinguished probability measure $d\nu_1$ on $\Omega \Mg^{(1)}$ invariant by the Teichmüller flow and ergodic (see e.g. \cite{AGY06}). This measure is called \textbf{Masur-Veech measure}. \index{Masur-Veech measure} In fact they proved the following key result: \label{glo_MVmeasure}
\begin{thm} \textbf{(Masur, Veech)} The Teichmüller flow $g_t$ preserves the (finite) measure $d\nu_1$ on $\Omega \Mg^{(1)}$ and is ergodic on each connected component of $\Omega \Mg^{(1)}$. \end{thm} 
\begin{proof} See \cite{Mas82}. \end{proof}
The Teichmüller flow is often also called \textbf{Teichmüller geodesic flow}. One might ask in which sense the Teichmüller flow is a geodesic flow. This question can be answered in the following way: Teichmüller showed that given two Riemann surfaces $S_0$ and $S_1$ there always exist maps $f_0: S_0 \to S_1$ which minimize the coefficient of quasiconformality (see Section~\ref{sec_moduli_space_rs}) and that the one parameter family of matrices $g_t$ 
applied to $S_0$ forms a geodesic with respect to the Teichmüller metric (see \cite{Zor06}, Section~8.1).\\ 
For $g=1$, the space $\Omega \mathcal{M}_1^{(1)}$ can be identified with $\SL_2(\RR) / \SL_2(\ZZ)$ and the Teichmüller metric coincides with the Poincaré metric on the modular surface. So for $g=1$ the Teichmüller flow is exactly the geodesic flow on the modular surface which was described in Section~\ref{sec_geodesic_flow} (see e.g. \cite{For06}, p. 556).


\subsubsection{The Teichmüller Flow on Teichmüller Curves} \label{Teich_flow_Teich_curves}
As Teichmüller curves have measure $0$ with respect to the Masur-Veech measure, one has to substitute the Masur-Veech measure in order to make sense of the notion of Lyapunov exponents of the Teichmüller flow on Teichmüller curves. To this end, let $(X,\omega)$ be a Veech surface of renormalized area 1 and consider its $\SL_2(\RR)$-orbit $C$ in $\Omega\Mg^{(1)}$. Since $C$ is closed, we obtain a finite measure $\mu_C$ on $\Omega\Mg^{(1)}$ with support $C$, namely the measure induced from the Haar measure $\lambda$ on $\SL_2(\RR)$. This is the desired measure since $\mu_C$ is $\SL_2(\RR)$-invariant and ergodic with respect to the Teichmüller flow (see \cite{CFS82}, Chapter~4, Theorem~1). We will see later that the choices of the two measures fit together in an appropriate way. \\[11pt]
We are now able to give an interpretation of the Teichmüller flow on Teichmüller curves as geodesic flow in the sense of Section~\ref{sec_geodesic_flow}. Recall that Teichmüller curves $f: \mathbb{H} \to \Mg$ are generated by flat surfaces. Indeed $f$ is simply the projection of the $\SL_2(\RR)$-orbit of some $(X,\omega)$ to $\Mg$.\\ 
Since a Teichmüller curve yields an isometric embedding $\HH / \SL(X,\omega) \to f(\HH) \subset \Mg$ (compare \cite{McM03}), we can identify each $Y \in f(\HH)$ with a point on the orbifold $\HH / \SL(X,\omega)$. 
As Teichmüller curves are complex geodesics for the Teichmüller metric 
we observe that via $f$ we can identify the Teichmüller flow on the Teichmüller curve with the geodesic flow on $T^1\HH / \SL(X,\omega)$.\\ 
The main advantage of the interpretation of the Teichmüller geodesic flow as the geodesic flow on $T^1\HH / \SL(X,\omega)$ is that it will in Section~\ref{sec_no_twist_Lyapunov} enable us to show that the Lyapunov exponents of twisted Teichmüller curves agree with the Lyapunov of the ordinary Teichmüller curves with respect to a certain cocycle which will be introduced next.

\subsection{The Kontsevich-Zorich Cocycle and its Lyapunov Exponents} \label{sec_Kontsevich_Zorich} 
In \cite{Kon97}, M. Kontsevich and A. Zorich introduced a certain renormalization cocycle over the Teichmüller flow. In this section we describe the cocycle and some of its properties. A more detailed review on this topic can be found e.g. in \cite{For06} and in \cite{Zor06}.\\[11pt]
As has already been explained, we must replace $\Mg$ by an appropriate fine moduli space by adding a level $l$-structure in order to find a universal family. We do not indicate this change in our notation, but let $f: \mathcal{X} \to \Mg$ be the universal family over $\Mg$ (for details see Section~\ref{sec_moduli_space_rs}). Let $\pi: \Omega\Mg \to \Mg$ be the usual projection. We now consider the local system $\pi^*(R^1f_{*}\RR)$, where $R^1$ is the first right derived functor. 
The fiber over $(X,\omega)$ is hence $H^1(X,\RR)$. We denote the corresponding real $C^\infty$-vectorbundle by $V$. Then $V$ has a natural norm, namely the \textbf{Hodge norm} (see e.g. \cite{Möl11a}, Chapter~4). Since $V$ carries a flat structure, we can lift the Teichmüller flow by parallel transport to a flow $F_t$ on $V$.  This is the \textbf{Kontsevich-Zorich cocycle}. By construction the Kontsevich-Zorich cocycle is indeed a cocycle over the Teichmüller flow. \index{Kontsevich-Zorich cocycle} The Kontsevich-Zorich cocycle is a continuous version of a cocycle introduced by G. Rauzy (see e.g. \cite{Zor06}).\\[11pt]
The real cohomology $H^1(X,\RR)$ of an orientable closed surface $X$ has dimension $2g$ and is endowed with a natural symplectic form, namely the intersection form (see \cite{FK92}, Chapter~III.1). So the Kontsevich-Zorich cocycle is defined on a symplectic vector bundle of dimension $2g$. Therefore, the Kontsevich-Zorich cocycle, both over the Teichmüller flow on $\Mg^{(1)}$ and over a Teichmüller curve $C$, has a symmetric Lyapunov spectrum (compare \cite{Via08} Proposition 5.8.): 
$$\lambda_1=1 \geq \lambda_2 \geq ... \geq \lambda_g \geq 0 \geq -\lambda_g \geq ... \geq - \lambda_2 \geq - \lambda_1 = -1.$$
The fact that $\lambda_1=1$ follows in both cases immediately by comparing the definition of the Teichmüller flow and the assertion of Oseledet's Theorem. Nevertheless there are also differences for the two types of Lyapunov exponents: it was an open problem for a decade whether the upper inequalities are strict for the Teichmüller flow on $\Omega\Mg^{(1)}$. In 2004, A. Avila and M. Viana succeeded in solving this problem. A weaker version of their theorem was earlier proven by G. Forni, see \cite{For02}. \newpage

\begin{thm} (\textbf{Avila, Viana, \cite{AV07}}) For any connected component of any stratum the first $g$ Lyapunov exponents are distinct and greater than 0:
$$1 = \lambda_1^{\nu_1} > \lambda_2^{\nu_1} > ... > \lambda_g^{\nu_1} >0.$$ \end{thm}

This statement is in general not true for the Teichmüller flow on Teichmüller curves (see e.g. \cite{For02} or \cite{BM10a}, Theorem~8.2).\\[11pt] Let us conclude this section by a proposition which shows that the involved measures fit together appropriately. Sometimes it is much easier to calculate the sum of the (positive) Lyapunov exponents than the individual ones. Let $C_d$ denote the union of all Teichmüller curves in a fixed stratum generated by square-tiled surfaces of $d$ squares and let $L(C_d)$ be the average of the Lyapunov exponents of the individual components weighted by the hyperbolic volume of the corresponding component. Then the following holds (see \cite{Che10}, Appendix~A):

\begin{prop} For $d \to \infty$ the weighted sum of Lyapunov exponents $L(C_d)$ of square-tiled surfaces in a component of a stratum of $\Omega\Mg$ converges to the sum of Lyapunov exponents $L_{\nu_1}$. \end{prop}

Finally, let us remark that the correspondence which was established in the last section implies that the Kontsevich-Zorich cocycle for Teichmüller curves can also be regarded as a cocycle over the geodesic flow $g_t$ on $T^1 \HH / \SL(X,\omega)$. For brevity reasons we will in the rest of these notes only speak of Lyapunov exponents when we mean the Lyapunov exponents of the Kontsevich-Zorich cocycle.

\newpage

\section{Kobayashi Curves Revisited} \label{chapter_comparing}

We now approach the answer to the question if all Kobayashi curves \index{Kobayashi curve} are twisted Teichmüller curves. In particular, one might ask if the genus~$4$ Teichmüller curve is a twist of a Teichmüller curve $C_{L,D}^\epsilon$ we saw that these curves \textit{always} lie on the \textit{same} Hilbert modular surface $X_D$. In the sequel we will answer this question negatively (Theorem~\ref{thm_notwist_general}). In those cases when the genus~$3$ Teichmüller curve also lies on $X_D$, we will furthermore see that it is not a twist of a Teichmüller curve $C_{L,D}^\epsilon$ either (Corollary~\ref{cor_lyap1/5}).\\[11pt] This chapter starts with a necessary criterion for any curve to be a twisted Teichmüller curve (Proposition~\ref{prop_twisted_implies_commensurable}). 
We then show by a very explicit calculation which uses the arithmetic of $\mathcal{O}_5$ that in discriminant $5$ the genus~$4$ Teichmüller curve is not a twisted Teichmüller curve (Proposition~\ref{prop_notwist_D5}). This arithmetic approach might also work for other discriminants. However, we use the concept of Lyapunov exponents (see Chapter~\ref{chapter_Lyapunov_exponents}) to prove the result for arbitrary discriminants in Section~\ref{sec_no_twist_Lyapunov}. 

\subsection{An Arithmetic Approach} \label{sec_arithmetic_approach}

It is not at all easy to decide whether a given curve $\HH / \Gamma$, where $\Gamma$ is a Fuchsian group, is a twisted Teichmüller curve or not. The first result gives a necessary criterion for this to be possible. As preparation for this criterion let us restate Proposition~\ref{prop_finite_index_GL2K} in the language of commensurators.

\begin{prop} We have $\Comm_{\GL_2(\mathbb{R})} (\SL_2(\OD)) = \GL_2(K).$ \end{prop}

The criterion itself is easy to state, but practically hard to check: a twisted Teichmüller curve is always commensurable with $\SL(L_D)$.

\begin{prop} \label{prop_twisted_implies_commensurable} Let $\Gamma \subset \SL_2(\OD)$ be a Fuchsian group. If $\mathbb{H}/\Gamma$ is a twisted Teichmüller curve then $\Gamma$ and $\SL(L_D)$ are commensurable in $\SL_2(\mathbb{R})$.
\end{prop}

Before we give the proof, let us note that this criterion does not depend on the choice of the generating surface (and thus $\SL(L_D)$) since changing the generating surface changes both, the stabilizer and the Veech group, only by conjugation.

\begin{proof} If $\mathbb{H}/\Gamma$ is a twisted Teichmüller curve then
$$ \textrm{M} \Stab(\Phi)  M^{-1} \cap SL_2(\OD)  = \Gamma.$$
for some $M \in \GL_2^+(K)$. This happens if and only if
$$\Stab(\Phi) \cap M^{-1} \SL_2(\OD) M = M^{-1} \Gamma M $$
Now divide every entry of $M$ by $\sqrt{\det(M)}$ and denote this matrix by $V$. Then $V$ lies in $\SL_2(\mathbb{R})$ and $V^{-1}M$ is a multiple of the identity matrix. Hence
$$V^{-1} (M \Stab(\Phi)  M^{-1} \cap SL_2(\OD))V = \Stab(\Phi) \cap M^{-1} \SL_2(\OD)M. $$
Thus it suffices to show that $\SL(L_D)$ and $M^{-1}\Gamma M$ are directly commensurable.  We have

$$[\SL(L_D) : \SL(L_D) \cap M^{-1}\Gamma M] = $$
$$[\Stab(\Phi) \cap \SL_2(\OD) : \Stab(\Phi) \cap \SL_2(\OD) \cap \Stab(\Phi) \cap M^{-1} \SL_2(\OD) M] \leq$$
$$[\SL_2(\OD) :  \SL_2(\OD) \cap M^{-1} \SL_2(\OD) M] < \infty.$$

And

$$[M^{-1}\Gamma M : M^{-1}\Gamma M \cap \SL(L_D)] = $$
$$[\Stab(\Phi) \cap M^{-1} \SL_2(\OD) M: \Stab(\Phi) \cap \SL_2(\OD) \cap \Stab(\Phi) \cap M^{-1} \SL_2(\OD) M] \leq$$
$$[\SL_2(\OD) :  \SL_2(\OD) \cap M \SL_2(\OD) M^{-1}] < \infty.$$

\end{proof}

\begin{cor} If $\mathbb{H}/\Gamma$ is a twisted Teichmüller curve of $\mathbb{H}/\SL(L_D)$ where $\SL(L_D)$ is maximal then $$\chi(\SL(L_D))|\chi(\Gamma).$$
\end{cor}

\begin{proof} Since $\Gamma$ and $\SL(L_D)$ are commensurable $\Gamma$ must be conjugated to a subgroup of $\SL(L_D)$ (Corollary~\ref{cor_max_conjugated}). 
\end{proof}

Recall that $\SL(L_5)=\left\langle S,T\right\rangle$. Moreover by calculating two cylinder decompositions of $X_5$ one finds that $\SL(X_5)$, i.e. the Veech group of the genus~$4$ Prym Teichmüller curve in discriminant 5, contains the subgroup $\widetilde{H}=\left\langle S,T^2,C\right\rangle$ with infinite index, where $C=\left( \begin{smallmatrix} -3-6w & 6+10w \\-2-4w & 5+6w \end{smallmatrix} \right)$. With the help of the algorithm described in Appendix~\ref{sec_check_elements} one may check that $C \notin \SL(L_5)$.

\begin{prop} \label{prop_notwist_D5} The groups $\SL(L_5)$ and $\SL(X_5)$ are not commensurable. \end{prop}

The proof will very explicitly make use of the arithmetic of both groups. This is of course not very satisfactory. In Section~\ref{sec_no_twist_Lyapunov} we will therefore present a much more general approach which will imply that any of the Prym Teichmüller curves \index{Teichmüller curve!Prym} is never a twisted Teichmüller curve of a Teichmüller curve in genus~$2$ and vice versa.

\begin{proof} We set $G:=\SL(L_5)$ and $H:=\SL(X_5)$. Suppose that $G$ and $H$ are commensurable. Since the Euler characteristic of $H$ is greater than the Euler characteristic of $G$ we know by the maximality that $H$ must be conjugated to a subgroup of $\SL(L_5)$. This means that there exists a matrix $M \in \SL_2(\mathbb{R})$ with 
$$MHM^{-1} < G.$$
Since $G$ and $H$ are both subgroups of $\SL_2(K)$ and since $M$ has to send cusps to cusps we must have $M \in \GL_2^+(K)$. 
Since $S \in G$ and $S \in H$ and since $S$ is the only elliptic element of order $2$ in $G$ up to conjugation, there exists $A \in G$ with $MSM^{-1} = ASA^{-1}$. Since $A \in G$ we may without loss of generality assume that $MSM^{-1}=S$. It is well-known that $M$ is of the form 
$$M= \begin{pmatrix} d & -c \\ c & d \end{pmatrix}.$$
By multiplying with the common denominator of $c$ and $d$ we may furthermore assume that $c,d \in \OD$. As we know that $T^2 \in H$ we then must have
$$MT^2M^{-1} = \begin{pmatrix} 1+\frac{2wcd}{c^2+d^2} & \frac{2wd^2}{c^2+d^2} \\ \frac{-2wc^2}{c^2+d^2} & 1 - \frac{2wcd}{c^2+d^2} \end{pmatrix} \in \SL_2(\OD).$$
In particular this means (recall that $w$ is the fundamental unit in $\mathcal{O}_5$):
$$\frac{2c^2}{c^2+d^2} \in \OD, \ \ \ \frac{2d^2}{c^2+d^2} \in \OD.$$
Now suppose that $c,d \neq 0$. The above is then equivalent to
$$\frac{2}{1+\frac{d^2}{c^2}} \in \OD, \ \ \ \frac{2}{1+\frac{c^2}{d^2}} \in \OD.$$
This implies that $1 + \frac{c^2}{d^2} = \frac{2w^m}{n}$ or equivalently $\frac{c^2}{d^2} = \frac{2w^m-n}{n}$ with $n \in \OD$ and $m \in \ZZ$. Now let $\pi$ be a prime element in $\OD$ with $\pi | 2w^m-n$ and $\pi|n$. Then $\pi|2$, which means $\pi=2$ because $2$ is an inert prime number. So $\frac{2w^m-n}{n} = \frac{w^m-k}{k}$ with $k \in \OD$. The latter fraction is completely reduced. This means that we have to distinguish two cases. \\[11pt]First suppose that $\frac{c^2}{d^2} = \frac{w^m-k}{k}$. Then $c^2 = (w^m-k)l$ and $d^2=kl$ for some $l \in \OD$. We can forget about $l$ because $l$ yields just a multiplication of $M$ by a multiple of the unit matrix which vanishes under conjugation. Therefore, $c^2 = w^m-k$ and $d^2=k$. Hence 
$$(*) \ \ \ c^2+d^2=w^m.$$ 
We now have again to distinguish two different subcases. First suppose that $m$ is even and suppose that $(c',d')$ is a solution for $(*)$. Then $(\frac{c'}{w^{m/2}},\frac{d'}{w^{m/2}})$ is a solution for $c^2+d^2=1$. Hence it suffices to look at this equation.
Setting $c=e+fw$ and $d=g+hw$ with $e,f,g,h \in \ZZ$ this yields
$$(e+fw)^2+(g+hw)^2 = 1$$
or equivalently 
$$e^2+f^2+g^2+h^2 + w (2fe + 2gh + f^2 + h^2)=1$$
In particular $e^2 + f^2 +g^2 + h^2 =1$ which means that at least three variables must be equal to zero and therefore $c=0$ or $d=0$ which is a contradiction. 
As second subcase suppose that $m=2v+1$ is odd. A solution $(c',d')$ of $(*)$ then yields a solution $(\frac{c'}{w^{v/2}},\frac{d'}{w^{v/2}})$ of $c^2+d^2=w$. By similar considerations as above there is no solution to the latter equation.
\\[11pt]Now look at the second case, namely $\frac{c^2}{d^2}$ is equal to a completely reduced fraction of the form $\frac{2w^m-n}{n}$. By the preceding arguments we must then have that $c^2=2w^m-n$ and $d^2=n$ and therefore $c^2+d^2=2w^m$. We again distinguish two different subcases. If $m$ is even it suffices to look at $c^2+d^2=2$. Using the same notation as above this yields $e^2+f^2+g^2+h^2 = 2$ which means that $2$ variables must be $0$ and the other two $\pm 1$. However, only $e= \pm 1$ and $g= \pm 1$ is possible since in all other cases we must either have that $c=0$ or $d=0$ or the non-integer part does not vanish. Finally, if $m$ is odd there then $M$ would be restricted to the following possibilities:
\begin{itemize}
\item[(i)] $c=0$: So $M=\left( \begin{smallmatrix} d & 0 \\ o & d \end{smallmatrix} \right)$ with $d \in \OD$. This is a multiple of the unit matrix and since the third element $C \in H$ does not lie in $G$, this is not possible.
\item[(ii)] $d=0$: So $M=\left( \begin{smallmatrix} 0 & -c \\ c & 0 \end{smallmatrix} \right)$ is a multiple of the matrix $S$ and thus by the same arguments as in $(i)$ not a possible candidate. 
\item[(iii)] $c,d=w^{2m}$: As we can again forget about multiplication with multiples of the unit matrix we may just look at $c,d=1$: So $M= \left( \begin{smallmatrix} 1 & -1 \\ 1 & 1  \end{smallmatrix} \right)$. Then $MT^2M^{-1} = \left( \begin{smallmatrix} 1-2w & 2w \\ -2w & 1-2w \end{smallmatrix} \right) \notin \SL(L_5).$ This is again a contradiction.
\item[(iv)] $c=-w^{2m}, d=w^{2m}$ which is equivalent to $c=-1,d=1$: So $M= \left( \begin{smallmatrix} 1 & 1 \\ -1 & 1  \end{smallmatrix} \right)$ and $MT^2M^{-1} = \left( \begin{smallmatrix} 1+2w & 2w \\ -2w & 1-2w \end{smallmatrix} \right) \notin \SL(L_5).$ This is again a contradiction.
\item[(v)] $c=-w^{2m}, d=-w^{2m}$ which is equivalent to $c=-1,d=-1$: This yields the same contradiction as in $(iii)$.
\item[(vi)] $c=w^{2m}, d=w^{2m}$ which is equivalent to $c=1,d=-1$: This yields the same contradiction as in $(iv)$.
\end{itemize}
Hence the groups $H$ and $G$ are not commensurable.
\end{proof}

In genus~$2$ there is one additional primitive Teichmüller curve for discriminant $5$, namely the one generated by the regular decagon. Its uniformizing group is given by the triangle group $\Delta(5,\infty,\infty)$. Also a twist of this Teichmüller curve will not yield to the genus~$4$ Teichmüller curve.

\begin{cor} The groups $\SL(X_5)$ and $\Delta(5,\infty,\infty)$ are not commensurable. \end{cor}

\begin{proof} This is true because $[\Delta(2,5,\infty):\Delta(5,\infty,\infty)]=2$, see e.g. \cite{Sin72}. \end{proof}

Of course there is a pattern hidden behind these calculations. In Appendix~\ref{sec_check_commens} we describe how to decide if two cofinite Fuchsian groups are commensurable when both of them contain an elliptic element. 

\subsection{The General Case} \label{sec_no_twist_Lyapunov}

In this section, we want to show that \textit{all} twisted Teichmüller curves are essentially new objects, i.e. not one of McMullen's Prym Teichmüller curves. For this we compare twisted Teichmüller curves with the Prym Teichmüller curves in genus~$3$ and $4$. In fact, we show that for any discriminant $D$ the Teichmüller curves in genus~$3$ and $4$ are not a twisted (genus~$2$) Teichmüller curve. This means that although there are infinitely many twisted Teichmüller curves we do never get any other of the low genus Teichmüller curves by this construction. This is really surprising when one recalls the large number of twisted Teichmüller curves (Theorem~\ref{thm_classificiation1}). The proof of this fact involves deep results from the theory of Lyapunov exponents. As we only consider Lyapunov exponents of Teichmüller curves here, we will omit the index of Lyapunov exponents indicating the measure. The main work of this section will consist in deriving the following result:

\begin{thm} \label{thm_notwist_general} For all discriminants $D$ the genus~$4$ Teichmüller curve is not a twist of a Teichmüller curve $C_{L,D}^\epsilon$. \end{thm}

For discriminant $D=5$ we have already proven this result in Proposition~\ref{prop_notwist_D5} by a very explicit calculation. We will now prove the general version without explicit knowledge of the involved Veech groups. The analogue statement also holds for the genus~$3$ Teichmüller curves whenever they also lie on $X_D$. Of course, one might also ask whether this theorem is still true for Teichmüller curves in $\Mg$ with $g \geq 5$, if those Teichmüller curves lie on a Hilbert modular surface. There are no such examples known today.\\[11pt]
The main idea of the proof is to look at the Lyapunov exponents of the involved Teichmüller curves (see Section~\ref{sec_Kontsevich_Zorich}) and to show that these are not compatible in a way which will be made more precise on the next pages. This will then immediately imply the assertion of the theorem.\\[11pt]
For the rest of this section we fix the discriminant $D$.\\[11pt]
The purpose of the next few lines is to give another (equivalent) description of the Lyapunov exponents of a Teichmüller curve which can be generalized to arbitrary quotients $\HH/\Gamma$. We first \index{Lyapunov exponents|(}consider the case of genus~$g=2$. Let us recall the setting: in Section~\ref{sec_Teichmüller_flow} we related the Teichmüller flow of the Teichmüller curve to the geodesic flow on $T^1\mathbb{H}/\SL(L_D)$. Then $T^1\mathbb{H}/\SL(L_D)$ has Lyapunov exponents in a natural way: by applying the Torelli map, $C_{L,D}^\epsilon=\mathbb{H}/\SL(L_D)$ can be considered as lying in $\mathcal{A}_2$. Let $A=\Jac(X)$ be a point on this curve. The lattice $\Lambda$ of $A$ can be interpreted as both $H_1(X,\ZZ)$ and $H_1(A,\ZZ)$ (see e.g. \cite{BL04}, Chapters~1.3 and 11.1). \index{Jacobian} This yields the following commutative diagram

$$
\begin{xy}
  \xymatrix{ H^1 (X,\RR) \ar[rrrr]^{\sim}& & & & H^1 (A,\RR)\\
  H_1 (X,\ZZ) \ar@{^(->}[d] \ar@{^(->}[u]& = & \Lambda & = & H_1 (A,\ZZ) \ar@{^(->}[d]            \ar@{^(->}[u]\\
  H^0 (X,\Omega^1 X)^\vee \ar[rrrr]^{\sim} & & & & H^0 (A,\Omega^1 A)^\vee
 	}
\end{xy}
$$
where the upper inclusions \index{Kontsevich-Zorich cocycle} are given by tensoring with $\RR$ and dualizing. As $H_1 (A,\ZZ)$ is a discrete group, $H^1 (A,\RR)$ yields a local system over the image of $T^1\HH/\SL(L_D)$ in $\mathcal{A}_2$ (and hence a cocycle): consider the following composition of maps

$$
\begin{xy}
  \xymatrix{ T^1\HH/\SL(L_D) \ar[r]^{\ \pi} & \HH/\SL(L_D) \ar@{^(->}[r]^{\ \ f} & \mathcal{M}_2  \ar@{^(->}[r]^{\Jac} & \mathcal{A}_2 \\
 	}
\end{xy}
$$
and pull back the local system to $T^1 \mathbb{H}/\SL(L_D)$. One can then look at the corresponding Lyapunov exponents. Note that by construction the flow extension of $g_t$ to the local system coincides with the Kontsevich-Zorich cocycle. 
Now let $\Gamma < \SL(L_D)$ be a subgroup of finite index. Then the same construction as above gives also \textit{natural} Lyapunov exponents for $T^1 \mathbb{H}/\Gamma$ (if we renormalize the Poincaré measure of $\mathbb{H} / \Gamma$ to $1$). 


\begin{lem} \label{lem_fund_subgroup} If $\Gamma < \SL(L_D)$ is a subgroup of finite index, then the Lyapunov exponents  of the geodesic flow on $T^1\HH/\Gamma$ and on $T^1\HH/\SL(L_D)$ as defined above are the same. 
\end{lem}

\begin{proof} Let $\mathcal{F}$ be a Dirichlet fundamental domain of $\SL(L_D)$. Then we can choose a (connected) fundamental domain $\mathcal{F}'$ for $\Gamma$ as $[\SL(L_D):\Gamma]$ copies of $\mathcal{F}$ such that $\mathcal{F} \subset \mathcal{F}'$. Let $A^t(x)$ be the cocycle over $g_t$ and $\lambda_1,\lambda_2$ be the Lyapunov exponents of the geodesic flow on $T^1\HH / \SL(L_D)$ (respectively $A'^t(x)$ and $\lambda_1',\lambda_2'$ for the geodesic flow on $T^1\HH / \Gamma$). By Poincaré's Recurrence Theorem (see e.g. \cite{KH97}, Theorem~4.1.19)  there exists for almost all starting points $x$ an unbounded sequence $(t_i)_{i \in \NN}$ such that the geodesic flow on $T^1\HH / \Gamma$ returns to $\mathcal{F}$. At each of the $t_i$ we have by construction that $A'^{t_i}(x)=A^{t_i}(x)$ and hence
\begin{align*} \lambda'_k &=  \lim_{t \to \infty} \frac{1}{t} \log \left\| A'^t(x) v \right\| = \lim_{i \to \infty} \frac{1}{t_i} \log \left\| A'^{t_i}(x) v \right\| \\ &= \lim_{i \to \infty} \frac{1}{t_i} \log \left\| A^{t_i}(x) v \right\| = \lambda_k. \end{align*}	
 \end{proof}


The principal significance of the lemma is the following corollary:

\begin{cor} \label{cor_twist_lyapunov} Twists do not change the Lyapunov exponents. \end{cor}

\begin{proof} Twisting involves only finite index subgroups and conjugation (compare Proposition~\ref{thm_twisted_finite_volume}). \end{proof}

Now let us analyze the genus~$4$ case. Consider the Prym Teichmüller curve in $W_D^4$ which has been described in Section~\ref{section_prym_teich}. The corresponding cocycle over $T^1\HH/\SL(X_D)$ has $4$ positive Lyapunov exponents since $H^1 (X,\RR)$ has dimension $8$. Let us denote these Lyapunov exponents by $\mu_1,...,\mu_4$. We now split up the Lyapunov exponents of the Teichmüller curve into two groups.\\[11pt] We embedded the genus~$4$ Teichmüller curve into $\mathcal{A}_2$ by mapping each point of the curve to its Prym variety. By definition the Prym variety is given as $H^0(X',\Omega(X')^-)^{\vee}/H_1(X,\ZZ)^{-}$. Thus a natural bundle over $X$ is given as $H^1(X,\RR)^-$ and therefore of dimension $4$. The positive Lyapunov exponents of this bundle are two of the $\mu_1,...,\mu_4$. Since the involution of the $X$-shaped billiard table is the rotation by $180^\circ$ the form $\omega=dz$ lies in $H_1(X,\ZZ)^{-}$ and hence $1=\mu_1$ is in the Lyapunov spectrum of the corresponding cocycle.\\[11pt]
So far the Lyapunov exponents of the genus~$2$ and of the genus~$4$ Teichmüller curves have been treated more or less separately. Of course, the next task is to link both of them. If the genus~$4$ Teichmüller curve was a twisted Teichmüller curve they would by Corollary~\ref{cor_twist_lyapunov} and the preceding considerations have the same Lyapunov exponents. In other words the embeddings of the genus~$2$ and of the genus~$4$ Teichmüller curves from Theorem~\ref{thm_modular_embedding} into the Hilbert modular surface uniquely determine two pairs of Lyapunov exponents and all twisted Teichmüller curves share their pair of Lyapunov exponents with the ordinary Teichmüller curve in genus~$2$. Now there is an obvious plan how to proceed: calculate the two pairs of Lyapunov exponents of the (ordinary) Teichmüller curves of genus~$2$ and $4$ and show that they are not equal.\\
Let $\lambda_1,\lambda_2$ be the Lyapunov exponents of the genus~$2$ curve and $\tilde{\mu}_1,\tilde{\mu}_2$ be the Lyapunov exponents of the genus~$4$ curve. We know that $\lambda_1=\tilde{\mu}_1=1$. At first we want to calculate $\lambda_2$ now. The Lyapunov exponents are in some cases known to depend only on the connected component of the stratum (see Theorem~\ref{thm_components_of_strata}) which the Teichmüller curve lies in. For Teichmüller curves in genus~$2$, the following theorem precisely calculates $\lambda_2$:












\begin{thm} \label{thm_lyapunov_M2} (\textbf{Bainbridge}, \cite{Bai07}, Theorem~15.1) 
If $\mu$ is any finite, ergodic, $\SL_2(\mathbb{R})$-invariant measure on $\Omega\mathcal{M}^{(1)}_2$ then
$$\lambda_2(\mu) = \left\{ \begin{matrix} 1/3 & \textit{if $\mu$ is supported on $\Omega\mathcal{M}^{(1)}_2(2)$} \\ 1/2 & \textit{if $\mu$ is supported on $\Omega\mathcal{M}^{(1)}_2(1,1)$} \end{matrix} \right. .$$
\end{thm}

As the Teichmüller curve in genus~$2$ lies in the stratum $\Omega\mathcal{M}^{(1)}_2(2)$ we therefore know that $\lambda_2=1/3$. For the sum of Lyapunov exponents of this Teichmüller curve we thus have $\sum_{i=1}^{2} \lambda_i=4/3$.\\[11pt]
Now we come to the more complicated case of the genus~$4$ Teichmüller curve and compute $\tilde{\mu}_2$. 
It was already pointed out that these Teichmüller curves lie in $\Omega\mathcal{M}_4(6)$. Another calculation shows that the Teichmüller curves in genus~$4$ in fact lie in the connected component $\Omega\mathcal{M}_4(6)^{even}$.\footnote{Note in the following that our proof would also work if the Teichmüller curves would lie in any other of the connected components.} Recently  D. Chen and M. Möller proved in \cite{CM11} that the sum of the Lyapunov exponents of Teichmüller curves depends in many cases only on the connected component which the Teichmüller curve lies in. The result for the stratum $\Omega\mathcal{M}_4(6)$ is the following: \index{spin invariant}

\begin{thm} \textbf{(Non-varying sum of Lyapunov exponents)} (\textbf{Chen, Möller}, \cite{CM11}) The sum of Lyapunov exponents of a Teichmüller curve in $\Omega\mathcal{M}_4(6)$ only depends on the connected component which the Teichmüller curve lies in. More precisely:
\begin{align*}
\Omega\mathcal{M}_4(6)^{even}: &\quad \sum_{i=1}^{4} \mu_i = 14/7 \\
\Omega\mathcal{M}_4(6)^{odd}: &\quad \sum_{i=1}^{4} \mu_i = 13/7\\
\Omega\mathcal{M}_4(6)^{hyp}: &\quad \sum_{i=1}^{4} \mu_i = 15/7
\end{align*}
			
\end{thm}

Indeed, D. Chen and M. Möller showed corresponding results for many different connected components of strata of small genus. Their result on its own does not give us enough information to calculate $\tilde{\mu}_2$: for $\Omega\mathcal{M}_4(6)^{even}$ we only get the trivial inequality $0 \leq \tilde{\mu_2} \leq 1$. We therefore have to combine this equality with another recent result by A. Eskin, M. Kontsevich and A. Zorich in \cite{EKZ11}. This will then do the job. \paragraph{Canonical double covering.} Before we can state their theorem we have to explain a general principle which is often called the canonical double covering: \index{canonical double cover} let $(C,q)$ be a pair consisting of a Riemann surface $C$ of genus~$g$ and a meromorphic quadratic differential $q$ with at most simple poles. Let $d_j$ be the orders of the zeroes of $q$ (respectively $d_j=-1$ for the poles). Then $\sum_j d_j = 4g-4$. One can then canonically associate to $(C,q)$ a double cover $p: \tilde{C} \to C$ such that $p^*q=(\omega)^2$ where $\omega$ is an Abelian differential on the Riemann surface $\tilde{C}$ of genus~$\tilde{g}$. One easily shows that this construction associates to each even $d_j>0$ a pair of zeroes of $\omega$ of orders $(d_j/2,d_j/2)$, to each odd $d_j>0$ a zero of order $d_j+1$ and nothing to simple poles (for details see \cite{Lan04} and \cite{KZ03}). Hence one gets a map from the stratum $\mathcal{Q}(d_1,...,d_n)$ \label{glo_stratq} of meromorphic quadratic differentials to the stratum $\Omega\Mg(k_1,...,k_m)$ of Abelian differentials. This map is in fact an immersion (\cite{KZ03}, Lemma~1).\\[11pt]
Now consider this double cover $p: \tilde{C} \to C$. There exists a natural involution $\sigma: \tilde{C} \to \tilde{C}$ interchanging the sheets of the cover. This decomposes $H^1(\tilde{C},\RR)$ into a direct of sum of subspaces, namely the invariant part $H^1(\tilde{C},\RR)^+$ (i.e. the $+1$-eigenspace of $\sigma^*$) and the anti-invariant part $H^1(\tilde{C},\RR)^-$ (i.e. the $-1$-eigenspace of $\sigma^*$) with respect to the induced involution $\sigma^*$ on the cohomology. Thus we get two natural vector bundles over $\mathcal{Q}(d_1,...,d_n)$, which we denote by $H^{1+}$ and $H^{1-}$. Evidently $H^{1+}$ is canonically isomorphic to the standard Hodge bundle $H^1_\RR$ over $\mathcal{Q}(d_1,...,d_n)$  as its fiber over a point $C$ is $H^1(C,\RR)$: it corresponds to cohomology classes pulled back from $C$ to $\tilde{C}$ via the projection. Now set $g_{eff}=\tilde{g}-g$. Hence $\dim H^{1-} = \dim H^1(\tilde{C},\RR)^- = \dim H^{1+} = \dim H^1(\tilde{C},\RR)^+ = 2g_{eff}$. Let $\lambda^-_1 \geq ... \geq \lambda^-_{g_{eff}}$ denote the top $g_{eff}$ Lyapunov exponents corresponding to the action of the Teichmüller geodesic flow and the vector bundle $H^{1-}$. Then one has:

\begin{thm} (\textbf{Eskin, Kontsevich, Zorich}, \cite{EKZ11})
Consider a stratum $\mathcal{Q}_1(d_1,...d_n)$ in the moduli space of quadratic differentials with at most simple poles, where $d_1+...+d_n=4g-4$. Let $\mathcal{S}$ be any regular $\PSL_2(\mathbb{R})$ suborbifold of $\mathcal{Q}_1(d_1,...d_n)$. Let $\lambda_1^+ \geq...\geq \lambda_g^+$ denote the Lyapunov exponents of the invariant subbundle $H^{1+}$ over $\mathcal{S}$ along the Teichmüller flow \index{Teichmüller flow} and let $\lambda_1^{-} \geq ... \geq \lambda_{g_{eff}}^-$ denote the Lyapunov exponents of the anti-invariant subbundle $H^{1-}$ over $\mathcal{S}$ along the Teichmüller flow. Then the Lyapunov exponents satisfy the following equation $$(\lambda_1^{-}+...+\lambda_{g_{eff}}^-) - (\lambda_1^++...+\lambda_g^+) = \frac{1}{4} \cdot \sum_{j \ such \ that \ d_j \ is \ odd} \frac{1}{d_j+2}.$$
The leading Lyapunov exponent $\mu_1^-$ is equal to one.
\end{thm}

We can apply this theorem to the case of Teichmüller curves of genus~$4$ as follows: the genus~$4$ Teichmüller curve is the projection of an $\SL_2(\RR)$-orbit of a flat surface $(X,\omega)$ (the X-shaped billiard table, see Section~\ref{section_prym_teich}) to $\mathcal{M}_4$. Let $\rho$ be the involution of $X$. It is a well-known fact that there exists a quadratic differential $q$ on $Y:=X/\rho$ such that $p^*q=\omega^2$, i.e. $X$ is the canonical double cover of $Y$. Therefore, $(Y,q)$ lies in $\mathcal{Q}_2(-5,1)$. Note furthermore that the bundle over the Prym variety corresponds to $H^{1-}$. In our notation the theorem hence gives:
$$2 \sum_i \tilde{\mu}_i - \sum_i \mu_i = \frac{1}{4} (1 + \frac{1}{7}) = \frac{2}{7}.$$ 
Combining this formula with the formula of Chen and Möller we finally get 
$$\sum_i \tilde{\mu}_i = (2 + \frac{2}{7})/2 = \frac{8}{7}.$$ Thus the second Lyapunov of the genus~$4$ Teichmüller curve is $\tilde{\mu}_2=\frac{1}{7}$. We conclude that the genus~$4$ Teichmüller curve is different from any of the twisted Teichmüller curves of  the genus~$2$ Teichmüller curve. This finishes the proof of Theorem~\ref{thm_notwist_general}.\\[11pt]
Although the types of the polarizations of the genus~$3$ Teichmüller curve and of the genus~$2$ Teichmüller are different, also the genus~$3$ Teichmüller curve might lie on $X_D$. This happens in the case, when the Hilbert modular surface parameterizing the corresponding Abelian surfaces is given by $\HH \times \HH^-/\SL(\OD\oplus\mathfrak{b}^2)$ for some fractional ideal $\mathfrak{a}$ since this Hilbert modular surface is then isomorphic to $X_D$ (compare \cite{Zag81}, \cite{vdG88}). We can then use the same methods as above to show:
\begin{cor} \label{cor_lyap1/5} 
\begin{itemize} 
\item[(i)] For all discriminants $D$ the genus~$3$ Teichmüller curve is not a twist of a  Teichmüller curve $C_{L,D}^\epsilon$.
\item[(ii)] The second Lyapunov exponent of the genus $3$ Teichmüller curve is equal to $\frac{1}{5}$.
\end{itemize}
\end{cor}
\begin{proof} Similarly as above, it follows from \cite{CM11} and \cite{EKZ11} that the second Lyapunov exponent of the genus $3$ Teichmüller curve is equal to $\frac{1}{5}$. By the same argument the genus~$4$ Teichmüller curve is then also not a twisted genus~$3$ Teichmüller curve for all discriminants $D$ (and vice versa). \end{proof}
As a consequence from the calculations we get the following inequalities for the Lyapunov exponents of the genus~$4$ Teichmüller curve.
\begin{cor} For the Lyapunov exponents of the Prym Teichmüller curves in genus $4$ the following inequalities hold:  \label{cor_lyap1/7} $3/7 \leq \mu_2 \leq 6/7$, $1/7 \leq \mu_3 < 3/7$, $0 < \mu_4 \leq 1/7$ and $\mu_3$ or $\mu_4$ is equal to $1/7$. \end{cor}
\index{Lyapunov exponents|)}

\newpage

\newpage

\appendix
\section{Appendix}

\subsection{Proof of Theorem~\ref{thm_Hilbert_modular_moduli}} \label{appendix_proof}

\begin{thm3} The Hilbert modular surface $X_D$ \index{Hilbert modular surface} is the moduli space of all pairs $(X,\rho)$, where $X$ is a principally polarized Abelian surface and $\rho$ is a choice of real multiplication on $X$ by $\OD$. \end{thm3}

\begin{proof} Given $\tau=(\tau_1,\tau_2) \in \HH \times \HH$ we define a map $\phi_{\tau}: \OD \oplus \OD^{\vee} \to \mathbb{C}^2$ by $$\phi_{\tau}(x,y)=(x+y\tau_1,x^{\sigma}+y^{\sigma}\tau_2).$$ 
Let $A_{\tau}$ be the complex torus $\mathbb{C}^2 / \phi_{\tau}(\OD \oplus \OD^{\vee})$ with the principal polarization induced by the standard symplectic pairing on $\OD \oplus \OD^{\vee}$, i.e. $\langle (x_1,y_1),(x_2,y_2) \rangle = \tr(x_1y_2-x_2y_1)$. We define real multiplication by $\OD$ on $A_{\tau}$ by $k(z_1,z_2)=(kz_1,k^{\sigma}z_2)$. This construction gives a map $\widetilde{\Psi}$ from $\mathbb{H} \times \mathbb{H}$ to the set of all triples $(X,\nu,\phi)$ where $(X,\nu)$ is a principally polarized Abelian surface $X=\mathbb{C}^2/\Lambda$ with real multiplication by $\OD$ and $\phi$ is a choice of an $\OD$-linear, symplectic isomorphism $\phi: \OD \oplus \OD^{\vee} \to \Lambda$. We need to show that ``forgetting'' the choice of the isomorphism $\phi$ means exactly factoring $\widetilde{\Psi}$ through $\SL(\OD \oplus \OD^\vee)$.\\ Let $g=\left( \begin{smallmatrix} a & b \\ c & d \end{smallmatrix}  \right)\in \SL(\OD \oplus \OD^\vee)$. Then also $g^*= \left( \begin{smallmatrix} a & -b \\ -c & d \end{smallmatrix} \right)$ is an automorphism of $\OD \oplus \OD^{\vee}$. Define
$$\chi(g,\tau) = \begin{pmatrix} (c\tau_1 + d)^{-1} & 0 \\ 0 & (c^{\sigma}\tau_2+d^{\sigma})^{-1} \end{pmatrix}.$$
This yields the following commutative diagram:
$$
\begin{xy}
 \xymatrix{ \OD \oplus \OD^{\vee} \ar[d]_{g^*} \ar[r]^{\ \ \ \phi_{\tau}} & \mathbb{C}^2 \ar[d]^{\chi(g,\tau)} \\ \OD \oplus \OD^{\vee} \ar[r]^{\ \ \ \phi_{g \cdot \tau}} & \mathbb{C}^2
 	}
\end{xy}
$$
Hence $\chi(g,\tau)$ induces an isomorphism between $A_{\tau}$ and $A_{g \cdot \tau}$, that  preserves polarizations and commutes with the action of real multiplication. Therefore, we have a well-defined map $\Psi$ from $X_D$ to the set of all principally polarized Abelian surfaces with a choice of real multiplication. This map will be shown to be a bijection.\\[11pt]
Let us first show that $\Psi$ is injective:  
assume that there exists an isomorphism $f: X=\mathbb{C}^2 / \Lambda \to X'= \mathbb{C}^2 / \Lambda'$. Let us choose two isomorphisms $\phi_{\tau} : \OD \oplus \OD^\vee \to \Lambda= \Lambda_{\tau}$ and $\phi_{\tau'} : \OD \oplus \OD^\vee \to \Lambda' = \Lambda_{\tau'}$. Let $\gamma:=(D+\sqrt{D})/2$. Following \cite{McM07}, Chapter~3, $A_{\tau}:=\CC^2/\phi_{\tau}(\OD \oplus \OD^\vee)$ is then isomorphic to $\CC^2/(\ZZ^2 \oplus \Pi \ZZ^2)$ where
$$\Pi = \frac{1}{D} \begin{pmatrix} \tau_1(\gamma^{\sigma})^2 + \tau_2(\gamma)^2 & -\tau_1 \gamma^{\sigma} - \tau_2 \gamma \\ -\tau_1 \gamma^{\sigma} - \tau_2 \gamma & \tau_1 + \tau_2 \end{pmatrix}$$
and similarly for $A_{\tau'}$. However, two such tori with corresponding matrices $\Pi$ and $\Pi'$ are isomorphic if and only if $\Pi$ and $\Pi'$ are equivalent modulo the action of $\Sp_4(\ZZ)$ on $\HH_2$. It can be checked that this implies that $\tau' = g \tau$ for some $g \in \SL(\OD \oplus \OD^\vee)$. It follows that $\Psi$ is injective. \\[11pt]
The longest step is to show that $\Psi$ is surjective. To do this, we consider an arbitrary complex torus $X=\mathbb{C}^2 / \Lambda$ with principal polarization given by an alternating form $E: \Lambda \times \Lambda \to \ZZ$ and real multiplication by $\rho: \OD \to \End(X)$. For dimension reasons $\Lambda$ is a projective rank 2 $\OD$-module. As first step let us show that $\Lambda \cong \OD \oplus \mathcal{I}$ for some ideal $\mathcal{I}$. This is an immediate consequence of the following two lemmas.

\begin{lem} Let $R$ be a Dedekind domain and $A_1,...,A_n$ be fractional ideals. Then
$$A_1 \oplus ... \oplus A_n \cong R^{n-1} \oplus A_1 \cdots A_n.$$
\end{lem}

\begin{proof} By induction, it suffices to consider the case $n=2$. So we are dealing with fractional ideals $A$ and $B$. Multiplying by suitable elements $x$ and $y$ of $R$ we may assume that $A$ and $B$ are relatively prime ideals in $R$. Define $\pi: A \oplus B \to R$ by $\pi(a,b) = a + b $. The kernel of $\pi$ is $A \cap B = AB$ since $A$ and $B$ are relatively prime. The short exact sequence
$$0 \longrightarrow AB \longrightarrow A \oplus B \stackrel{\pi}{\longrightarrow} R \longrightarrow 0$$
splits since $R$ is (as Dedekind domain) free, and this yields the claim.
\end{proof}

\begin{lem} For a finitely generated projective $\OD$-module $M$ of rank n
$$M = \OD^{n-1} \oplus I$$
where $I$ is an ideal.
\end{lem}

\begin{proof}
As $M$ is projective, it is torsion free. Now we proceed the proof by induction. If $n=1$, then $M$ is a $\OD$-submodule of $M \otimes_{\OD} K \cong K$ and is therefore isomorphic to a fractional ideal. By choosing $n-1$ elements of M that span a vector space of that dimension in $M \otimes_{\OD} K$, we can construct an $\OD$-submodule $N$ of rank $n-1$. The exact sequence
$$0 \longrightarrow N \longrightarrow M \longrightarrow M/N \longrightarrow 0$$
remains exact upon tensoring with K, hence $M/N$ has rank 1. Thus $M/N$ is projective and the sequence splits. Since any fractional ideal is isomorphic to an ideal, the conclusion follows from the inductive hypothesis and the last lemma.
\end{proof}

Hence we have that as $\OD$-modules $\Lambda \cong \OD \oplus \mathcal{I}$ for some ideal $\mathcal{I}$. So we may assume that $X= \mathbb{C}^2 / \varphi(\OD \oplus \mathcal{I})$ with real multiplication $\rho$ by $\OD$ for some embedding $\varphi: \OD \oplus \mathcal{I} \to \CC^2$. 

\begin{lem} There exists a symplectic isomorphism $\Theta: (\OD \oplus \mathcal{I}, E,\rho) \to (\OD \oplus \OD^\vee, \left\langle,\right\rangle, \rho')$, where $\left\langle ,\right\rangle$ is the symplectic pairing on $\OD \oplus \OD^\vee$ that is compatible with the real multiplications $\rho$ and $\rho'$. \end{lem}

\begin{proof} We can always choose a symplectic basis with $a_1,a_2 \in \OD$, $b_1,b_2 \in \mathcal{I}$ for $\Lambda$, i.e. $E(a_i,b_j) = \delta_{ij}$ and $E(a_i,a_j)=0$ and $E(b_i,b_j)=0$. The form $E$ extends to $E: K^2 \times K^2 \to \mathbb{Q}$ such that $E(kx,y)=E(x,ky)$ for all $x,y \in \Lambda$ and $k \in K$. Let $\left\langle ,\right\rangle$ be the standard symplectic form on $\OD \oplus \OD^\vee$ with its standard basis $c_1,c_2,d_1,d_2$ and together with its standard real multiplication $\rho'$ (see e.g. \cite{McM07}). We choose a $\ZZ$-linear symplectic isomorphism $\Theta : \OD \oplus \mathcal{I} \to \OD \oplus \OD^\vee$ with $\Theta(a_i)=c_i$ and $\Theta(b_i)=d_i$ for $i=1,2$. Then we may identify the $a_i$ with the $c_i$ since the $c_i$ are a basis of $\OD$ as $\ZZ$-module. We now want to prove that this symplectic isomorphism is $\OD$ linear. By tensoring with $K$ we also extend this map to $\Theta: K \times K \to K \times K$ and verify that $\Theta$ is indeed $K$-linear. Since $\Theta$ is $\mathbb{Q}$-linear, it suffices to show that $\Theta(kx)=k\Theta(x)$ for a fixed $k \in K \smallsetminus \mathbb{Q}$ and for all $x \in K \times K$.  We now choose $k = \frac{b_1}{b_2}$. By definition we have $\Theta(ka_1) = k \Theta(a_1)$ and $\Theta(ka_2) = k \Theta(a_2)$. We now show that $\Theta(kb_1)=k\Theta(b_1)$ for this fixed $k \in K \smallsetminus \mathbb{Q}$. This is equivalent to showing that $kd_1=d_2$. By definition of $k$ we have
$$1 = E(a_2,b_2) = E(a_2,kb_1) = E(ka_2,b_1),$$
which yields 
\begin{eqnarray} \label{eqn_sym1} 1 = \left\langle \Theta(ka_2), \Theta(b_1)\right\rangle = \left\langle k\Theta(a_2),d_1\right\rangle=\left\langle c_2,kd_1\right\rangle. \end{eqnarray}
Similarly we have
$$0 = E(a_1,b_2) = E(a_1,kb_1) = E(ka_1,b_1)$$
which yields 
\begin{eqnarray} \label{eqn_sym2} 0 = \left\langle \Theta(ka_1),\Theta(b_1)\right\rangle=\left\langle k\Theta(a_1),d_1\right\rangle=\left\langle c_1,kd_1\right\rangle. \end{eqnarray}
Equations (\ref{eqn_sym1}) and (\ref{eqn_sym2}) imply that $kd_1=d_2$. By considering appropriate equations it follows similarly as above that $\Theta(kb_2) = k \Theta(b_2)$, which implies $K$-linearity of $\Theta$. This means that the following diagram commutes:
$$
\begin{xy}
  \xymatrix{ \OD \oplus \mathcal{I} \ar[r]^{\ \Theta} \ar[d]_{\rho} & \OD \oplus \OD^\vee  \ar[d]^{\rho'} \\
   \OD \oplus \mathcal{I} \ar[r]^{\ \Theta} & \OD \oplus \OD^\vee
 	}
\end{xy}
$$
Or in other words the chosen symplectic isomorphism respects the real multiplication. 
\end{proof}

We may therefore assume that $\Lambda = \varphi(\OD \oplus \OD^\vee)$ where $\varphi: \OD \oplus \OD^\vee \to \mathbb{C}^2$ is an embedding. 
It remains to show that there exists a $\tau \in \HH^2$ with $\phi_\tau = \varphi$.\\
Recall that $$(a_1,a_2,b_1,b_2)=((1,0),(\gamma,0),(0,-\gamma^\sigma/\sqrt{D}),(0,1/\sqrt{D}))$$
is the standard basis of $\OD \oplus \OD^\vee$ with respect to the standard symplectic form, where $\gamma= (D+\sqrt{D})/2$. With respect to this basis $\phi_\tau: \OD \oplus \OD^\vee \to \CC^2$ is given by the matrix
$$\phi_\tau= \begin{pmatrix} 1 & \gamma & - \tau_1 \gamma^\sigma /\sqrt{D} & \tau_1/\sqrt{D} \\
1 & \gamma^\sigma & \tau_2 \gamma /\sqrt{D} & -\tau_2/\sqrt{D} \end{pmatrix}$$
or equivalently
$$\phi_\tau=(B,D_\tau(B^{t})^{-1})$$
where $B=\left( \begin{smallmatrix} 1 & \gamma \\ 1 & \gamma^\sigma \end{smallmatrix} \right)$ and $D_\tau= \left( \begin{smallmatrix}\tau_1 & 0 \\ 0 & \tau_2 \end{smallmatrix} \right)$.
Consequently we have $\Psi(\tau) \cong \CC^2 / (\ZZ^2 \oplus \Pi \ZZ^2)$, where
$$\Pi=B^{-1}D_\tau(B^t)^{-1} = \frac{1}{D} \begin{pmatrix} \tau_1(\gamma^\sigma)^2 + \tau_2 \gamma^2  & - \tau_1 \gamma^\sigma - \tau_2 \gamma \\ - \tau_1 \gamma^\sigma - \tau_2 \gamma & \tau_1 + \tau_2 \end{pmatrix}.$$
Let us now consider $\varphi$. As $\rho(\gamma)$ has two distinct eigenvalues, namely $\gamma$ and $\gamma^\sigma$, we may choose a basis of $\CC^2$ such that the action of $\rho(\gamma)$ is given by $\left( \begin{smallmatrix} \gamma & 0 \\ 0 & \gamma^\sigma \end{smallmatrix} \right)$. One easily calculates that the action of $\rho(\gamma)$ on the lattice $\OD \oplus \OD^\vee$ with respect to the basis $(a_1,a_2,b_1,b_2)$ is given by 
$$\rho(\gamma) (a_1,a_2) = C \begin{pmatrix} a_1 \\ a_2 \end{pmatrix}$$
$$\rho(\gamma) (b_1,b_2) = C^t \begin{pmatrix} b_1 \\ b_2 \end{pmatrix},$$
where $C= \left( \begin{smallmatrix} 0 & 1 \\ (D-D^2)/4 & D \end{smallmatrix} \right)$.\\ 
Since the real multiplication commutes with $\varphi$, we have
$$\varphi(\rho(\gamma) a_i) = \begin{pmatrix} \gamma & 0 \\ 0 & \gamma^\sigma \end{pmatrix} \varphi(a_i)$$
$$\varphi(\rho(\gamma) b_i) = \begin{pmatrix} \gamma & 0 \\ 0 & \gamma^\sigma \end{pmatrix} \varphi(b_i).$$
Setting $\varphi(a_1)= \left( \begin{smallmatrix} r \\ s \end{smallmatrix} \right)$ and $\varphi(b_2) = \left( \begin{smallmatrix} v/\sqrt{D} \\ - w/\sqrt{D}\end{smallmatrix} \right)$ with $r,s,v,w \in \CC$ we get that $\varphi: \OD \oplus \OD^\vee \to \CC^2$ is given by the matrix
$$\varphi=(\begin{pmatrix} r & 0 \\ 0 & s \end{pmatrix}B, \begin{pmatrix} v & 0 \\ 0 & w \end{pmatrix}(B^t)^{-1}) .$$
Consequently $\CC^2 / \varphi(\OD \oplus \OD^\vee) \cong \CC^2 / (\ZZ \oplus \Pi' \ZZ)$, where $$\Pi'= \begin{pmatrix} \tau_1(\gamma^\sigma)^2 + \tau_2 \gamma^2  & - \tau_1 \gamma^\sigma - \tau_2 \gamma \\ - \tau_1 \gamma^\sigma - \tau_2 \gamma & \tau_1 + \tau_2 \end{pmatrix}.$$
with $\tau_1,\tau_2 \in \CC$. Since $\CC^2 / (\ZZ \oplus \Pi' \ZZ)$ is a principal polarized Abelian variety we have in fact that $\Pi' \in \HH_2$ and hence $\tau_1,\tau_2 \in \HH$. This shows that $\Psi$ is indeed surjective. So $\Psi$ is a bijection and we have finally proven the assertion of the theorem.
\end{proof}

\newpage

\subsection{Elements of Non-cocompact Cofinite Fuchsian Groups} \label{sec_check_elements}
\index{Fuchsian group} 
Let $\Gamma \subset \PSL_2(\RR)$ be an arbitrary Fuchsian group. It is in general a very hard problem to decide whether a given element $N \in \PSL_2(\RR)$ is in $\Gamma$ or not. For example one would like to check if a certain matrix lies inside a Veech group $\SL(X,\omega)$ or not. In fact, one is often even interested in writing $N$ as a word in given generators if possible. These problems are very much reminiscent of the famous - generally unsolvable - word problem (see e.g. \cite{Bau93}).
For non-cocompact, cofinite Fuchsian groups we are able to give an algorithm which does solve both of the above problems - at least if the Dirichlet fundamental domain of the Fuchsian group is known sufficiently well.\\[11pt]
Let $\Gamma$ until the end of Appendix~\ref{sec_check_elements} be a non-cocompact cofinite Fuchsian groups. Then $\Gamma$ contains at least one parabolic element (see \cite{Kat92}, Corollary~4.2.7). Before we explain the algorithm, let us collect a bunch of well-known facts. The first one describes the Dirichlet fundamental domain of a Fuchsian group in Euclidean metric (compare \cite{Kat92}, Chapter~3).

\begin{lem} The Dirichlet fundamental domain of $\Gamma$ can be described using Euclidean metric as follows:
$$ D_p(\Gamma) = \left\{ z \in \HH \mid \left|  \frac{T(z)-p}{z-p} \right| \geq \frac{1}{|cz+d|} \ \forall T=\begin{pmatrix} a & b \\ c & d \end{pmatrix} \in \Gamma \right\}$$
\end{lem}

A long but straightforward calculation then yields:

\begin{cor} \label{cor_Dirichlet_infinity_domain} For $\lim_{k \to \infty} D_{ki}(\Gamma)=: D_{\infty}(\Gamma)$ all bounding geodesics of the Dirichlet fundamental domain which are not vertical lines are given by
$$ \left| z - \left( - \frac{d}{c} \right) \right| = \frac{1}{|c|}$$
for some $\left( \begin{smallmatrix} a & b \\ c & d \end{smallmatrix}\right) \in \Gamma$. The interior of this limit fundamental domain is
$$ \inter(D_{\infty}(\Gamma)) = \left\{ z \in \HH \mid  \left| z - \left( - \frac{d}{c} \right) \right| > \frac{1}{|c|} \right\}.$$ \end{cor}

Furthermore we remind the reader of the following fact (see \cite{FB06}, Hilfssatz~V.7.1):

\begin{lem} If $M=\left( \begin{smallmatrix} a & b \\ c & d \end{smallmatrix} \right) \in SL_2(\mathbb{R})$, then for all $z \in \mathbb{H}$
$$ \Imag (Mz) = \frac{\Imag (z)}{|cz+d|^2}$$
holds.
\end{lem}

The algorithm is based on the following lemma:  

\begin{lem} \label{lem_app_orbits} Let $\Gamma$ be a Fuchsian group which is a lattice in $\PSL_2(\mathbb{R}$). Then we have:
\begin{itemize}
\item[(i)] For all $z \in \mathbb{H}$ there are only finitely many $a_1,...,a_n \in \mathbb{R}$ with $a_i \geq \Imag(z)$ which appear as imaginary parts in the elements of $\Gamma z$.
\item[(ii)] Every orbit $\Gamma z$ contains points of maximal imaginary part. 
\item[(iii)] The points $z \in \HH$ which have maximal imaginary part in their $\Gamma$-orbit are those with
$$|cz+d| \geq 1 \ \forall c,d \ \rm{with} \ \ \ \begin{pmatrix} * & * \\ c & d \end{pmatrix} \in \Gamma.$$
\end{itemize}
\end{lem} 

\begin{proof} Let $z=x+iy$. The last lemma gives $$\Imag(Mz) \geq \Imag(z) \ \textrm{if and only if} \ |cz+d| \leq 1.$$ This proves $(iii)$. The inequality $|cz+d| \leq 1$ has only finitely many solutions since $\Gamma$ is discrete. This yields $(i)$ and $(ii)$. \end{proof}

In \cite{Koh06} one finds an algorithm which decomposes each element $\SL_2(\ZZ)$ into a word in the standard generators of $\SL_2(\ZZ)$. Simultaneously it gives a possibility to decide whether a given matrix lies in $\SL_2(\ZZ)$, although this is of course trivial. We now imitate this algorithm.\\[11pt]
Recall that $\Gamma$ is a Fuchsian group containing at least one parabolic element. By conjugation we may assume that $\Gamma$ contains $A=\left( \begin{smallmatrix} 1 & s \\ 0 & 1 \end{smallmatrix} \right)$ with $s \in \mathbb{R}$. Furthermore we may assume that $|s|$ is minimal, i.e. there is no matrix of this form in $\Gamma$ with an upper right entry of smaller absolute value. The Dirichlet fundamental domain $D_{\infty}(\Gamma)$ then looks like:
\begin{center}
\psset{xunit=1cm,yunit=1cm,runit=1cm}
\begin{pspicture}(-0.5,-0.5)(8,4) 
\psline[linewidth=0.5pt]{->}(0,0)(8,0)
\psline[linewidth=0.5pt]{->}(4,-0.1)(4,4)
\psline[linewidth=0.5pt]{-}(7,0)(7,4)
\psline[linewidth=0.5pt]{-}(1,0)(1,4)
\psarc[linewidth=0.5pt](6.5,0){0.5cm}{0}{180}
\psarc[linewidth=0.5pt](5.7,0){0.3cm}{0}{180}
\psarc[linewidth=0.5pt](4.6,0){0.8cm}{0}{180}
\psarc[linewidth=0.5pt](3.5,0){0.3cm}{0}{180}
\psarc[linewidth=0.5pt](3.2,0){0.4cm}{0}{180}
\psarc[linewidth=0.5pt](2.6,0){0.3cm}{0}{180}
\psarc[linewidth=0.5pt](2.2,0){0.5cm}{0}{180}
\psarc[linewidth=0.5pt](1.35,0){0.35cm}{0}{180}

\end{pspicture}\\
Figure A.1. A typical fundamental domain.
\end{center}
The left and right boundary of $D_{\infty}(\Gamma)$ are the geodesics joining $-s/2$ (respectively $s/2$) and $\infty$. The other boundary geodesics fill the gap between $-s/2$ and $s/2$. 
Now let $X \in \PSL_2(\RR)$. Let us try to write $X$ as a word in the generators of $\Gamma$. Choose an arbitrary point $z_0 \in \inter(D_{\infty}(\Gamma))$ and\\[11pt]
(1) look at $y_0 = X z_0$ and apply $k$-times (with $k \in \ZZ$) the matrix $A$ until $|\Real(A^k y_0)| \leq \frac{s}{2}$.\\[11pt]
(2) If $y_1 := A^k y_0\in \inter(D_{\infty}(\Gamma))$, then check if $y_1=z_0$. If $y_1 \neq z_0$ then $X \notin \Gamma$ since $\mathcal{F}$ is a fundamental domain.\\[11pt]
(3) If $y_1 = z_0$, then check if $X = A^{-k}$ in $\PSL_2(\mathbb{R}$). In either case we are finished.\\[11pt]
(4) If $y_1 \notin \inter(D_{\infty}(\Gamma))$, then $y_1$ lies beneath (at least) one of the boundary geodesics and above the $x$-axis. Then find a matrix $B=\left( \begin{smallmatrix} a & b \\ c & d \end{smallmatrix} \right) \in \Gamma$ which generates this geodesic in the sense of Corollary~\ref{cor_Dirichlet_infinity_domain}. And set $y_2=By_1$.\\[11pt]
Since $B$ generates the geodesic
$$ \left|y_1 - \left(- \frac{d}{c} \right) \right| < \frac{1}{|c|}$$
or since equivalently $|cy_1 +d| < 1$ holds, we must have that $\Imag(y_2) > \Imag (y_1)$. If $y_2 \in \inter(D_{\infty}(\Gamma))$, then we are again finished (after checking if $y_2 = z_0$ and if $A^{-k}B^ {-1}=X$). Otherwise we continue with step (I) with $y_2$ instead of $y_0$.\\[11pt] 
As the imaginary part grows with every iteration, the algorithms really determines by Lemma~\ref{lem_app_orbits} (iii). We have thus shown:

\begin{thm} Let $\Gamma$ be a non-cocompact, cofinite Fuchsian group. If $A_1,...,A_n$ generate $\Gamma$ in the sense of Corollary~\ref{cor_Dirichlet_infinity_domain} then the algorithm above decides whether an arbitrary element $X \in \PSL_2(\RR)$ lies in $\Gamma$. If so, the algorithm returns $X$ as a word in the $A_i$. \end{thm}

\newpage

\subsection{Checking Commensurability} \label{sec_check_commens}
\index{commensurator} \index{Fuchsian group}

If we knew that in genus~$2$ all Veech groups of Teichmüller curves were maximal and if we wanted to check if a curve $\HH / \Gamma$ is a twisted Teichmüller curve then Proposition~\ref{prop_twisted_implies_commensurable} and Lemma~\ref{lem_commensurable_implies_subgroup} make it necessary to find a sufficiently general algorithm which decides if a Fuchsian group $H \subset \SL_2(\mathbb{R})$ is conjugated to a subgroup of another Fuchsian group $G \subset \SL_2(\mathbb{R})$.\\[11pt]
Since the signature of a Fuchsian group does not change under conjugation it is necessary that $G$ contains a subgroup of the same signature as $H$. This is of course a much stronger criterion than just looking at the Euler characteristics. In \cite{Sin70}, D. Singerman gave a criterion which decides whether a given group $G$ contains a subgroup of the same signature as $H$.

\begin{defi} If $\Gamma$ is of signature $(g;m_1,...,m_r;s;t)$  then the integers $m_1,...,m_r$ are called the \textbf{periods} of $\Gamma$. \index{Fuchsian group!signature} \end{defi}

Let $n$ now be a period of $\Gamma' \subset \Gamma$. Then $n$ is by definition the order of an elliptic element $y \in \Gamma'$ and $y$ is a power of a conjugate of one of the elliptic generators $x_j \in \Gamma$ with order $m_j$. We shall then say that $n$ has been induced by $m_j$. Of course, this implies that $n|m_j$. 

\begin{thm} \label{thm_singerman_1970} \textbf{(Singerman, 1970)} Let $\Gamma$ be a Fuchsian group of signature $(g;m_1,...,m_r;s;t)$. Then $\Gamma$ contains a subgroup $\Gamma_1$ of index $N$ with signature $(g'n_{11},n_{12},...,n_{1p_1},...,n_{r1}, ...,n_{rp_r};s';t')$ if and only if
\begin{itemize}
\item[(I)] There exists a finite permutation group $G$ transitive on $N$ points and an epimorphism $\theta : \Gamma \to G$ satisfying the following properties
\begin{itemize}
\item[(i)] The permutation $\theta(x_j)$ has precisely $p_j$ cycles of lengths less than $m_j$, the lengths of these cycles being $m_j/n_{j1},...,m_j/n_{jp_j}$,
\item[(ii)] If we denote the number of cycles in the permutation $\theta(\gamma)$ by $\delta(\gamma)$ then
$$s'=\sum_{k=1}^s \delta(p_k), \ \ \ t'=\sum_{l=1}^t \delta(h_l)$$
where $p_k$ are the parabolic and $h_l$ are the hyperbolic generators of $\Gamma$.
\end{itemize}
\item[(II)] $vol(\Gamma_1)/vol(\Gamma)=N$.
\end{itemize}
\end{thm}

Nevertheless, there may exist non-conjugated groups of the same signature (see e.g. \cite{SW00}). So Theorem~\ref{thm_singerman_1970} only gives necessary conditions for $H$ and $G$. Hence we really need an algorithm to decide, whether there exists an $M \in SL_2(\mathbb{R})$ with $MHM^{-1} \leq G$. For the rest of the section we will as two assumptions suppose that $G$ is a cofinite Fuchsian group and $H$ (and thus necessarily also $G$) contains elliptic elements. Note that both assumptions are fulfilled for the Veech groups of Teichmüller curves of genus~$2$ and $4$. The idea of the algorithm is that, if $MHM^{-1} \leq G$, then there must exist a fundamental domain $\mathcal{F}_H$ of $H$ and a fundamental domain $\mathcal{F}_G$ of $G$ such that the tessellation of $\mathbb{H}$ by copies of $\mathcal{F}_G$ is a refinement of the tessellation of $\mathbb{H}$ by copies of $M\mathcal{F}_HM^{-1}$. The algorithm now works in the following way:
\begin{itemize}
\item[(1)] Choose an arbitrary elliptic fixed point $x$ of $H$ of period $n_x$. 
\item[(2)] For each conjugacy class of elliptic fixed points of period $m_i$ with $n_x|m_i$ in $G$, choose an arbitrary elliptic fixed point $z_i$ and calculate all matrices $M_{x,z_i}$ which send $x$ to $z_i$.
\end{itemize}
Of course, elliptic fixed points of $H$ must be sent by $M_{x,z_i}$ to elliptic fixed points of $G$ of suitable period. Thus the upper matrices $M_{x,z_i}$ are the only possible candidates for $M$. Note that the choices of the elliptic fixed points $z_i$ and $x$ are really free since choosing another representative than $x$ means multiplying $M_{x,z_i}$ by a matrix in $H$ from the right and choosing another representative than $z_i$ means multiplying $M_{x,z_i}$ by a matrix in $G$ from the left.
\begin{itemize}
\item[(3)] Choose another arbitrary elliptic fixed point $y$ of $H$ of period $n_y$ and calculate $d(x,y)$.
\end{itemize}
For computational reasons it is convenient to choose $y$ with minimal distance $d(x,y)$.
\begin{itemize}
\item[(4)] Calculate parts of the tessellation of $\mathbb{H}$ by copies of $\mathcal{F}_G$ until $B_{d(x,y)}(z_i)$ is completely covered.
\end{itemize}
Note that step (4) can be performed in finite time since $G$ is a cofinite Fuchsian group. Since $M_{x,z_i}$ is an isometry the next step of the algorithm is naturally given by:
 	\begin{itemize}
\item[(5)] For all elliptic fixed points $z_k$ of $G$ of period $m_k$ with $z_k \in \delta B_{d(x,y)}(z_i)$ and $n_y|m_k$ calculate if there exists a matrix $M_{x,z_i}$ with $M_{x,z_i}y=z_k$. If so this matrix is unique and called $M_{x,y,z_i,z_k}$.
\end{itemize}
This yields a complete list of possible candidates $M_{x,y,z_i,z_k}$. For each of these matrices now perform step (6).
\begin{itemize}
\item[(6)] Set $M:=M_{x,y,z_i,z_k}^{-1}$. Then for all generators $u_j$ of $H$ decide whether $Mu_jM^{-1} \in G$. If so (for all $u_j$), then $MHM^{-1} \leq G$. 
\end{itemize}
Step (6) can for example be done by the algorithm given in section \ref{sec_check_elements}.\\[11pt]
In the following example we now show how the algorithm explicitly works:

\begin{exa} Let $H=\left\langle S, \left( \begin{smallmatrix} 1 + 2/9 \sqrt{8} & 2/9 \\ -16/9 & 1 - 2/9 \sqrt{8} \end{smallmatrix} \right) \right\rangle$ and let $G=\SL_2(\ZZ)$. Both groups $H$ and $G$ have only one conjugacy class of elliptic fixed points of order $2$. Thus we may choose $x=z_i=i$ and it is well-known that $M_{x,z_i} = \left( \begin{smallmatrix} c & d \\ -d & c \end{smallmatrix} \right)$ with $c^2+d^2=1$. Another elliptic fixed point of $H$ is 
$$\begin{pmatrix} 1 + 2/9 \sqrt{8} & 2/9 \\ -16/9 & 1 - 2/9 \sqrt{8} \end{pmatrix}i =  \underbrace{\frac{1}{1553} \left( -702 - 220 \sqrt{8} + i (369 + 36 \sqrt{8})\right)}_{=:y}.$$ $d(x,y)=d(i,i+2)$ and $i+2$ is an elliptic fixed point of $G$. Solving the equation $M_{x,z_i} y = i+2$ yields  $$M:=M_{i,y,i,i+2}^{-1}= \begin{pmatrix} 1/3 & -1/3\sqrt{8} \\ 1/3\sqrt{8} & 1/3 \end{pmatrix}.$$ We have  $$M\begin{pmatrix} 0 & -1 \\ 1 & 0 \end{pmatrix}M^{-1}=\begin{pmatrix} 0 & -1 \\ 1 & 0 \end{pmatrix}$$ and $$M \begin{pmatrix} 1 + 2/9 \sqrt{8} & 2/9 \\ -16/9 & 1 - 2/9 \sqrt{8} \end{pmatrix} M^{-1} = \begin{pmatrix} 1 & 2 \\ 0 & 1 \end{pmatrix}$$ and therefore $MHM^{-1} \leq G$. \end{exa}

\newpage

\section{Tables} \label{appendix_tables}

Tables with some numerical data for the volume of diagonal twisted Teichmüller curves in the cases $D=13$ and $D=17$ can be found on the following pages. As we have seen in Chapter~\ref{chapter_calculations} diagonal twisted Teichmüller curves carry the  information about the volume of almost all twisted Teichmüller curves in the case $h_D=1$. Therefore, diagonal twisted Teichmüller curves are especially interesting from a numerical point of view. More precisely, the tables on the next two pages contain the indexes $[\SL(L_D) : \SL(L_D) \cap \Gamma^{D}_0(m) \cap \Gamma^{D,0}(n)]$ for many $m, n \in \OD$. From this data one can calculate the volume of the corresponding twisted Teichmüller curves (Theorem~\ref{thm_summarize_euler_II}). All calculations were done with PARI/GP.

\newpage 

\begin{landscape}

\begin{tabular}{|p{0.95cm}||p{0.95cm}|p{0.95cm}|p{0.95cm}|p{0.95cm}|p{0.95cm}|p{0.95cm}|p{0.95cm}|p{0.95cm}|p{0.95cm}|p{0.95cm}|p{0.95cm}|p{0.95cm}|p{0.95cm}|p{0.95cm}|p{0.95cm}|p{0.95cm}|}
\hline
$n \; \backslash \; m$ & 1 & 2 & 3 & 4 & \textit{w} & \textit{w}+1 & \textit{w}+2 & \textit{w}+3 & 2\textit{w} & 2\textit{w}+1 & 2\textit{w}+2 & 2\textit{w}+3 & 3\textit{w} & 3\textit{w}+1 & 3\textit{w}+2 & 3\textit{w}+3 \\
\hline \hline
1 & \textbf{1} & \textbf{5} & \textbf{16} & \textbf{20} & \textbf{4} & \textbf{1} & \textbf{4} & \textbf{12} & \textbf{20} & \textbf{12} & \textbf{5} & \textbf{4} & \textbf{48} & \textbf{24} & \textbf{18} & \textbf{16}\\
\hline
2 & \textbf{5} & - & \textbf{80} & - & \textbf{20} & \textbf{5} & \textbf{20} & \textbf{60} & - & \textbf{60} & - & \textbf{20} & \textbf{240} & \textbf{120} & \textbf{90} & \textbf{80}\\
\hline
3 & \textbf{16} & \textbf{80} & - & \textbf{320} & - & \textbf{16} & - & - & - & - & \textbf{80} & - & - & \textbf{384} &  \textbf{288} & -\\
\hline
4 & 	\textbf{20} &	- &	\textbf{320} &	- &	\textbf{80} &	\textbf{20} &	\textbf{80} &	\textbf{240} &	- &	\textbf{240} &	- &	\textbf{80} &	\textbf{960} &	\textbf{480} &	\textbf{360} &	\textbf{320}  \\
\hline
\textit{w} & 	\textbf{4} &	\textbf{20} &	- &	\textbf{80} &	- &	\textbf{4} &	\textbf{16} & - &	- &	\textbf{48} &	\textbf{20} &	- &	- &	\textbf{96} &	\textbf{72} &	- \\
  
\hline
\textit{w}+1 & 	\textbf{1} & 	\textbf{5} & 	\textbf{16} & \textbf{20} & \textbf{4} & 	\textbf{1} & 	\textbf{4}  & 	\textbf{12} &	\textbf{20} &	\textbf{12} &	\textbf{5} &	\textbf{4} &	\textbf{48} &	\textbf{24} &	\textbf{18} &	\textbf{16} \\
\hline
\textit{w}+2 & 	\textbf{4} &	\textbf{20} &	- &	\textbf{80} &	\textbf{16} &	\textbf{4} &	- &	\textbf{48} &	\textbf{80} &	- &	\textbf{20} &	\textbf{16} &	- &	\textbf{96} &	\textbf{72} &	- \\
\hline
\textit{w}+3 &	\textbf{12} &	\textbf{60} &	- &	\textbf{240} &	- &	\textbf{12} &	\textbf{48} &	- &	- &	\textbf{144} &	\textbf{60} &	- &	- &	\textbf{288} &	\textbf{216} &	- \\
\hline
2\textit{w} & 	\textbf{20} &	- &	- &	- &	- &	\textbf{20} &	\textbf{80} &	- &	- &	\textbf{240} & - &	- &	- &	\textbf{480} &	\textbf{360} &	- \\
\hline
2\textit{w}+1 &	\textbf{12} &	\textbf{60} &	- &	\textbf{240} &	\textbf{48} &	\textbf{12} &	- &	\textbf{144} &	\textbf{240} &	- &	\textbf{60} &	\textbf{48} &	\textbf{48} &	\textbf{288} &	\textbf{216} &	- \\
\hline
2\textit{w}+2 & \textbf{5} &	- &	\textbf{80} &	- &	\textbf{20} &	\textbf{5} &	\textbf{20} &	\textbf{60} &	- &	\textbf{60} &	- &	\textbf{20} &	\textbf{240} &	\textbf{120} &	\textbf{90} &	\textbf{80} \\
\hline
2\textit{w}+3 &	\textbf{4} &	\textbf{20} & - &	\textbf{80} &	- &	\textbf{4} &	\textbf{16} &	- &	- &	\textbf{48} &	\textbf{20} &	- &	- &	\textbf{96} &	\textbf{72} &	- \\
\hline 
3\textit{w} &	\textbf{48} &	\textbf{240} &	- &	\textbf{960} & - &	\textbf{48} &	- &	- &	- &	\textbf{48} &	\textbf{240} &	- &	- &	\textbf{1152} &	\textbf{864} &	- \\
\hline
3\textit{w}+1 &	\textbf{24} &	\textbf{120} &	\textbf{384} &	\textbf{480} &	\textbf{96} &	\textbf{24} &	\textbf{96} &	\textbf{288} &	\textbf{480} &	\textbf{288} &	\textbf{120} &	\textbf{96} &	\textbf{1152} &	- &	\textbf{432} &	\textbf{384} \\
\hline
3\textit{w}+2 &	\textbf{18} &	\textbf{90} &	\textbf{288} &	\textbf{360} &	\textbf{72} &	\textbf{18} &	\textbf{72} &	\textbf{216} &	\textbf{360} &	\textbf{216} &	\textbf{90} &	\textbf{72} &	\textbf{864} &	\textbf{432} &	- &	\textbf{288} \\
\hline
3\textit{w}+3 &	\textbf{16} &	\textbf{80} &	- &	\textbf{320} &	- &	\textbf{16} &	- &	- &	- &	- &	\textbf{80} &	- &	- &	\textbf{384} &	\textbf{288} &	- \\
\hline
\end{tabular}

\begin{center} 
\begin{tabular}{ll} Table 1: & Index of $\SL(L_{13}) \cap \Gamma^{13}_0(m) \cap \Gamma^{13,0}(n)$ in $\SL(L_{13})$ if $(m,n)=1$.\\
& It is equal to the index of $\Gamma^{13}_0(m) \cap \Gamma^{13,0}(n)$ in $\SL(\mathcal{O}_{13})$ if $(n,m)=1$.
\end{tabular}
\end{center}
\end{landscape}

\newpage

\begin{landscape}

\begin{tabular}{|p{0.95cm}||p{0.95cm}|p{0.95cm}|p{0.95cm}|p{0.95cm}|p{0.95cm}|p{0.95cm}|p{0.95cm}|p{0.95cm}|p{0.95cm}|p{0.95cm}|p{0.95cm}|p{0.95cm}|p{0.95cm}|p{0.95cm}|p{0.95cm}|p{0.95cm}|}
\hline
$n \; \backslash \; m$ & 1 & 2 & 3 & 4 & \textit{w} & \textit{w}+1 & \textit{w}+2 & \textit{w}+3 & 2\textit{w} & 2\textit{w}+1 & 2\textit{w}+2 & 2\textit{w}+3 & 3\textit{w} & 3\textit{w}+1 & 3\textit{w}+2 & 3\textit{w}+3 \\
\hline \hline
1 &	\textbf{1} &	6 &	\textbf{10} &	24 &	4 &	\textbf{3} &	2 &	\textbf{12} &	24 &	\textbf{14} &	12 &	\textbf{1} &	40 &	\textbf{48} &	28 &	\textbf{30} \\
\hline
2 &	6 &	- &	60 &	- &	- &	- &	- &	- &	- &	84 &	- &	6 &	- &	- &	- &	- \\
\hline
3 &	\textbf{10} &	60 &	- &	240 &	40 &	\textbf{30} &	20 &	\textbf{120} &	240 &	\textbf{140} &	120 &	\textbf{10} &	- &	\textbf{480} &	280 &	- \\
\hline
4 &	24 &	- &	240 &	- &	- &	- &	- &	- &	- &	336 & - &	24 &	- &	- &	- &	- \\
\hline
\textit{w}&	4 &	- &	40 &	- &	- &	12 &	- &	48 &	- &	56 &	- &	4 &	- &	192 &	- &	120 \\
\hline
\textit{w}+1 &	\textbf{3} &	- &	\textbf{30} &	- &	12 &	- &	6 &	- & -&	\textbf{42} &	- &	\textbf{3} &	120 &	- &	84 &	- \\
\hline
\textit{w}+2 &	2 &	- &	20 &	- &	- &	6 &	- &	24 &	- &	28 &	- &	2 &	- &	96 &	- &	60 \\
\hline
\textit{w}+3 &	\textbf{12} &	- &	\textbf{120} &	- &	48 &	- &	24 &	- &	- &	\textbf{168} &	- &	\textbf{12} &	480 &	- &	336 &	- \\
\hline
2\textit{w} &	24 &	- &	240 &	- &	- &	- &	- &	- &	- &	336 &	- &	24 &	- &	- &	- &	- \\
\hline
2\textit{w}+1 &	\textbf{14} &	84 &	\textbf{140} &	336 &	56 &	\textbf{42} &	28 &	\textbf{168} & 336 &	- &	168 &	\textbf{14} &	560 &	\textbf{672} &	392 &	\textbf{420} \\
\hline
2\textit{w}+2 &	12 &	- &	120 &	- &	- &	- &	- &	- &	- &	168 &	- &	12 &	- &	- &	- &	- \\
\hline
2\textit{w}+3 &	\textbf{1} &	6 &	\textbf{10} &	24 &	4 &	\textbf{3} &	2 &	\textbf{12} &	24 &	\textbf{14} &	12 &	\textbf{1} &	40 &	\textbf{48} &	28 &	\textbf{30} \\
\hline 
3\textit{w} &	40 &	- &	- &	- &	- &	120 &	- &	480 &	- &	560 &	- &	40 &	- &	1920 & - &	- \\
\hline
3\textit{w}+1 &	\textbf{48} &	- &	\textbf{480} &	- &	192 &	- &	96 &	- &	- &	\textbf{672} &	- &	\textbf{48} &	1920 &	- &	1344 &	- \\
\hline
3\textit{w}+2 &	28 &	- &	280 &	- &	- &	84 &	- &	336 &	- &	392 &	- &	28 &	- &	1344 &	- &	840 \\
\hline
3\textit{w}+3 &	\textbf{30} &	- &	- &	- &	120 &	- &	60 &	- &	- &	\textbf{420} &	- &	\textbf{30} &	- &	- &	840 &	- \\
\hline\end{tabular}
\begin{center} 
\begin{tabular}{ll} Table 2: & Index of $\SL(L_{17}^1) \cap \Gamma^{17}_0(m) \cap \Gamma^{17,0}(n)$ in $\SL(L_{17}^1)$.\\
& The index is either the same (bold face) or $2/3$ as large (normal face)\\ & as the index of  $\Gamma^{17}_0(m) \cap \Gamma^{17,0}(n)$ in $\SL(\mathcal{O}_{17})$.
\end{tabular}
\end{center}
\end{landscape}

\newpage
\section*{List of Symbols}
\addcontentsline{toc}{section}{List of Symbols}
\begin{tabular}{ll}
$K$ & Real quadratic number field $\QQ(\sqrt{D})$, \pageref{glo_K_real_quadratic}\\
$\textrm{N}(\cdot)$ & Norm of an element in $K$, \pageref{glo_N}\\
$\mathcal{N}(\cdot)$ & Absolute value of the norm of an \\&element in $K$, \pageref{glo_N}\\
$\tr(\cdot)$ & Trace of an element in $K$, \pageref{glo_tr}\\
$\OD$ & Ring of integers in $K$, \pageref{glo_OD}\\
	& resp. real quadratic order, \pageref{glo_ODgen}\\
$w$ & Second basis element of $\OD$, \pageref{glo_w}, \pageref{glo_ODgen}\\
$\OD^*$ & Units in $\OD$, \pageref{glo_OD*}\\
$\epsilon$ & Fundamental unit, \pageref{glo_epsilon}\\
$\OD^\vee$ & Inverse different, \pageref{glo_ODvee}\\
$h_D$ & Class number of $\OD$, \pageref{glo_hK}\\
$(a)$ & Fractional ideal generated by $a \in K^*$, \pageref{glo_(a)}\\
$\mathfrak{p}$ & Prime ideal, \pageref{glo_matp}\\
$\left( \frac{D}{p} \right)$ & Legendre symbol, \pageref{glo_Legendre}\\
$\HH$ & Complex upper half plane, \pageref{glo_H}\\
$\mathbb{P}^1(\CC)$ & Projective plane, \pageref{glo_P1}\\
$\overline{\HH}$ & Closure of $\HH$, \pageref{glo_Hbar}\\
$\mu(\cdot)$ & Hyperbolic area, \pageref{glo_hyp}\\
$\mathbb{D}$ & Unit disc, \pageref{glo_D}\\
$\SL_2(\cdot)$ & Special linear group, \pageref{glo_SL2}\\
$\PSL_2(\RR)$ & $\SL_2(\RR)/{\pm \Id}$, \pageref{glo_PSL2}\\
$\Gamma$ & Fuchsian group, \pageref{glo_Gamma}\\
$\tr(\cdot)$ & Trace of a matrix, \pageref{glo_trace}\\
$\mathcal{F}$ & Fundamental domain, \pageref{glo_F}\\
$\Gamma^M$ & $M^{-1}\Gamma M$, \pageref{glo_conj}\\
$(g;m_1,...m_r;s)$ & Signature of a Fuchsian group, \pageref{glo_signature}\\
$\chi(\cdot)$ & Euler characteristic, \pageref{glo_eul}\\
$[A:B]$ & Index of an subgroup $B$ in $A$, \pageref{glo_Index}\\
$\Comm_G(A)$ & Commensurator of $A$ in $G$, \pageref{glo_Comm}\\
$\Gamma^D(n)$ & Principal congruence subgroup, \pageref{glo_GammaD}\\
$\Gamma^D_0(n)$ & Hecke congruence subgroup, \pageref{glo_GammaD0}\\
$\Gamma^{D,0}(n)$ & Hecke congruence subgroup, \pageref{glo_GammaD0}\\
$\Gamma^{D}(m,n)$ & $\Gamma^D_0(m) \cap \Gamma^{D,0}(n)$, \pageref{glo_GammaD0}\\
$(X,\omega)$ & Flat surface, \pageref{glo_Xw}\\
$\SL(X,\omega)$ & Veech group of a flat surface, \pageref{glo_SLXw}\\
$\mathcal{T}(S), \mathcal{T}_g$ & Teichmüller space, \pageref{glo_TS}\\
$\Mod(S)$ & Mapping class group, \pageref{glo_ModS}\\
$\Mg$ & Moduli space of compact Riemann surfaces\\ & of genus g, \pageref{glo_Mg}\\
$\overline{\Mg}$ & Deligne-Mumford compactification of $\Mg$, \pageref{glo_Mgbar}\\
\end{tabular}

\newpage

\begin{tabular}{ll}
$\Omega(X)$ & Nonzero holomorphic $1-$forms on $X$, \pageref{glo_Omx}\\
$\Omega\Mg$ & Set of flat surfaces, \pageref{glo_OmMg}\\
$\Omega\Mg(k_1,...,k_n)$ & Stratum of $\Omega\Mg$, \pageref{glo_strat}\\
$H^i(X,S)$ & $i$-th Cohomology of $X$ with coefficients in $S$, \pageref{glo_coh}\\
$H_i(X,S)$ & $i$-th Homology of $X$ with coefficients in $S$, \pageref{glo_hom}\\
$\epsilon(X,\omega)$ & Spin invariant, \pageref{glo_spin}\\
$\Omega\Mg^{\textrm{hyp}}(\cdot), \Omega\Mg^{\textrm{odd}}(\cdot),$ & Connected components of $\Omega\Mg(\cdot)$, \pageref{thm_components_of_strata}\\
$\Omega\Mg^{\textrm{even}}(\cdot)$\\
$\Lambda$ & Lattice, \pageref{glo_Lambda}\\
$c_1(L)$ & Chern class of the line bundle $L$, \pageref{glo_c1}\\
$\Pi$ & Period matrix, \pageref{glo_Pi}\\
$\Jac(X)$ & Jacobian of $X$, \pageref{glo_Jac}\\
$\End(X)$ & Endomorphisms of $X$, \pageref{glo_End}\\
$\rho$ & Real multiplication, \pageref{glo_End}\\
$\HH_g$ & Siegel upper half space, \pageref{glo_Hg}\\
$\mathcal{A}_g^D$ & Moduli space of Abelian varieties\\ & of polarization $D$, \pageref{glo_Agd}\\
$\mathcal{A}_g$ & Moduli space of principally polarized	\\ & Abelian varieties, \pageref{glo_Ag}\\
$\SL(\OD \oplus \mathfrak{a})$ & Hilbert modular group with respect to $\mathfrak{a}$, \pageref{glo_HMG}\\
$X_D$ & Hilbert modular surface, \pageref{glo_XD}\\
$G(M,V)$ & Type of a cusp, \pageref{glo_type}\\
$\zeta_K$ & Zeta-function of $K$, \pageref{glo_zeta}\\
$\sigma_1(a)$ & Sum of divisors of $a \in \ZZ$, \pageref{glo_sigma}\\
$\GL_2^+(K)$ & Group of matrices with totally \\ &positive determinant, \pageref{glo_gl2+}\\ $M^{\sigma}$ & Galois conjugate of a Matrix in $\GL_2^+(K)$, \pageref{glo_Msigma}\\
$T_N$ & Curve on $X_D$, \pageref{glo_TN}\\
$P_D$ & Reducible locus in $X_D$, \pageref{glo_PD}\\
$k_W(\cdot,\cdot)$ & Kobayashi metric, \pageref{glo_KW}\\
$g(X)$ & Genus of Riemann surface $X$, \pageref{glo_g(X)}\\
$O_2(\RR)$ & Orthogonal group, \pageref{glo_O2}\\
$P(a,b)$ & L-shaped polygon, \pageref{glo_billiard}\\
$\SL(L_D^1)$ & Veech group of an odd spin L-shaped \\ & polygon of discriminant $D$, \pageref{glo_SLLD1}\\
$C_{L,D}^\epsilon$ & Teichmüller curve in $\mathcal{M}_2$\\ & with discriminant $D$ and spin $\epsilon$, \pageref{glo_CLD}\\
$\SL(L_D^0)$ & Veech group of an even spin L-shaped\\ & polygon of discriminant $D$, \pageref{glo_SLLD0}\\
$\SL(L_D)$ & Veech group of an arbitrary  L-shaped\\ & polygon of discriminant $D$, \pageref{glo_SLLD}\\
$\Phi(z)=(z,\varphi(z))$ & Graph of the Teichmüller curve, \pageref{glo_Phi}\\
\end{tabular}

\newpage

\begin{tabular}{ll}
$T$ & $\begin{pmatrix} 1 & w \\ 0 & 1 \end{pmatrix}$ or $\begin{pmatrix} 1 & w-1 \\ 0 & 1 \end{pmatrix}$ or $\begin{pmatrix} 1 & w+1 \\ 0 & 1 \end{pmatrix}$, \pageref{glo_T}, \pageref{glo_T2}, \pageref{glo_T3}\\
$S$ & $\begin{pmatrix} 0 & -1 \\ 1 & 0 \end{pmatrix}$, \pageref{glo_S}\\
$Z$ & $\begin{pmatrix} 1 & 0 \\ w & 1 \end{pmatrix}$ or $\begin{pmatrix} 1 & 0 \\ w+1 & 1 \end{pmatrix}$, \pageref{glo_Z}, \pageref{glo_Z2}, \pageref{glo_Z3}\\
$\Stab(\Phi)$ & Stabilizer of the graph of a Teichmüller curve, \pageref{glo_stabphi}\\
$C_M$ & Twisted Teichmüller curve, \pageref{glo_C_M}\\
$\SL_M(L_D)$ & Stabilizer of the graph of the \\&twisted Teichmüller curve, \pageref{glo_SLMLD}\\
$\SL_2(\OD,M)$ & $\SL_2(\mathcal{O_D}) \cap M^{-1}\SL_2(\OD)M$, \pageref{glo_XDM}\\
$X_D(M)$ & Level covering of $X_D$, \pageref{glo_XDM}\\
$\SL(L_D,M)$ & $\SL(L_D) \cap M\SL_2(\OD)M^{-1}$, \pageref{glo_SL(L_D,M)}\\
$C(M)$ & $\HH / \SL(L_D,M)$, \pageref{glo_C(M)}\\
$\SL_M(L_D,M)$ & $\SL_M(L_D) \cap M\SL_2(\OD)M^{-1}$, \pageref{glo_SL_M(L_D,M)}\\
$C_M(M)$ & $\HH / \SL_M(L_D,M)$, \pageref{glo_C_M(M)}\\
$\SL^M(L_D)$ & $\SL_M(L_D)^M$, \pageref{glo_SLMLD2}\\
$\SL^M(L_D,M)$ & $\SL_M(L_D,M)^M$, \pageref{glo_SLMLD2}\\
$C^M(M)$ & $\HH / \SL^M(L_D,M)$, \pageref{glo_SLMLD2}\\
$\mathfrak{p}_2$ & Common prime ideal divisor of $(2)$ and $(w)$, \pageref{glo_P2}\\
$\SL_2^{tr}(K)$ & $\left\{x \in \SL_2(K) \mid \tr(x) \in \OD \right\}$, \pageref{glo_SL2tr}\\
$\eta^+$ & Upper right entry of $T$, \pageref{glo_etas}\\
$\eta^-$ & Lower left entry of $Z$, \pageref{glo_etas}\\
$\eta^*$ & Least common multiple of $\eta^+$ and $\eta^-$, \pageref{glo_etas}\\
$\St_{\SL_2(\RR)^2}(C)$ & Twisting stabilizer of Teichmüller curve $C$ in $\SL_2(\RR)^2$, \pageref{glo_St}\\
$\St_{\GL_2^+(K)}(C)$ & Twisting stabilizer of Teichmüller curve $C$ in $\GL_2^+(K)$, \pageref{glo_St}\\
$\mathbb{P}^1(\cdot)$ & Projective space, \pageref{glo_P12}\\
$\Omega(X)^+$ & Even Abelian differentials, \pageref{glo_Om+-}\\
$\Omega(X)^-$ & Odd Abelian differentials, \pageref{glo_Om+-}\\
$\Prym(X',\rho)$ & Prym variety of $(X',\rho)$, \pageref{glo_Prym}\\
$\Omega E_D^g$ & Prym eigenspace, \pageref{glo_Wei}\\
$\Omega W_D^g$ & Weierstrass locus, \pageref{glo_Wei}\\
$(X,\mu)$ & Probability space, \pageref{glo_Xmu}\\
$L^1(X,\mu)$ & Space of integrable, measurable functions \\
& $f: X \to \RR$ with respect to $\mu$, \pageref{glo_Xmu}\\
$\lambda_i$ & $i$-th Lyapunov exponent, \pageref{glo_lambdanox}\\
$l_i$ & multiplicity of the $i$-th Lyapunov exponent, \pageref{glo_lambdanox}\\
$T_z\mathbb{H}$ & Tangent plane of $\HH$ at $z$, \pageref{glo_tangent_loc}\\
$T^1\HH$ & Tangent space of $\HH$, \pageref{glo_tangent}\\
$a_t$ & Geodesic flow, \pageref{glo_geo_flow}\\
$g_t$ & Teichmüller flow, \pageref{glo_teich_flow}\\
\end{tabular}

\newpage

\begin{tabular}{ll}
$\Omega\Mg^{(1)}$ & Flat surfaces with renormalized area, \pageref{glo_OmMg1}\\
$d\nu_1$ & Masur-Veech-measure, \pageref{glo_MVmeasure}\\
$\mathcal{Q}(d_1,...,d_n)$ & Stratum of meromorphic differentials, \pageref{glo_stratq}\\
\end{tabular}

\newpage

\addcontentsline{toc}{section}{Index}
\printindex

\textsc{Goethe-Universtität Frankfurt, Institut für Mathematik, Robert-Mayer-Str. 6-8, D-60325 Frankfurt (Main)}\\
\textit{E-mail address:} \texttt{weiss@math.uni-frankfurt.de}

\end{document}